\documentclass[10pt,reqno]{amsart}

\usepackage{amssymb}
\usepackage{amsmath}
\usepackage{amsfonts, dsfont}
\usepackage{kotex}
\counterwithin{figure}{section}

\usepackage{amsthm} 
\usepackage[utf8]{inputenc} 
\usepackage{enumitem} 

\usepackage{graphicx}

\usepackage[utf8]{inputenc}
\usepackage{anyfontsize}

\usepackage{tikz}
\usetikzlibrary{
	calc,intersections,arrows.meta,patterns,
	decorations.pathmorphing,
	decorations.fractals
}

\DeclareGraphicsExtensions{.pdf,.png,.jpg}

\usepackage{color}

\theoremstyle{plain}

\newtheorem{thm}{Theorem}[section]
\newtheorem*{thm*}{Theorem}
\newtheorem{corollary}[thm]{Corollary}

\newtheorem{lemma}[thm]{Lemma}
\newtheorem{prop}[thm]{Proposition}
\newtheorem{statement}[thm]{Statement}
\newtheorem*{statement*}{Statement}
\newtheorem{assumption}[thm]{Assumption}

\theoremstyle{definition}   
\newtheorem{defn}[thm]{Definition}

\theoremstyle{remark}  
\newtheorem{remark}[thm]{Remark}
\newtheorem{example}[thm]{Example}

\newtheorem{innercustomgeneric}{\customgenericname}
\providecommand{\customgenericname}{}

\newcommand{\newcustomtheorem}[2]{\newenvironment{#1}[1]
	{\renewcommand\customgenericname{#2}
		\renewcommand\theinnercustomgeneric{##1}\innercustomgeneric}{\endinnercustomgeneric}}

\newcustomtheorem{customthm}{Theorem}

\newcommand\trho{\widetilde{\rho}}

\newcommand\bL{\mathbb{L}}

\newcommand\bR{\mathbb{R}}
\newcommand\bC{\mathbb{C}}
\newcommand\bH{\mathbb{H}}
\newcommand\bZ{\mathbb{Z}}
\newcommand\bW{\mathbb{W}}

\newcommand\bS{\mathbb{S}}

\newcommand\bN{\mathbb{N}}

\newcommand\cA{\mathcal{A}}

\newcommand\cD{\mathcal{D}}
\newcommand\cF{\mathcal{F}}

\newcommand\cH{\mathcal{H}}
\newcommand\cI{\mathcal{I}}

\newcommand\cL{\mathcal{L}}

\newcommand\cM{\mathcal{M}}

\newcommand\cO{\mathcal{O}}

\newcommand\cW{\mathcal{W}}
\newcommand{\domain}{\mathcal{O}}

\newcommand\la{\langle}
\newcommand\ra{\rangle}
\newcommand\dd{\,\mathrm{d}}

\newcommand{\mysection}[1]{\section{#1}
\setcounter{equation}{0}}

\newcount\pgfdecorationorder
\pgfkeys{/pgf/decoration/.cd,
	Koch angle/.store in=\pgfkochangle, Koch angle=85,
	Koch order/.code={\global\pgfdecorationorder=#1}, Koch order=1
}
\pgfdeclaredecoration{Koch}{calculate}{
	\state{calculate}[width=0pt, next state=draw, persistent precomputation={
		\pgfmathparse{\pgfdecoratedinputsegmentlength/(2*(1+cos(\pgfkochangle)))}%
		\let\pgfkochsegmentlength=\pgfmathresult%
		\pgfmathparse{\pgfkochsegmentlength*sin(\pgfkochangle)}%
		\let\pgfkochy=\pgfmathresult%
		\pgfmathparse{\pgfkochsegmentlength*(1 + cos(\pgfkochangle))}%
		\let\pgfkochxa=\pgfmathresult%
		\pgfmathparse{\pgfkochsegmentlength*(1 + 2*cos(\pgfkochangle))}%
		\let\pgfkochxb=\pgfmathresult%
	}]{}
	\state{draw}[width=\pgfdecoratedinputsegmentlength]{
		\pgfpathmoveto{\pgfpointorigin}%
		\pgfpathlineto{\pgfqpoint{\pgfkochsegmentlength pt}{0pt}}%
		\pgfpathlineto{\pgfqpoint{\pgfkochxa pt}{\pgfkochy pt}}%
		\pgfpathlineto{\pgfqpoint{\pgfkochxb pt}{0pt}}%
		\pgfpathlineto{\pgfqpoint{\pgfdecoratedinputsegmentlength}{0pt}}%
	}
	\state{final}{
		\global\advance\pgfdecorationorder by -1\relax%
		\ifnum\pgfdecorationorder>0\relax%
		\pgfgetpath\decoratedpath%
		\pgfsetpath\empty%
		\begin{pgfdecoration}{{Koch}{\pgfdecoratedpathlength}}%
			\pgfsetpath\decoratedpath
		\end{pgfdecoration}%
		\fi
}}

\begin{document}


\title[$L_p$-theory for PDEs in non-smooth domains]
{Sobolev space theory for Poisson's and the heat equations in non-smooth domains via superharmonic functions and Hardy's inequality}





 
\thanks{The author was supported by the National Research Foundation of Korea(NRF) grant funded by the Korea government(MSIT) (No. NRF-2020R1A2C1A01003354)} 

%

\author{jinsol Seo}
\address{Jinsol Seo, Department of Mathematics, Korea University, Anam-ro 145, Sungbuk-gu, Seoul, 02841, Republic of Korea}
\email{seo9401@korea.ac.kr}



\begin{abstract}
We prove the unique solvability for the Poisson and heat equations in non-smooth domains $\Omega\subset \bR^d$ in weighted Sobolev spaces.
The zero Dirichlet boundary condition is considered, and domains are merely assumed to admit the Hardy inequality:
$$
\int_{\Omega}\Big|\frac{f(x)}{d(x,\partial\Omega)}\Big|^2\dd x\leq N\int_{\Omega}|\nabla f|^2 \dd x\,\,\,\,,\,\,\,\,   \forall f\in C_c^{\infty}(\Omega)\,.
$$
To describe the boundary behavior of solutions, we introduce a weight system that consists of superharmonic functions and the distance function to the boundary.
The results provide separate applications for the following domains: convex domains, domains with exterior cone condition, totally vanishing exterior Reifenberg domains, conic domains, and domains $\Omega\subset\bR^d$ which the Aikawa dimension of $\Omega^c$ is less than $d-2$. 
\end{abstract}

\maketitle

\setcounter{tocdepth}{3}

\let\oldtocsection=\tocsection

\let\oldtocsubsection=\tocsubsection

\let\oldtocsubsubsection=\tocsubsubsectio

\renewcommand{\tocsection}[2]{\hspace{0em}\oldtocsection{#1}{#2}}
\renewcommand{\tocsubsection}[2]{\hspace{1em}\oldtocsubsection{#1}{#2}}
\renewcommand{\tocsubsubsection}[2]{\hspace{2em}\oldtocsubsubsection{#1}{#2}}

\tableofcontents

\mysection{Introduction}\label{sec:Introduction}
The Poisson and heat equations are among the most classical partial differential equations.
Together with the Schauder and $L_2$-theories, the $L_p$-theory for these equations in $\bR^d$ and $C^2$-domains has been developed long before.
In particular, there are extensions in various directions, including variable coefficients \cite{DongKim, Krylov2009}, semigroups \cite{Pseudodiff, semigroup}, and non-smooth domains.
This paper concentrates on non-smooth domains, where unweighted or weighted $L_p$-theories have been developed for several tyes of domains:
$C^1$-domains \cite{doyoon, KK2004}, Reifenberg domains \cite{Relliptic, Rparabolic}, convex domains \cite{convexAdo,convexFromm}, Lipschitz domains \cite{kenig,IW}, smooth cones \cite{Kozlov,Naza,Sol}, and polyhedrons \cite{KMR,MNP,MR}.

In this paper, we present a weighted $L_p$-theory for the Poisson equation
\begin{alignat}{3}
	\Delta u=f&\quad\text{in}\,\,\,\Omega\quad&&;\quad u|_{\partial\Omega}\equiv 0\qquad \label{ellip}
\end{alignat}
and the heat equation
\begin{alignat}{3}
	u_t=\Delta u+f&\quad\text{in}\,\,\,(0,\infty)\times \Omega\quad&&;\quad u(0,\cdot)=u_0\,\,,\,\,u|_{(0,\infty)\times \partial\Omega}\equiv 0\label{para}\,.
\end{alignat}
Here, $\Omega\subsetneq\bR^d$ is an open set admitting the ($L_2$-)Hardy inequality, \textit{i.e.}, when there exists a constant $\mathrm{C}_0(\Omega)>0$ such that
\begin{align}\label{hardy}
	\int_{\Omega}\Big|\frac{f(x)}{d(x,\partial\Omega)}\Big|^2\dd x\leq \mathrm{C}_0(\Omega)\int_{\Omega}|\nabla f(x)|^2 \dd x\quad\text{for all}\quad f\in C_c^{\infty}(\Omega)\,,
\end{align}
where $d(\,\cdot\,,\partial\Omega)$ is the distance function to the boundary of $\Omega$.
One of notable sufficient conditions for the Hardy inequality is the volume density condition:
\begin{align}\label{230212413}
	\inf_{\substack{p\in\partial\Omega\\r>0}}\frac{m(\Omega^c \cap B_r(p))}{m(B_r(p))}>0\,,
\end{align}
where $m$ is the Lebesgue measure on $\bR^d$ (see Remark~\ref{21.07.06.1}).

Our main results are introduced in a simplified manner in Subsection~\ref{0002}.
The results provide separate applications for the following domain conditions:
\begin{enumerate}
	\item Domains $\Omega$ satisfying \eqref{230212413};
	\item Domains $\Omega\subset \bR^d$ with $\dim_{\cA}\Omega^c<d-2$;
	\item Domains satisfying the exterior cone condition, and planar domains satsifying the exterior line segment condition;
	\item Convex domains;
	\item Domains satsifying the totally vanishing exterior Reifenberg condition;
	\item Conic domains (containing smooth cones and polyhedral cones).
\end{enumerate}
These applications are presented in Subsubsections~\ref{0003}.1 - \ref{0003}.6, sequentially.
Before summarizing the main results and their applications, we first introduce several studies related to the $L_p$-theory of the Poisson and heat equations in various domains.

\subsection{Historical remarks and aims of this paper.}\label{230214201}
\,\,

\vspace{2mm}
\noindent
\textbf{Remark on $L_p$-results for non-smooth domains.}\hspace{2mm}
One of the most remarkable studies on the Poisson equation in non-smooth domains is the work of Jerison and Kenig \cite{kenig}, where the authors proved the following result:

\begin{thm*}[\cite{kenig}, Theorems 1.1-1.3 and Proposition 1.4]
	For a domain $\Omega\subsetneq \bR^d$ and $p\in(1,\infty)$, we denote
	\begin{equation}\label{230214355}
		\begin{alignedat}{2}
			&\mathring{W}^1_p(\Omega)&&=\text{the closure of $C_c^{\infty}(\Omega)$ in $W_p^1(\Omega)$}\,;\\
			&W_p^{-1}(\Omega)&&=\text{the dual space of $\mathring{W}^1_{p/(p-1)}(\Omega)$}\,.
		\end{alignedat}	
	\end{equation}
	\begin{enumerate}
		\item
		For any bounded Lipschitz domain $\Omega\subset \bR^d$, there exists $\epsilon>0$ such that if
		\begin{align*}
			\begin{cases}
				4/3-\epsilon< p<4+\epsilon\quad\text{when}\,\,\,\,d=2\,;\vspace{1mm}\\
				3/2-\epsilon< p<3+\epsilon\quad\text{when}\,\,\,\,d\geq 3\,,
			\end{cases}
		\end{align*}
		then for any $f\in W_p^{-1}(\Omega)$, the equation $\Delta u=f$ is uniquely solvable in $\mathring{W}^1_p(\Omega)$.
		For the solution $u$, we have
		\begin{align*}
			\|u\|_{W_p^1(\Omega)}\leq N(\Omega,p)\|f\|_{W_p^{-1}(\Omega)}\,.
		\end{align*}
		\item If $p>4$ when $d=2$, and $p>3$ when $d\geq 3$,
		then there exists a bounded Lipschitz domain $\Omega$ and $f\in W_p^{-1}(\Omega)$ such that the equation $\Delta u=f$ does not have solution $u$ in $\mathring{W}^1_p(\Omega)$.
	\end{enumerate}
\end{thm*}
\noindent
This theorem establishes that the Poisson equation is \textit{not uniquely solvable} in unweighted Sobolev spaces $\mathring{W}_p^1$, in general, for non-smooth domains $\Omega$ and values of $p\in(1,\infty)$.
For (1) of the above theorem, Jerison and Kenig investigated the trace map $w\mapsto w|_{\partial\Omega}$ for $w\in W_p^1(\bR^d)$ satisfying $\Delta w=f 1_{\Omega}$, and the homogeneous Dirichlet problem $\Delta v=0\,;\,v|_{\partial\Omega}=w|_{\partial\Omega}$.
The Lipschitz boundary condition for $\Omega$ plays a crucial role in this context.

Elliptic and parabolic equations in smooth cones and polyhedrons have been extensively studied in the literature, including studies in \cite{KMR,MNP,MR} (for elliptic equations) and \cite{Kozlov,Naza,Sol} (for parabolic equations).
Here, a smooth cone is a domain $\Omega\subset \bR^d$ defined as 
$$
\Omega=\{r\sigma\,:\,r>0\,\,\,\,\text{and}\,\,\,\, \sigma\in\cM\}\,,
$$
where $\cM$ is a a smooth subdomain of $\bS^{d-1}:=\partial B_1(0)$ (see Figure \ref{230222651}).
The references provide the unique solvability of the equations in specific weighted $L_p$-Sobolev spaces for all $p\in(1,\infty)$, by using the spectral theory of so-called operator pencils for elliptic equations and by using Green function estimates for parabolic equations.
The weight system in these Sobolev spaces (for smooth cones and polyhedrons) consists of distance functions for each vertex and edge; the range of weights for the solvability is closely related to the \textit{eigenvalues} of the spherical Laplacian on $\cM$.
Furthermore, by Sobolev-H\"older embedding theorems (introduced in \cite[Lemma 1.2.3, Lemma 3.1.4]{MR}), the pointwise behavior of solutions near vertices and edges is also obtained.

The aforementioned studies suggest considering weight systems associated with each domain and Laplace operator to investigate the solvability of the Poisson and heat equations in various non-smooth domains and to describe the boundary behavior of solutions.

There are many other notable studies in this area.
In Section \ref{0003}, dealing with several types of non-smooth domains, we mention works relevant to each situation.

	\vspace{4mm}
	
	\noindent\textbf{Remark on the method of this paper.}\hspace{2mm}
	Our approach is based on the localization argument developed by Krylov \cite{Krylov1999-1}, where the author investigated the Poisson and heat equations in the half space $\bR_+^d:=\{(x_1,\ldots,x_d)\,:\,x_1>0\}$.
	Krylov provides the following weighted $L_p$-estimates (see \cite[Theorem 4.1]{Krylov1999-1}): if $\theta\in(-p-1,-1)$, then for any $n\in\bN_0$, $u\in C_c^{\infty}(\Omega)$ and $f:=\Delta u$,
	\begin{align}\label{220930316}
		\int_{\bR_+^d}\bigg(\sum_{k=0}^{n+2}|\rho^{k}D^ku|\bigg)^p\rho^{\theta}\dd x\lesssim\int_{\bR_+^d}\bigg(\sum_{k=0}^{n}|\rho^{k+2}D^kf|\bigg)^p\rho^{\theta}\dd x\,,
	\end{align}
	where $\rho(x)=d(x,\partial\bR_+^d)=x_1$ for $x=(x_1,\ldots,x_d)$.
	By setting $\theta=-p$, this implies
	\begin{align*}
		\int_{\bR_+^d}|\rho^{-1}u|^p+|Du|^p+\cdots+|\rho^{n+1}D^{n+2}u|^p\dd x\lesssim\int_{\bR_+^d}|\rho f|^p+\cdots+|\rho^{n+1}D^nf|^p\dd x\,.
	\end{align*}
	The value of $\theta$ in \eqref{220930316} describes the boundary behavior of solutions and their derivatives.
	We further refer to \cite[Theorem 3.1]{Krylov1999-1} for Sobolev-H\"older embedding theorems for the above weight system.
	
	Briefly speaking, the proof of \cite[Theorem 4.1]{Krylov1999-1} can be divided into two steps.
	Firstly, a \textit{localization argument} is applied to estimate higher order derivatives of the solution $u$ (the left-hand side of \eqref{220930316}) by the zeroth-order derivative of $u$ ($\int|u|^p\rho^{\theta}\dd x$) and the force term $f$ (the right-hand side of \eqref{220930316}).
	Secondly, the author estimates the zeroth-order derivative of $u$ by $f$, using the {\textit{weighted Hardy inequalities} for $\bR_+$; the sharp constants in the weighted Hardy inequalities play a crucial role.
	
	The localization argument used in \cite{Krylov1999-1} is applicable to any domain and any $\theta\in\bR$, not just to $\bR^d_+$ and $\theta\in (-p-1,-1)$ (see, \textit{e.g.}, \cite{KK2004,ConicPDE} or Lemma \ref{21.05.13.8}).
	However, the second step of the proof for \cite[Theorem 4.1]{Krylov1999-1} cannot be directly applied to other domains; this step relies on the weighted Hardy inequalities for $\bR_+$.
	For instance, the authors of \cite{ConicPDE} employ the localization argument for parabolic equations in smooth cones.
	However, in contrast to the approach in \cite{Krylov1999-1}, they use pointwise estimates of Green functions to estimate zeroth-order derivatives of solutions since weighted Hardy inequalities on conic domains have yet to be explored as extensively as those on $\bR_+$.
	
	We concentrate on the (unweighted) Hardy inequality \eqref{hardy} to estimate zeroth-order derivatives of solutions.
	We do this because the Hardy inequality holds on various non-smooth domains (see \eqref{230212413}), and the approach used in \cite{Krylov1999-1} is independent of the kernels of the Poisson and heat equations.
	To the best of our knowledge, the class of domains admitting the Hardy inequality is broader than the class of domains for which sharp estimates for the Poisson kernel have been investigated.
	
	To focus on the Hardy inequality, we note the work of Kim \cite{Kim2014}, where the author investigated stochastic parabolic equations in bounded domains $\Omega$ admitting the Hardy inequality.
	In particular, in \cite[Theorem 2.12]{Kim2014} the author provides a \eqref{220930316} type estimate, in which $(\bR^d_+,\rho(\cdot))$ is replaced by $(\Omega,d(\,\cdot\,,\partial\Omega))$ and the range of $\theta$ is restricted to to around $-2$.
	
	This work revealed a connection between the Hardy inequality \eqref{hardy} and the approach used in \cite{Krylov1999-1}.
	However, it should be noted that the range of $\theta$ in \cite[Theorem 2.12]{Kim2014} is not specified.
	Therefore, the results in \cite{Kim2014} may not fully describe the boundary behavior of solutions sufficiently well and may not include the results on $\bR_+^d$ \cite{Krylov1999-1} and $C^1$-domains \cite{KK2004}.

	\vspace{4mm}
	\noindent
	\textbf{Objective and approach of this paper.}
	\hspace{2mm}
	This paper aims to develop a general $L_p$-theory for the Poisson equation \eqref{ellip} and the heat equation \eqref{para} in a variety of non-smooth domains.
	We focus on domains that merely admit the Hardy inequality, following \cite{Kim2014}.
	A distinguishing feature of this paper from earlier studies is the use of \textit{superharmonic functions}.
	These functions are used with the Hardy inequality to estimate zeroth-order derivatives of solutions, as shown in Theorem~\ref{21.05.13.2}.
	
	Furthermore, we introduce the concept of \textit{Harnack functions} and \textit{regular Harnack functions} (see Definition~\ref{2210261255}) to extend the localization argument employed in \cite{Krylov1999-1} to a broader class of weight functions.
	These notions enable us to obtain a unified formulation for the main theorem.
	While the weight system used in most applications within this paper consists only of the distance function to the boundary, the notion of Harnack functions helps us to derive a result for conic domains, as presented in Subsection~\ref{0074}.
	
	The main result of this paper, Theorem~\ref{2207291120}, establishes that for a domain $\Omega$ admitting the Hardy inequality \eqref{hardy} and a superharmonic Harnack function $\psi$ on $\Omega$, equations \eqref{ellip} and \eqref{para} are uniquely solvable in weighted Sobolev spaces related to $\psi$.
	This result has applications to various non-smooth domains listed below \eqref{230212413} (see Subsection \ref{0003}).
	By proving the existence of suitable superharmonic functions reflecting geometric conditions for domains, we obtain unique solvability results that differ for each domain condition (see Theorems~\ref{21.11.08.1}, \ref{21.10.18.1}, \ref{22.02.19.5}, \ref{22.02.19.6}).
	
	Our results bridge the gap between \cite{KK2004, Krylov1999-1} and \cite{Kim2014}.
	Since we only assume the Hardy inequality for domains, this paper can be seen as an extension of \cite{Kim2014}.
	In addition, when focusing only on the Poisson and heat equation, Corollaries~\ref{2208131026} and \ref{22.07.17.109} encompass \cite[Theorem 4.1, Theorem 5.6]{Krylov1999-1} and \cite[Theorem 2.10]{KK2004}, respectively.

	Finally, we mention that the approach presented in this paper can be applied not only to the Poisson and heat equations but also to extended evolution equations, such as the time-fractional heat equations and the stochastic heat equation (for definitions, see, \textit{e.g.}, \cite{PDEfracC1,KKL2017} and \cite{K2004SPDE,ConicsPDE,aapp}, respectively).
	The localization argument presented in Section~\ref{0050} and the results that provide appropriate superharmonic functions for each domain (see Sections~\ref{app.} and \ref{app2.}) can be directly applied to these equations.
	In future work, we plan to extend the results obtained in this paper to these extended evolution equations.

	\vspace{2mm}

	\subsection{Summary of the main result}\label{0002}

	Let $d\in\bN$, $\Omega\subsetneq\bR^d$ is an open set, and $T\in(0,\infty]$.
	We denote
	$$
	\rho(x):=d(x,\partial\Omega):=\mathrm{dist}(x,\partial\Omega)
	$$
	and when $T=\infty$, we adopt the convention that $[0,T]=[0,\infty)$.
	\begin{defn}\label{2210261255}
		Let $\psi:\Omega\rightarrow \bR_+$ be a locally integrable function.
		\begin{enumerate}
			\item $\psi$ is said to be \textit{superharmonic} if $\Delta \psi\leq 0$ in the sense of distribution, \textit{i.e.},
			\begin{align*}
				\int_{\Omega}\psi \Delta \zeta \dd x\leq 0\quad\text{for all $\zeta\in C_c^{\infty}(\Omega)$ with $\zeta\geq 0$.}
			\end{align*}
			\item We call $\psi$ a \textit{Harnack} function if $\psi>0$, and there exists a constant $\mathrm{C}_1(\psi)$ such that
			\begin{align*}
				\underset{B(x,\rho(x)/2)}{\mathrm{ess\,sup}}\,\psi\leq \mathrm{C}_1(\psi)\underset{B(x,\rho(x)/2)}{\mathrm{ess\,inf}}\,\psi\quad\text{for all}\,\,x\in\Omega.
			\end{align*}
		\end{enumerate}
	\end{defn}
	\vspace{1mm}
	
	The primary motivation for the concept of Harnack functions is a localization argument.
	The following is proved in Lemma~\ref{21.05.27.3}: $\psi$ is Harnack if and only if there exists $\Psi\in C^{\infty}(\Omega)$ such that
	\begin{equation}\label{221027303}
		\begin{aligned}
			&\Psi\simeq \psi\quad \text{almost everywhere on}\,\,\Omega\,;\\
			&|D^k\Psi|\lesssim \rho^{-k}\Psi\quad \text{for all }k\in\bN\,.
		\end{aligned}
	\end{equation}
	We call $\Psi$ a regularization of $\psi$.
	The concept of regularization enables us to generalize a localization argument used in \cite{Krylov1999-1} to a broader class of weight functions; see Lemmas~\ref{21.05.13.8} and \ref{21.05.13.9} for this generalization.

	We introduce weighted Sobolev spaces and weighted Sobolev-Slobodeckij spaces.
	\begin{defn}
		Let $p\in (1,\infty)$, $\theta,\,\sigma\in\bR$, and $\psi$ is a Harnack function.
		\begin{enumerate}
			\item For $n\in \{0,\,1,\,2,\,\ldots\}$ and $0<s<1$, we denote
			\begin{alignat*}{2}
				\|&f\|_{W_{p,\theta}^n(\Omega,\psi^\sigma)}^p&=&\,\sum_{k=0}^n\int_{\Omega}|\rho^{k}D^k f|^p\psi^{\sigma} \rho^{\theta}\dd x\quad \Big(=\sum_{k=0}^n\big\|\rho^kD^kf\big\|_{W_{p,\theta}^0(\Omega,\psi^\sigma)}^p\Big)\,,\\[8pt]
				\|&f\|_{W_{p,\theta}^{n+s}(\Omega,\psi^\sigma)}^p&=&\,\|f\|_{W_{p,\theta}^n(\Omega,\psi^\sigma)}^p+[D^nf]_{W_{p,\theta+np}^{s}(\Omega,\psi^\sigma)}^p\,,
			\end{alignat*}
			where
			\begin{alignat*}{2}
				[&h]_{W_{p,\theta+np}^{s}(\Omega,\psi^\sigma)}^p&:=&\int_{\Omega}\bigg(\int_{\{y:|x-y|\leq \rho(x)/2\}}\frac{|h(x)-h(y)|^p}{|x-y|^{d+s p}}\dd y\bigg)\psi(x)^{\sigma} \rho(x)^{(n+s)p+\theta}\dd x\,.
			\end{alignat*}
			\vspace{1mm}
			
			\item For $n\in\bN$ and $s\in[0,1)$, we denote
			\begin{align*}
				\|f\|_{W_{p,\theta}^{-n+s}(\Omega,\psi^\sigma)}=\inf\Big\{\sum_{|\alpha|\leq n}\|\rho^{-|\alpha|}f_{\alpha}\|_{W_{p,\theta}^{s}(\Omega,\psi^\sigma)}\,:\,f=\sum_{|\alpha|\leq n}D^{\alpha}f_{\alpha}\Big\}\,.
			\end{align*}
			\item For $\gamma\in\bR$, we denote
			$$
			W_{p,\theta}^{\gamma}(\Omega,\psi^\sigma)=\big\{f\in\cD'(\Omega)\,:\,\|f\|_{W_{p,\theta}^{\gamma}(\Omega,\psi^\sigma)}<\infty\big\}\,,
			$$ where $\cD'(\Omega)$ denotes the spaces of all distributions on $\Omega$.
		\end{enumerate}
	\end{defn}
	
	\begin{defn} Let $p\in (1,\infty)$, $\theta,\,\sigma\in\bR$, $n\in\bZ$.
		\begin{enumerate}
			\item We denote $\bW_{p,\theta}^{n}(\Omega_T,\psi^{\sigma})=L_p\big((0,T);W_{p,\theta}^{n}(\Omega,\psi^{\sigma})\big)$.
			\vspace{1mm}
			
			\item By $\cW_{p,\theta}^{n+2}(\Omega_T,\psi^{\sigma})$, we denote the set of all $u:[0,T]\rightarrow \cD'(\Omega)$ satisfying the following:
			\begin{itemize}
				\item $u\in\bW_{p,\theta}^{n+2}(\Omega_T,\psi^{\sigma})$ and $u(0,\cdot)\in W_{p,\theta+2}^{n+2-2/p}(\Omega,\psi^\sigma)$;
				\item there exists $f\in \bW_{p,\theta+2p}^n(\Omega_T,\psi^{\sigma})$ such that 
				\begin{align}\label{230417643}
				\big\langle u(t,\cdot),\zeta\big\rangle=\big\langle u(0,\cdot),\zeta\big\rangle+\int_0^t\big\langle f(s,\cdot),\zeta\big\rangle\, ds\,.
				\end{align}
				for all $t\in(0,T]$ and $\zeta\in C_c^{\infty}(\Omega)$.
			\end{itemize}
			In the case \eqref{230417643}, we denote $\partial_tu=f$.
			The norm of $\cW_{p,\theta}^{n+2}(\Omega_T,\psi^{\sigma})$ is defined by
			$$
			\|u\|_{\cW_{p,\theta}^{n+2}(\Omega,\psi^{\sigma})}:=\|u\|_{\bW_{p,\theta}^{n+2}(\Omega,\psi^{\sigma})}+\|u(0,\cdot)\|_{W_{p,\theta+2}^{n+2-2/p}(\Omega,\psi^{\sigma})}+\|\partial_t u\|_{\bW_{p,\theta+2p}^{n}(\Omega,\psi^{\sigma})}\,.
			$$
			
			%
			%
		\end{enumerate}
	\end{defn}

	\begin{remark}\label{230214208}
		\,\,
		
		\begin{enumerate}
			\item 	The spaces $W_{p,\theta}^{\gamma}(\Omega,\psi^\sigma)$, $\bW_{p,\theta}^{n}(\Omega,\psi^\sigma)$ and $\cW_{p,\theta}^{n}(\Omega,\psi^\sigma)$ appear only in this section. 
			However, these spaces have the following equivalent relation (see Propositions \ref{220528651}, \ref{220418435}, Corollary \ref{21.05.26.3} and Remark~\ref{22.04.18.5}):
			\begin{itemize}
			\item Let $n\in\bZ$, $0<s<1$, and let $\Psi$ be a function satisfying \eqref{221027303}.
					\begin{align*}
					W_{p,\theta}^{n}(\Omega,\psi^\sigma)=\Psi^{-\sigma/p} H_{p,\theta+d}^{n}(\Omega)\quad\text{and}\quad W_{p,\theta}^{n+s}(\Omega,\psi^\sigma)=\Psi^{-\sigma/p} B_{p,\theta+d}^{n+s}(\Omega)\,,
				\end{align*}
				where $\Psi^{-\sigma/p} H_{p,\theta+d}^{\gamma}$ and $\Psi^{-\sigma/p} B_{p,\theta+d}^{\gamma}$ are introduced in Subsections~\ref{0042} and \ref{0052}.
				In addition,
				\begin{align*}
					\bW_{p,\theta}^{n}(\Omega_T,\psi^\sigma)=\Psi^{-\sigma/p} \bH_{p,\theta+d}^{n}(\Omega,T)\quad\text{and}\quad \cW_{p,\theta}^{n+2}(\Omega_T,\psi^\sigma)=\Psi^{-\sigma/p} \cH_{p,\theta+d}^{n+2}(\Omega,T)\,,
				\end{align*}
				where $\Psi^{-\sigma/p} \bH_{p,\theta+d}^{n}$ and $\Psi^{-\sigma/p} \cH_{p,\theta+d}^{n+2}(\Omega)$ are introduced in \eqref{221015645} and the below of \eqref{221015645}, respectively.
			\end{itemize}
			\vspace{1mm}
			
			\item Properties of $W_{p,\theta}^{\gamma}(\Omega,\psi^{\sigma})$ and $\cW_{p,\theta}^{n}(\Omega_T,\psi^\sigma)$) are introduced in Subsections \ref{0042} and \ref{0052}.
			Especially, Lemmas~\ref{21.09.29.4}, \ref{22.04.15.148}, and Proposition~\ref{2205241111} provide that the dual space of $W_{p,\theta}^{s}(\Omega,\psi^{\sigma})$ is 
			$W_{p',\theta'}^{-s}(\Omega,\psi^{\sigma'})$, where
			$$
			\frac{1}{p}+\frac{1}{p'}=1\,\,,\quad \frac{\theta}{p}+\frac{\theta'}{p'}=\frac{\sigma}{p}+\frac{\sigma'}{p'}=0\,.
			$$
			In addition, $W_{p,\theta}^{\gamma}(\Omega,\psi^{\sigma})$ is a Banach space, and $C_c^{\infty}(\Omega)$ is dense in $W_{p,\theta}^{\gamma}(\Omega,\psi^{\sigma})$.
			Similarly, $\cW_{p,\theta}^{\gamma}(\Omega,\psi^{\sigma})$ is a Banach space, and $C_c^{\infty}\big([0,\infty)\times \Omega\big)$ is dense in $\cW_{p,\theta}^{\gamma}(\Omega,\psi^{\sigma})$.
		\end{enumerate}
	\end{remark}
	\vspace{1mm}	

	For $0<\nu_1\leq \nu_2<\infty$ and $T\in (0,\infty]$, we denote
	\begin{itemize}
	\item $\mathrm{M}(\nu_1,\nu_2)$ : the set of all $d\times d$ real-valued symmetric matrices $(\alpha^{ij})_{d\times d}$ satisfying
$$
\nu_1|\xi|^2\leq \sum_{i,j=1}^d\alpha^{ij}\xi_i\xi_j\leq \nu_2|\xi|^2\qquad\forall\,\, \xi\in\bR^d;
$$

\item $\cM_T(\nu_1,\nu_2)$ : 
the set of all $\cL:=\sum_{i,j=1}^da^{ij}(\cdot)D_{ij}$, where $\{a^{ij}(\cdot)\}_{i,j=1,....,d}$ is a family of time measurable function on $\bR_+$ such that $\big(a^{ij}(t)\big)_{d\times d}\in \mathrm{M}(\nu_1,\nu_2)$ for all $t\in(0,T]$.
	\end{itemize}
	
	We state main results of this paper as a version by $W_{p,\theta}^\gamma(\Omega,\psi^\sigma)$.
	
	\begin{thm}[see Theorems~\ref{21.09.29.1}, \ref{22.02.18.6} with Proposition~\ref{05.11.1} and Remarks~\ref{220617}, \ref{230212657}]\label{2207291120}
		Suppose that \vspace{1mm}
		\begin{itemize}
			\item[]\qquad\qquad $p\in(1,\infty)$, $n\in\bZ$, $\sigma\in (-p+1,1)$;\vspace{1mm}
			
			\item[]\qquad\qquad $\Omega$ admits the Hardy inequality \eqref{hardy};\vspace{1mm}
			
			\item[]\qquad\qquad $\psi$ is a superharmonic Harnack function on $\Omega$.\vspace{1mm}
		\end{itemize}
		
		\begin{enumerate}
			\item For any $\lambda\geq 0$ and $f\in W_{p,2p-2}^{n}(\Omega,\psi^\sigma)$, the equation
			$$
			\Delta u-\lambda u=f\,
			$$ 
			has a unique solution $u$ in $W_{p,-2}^{n+2}(\Omega,\psi^{\sigma})$.
			Moreover, we have
			$$
			\|u\|_{W_{p,-2}^{n+2}(\Omega,\psi^{\sigma})}+\lambda \|u\|_{W_{p,2p-2}^{n}(\Omega,\psi^{\sigma})}\leq N\|f\|_{W_{p,2p-2}^{n}(\Omega,\psi^{\sigma})}\,,
			$$
			where $N$ depends only on $d$, $p$, $n$, $\sigma$, $\mathrm{C}_0(\Omega)$ and $\mathrm{C}_1(\psi)$.
			\vspace{1mm}
			
			\item For any $f\in \bW_{p,2p-2}^{n}(\Omega_T,\psi^{\sigma})$ and $u_0\in W_{p,0}^{n+2-2/p}(\Omega,\psi^{\sigma})$, the equation
			$$
			u_t=\Delta u+f\quad;\quad u(0)=u_0
			$$ 
			has a unique solution $u$ in $\cW_{p,-2}^{n+2}(\Omega_T,\psi^{\sigma})$.
			Moreover, we have
			$$
			\|u\|_{\cW_{p,-2}^{n+2}(\Omega_T,\psi^{\sigma})}\leq N\left(\|u_0\|_{W_{p,0}^{n+2-2/p}(\Omega,\psi^{\sigma})}+\|f\|_{\bW_{p,2p-2}^{n}(\Omega_T,\psi^{\sigma})}\right)\,,
			$$
			where $N$ depends only on $d$, $p$, $n$, $\sigma$, $\mathrm{C}_0(\Omega)$ and $\mathrm{C}_1(\psi)$.
			\vspace{1mm}
			
			\item Let $0<\nu_1\leq \nu_2<\infty$ and $\cL\in \cM_T(\nu_1,\nu_2)$, and
			additionally assume that $\psi$ satisfies
			\begin{align*}
				\alpha^{ij}D_{ij}\psi\leq 0
			\end{align*}
			in the sense of distribution for all $(\alpha^{ij})_{d\times d}\in \mathrm{M}(\nu_1,\nu_2)$.
			Then for any $f\in \bW_{p,2p-2}^{n}(\Omega_T,\psi^{\sigma})$ and $u_0\in W_{p,0}^{n+2-2/p}(\Omega,\psi^{\sigma})$, the equation
			$$
			u_t=\cL u+f\quad;\quad u(0)=u_0
			$$ 
			has a unique solution $u$ in $\cW_{p,-2}^{n+2}(\Omega_T,\psi^{\sigma})$.
			Moreover, we have
			$$
			\|u\|_{\cW_{p,-2}^{n+2}(\Omega_T,\psi^{\sigma})}\leq N\left(\|u_0\|_{W_{p,0}^{n+2-2/p}(\Omega,\psi^{\sigma})}+\|f\|_{\bW_{p,2p-2}^{n}(\Omega_T,\psi^{\sigma})}\right)\,,
			$$
			where $N$ depends only on $d$, $p$, $n$, $\nu_1$, $\nu_2$, $\sigma$, $\mathrm{C}_0(\Omega)$ and $\mathrm{C}_1(\psi)$.
		\end{enumerate}
	\end{thm}
	The constant function $1_{\Omega}$ is a trivial example of superharmonic Harnack functions.
	As another example, it is provided in Example~\ref{220912411} that if $\Omega$ is a domain (connected open set) admitting the Hardy inequality, then $G_{\Omega}(x_0,\,\cdot\,)\wedge 1$ is a superharmonic Harnack function, where $G_{\Omega}$ is the Green function for the Poisson equation in $\Omega$ and $x_0$ is an arbitrary fixed point in $\Omega$.
	
	\vspace{2mm}
	
	\subsection{Summary of applications}\label{0003}

	This subsection considers a domain $\Omega\subset \bR^d$, where $d\geq 2$.
	For convenience, we denote
	$$
	W_{p,\theta}^{\gamma}(\Omega)=W_{p,\theta}^{\gamma}(\Omega,1)\,\,\,,\,\,\,\,\bW_{p,\theta}^{n}(\Omega_T)=\bW_{p,\theta}^{n}(\Omega_T,1)\,\,\,,\,\,\,\,\cW_{p,\theta}^{n}(\Omega_T)=\cW_{p,\theta}^{n}(\Omega_T,1)\,,
	$$
	and define the following statement:
	\begin{statement}[$\Omega,p,\theta$]\label{230203621}$\,\,$
		\begin{enumerate}
			\item[$\mathrm{[Pois]}$] Let $\lambda\geq 0$. For any $n\in\bZ$, if $f\in W_{p,\theta}^{n}(\Omega)$, the equation
			\begin{align}\label{230214224}
				\Delta u-\lambda u=f
			\end{align}
			has a unique solution $u$ in $W_{p,\theta+2p}^{n+2}(\Omega)$.
			Moreover, we have
			\begin{align}\label{220923207}
				\|u\|_{W_{p,\theta}^{n+2}(\Omega)}+\lambda \|u\|_{W_{p,\theta}^{n+2}(\Omega)}\leq N_1\|f\|_{W_{p,\theta}^{n+2}(\Omega)}\,,
			\end{align}
			where $N_1$ is independent of $f$, $u$, and $\lambda$.
			\vspace{1mm}
			
			\item[$\mathrm{[Heat]}$] For any $n\in\bZ$, if $f\in \bW_{p,\theta+2p}^{n}(\Omega_T)$ and $u_0\in W_{p,\theta+2}^{n+2-2/p}(\Omega)$, then the equation
			$$
			u_t=\Delta u+f\quad;\quad u(0)=u_0
			$$ 
			has a unique solution $u$ in $\cW_{p,\theta}^{n+2}(\Omega_T)$.
			Moreover, we have
			\begin{align}\label{2209232071}
				\|u\|_{\cW_{p,\theta}^{n+2}(\Omega_T)}\leq N_2\big(\|u_0\|_{W_{p,\theta+2}^{n+2-2/p}(\Omega)}+\|f\|_{\bW_{p,\theta}^{n+2}(\Omega_T)}\big)\,,
			\end{align}
			where $N_2$ is independent of $f$, $u$, and $T$.
			\vspace{1mm}
			
			\item[$\mathrm{[Para]}$] Let $\cL\in\cM_T(\nu,\nu^{-1})$ for some $\nu\in(0,1]$.
			For any $n\in\bZ$, if $f\in \bW_{p,\theta+2p}^{n}(\Omega_T)$ and $u_0\in W_{p,\theta+2}^{n+2-2/p}(\Omega)$, then the equation
			$$
			u_t=\cL u+f\quad;\quad u(0)=u_0
			$$ 
			has a unique solution $u$ in $\cW_{p,\theta}^{n+2}(\Omega_T)$.
			Moreover, we have
			\begin{align}\label{2209232072}
				\|u\|_{\cW_{p,\theta}^{n+2}(\Omega_T)}\leq N_3\left(\|u_0\|_{W_{p,\theta+2}^{n+2-2/p}(\Omega)}+\|f\|_{\bW_{p,\theta}^{n+2}(\Omega_T)}\right)\,,
			\end{align}
			where $N_3$ is independent of $f$, $u$, and $T$.
		\end{enumerate}
	\end{statement}

	
	\vspace{6mm}\noindent
	\textbf{1.3.1. (Subsections \ref{fatex} and \ref{fatex2}) Domains with fat exterior.}
	\vspace{2mm}
	
	Consider a domain $\Omega$ satisfying the \textit{capacity density condition} for $\Omega^c$:
	\begin{align}\label{2302101253}
		\inf_{\substack{p\in\partial\Omega\\r>0}}\frac{\mathrm{Cap}\left(\Omega^c\cap \overline{B}_r(p),B_{2r}(p)\right)}{\mathrm{Cap}\left(\overline{B}_r(p),B_{2r}(p)\right)}\geq \epsilon_0>0\,,
	\end{align}
	where $\mathrm{Cap}(K,U)$ denotes the capacity of $K$ relative to $U$ (for the definition, see \eqref{230324942}).
	It is worth noting that this condition has been studied in the literature, including \cite{aikawa2002, Aikawa2009, AA, lewis}, and the volume density condition \eqref{230212413} is a sufficient condition for \eqref{2302101253} (see Remark~\ref{21.07.06.1}).
	
	\begin{thm}[see Corollary~\ref{22.02.19.3} with Remark~\ref{21.07.06.1}]
		Let $\Omega$ satisfy \eqref{2302101253}.
		Then there exists $\alpha>0$, which depends only on $d$ and $\epsilon_0$, such that for any $p\in(1,\infty)$ and $\theta\in\bR$ satisfying
		$$
		-2-(p-1)\alpha<\theta<-2+\alpha\,,
		$$
		Statement~\ref{230203621} $(\Omega,p,\theta)$-$\mathrm{[Pois,Heat]}$ holds.
		In addition, $N_1$ (in \eqref{220923207}) and $N_2$ (in \eqref{2209232071}) depend only on $d$, $p$, $n$, $\theta$, and $\epsilon_0$.
	\end{thm}
	
	Moreover, we also obtain a solvability result for the Poisson and heat equations in unweighted Sobolev spaces $\mathring{W}_{p}^1(\Omega)$, where $\mathring{W}_{p}^1(\Omega)$ is the closure of $C_c^{\infty}(\Omega)$ in
	$$
	W_p^1(\Omega):=\{f\in\cD'(\Omega)\,:\,\|f\|_{L_p(\Omega)}+\|\nabla f\|_{L_p(\Omega)}<\infty\}\,.
	$$
	
	\begin{thm}[see Theorem~\ref{230210356}]\label{2304141256}
		Let $\Omega$ satisfies \eqref{2302101253} and
		\begin{align*}
			\lambda\geq 0\quad \text{if}\quad d_{\Omega}:=\sup_{x\in\Omega}d(x,\partial\Omega)<\infty\quad\text{and}\quad \lambda> 0\quad \text{if}\quad d_{\Omega}=\infty\,.
		\end{align*}
		Then there exists $\epsilon\in(0,1)$ depending only on $d$, $\epsilon_0$ (in \eqref{2302101253}) such that for any $p\in(2-\epsilon,2+\epsilon)$, the following holds:
		\begin{itemize}
			\item[] For any $f^0,\,\ldots,\,f^d\in L_p(\Omega)$, the equation
			\begin{align*}
				\Delta u-\lambda u=f^0+\sum_{i=1}^dD_if^i
			\end{align*}
			is uniquely solvable in $\mathring{W}^{1}_{p}(\Omega)$.
			Moreover, we have
			\begin{align*}
				\begin{split}
					\|\nabla u\|_{L_p(\Omega)}+\frac{1}{\min\big(\lambda^{-1/2},d_\Omega\big)}\|u\|_{L_p(\Omega)}\lesssim_{d,p,\epsilon_0} \min\big(\lambda^{-1/2},d_\Omega\big)\|f^0\|_{L_p(\Omega)}+\sum_{i=1}^d\|f^i\|_{L_p(\Omega)}\,.
				\end{split}
			\end{align*}
		\end{itemize}
	\end{thm}

	A counterpart of Theorem~\ref{2304141256} for parabolic equations is provided in Theorem~\ref{230210357}.

	\vspace{6mm}\noindent
	\textbf{1.3.2. (Subsection \ref{0062}) Domains with thin exterior.}
	\vspace{2mm}
	
	$\dim_{\cA} \Omega^c$ denote the \textit{Aikawa dimension} of $\Omega^c$, which is defined as the infimum of $\beta\geq 0$ such that
	$$
	\sup_{p\in\Omega^c,r>0}\frac{1}{r^{\beta}}\int_{B(p,r)}\rho(x)^{-d+\beta}\dd x\leq A_{\beta}<\infty\,,
	$$		
	with considering $0^{-1}=\infty$.
	We consider a domain $\Omega$ for which $\mathrm{dim}_{\cA}(\Omega^c)<d-2$.
	A relation between the Aikawa dimension, the Hausdorff dimension, and the Assouad dimension is mentioned in Remark~\ref{22.02.24.1}.
	For instance, for a Cantor set $C\subset \{(t,0,0)\,:\,0\leq t\leq 1\}$, $\Omega:=\bR^3\setminus C$ satisfies
	$$
	\mathrm{dim}_{\cA}(\Omega^c)=\text{Hausdorff dimension of $C$}=\log_3 2<3-2\,.
	$$

	\begin{thm}[see Corollary \ref{22.02.19.300}]
		Let $d\geq 3$ and  $\dim_{\cA}(\Omega^c)=:\beta_0<d-2$.
		For any $p\in(1,\infty)$ and $\theta\in\bR$ satisfying 
		$$
		-d+\beta_0<\theta<(p-1)(d-\beta_0)-2p\,,
		$$
		Statement \ref{230203621} $(\Omega,p,\theta)$-$\mathrm{[Pois,Heat]}$ holds.
		In addition, $N_1$ (in \eqref{220923207}) and $N_2$ (in \eqref{2209232071}) depend only on $d$, $p$, $n$, $\theta$, $\beta_0$, and $\{A_{\beta}\}_{\beta>\beta_0}$.
		
	\end{thm}
	
	\vspace{6mm}\noindent
	\textbf{1.3.3. (Subsection \ref{0071}) Domains with exterior cone condition.}
	\vspace{2mm}
	
	For $\delta\in[0,\pi/2)$ and $R>0$, $\Omega$ is said to satisfy the \textit{exterior $(\delta,R)$-cone condition} if, for every $p\in\partial\Omega$, there exists a unit vector $e_p\in\bR^d$ such that
	$$
	\{x\in\bR^d\,:\,(x-p)\cdot e_p\geq |x-p|\cos\delta\,\,,\,\,|x-p|<R\}\subset \Omega^c\,;
	$$
	when $\delta=0$, this condition is often called the exterior $R$-line segment condition.
	Examples for this condition are given in Example~\ref{220816845} and illustrated in Figure \ref{230212745}.
	
	Given $\delta>0$, we denote
	$$
	\lambda_{\delta}=-\frac{d-2}{2}+\sqrt{\Big(\frac{d-2}{2}\Big)^2+\Lambda_{\delta}}\,,
	$$
	where $\Lambda_{\delta}>0$ is the first eigenvalue for Dirichlet spherical Laplacian on
	$$
	\{\sigma=(\sigma_1,\ldots,\sigma_d)\in \partial B_1(0)\,:\,\sigma_1<\cos\delta\}\,.
	$$
	When $d=2$ and $\delta=0$, we set $\lambda_{\delta}=1/2$.
	We provide information on $\Lambda_{\delta}$ in \eqref{230212803} and Proposition~\ref{230212757}.
	Note that $0<\lambda_{\delta}<1$ for $0<\delta<\pi/2$, and
	\begin{align}\label{230214226}
		\lim_{\delta\searrow 0}\lambda_\delta=0\quad\text{if $d\geq 3$}\quad\text{and}\quad \lim_{\delta\searrow 0}\lambda_\delta=\frac{1}{2}\quad\text{if $d=2$}.
	\end{align}

	\begin{thm}[see Corollary \ref{221026914}]\label{2302141109}
		Let $\delta\in(0,\pi)$ if $d\geq 3$, and $\delta\in[0,\pi)$ if $d= 2$.
		Assume that $\Omega$ satisfies the $(\delta,R)$-exterior cone condition, where
		\begin{alignat*}{2}
			&0<R\leq \infty &&\quad \text{if\,\,\,$\Omega$\,\,\,is\,\,\,bounded}\,;\\
			&R=\infty&&\quad\text{if\,\,\,$\Omega$\,\,\,is\,\,\,unbounded.}
		\end{alignat*}
		Then, for any $p\in(1,\infty)$ and $\theta\in\bR$ satisfying 
		$$
		-\lambda_{\delta}(p-1)-2<\theta<\lambda_{\delta}-2\,,
		$$
		Statement \ref{230203621} $(\Omega,p,\theta)$-$\mathrm{[Pois,Heat]}$ holds.
		In addition, if $\Omega$ is bounded, then $N_1$ (in \eqref{220923207}) and $N_2$ (in \eqref{2209232071}) depend only on $d$, $p$, $n$, $\theta$, $\delta$, and $\mathrm{diam}(\Omega)/R$.
		If $\Omega$ is unbounded (and $R=\infty$), then $N_1$ and $N_2$ depend only on the same parameters, except for $\mathrm{diam}(\Omega)/R$.
	\end{thm}
	
	Corollary~\ref{221026914} deals with the exterior cone condition, which can be considered as a generalization of the Lipschitz boundary condition.
	One of the most well-known studies on Lipschitz domains is the work of Jerison and Kenig \cite{kenig}.
	It should be noted that Corollary~\ref{221026914} and \cite[Theorems 1.1, 1.3]{kenig} address different aspects of the Poisson equation in non-smooth domains, and hence cannot be directly compared.
	
	For instance, let $\Omega\subset \bR^2$ be a bounded domain satisfying the exterior $(0,R)$-cone condition, $R>0$.
	Theorem~\ref{2302141109} guarantees the unique solvability of equation \eqref{230214224} in $\mathring{W}_p^1(\Omega)$, for $p\in[3/2.3]$ and $f\in W_p^{-1}(\Omega)$ (see Remark~\ref{230215958}.(1) and \eqref{230214226}). 
	On the other hand, in Theorem 1.3 of \cite{kenig}, Jerison and Kenig showed that if $\Omega\subset \bR^2$ is a bounded Lipschitz domain, then the unique solvability is ensured for $p\in[4/3,4]$. 
	Therefore, for bounded Lipschitz domains, the range of $p$ provided in \cite{kenig} is broader than what is implied by Theorem \ref{2302141109}.
	However, the class of domains considered in Theorem \ref{2302141109} is broader than the class of Lipschitz domains.
	
	Notably, the results in \cite{kenig} are more comprehensive than what has been described above, especially regarding the regularity of solutions.
	To compare \cite{kenig} with Corollary \ref{221026914} in general cases, we refer the reader to the following  remark on function spaces:
	
	\begin{remark}\label{230215958}
		This remark explains the relation between the function spaces $W_{p,\theta}^k(\Omega)$ (and $H_{p,\theta+d}^\gamma(\Omega)$ in Definition~\ref{220610533}) and other types of Sobolev spaces.
		
		\begin{enumerate}
			\item Recall the definition of $\mathring{W}_p^1(\Omega)$ and $W_p^{-1}(\Omega)$ in \eqref{230214355}.
			If $\Omega$ is a bounded domain satisfying \eqref{230212413}, then there exists $N=N(\Omega)>0$ such that
			$$
			\int_{\Omega}|f|^p+\Big|\frac{f}{\rho}\Big|^p\dd x\leq N\int_{\Omega}|\nabla f|^p\dd x\qquad\forall\,\,f\in C_c^{\infty}(\Omega)
			$$
			(see, \textit{e.g.}, \cite[(7.44)]{GT} and \cite[page 60]{ward}).
			This implies that $\mathring{W}_p^1(\Omega)=W_{p,-p}^1(\Omega)$.
			Furthermore, by Remark~\ref{230214208}.(2), we also have $W_p^{-1}(\Omega)=W_{p,p}^{-1}(\Omega)$.
			
			\item Let $\Omega$ be a bounded Lipschitz domain, and let $L_s^p(\Omega)$ and $L_{s,\mathrm{o}}^p$ denote function spaces introduced in \cite[Section 2]{kenig}, where $p\in(1,\infty)$ is the integrability parameter, and $s\in\bR$ is the regularity parameter.
			To avoid any ambiguity, we use the notation $\mathring{L}_s^p(\Omega)$ to refer to $L_{s,\mathrm{o}}^p$.
			We recall that for any $k=0,\,1,\,2,\,\ldots$, 
			\begin{alignat*}{2}
			&L_k^p(\Omega)&&=W_p^k(\Omega):=\{f\in\cD'\Omega)\,:\,\sum_{i=0}^k\|D^if\|_p<\infty\}\,;\\
			&\mathring{L}_k^p(\Omega)&&=\text{the closure of $C_c^{\infty}(\Omega)$ in $L_k^p(\Omega)$}\,;\\
			&\mathring{L}_{-k}^p(\Omega)&&=\text{the dual of $\mathring{L}_{k}^{p/(p-1)}(\Omega)$}\,.
			\end{alignat*}
		In addition, $H_{p,\theta+d}^k(\Omega)=W_{p,\theta}^k(\Omega)$, where $H_{p,\theta+d}^k(\Omega)$ is the space defined in Definition~\ref{220610533}.
		
			Since $C_c^{\infty}(\Omega)$ is dense in $H_{p,\theta+d}^k$, we have
			\begin{align*}
				H_{p,d}^0(\Omega)=\mathring{L}_{0}^p\quad\text{and}\quad H_{p,d-kp}^k(\Omega)\subset \mathring{L}_k^p(\Omega)\qquad \forall\,\,k\in\bN\,,\,\,p\in(1,\infty)\,.
			\end{align*}
			Using the interpolation properties for $L_s^p(\Omega)$ and $H_{p,\theta}^{\gamma}(\Omega)$ (see \cite[Corollary 2.10]{kenig} and Proposition~\ref{220527502111}.(3), respectively),
			we obtain that for any $s\geq 0$, $H_{p,d-sp}^s(\Omega)\subset \mathring{L}_s^p(\Omega)$.
			We also obtain that for $s<0$, $L_{s}^p(\Omega)\subset H_{p,d-sp}^{s}(\Omega)$.
			Indeedn, $L_{s}^p(\Omega)$ and $H_{p,d-sp}^{s}(\Omega)$ are the dual spaces of $\mathring{L}_{-s}^{p'}(\Omega)$ and $H_{p,d+sp'}^{-s}(\Omega)$, respectively, where $p'=p/(p-1)$.
		\end{enumerate}
		
	\end{remark}
	%
	
	\vspace{6mm}\noindent
	\textbf{1.3.4. (Subsection \ref{convex}) Convex domain.}
	\vspace{2mm}
	
	$\Omega$ is said to be \text{convex} if for any $x,\,y\in\Omega$ and $t\in[0,1]$, $(1-t)x+ty\in\Omega$.
	
	\begin{thm}[see Corollary \ref{2208131026}]\label{221005646}
		Let $d\geq 2$ and $1<p<\infty$.	
		Suppose that $\Omega$ is convex (not necessarily bounded).
		For any $p\in(1,\infty))$ and $\theta\in\bR$ satisfying
		$$
		-p-1<\theta<-1\,,
		$$
		Statement \ref{230203621} $(\Omega,p,\theta)$-$\mathrm{[Pois,Para]}$ holds.
		In addition, $N_1$ (in \eqref{220923207}) depends only on $d,\,p,\,n,\,\theta$, and $N_3$ (in \eqref{2209232072}) depends only on the same parameters and $\nu$;
		in particular, $N_1$ and $N_3$ are independent of $\Omega$.
	\end{thm}
	
	Adolfsson \cite{convexAdo} and Fromm \cite{convexFromm} have established the solvability of the Poisson equation in \textit{bounded} convex domains.
	In terms of unweighted estimates for higher regularity, their result is more general than Corollary~\ref{2208131026}.
	However, Corollary~\ref{2208131026} considers convex domains that are not necessarily bounded and also provides solvability results in \textit{weighted} Sobolev spaces; when comparing these results with Corollary~\ref{2208131026}, it is useful to note Remark~\ref{230215958} and that bounded convex domains are Lipschitz domains (see, \textit{e.g.}, \cite[Corollary 1.2.2.3]{nonsmoothGris}).

	Combining the results of Corollary~\ref{2208131026} with \cite[Theorem 3.1.2.1]{nonsmoothGris} may yield results similar to \cite[Corollary 1]{convexFromm}.
	However, we do not pursue this direction in this paper.

	\vspace{6mm}\noindent
	\textbf{1.3.5. (Subsection \ref{ERD}) Totally vanishing exterior Reifenberg condition.}
	\vspace{2mm}
	
	This subsubsection discusses the \textit{totally vanishing exterior Reifenberg} condition (abbreviate to `$\langle\mathrm{TVER}\rangle$'), which is a generalization of the concept of bounded vanishing Reifenberg domains introduced below \eqref{22.02.26.41}.
	
	To clarify the main point of $\langle\mathrm{TVER}\rangle$, in Definition \ref{221013228}, we provide a simplified version of the concept in Definition~\ref{2209151117}.(3); $\langle\mathrm{TVER}\rangle$ in Definition \ref{221013228} is a sufficient condition for the totally vanishing exterior Reifenberg condition in Definition~\ref{2209151117}.(3).
	(Figure~\ref{230212856} illustrates the differences between the vanishing Reifenberg condition, $\langle\mathrm{TVER}\rangle$ in Definition \ref{221013228}, and the totally vanishing exterior Reifenberg condition in Definition~\ref{2209151117}.(3).)
	
	\begin{defn}\label{221013228}
		We say that $\Omega$ satisfies the \textit{totally vanishing exterior Reifenberg condition} (abbreviate to `$\langle\mathrm{TVER}\rangle$') if for any $\delta\in(0,1)$, there exist $R_{0,\delta},\,R_{\infty,\delta}>0$ satisfying the following: for every $p\in\partial \Omega$ and $r\in\bR_+$ with $r\leq R_{0,\delta}$ or $r\geq R_{\infty,\delta}$, there exists a unit vector $e_{p,r}\in\bR^d$ such that
		\begin{align}\label{230203624}
			\Omega\cap B_r(p)\subset \{x\in B_r(p)\,:\,(x-p)\cdot e_{p,r}<\delta r\}\,.
		\end{align}
	\end{defn}
\vspace{1mm}

	As shown in Example \ref{220910305}, $\langle\mathrm{TVER}\rangle$ is fulfilled by bounded domains of the following types: the vanishing Reifenberg domains, $C^1$-domains, domains with the exterior ball condition, and finite intersections of Reifenberg domains.
	Moreover, several unbounded domains also satisfy $\langle\mathrm{TVER}\rangle$ (see Proposition \ref{220918248}).
	
	We now present our result for the Poisson and heat equations in domains satisfying $\langle \mathrm{TVER} \rangle$.

	\begin{thm}[see Corollary \ref{22.07.17.109}]\label{230214437}
		Suppose that $\Omega$ satisfies $\langle\mathrm{TVER}\rangle$.
		For any $p\in(1,\infty)$ and $\theta\in\bR$ satisfying
		$$
		-p-1<\theta<-1\,,
		$$
		$\mathrm{Statement}(\Omega,p,\theta)$-$\mathrm{[Pois,Para]}$ holds.
		In addition, $N_1$ (in \eqref{220923207}) depends only on $d,\,p,\,n,\,\theta$, and $\big\{R_{0,\delta}/R_{\infty,\delta}\big\}_{\delta\in(0,1]}$, and $N_3$ (in \eqref{2209232072}) depends only on the same parameters and $\nu$.
	\end{thm}

	The Poisson and heat equations in bounded vanishing Reifenberg domains have been investigated in the literature, such as the works of Byun and Wang \cite{Relliptic, Rparabolic}, Choi and Kim \cite{Reifweight2}, and Dong and Kim \cite{DongKim}.
	More specifically, these studies focus on the elliptic and parabolic equations with variable coefficients, and the results in \cite{Reifweight2, DongKim} also provide weighted $L_p$-estimates for Muckenhoupt $A_p$-weight functions.
	It is worth noting, however, that these studies mostly dealt with bounded vanishing Reifenberg domains.
	In contrast, Theorem \ref{230214437} considers the class of domains satisfying $\langle\mathrm{TVER}\rangle$, which includes bounded vanishing Reifenberg domains.

	\vspace{6mm}\noindent
	\textbf{1.3.6. (Subsection \ref{0074}) Conic domain.}\hspace{2mm}
	\vspace{2mm}
	
	Let $\cM$ be a subdomain of $\bS^{d-1}:=\{\,x\in\bR^d\,:\,|x|=1\}$ and $\Omega$ be a conic domain generated by $\cM$, \textit{i.e.},
	$$
	\Omega=\{r\sigma\,:\,r>0\,\,,\,\,\,\,\sigma\in\cM\}\,.
	$$
	We consider $\cM$ satisfying Assumption~\ref{2207111151}; this assumption is satisfied if $\Omega$ is a smooth cone or polyhedral cone (see Proposition~\ref{230413525}.(3)).
	
	For $r\in(0,1]$, we denote
	$$
	B_r^{\Omega}:=\Omega\cap B_r(0)\subset \bR^d\quad\text{and}\quad Q_r^{\Omega}=(1-r^2,1]\times B_r^{\Omega}\,.
	$$
	$\Lambda_0>0$ represents the first Dirichlet eigenvalue for spherical laplacian $\Delta_{\bS}$ on $\cM$;
	for the definition and more information of $\Delta_{\bS}$ and $\Lambda_0$, see \eqref{230203650} and Proposition~\ref{230413525}.(1), respectively.
	We define
	$$
	\lambda_0=-\frac{d-2}{2}+\sqrt{\Lambda_0+\Big(\frac{d-2}{2}\Big)^2}\,.
	$$ 
	
	We obtain the following pointwise estimate for homogeneous solution to the heat equation in $\Omega$:
	\begin{thm}[see Theorem~\ref{2208221223} and Remark~\ref{2301111204}]\label{230218744}
		Let $\cM\subset \bS^{d-1}\,(d\geq 2)$ satisfy Assuption~\ref{2207111151}, and suppose that $u\in C^{\infty}(Q_1^{\Omega})$ satisfies that
		\begin{align*}
			&u_t=\Delta u\quad\text{in}\quad Q_1^{\Omega}\,;\\
			&\lim_{(t,x)\rightarrow (t_0,x_0)}u(t,x)=0\quad\text{whenever}\quad 0<t_0\leq 1\,\,,\,\,x_0\in(\partial\Omega)\cap B_1.
		\end{align*}
		Then for any $\lambda\in(0,\lambda_0)$ and $R \in(0,1)$,
		\begin{align}\label{230212805}
			|u(t,x)|\leq N\Big(\sup_{Q_1^{\Omega}}|u|\Big)|x|^{\lambda}\qquad \forall\,\, (t,x)\in Q_R^{\Omega}\,,
		\end{align}
		where $N=N(\Omega,\epsilon,R)>0$.
	\end{thm}
	
	When $\Omega$ is a smooth cone, \textit{i.e.}, $\cM\subset \bS^{d-1}$ has a smooth boundary, estimate \eqref{230212805} is already established in the literature (see, \textit{e.g.}, \cite[Theorem 2.1.3]{Kozlov}).
	In Lemma 3.8 of \cite{Kozlov}, an estimate of the same type as \eqref{230212805} was employed to obtain pointwise estimates for Green functions for parabolic equations in smooth cones.
	Following the approach in \cite{Kozlov}, we anticipate that Theorem~\ref{230218744} can be used to derive estimates of the heat kernels for polyhedral cones, as Assumption~\ref{2207111151} holds for such cones.
	However, we leave the details for future work.

	\subsection{Plan for the paper and notation}\label{0004}
	We provide an outline of Sections~\ref{0030} - \ref{app2.} and Appendix~\ref{008}.
	
	\vspace{2mm}
	
	In Section~\ref{0030}, we present key estimates associated with superharmonic functions and provide weighted $L_p$-estimates for zeroth-order derivatives of solutions to the Poisson equation.
	
	\vspace{1mm}
	
	Section~\ref{0040} is devoted to function spaces for the Poisson equation and the solvability of this equation.
	Subsection~\ref{0041} introduces the notions of Harnack functions and regular Harnack functions.
	Subsection~\ref{0042} presents the weighted Sobolev spaces $\Psi H_{p,\theta}^{\gamma}(\Omega)$, where $\Psi$ is a regular Harnack function and $H_{p,\theta}^{\gamma}(\Omega)$ are the Sobolev spaces introduced in \cite{Krylov1999-1,Lo1}.
	Additionally, we provide properties of $\Psi H_{p,\theta}^{\gamma}(\Omega)$ in this subsection.
	In Subsection~\ref{0043}, we prove the unique solvability of the Poisson equation in the context of $\Psi H_{p,\theta}^{\gamma}(\Omega)$.
	
	\vspace{1mm}
	
	Section~\ref{0050} focuses on the heat equation.
	Subsection~\ref{0051} presents results for the heat equation corresponding to Section~\ref{0030}, while 
	Subsection~\ref{0052} introduces the function spaces for parabolic equations.
	In Subsection~\ref{0053}, we prove the unique solvability of parabolic equations.
	
	\vspace{1mm}
	In Section~\ref{app.}, we begin by exploring the relationship between the Hardy inequality and dimensional notions.
	We also recall classical results for superharmonic functions.
	Subsections~\ref{fatex} and \ref{fatex2} present results for domains introduced in Subsubsec. 1.3.1, while Subsection \ref{0062} provides results for domains introduced in Subsubsec. 1.3.2.
	\vspace{1mm}
	
	Section~\ref{app2.} sequentially provides results for domains introduced in Subsubsections~1.3.3 - 1.3.6.
	
	\vspace{1mm}
	Appendix~\ref{008} discusses the function spaces $H_{p,\theta}^{\gamma}(\Omega)$, $\Psi H_{p,\theta}^{\gamma}(\Omega)$, $B_{p,\theta}^{\gamma}(\Omega)$, and $\Psi B_{p,\theta}^{\gamma}(\Omega)$.
	Appendix~\ref{0082} complies properties of $H_{p,\theta}^{\gamma}(\Omega)$ and $B_{p,\theta}^{\gamma}(\Omega)$, based on the analysis in \cite{Krylov1999-1,Lo1}.
	In Appendix~\ref{2302131019}, we provide auxiliary results used in the proofs of Lemmas~\ref{22.02.16.1} and \ref{21.09.29.4}.(5).
	Finally, Appendix~\ref{0083} offers equivalent norms for $\Psi H_{p,\theta}^{n}(\Omega)$ and $\Psi B_{p,\theta}^{n+s}(\Omega)$, where $n\in\bN_0$ and $s\in(0,1)$.
	
	\vspace{4mm}
	
	\noindent\textbf{Notations.}
	\begin{itemize}
		\item We use $:=$ to denote a definition.

		\item  Throughout the paper, the letter $N$ denotes a finite positive constant which may have different values along the argument  while the dependence  will be informed;  $N=N(a,b,\cdots)$, meaning that  $N$ depends only on the parameters inside the parentheses.
		
		\item  $A\lesssim_{a,b,\ldots} B$ means that $A\leq N(a,b,\ldots)B$, and $A\simeq _{a,b,\ldots}B$ means that $A\lesssim_{a,b,\ldots} B$ and $B\lesssim_{a,b,\ldots} A$.
		
		\item $a \vee b :=\max\{a,b\}$, $a \wedge b :=\min\{a,b\}$.

		
		\item $\bR^d$ stands for the $d$-dimensional Euclidean space of points $x=(x^1,\ldots,x^d)$, and  $\bR^d_+:=\{x=(x^1,\ldots,x^d): x^1>0\}$.

		\item $\bS^{d-1}$ denotes $\bS^{d-1}=\big\{(x_1,\ldots,x_d)\in\bR^d\,:\,\sqrt{(x_1)^2+\cdots+(x_d)^2}=1\big\}$.
		
		\item  $\bN$ denotes the natural number system, $\bN_0=\{0\}\cup \bN$,  and   $\bZ$ denotes the set of integers.
		
		\item For $x=(x_1,\ldots,x_d)$, $y=(y_1,\ldots,y_d)$ in $\bR^d$,  $x\cdot y:=(x,y)_d :=\sum^d_{i=1}x_iy_i$ denotes the standard inner product. $|x|$ denotes $\sqrt{x\cdot x}$.
		
		\item For an open set $\domain$ in $\bR^d$, $\partial \domain$ denotes the boundary of $\domain$, $\overline{\cO}\setminus\cO$.
		
		\item A non-empty connected open set is called a domain.
		
		\item For a set $E\subset \bR^d$, $d(x,E)$ denotes the distance between a point $x$ and a set $\domain\in\bR^d$, defined by $\inf_{y\in E}|x-y|$.
		For two sets $E_1,\,E_2\subset \bR^d$, $d(E_1,E_2):=\inf_{x\in E_1}d(x,E_2)$.

		\item For a set $E\subset \bR^d$, $1_E$ denotes the indicator function on $E$ so that $1_E(x)=1$ if $x\in E$, and $1_E(x)=0$ if $x\notin E$.

		\item  For a measure space $(A, \cA, \mu)$ and a measurable function $f:A\rightarrow [-\infty,\infty]$,
		\begin{align*}
			&\underset{A}{\mathrm{ess\,sup}}\,f:=\underset{x\in A}{\mathrm{ess\,sup}}\,f(x):=\inf\{a\in[-\infty,\infty]\,:\,\mu\big(\{x\in A\,:\,f(x)> a\}\big)=0\}\,,\\
			&\underset{A}{\mathrm{ess\,inf}}\,f:=\underset{x\in A}{\mathrm{ess\,inf}}\,f(x):=-\underset{A}{\mathrm{ess\,sup}}\,(-f)\,.
		\end{align*}

		\item  For a measure space $(A, \cA, \mu)$, a Banach space $(B,\|\cdot\|_B)$, and $p\in[1,\infty]$, we write $L_p(A,\cA, \mu;B)$ for the collection of all $B$-valued $\overline{\cA}$-measurable functions $f$ such that
		\begin{alignat*}{2}
		&\|f\|^p_{L_p(A,\cA,\mu;B)}:=\int_{A} \lVert f\rVert^p_{B} \dd \mu<\infty&&\qquad\text{if}\quad p\in[1,\infty)\,;\\
		&\|f\|_{L_\infty(A,\cA,\mu;B)}:=\underset{x\in A}{\mathrm{ess\,sup}}\,\|f(x)\|_B<\infty&&\qquad\text{if}\quad  p=\infty\,.
		\end{alignat*}
		Here, $\bar{\cA}$ is the completion of $\cA$ with respect to $\mu$.  
		We will drop $\cA$ or $\mu$ or even $B$ in $L_p(A,\cA, \mu;B)$ when they are obvious in the context. 
		
		\item  For  any multi-index $\alpha=(\alpha_1,\ldots,\alpha_d)$, $\alpha_i\in \{0\}\cup \bN$,   
		$$
		\partial_tf:=\frac{\partial f}{\partial t}, \quad f_{x^i}:=D_if:=\frac{\partial f}{\partial x^i}, \quad D^{\alpha}f(x):=D^{\alpha_d}_d\cdots D^{\alpha_1}_1f(x).
		$$
		We denote $|\alpha|:=\sum_{i=1}^d \alpha_i$.  For the second order derivatives we denote $D_jD_if$ by $D_{ij}f$. We often use the notation 
		$|gf_x|^p$ for $|g|^p\sum_i|D_if|^p$ and $|gf_{xx}|^p$ for $|g|^p\sum_{i,j}|D_{ij}f|^p$.  We also use $D^m f$ to denote arbitrary partial derivatives  of order $m$ with respect to the space variable.
		
		\item $\Delta f:=\sum_{i=1}^d D_{ii}f$ denotes the Laplacian for a function $f$ defined on $\cO$.
		
		\item For $n\in \{0\}\cup \bN$,  $W^n_p(\domain):=\{f: \sum_{|\alpha|\le n}\int_{\domain}|D^{\alpha}f|^p \dd x<\infty\}$, the Sobolev space.

		\item  For an open set $\domain\subseteq\bR^d$ and a Banach space $B$, $C(\domain;B)$ denotes the set of all $B$-valued continuous functions $f$ in $\domain$ such that $|f|_{C(\domain;B)}:=\sup_{\domain}\|f\|_B<\infty$.
		For $n\in\bN$, by $C^n(\cO;B)$ we denote the set of all $f:\cO\rightarrow B$ which is strongly $n$-times continuously diffrentiable on $\cO$ with
		$$
		\|f\|_{C^n(\cO;B)}:=\sum_{k=0}^n\Big(\sup_{x\in \Omega}\|D^kf(x)\|_B\Big)<\infty\,.
		$$
		For $n\in\bN_0$ and $\alpha\in (0,1]$, by $C^{n,\alpha}(\cO;B)$ we denote the set of all $f\in C^n(\cO;B)$ such that
		\begin{align*}
			\|f\|_{C^{n,\alpha}(\cO;B)}\,&:=\|f\|_{C^n(\cO;B)}+[f]_{C^{n,\alpha}(\domain;B)}\\
			&:=\|f\|_{C^n(\cO;B)}+\sup_{x\neq y\in \domain} \frac{\|D^nf(x)-D^nf(y)\|_B}{|x-y|^{\alpha}}<\infty.
		\end{align*}

		\item  $supp(f)$ denotes the support of the function $f$ defined as the closure of $\{x\,:\,f(x)\neq 0\}$.
		For an open set $\domain\subseteq\bR^d$, $C^{\infty}_c(\domain)$ is the the space of infinitely differentiable functions $f$ for which $\mathrm{supp}(f)$ is a compact subset of $\domain$.
		Also, $C^{\infty}(\domain)$ denotes the the space of infinitely differentiable functions in $\domain$.
		
		\item Let $\cO\subseteq \bR^d$ be an open set.
		For $X(\cO)=L_p(\cO)$ or $C^n(\cO)$ or $C^{n,\alpha}(\cO)$, $X_{loc}(\cO)$ denotes the set of all function $f$ on $\cO$ such that $f\zeta\in X(\cO)$ for all $\zeta\in C_c^{\infty}(\cO)$. 
		
		\item For an open set $\cO\subseteq \bR^d$, $\cD'(\cO)$ denotes the set of all distrubitions on $\cO$, which is the dual of $C_c^{\infty}(\Omega)$.
		If $f$ is a distribution with the reference domain $\domain$, then the expression $\la f,\varphi\ra $, $\varphi\in C^{\infty}_c(\domain)$, will denote the evaluation of $f$ with the test function $\varphi$.
		
		\item For $F\in\cD'(\Omega)$, the notation $F\geq 0$ denotes that $\la F,\zeta\ra\geq 0$ for any $\zeta\in C_c^{\infty}(\Omega)$ with $\zeta\geq 0$.

		

		
		
	\end{itemize}
	
\vspace{2mm}

\mysection{Key estimates for the Poisson equation}\label{0030}

This section aims to obtain estimates for the zeroth-order derivatives (the function itself) of solutions to the Poisson equation
\begin{align*}
	\Delta u-\lambda u=f\quad\text{in}\,\,\,\Omega\quad;\quad u|_{\partial\Omega}=0\,,
\end{align*}
where $\lambda\geq 0$ and $\Omega$ admits the Hardy inequality (see Theorem~\ref{21.05.13.2}).
In this estimates, superharmonic functions are used as weight functions.
We begin with the definition and elementary properties of superharmonic functions.

\begin{defn}\label{21.01.19.1}
	\,\,
	
	\begin{enumerate} 
		\item A function $\phi\in L^1_{loc}(\Omega)$ is said to be \textit{superharmonic} if $\Delta \phi\leq 0$ in the sense of distribution on $\Omega$, \textit{i.e.},
		$$
		\int_{\Omega}\phi\,\Delta \zeta\, \dd x\leq 0\qquad\forall\,\,\zeta\in C_c^{\infty}(\Omega)\,.
		$$
		
		\item A function $\phi:\Omega\rightarrow (-\infty,+\infty]$ is called a \textit{classical superharmonic function} if the following are satisfied:
		\begin{enumerate}
			\item $\phi$ is lower semi-continuous on $\Omega$.
			\item For any $x\in \Omega$ and $r>0$ satisfying $\overline{B}_r(x)\subset \Omega$, 
			\begin{align*}
				\phi(x)\geq \frac{1}{m\big(B_r(x)\big)}\int_{B_r(x)}\phi(y)\dd y\,,
			\end{align*}
			where $m$ is the Lebesgue measure on $\bR^d$. 
			\item $\phi\not\equiv +\infty$ on each connected component of $\Omega$.
		\end{enumerate}
	\end{enumerate}
\end{defn}
\vspace{1mm}

Recall that $\phi$ is said to be $harmonic$ if both $\phi$ and $-\phi$ are classical superharmonic functions.
\begin{remark}
	Equivalent definitions of classical superharmonic functions are introduced in \cite[Definition 3.1.2, Theorem 3.2.2]{AG}.
	It follows that if $\phi$ is a classical superharmonic function on a neighborhood of each point in $\Omega$, then $\phi$ is a classical superharmonic function on $\Omega$.
\end{remark}
%

\begin{remark}\label{22.01.25.1}
	It is well known that every classical superharmonic function is superharmonic.
	Conversely, if $\phi$ is a superharmonic function, then there exists a classical superharmonic function $\phi_0$ such that $\phi=\phi_0$ almost everywhere on $\Omega$.
	They can be found in \cite[Theorem 4.3.2]{AG} and \cite[Proposition 30.6]{TF}, respectively.
\end{remark}

\begin{prop}\label{21.04.23.3}
	Let $\phi$ be a classical superharmonic function on $\Omega$.
	
	\begin{enumerate}
		\item If $\phi$ is twice continuously differentiable, then $\Delta\phi\leq 0$.
		
		\item $\phi$ is locally integrable on $\Omega$.
		
		\item For any compact set $K\subset \Omega$, $\phi$ has the minumum value on $K$. 
		
		\item 
		For $\epsilon>0$, put
		\begin{align}\label{21.04.23.1}
			\phi^{(\epsilon)}(x)=\int_{B_1}\big(\phi 1_{\Omega}\big)(x-\epsilon y)\widetilde{\zeta}(y)\dd y\,,
		\end{align}
		where
		\begin{align*}
			\widetilde{\zeta}(x):=N_0\,e^{-1/(1-|x|^2)}1_{B_1(0)}(x)
		\end{align*}
		and $N_0$ is a positive constant such that $\int_{\bR^d}\widetilde{\zeta}\dd x=1$.
		Then for any compact set $K\subset \Omega$ and $\epsilon\in\big(0,d(K,\Omega^c)\big)$, the following hold:
		\begin{enumerate}
			\item $\phi^{(\epsilon)}$ is infinitely smooth on $\bR^d$.
			\item $\phi^{(\epsilon)}$ is a classical superharmonic function on $K^{\circ}$.
			\item For any $x\in K$, $\phi^{(\epsilon)}(x)\nearrow \phi(x)$ as $\epsilon\searrow 0$.
		\end{enumerate}
	\end{enumerate}
\end{prop}
For this proposition, (1) - (3) follow from Definition~\ref{21.01.19.1} and Remark~\ref{22.01.25.1}, and (4) can be found in \cite[Theorem 3.3.3]{AG}.

\begin{remark}\label{2303081145}
	If $\phi$ is a positive classical superharmonic function on $\Omega$ and $c\leq 1$, then $\phi^c$ is locally integrable on $\Omega$.
	Indeed, for any comapct set $K\subset \Omega$, if $c\in(0,1]$, then by Proposition~\ref{21.04.23.3}.(2),
	$$
	\int_K\phi^c\dd x\leq |K|^{1-c}\big(\int_K\phi\dd x\big)^c<\infty\,.
	$$
	If $c\leq 0$, then by Proposition~\ref{21.04.23.3}.(3), $\max\limits_K(\phi^c)= \big(\min\limits_K\phi\big)^c<\infty$.
\end{remark}

\begin{lemma}\label{21.04.23.5}
	Let $\phi$ be a positive classical superharmonic function on $\Omega$.
	If $f\in L^1(\Omega)$ and  $\mathrm{supp}(f)$ is a compact subset of $\Omega$, then for any $c\in\bR$,
	\begin{align}\label{220429449}
		\lim_{\epsilon\rightarrow 0}\int_{\Omega}|f|\big(\phi^{(\epsilon)}\big)^c\dd x= \int_{\Omega}|f|\phi^c\dd x\,,
	\end{align}
	where $\phi^{(\epsilon)}$ is defined in \eqref{21.04.23.1}.
\end{lemma}

\begin{proof}
	Take a bounded open set $U$ such that $\text{supp}(f)\subset U$ and $\overline{U}\subset\Omega$.
	Proposition~\ref{21.04.23.3} implies that for $0<\epsilon<d(\text{supp}(f),U^c)$ and $x\in \text{supp}(f)$,
	\begin{align*}
		\phi^{(\epsilon)}(x)\nearrow\phi(x)\,\,\,\,\text{as}\,\,\,\,\epsilon\searrow 0\,\,,\,\,\,\,\text{and}\quad 0<\min_{\overline{U}}\phi=:m\leq \phi^{(\epsilon)}(x)\,. 
	\end{align*}
	If $c\geq 0$, then \eqref{220429449} follows from the monotone convergence theorem.
	If $c<0$, then $|f|\big(\phi^{(\epsilon)}\big)^c\leq m^{c}|f|$, and therefore \eqref{220429449} follows from the Lebesgue dominated convergence theorem.
\end{proof}

\begin{remark}\label{21.11.15.2}
	Under the assumption in Lemma~\ref{21.04.23.5}, we additionally assume that $c\leq 1$ and $f$ is bounded.
	Then $f\phi^c$ is integrable on $\Omega$ (see Remark~\ref{2303081145}).
	By applying Lemma~\ref{21.04.23.5} with $f$ replaced by $\max(f,0)$ and $\max(-f,0)$, we obtain
	$$
	\lim_{\epsilon\rightarrow 0}\int_{\Omega}f\,\big(\phi^{(\epsilon)}\big)^c\dd x= \int_{\Omega}f\,\phi^c\dd x.
	$$
\end{remark}

\vspace{1mm}

%
%

The following is the key lemma of this section.

\begin{lemma}\label{03.30}
	Let $p\in(1,\infty)$ and $c\in(-p+1,1)$ and suppose that $u\in C(\Omega)$ satisfies that
	\begin{equation}\label{22.01.25.2}
		\begin{gathered}
			\mathrm{supp}(u)\,\,\text{is a compact subset of}\,\,\,\Omega\,,\\
			u\in C_{\mathrm{loc}}^2(\{x\in\Omega\,:\,u(x)\neq 0\})\,\,\,,\,\,\,\text{and}\quad \int_{\{u\neq 0\}}|u|^{p-1}|D^2u|\dd x<\infty\,,\
		\end{gathered}
	\end{equation}
	and $\phi$ is a positive superharmonic function on a neighborhood of $\mathrm{supp}(u)$.
	\begin{enumerate}
		\item If $\phi$ is twice continuously differentiable, then
		\begin{align}\label{230121211}
			\int_{\Omega}|u|^p\phi^{c-2}|\nabla \phi|^2 \dd x \leq \Big(\frac{p}{1-c}\Big)^2\int_{\Omega\cap\{u\neq 0\}}|u|^{p-2}|\nabla u|^2\phi^c \dd x.
		\end{align}
		
		\item If $(\Delta u)1_{\{u\neq 0\}}$ is bounded, then
		\begin{align}\label{22062422511}
			\int_{\Omega\cap\{u\neq 0\}}|u|^{p-2}|\nabla u|^2\phi^c \dd x\leq N\int_{\Omega\cap\{u\neq 0\}}(-\Delta u)\cdot u|u|^{p-2}\phi^c \dd x\,,
		\end{align}
		where $N=N(p,c)>0$.
		\vspace{1mm}
		
		\item If the Hardy inequality \eqref{hardy} holds for $\Omega$, then
		\begin{align}\label{22062422512}
			\int_{\Omega}|u|^p\phi^c\rho^{-2}\dd x\leq N\int_{\Omega\cap\{u\neq 0\}}|u|^{p-2}|\nabla u|^2 \phi^c \dd x\,,
		\end{align}
		where $N=N(p,c,\mathrm{C}_0(\Omega))>0$.
	\end{enumerate}
\end{lemma}

Lemma~\ref{03.30} is mainly used for $u\in C_c^{\infty}(\Omega)$.
However, in order to obtain appropriate solutions of the Poisson equation that are aruitable for our purpose (see  Lemma~\ref{21.05.25.3}), we consider the condition \eqref{22.01.25.2} in Lemma~\ref{03.30}.
Before proving Lemma~\ref{03.30}, we introduce a lemma that help us handle functions satisfying \eqref{22.01.25.2}.
\begin{lemma}\label{21.04.23.4}
	Let $p\in(1,\infty)$ and $u\in C(\bR^d)$ satisfy \eqref{22.01.25.2}.
	
	\begin{enumerate}
		\item $|u|^{p/2-1}u\in W_2^1(\bR^d)$ and $D_i(|u|^{p/2-1}u)=\frac{p}{2}|u|^{p/2-1}(D_iu)1_{\{u\neq 0\}}$.
		\vspace{1mm}
		
		\item  $|u|^p\in W_1^2(\bR^d)$ and 
		\begin{align*}
			D_i\big(|u|^p\big)\,&=p|u|^{p-2}uD_iu 1_{\{u\neq 0\}}\,\,;\\
			D_{ij}\big(|u|^p\big)\,&=\big(p|u|^{p-2}uD_{ij}u+p(p-1)|u|^{p-2}D_iuD_ju\big)\,1_{\{u\neq 0\}}.
		\end{align*}
	\end{enumerate}
\end{lemma}
The proof of Lemma~\ref{21.04.23.4} is provided in the end of this subsection.

\begin{remark}\label{220613813}
	If \eqref{hardy} holds, then the inequality in \eqref{hardy} also holds for all $f\in \mathring{W}^1_2(\Omega)$, where $\mathring{W}^1_2(\Omega)$ denotes the closure of $C_c^{\infty}(\Omega)$ in $W_2^1(\Omega)$.	
\end{remark}

\begin{proof}[Proof of Lemma~\ref{03.30}]
	By Remark~\ref{22.01.25.1}, we may assume that $\phi$ is a classical superharmonic function on a neighborhood of $\text{supp}(u)$.
	In this proof, all of the integrations by parts are based on Lemma~\ref{21.04.23.4}.
	
	(1) Recall that $\phi$ is twice continuouly differentiable on a neighborhood of $\mathrm{supp}(u)$.
	Integrate by parts to obtain
	\begin{align}\label{220530132}
		\begin{split}
			&(1-c)\int_{\Omega}|u|^p\phi^{c-2}|\nabla \phi|^2\dd x\\
			=\,&-\int_{\Omega}|u|^p\nabla\phi\cdot\nabla (\phi^{c-1})\,\dd x\\
			=\,&p\int_{\Omega\cap\{u\neq 0\}}|u|^{p-2}u\,\phi^{c-1}(\nabla u\cdot\nabla \phi)\dd x+\int_{\Omega}|u|^p\phi^{c-1} \Delta \phi\, \dd x\\
			\leq\,& p\,\Big(\int_{\Omega\cap\{u\neq 0\}} |u|^{p-2}|\nabla u|^2\phi^c\dd x\Big)^{1/2}\Big(\int_{\Omega} |u|^p\phi^{c-2}|\nabla \phi|^2\dd x \Big)^{1/2}\,,
		\end{split}
	\end{align}
	where the last inequality follows from the H\"older inequality and that $\Delta\phi\leq 0$ on $\{u\neq 0\}$.
	Since the first term of \eqref{220530132} is finite, we obtain \eqref{230121211}.
	\vspace{1mm}
	
	Although in (2) and (3), we do not assume that $\phi$ is infinitely smooth, we can restrict our attention to this case.
	This is because  if \eqref{22062422511} and \eqref{22062422512} hold for $\phi^{(\epsilon)}$ instead of $\phi$, for all sufficiently small $\epsilon>0$,
	then \eqref{22062422511} and \eqref{22062422512} also hold for $\phi$ by Lemma~\ref{21.04.23.5} and Remark~\ref{21.11.15.2}.
	Note that if $0<\epsilon<d(\mathrm{supp}(u),\partial\Omega)$, then that $\phi^{(\epsilon)}$ is a positive superharmonic function on $\mathrm{supp}(u)$ (see Proposition~\ref{21.04.23.3} ).
	In addition, $|u|^{p-2}|\nabla u|^21_{\{u\neq 0\}}$ and $|u|^p\rho^{-2}$ are integrable (see Lemma~\ref{21.04.23.4}) and $-\Delta u\cdot u|u|^{p-2}1_{\{u\neq 0\}}$ in \eqref{22062422511} is bounded.
	Therefore, in the proof of (2) and (3), we additionally assume that $\phi$ is infinitely smooth.
	
	
	(2) \textbf{Case 1.} $c\in [0,1)$
	
	Integrate by parts to obtain
	\begin{align*}
		\int_{\Omega}-\Delta u\cdot u|u|^{p-2}\phi^c \dd x=\,& (p-1)\int_{\Omega\cap\{u\neq 0\}}|u|^{p-2}|\nabla u|^2\phi^c\dd x-\frac{1}{p}\int_{\Omega}|u|^p\Delta(\phi^c)\dd x\,.
	\end{align*}
	Since
	\begin{align}\label{22.04.12.1104}
		\Delta (\phi^c)=c\,\phi^{c-1}\Delta\phi+c(c-1)\phi^{c-2}|\nabla\phi|^2\leq 0\qquad\text{on}\quad \text{supp}(u)\,,
	\end{align}
	\eqref{22062422511} is obtained.
	
	\textbf{Case 2.} $c\in(-p+1,0)$
	
	Due to integration by parts, H\"older inequality, and \eqref{230121211}, we have
	\begin{align*}
		&\int_{\Omega}-\Delta u\cdot u|u|^{p-2}\phi^c \dd x\\
		=\,&(p-1)\int_{\Omega}|u|^{p-2}|\nabla u|^2\phi^c\dd x+c\int_{\Omega}(\nabla u)\cdot(\nabla\phi) u|u|^{p-2}\phi^{c-1}\dd x\\
		\geq\,& (p-1)\int_{\Omega}|u|^{p-2}|\nabla u|^2\phi^c\dd x\\
		&+c\left(\int_{\Omega\cap\{u\neq 0\}}|u|^{p-2}|\nabla u|^2\phi^c\dd x\cdot\int_{\Omega}|u|^p\phi^{c-2}|\nabla \phi|^2\dd x\right)^{1/2}\\
		\geq\,& \frac{p+c-1}{1-c}\int_{\Omega}|u|^{p-2}|\nabla u|^2\phi^c\dd x\,.
	\end{align*}
	
	(3)
	Recall that $\phi$ is assumed to be positive and smooth on a neighborhood of $\text{supp}(u)$.
	Due to Lemma~\ref{21.04.23.4}, $|u|^{p/2-1}u\phi^c$ belongs to $\mathring{W}^1_2(\Omega)$ and
	$$
	\nabla\big(|u|^{p/2-1}u\phi^c\big)=\frac{p}{2}|u|^{p/2-1}(\nabla u)1_{\{u\neq 0\}}\phi^{c/2}+\frac{c}{2}|u|^{p/2}\phi^{c/2-1}\nabla\phi\,.
	$$
	Therefore, due to the Hardy inequality (see Remark~\ref{220613813}) and \eqref{230121211}, we have
	\begin{align*}
		\,&\int_{\Omega}\big||u|^{p/2-1}u\phi^{c/2}\big|^2\rho^{-2}\dd x\\
		\lesssim_{p,c}\,&\mathrm{C}_0(\Omega) \int_{\Omega}\Big(|u|^{p-2}|\nabla u|^2\phi^c1_{\{u\neq 0\}}+|u|^p\phi^{c-2}|\nabla\phi|^2\Big)\dd x\\
		\lesssim_{p,c}\,& \mathrm{C}_0(\Omega) \int_{\Omega\cap\{u\neq 0\}}|u|^{p-2}|\nabla u|^2\phi^c \dd x.
	\end{align*}
\end{proof}

\begin{thm}\label{21.05.13.2}
	Let $p\in(1,\infty)$ and suppose that 
	\begin{align*}
		&\text{$\Omega$ admits the Hardy inequality \eqref{hardy}}\,;\\
		&\text{$\phi$ is a positive superharmonic function on $\Omega$, and }-p+1<c<1\,.
	\end{align*}
	If $u\in C(\Omega)$ satisfies \eqref{22.01.25.2} and $(\Delta u)1_{\{u\neq 0\}}$ is bounded, then for any $\lambda\geq 0$,
	\begin{align*}
		\int_{\Omega}|u|^{p}\phi^c\rho^{-2}\dd x\leq N\int_{\Omega}|\Delta u-\lambda u|^p\phi^c\rho^{2p-2}\dd x\,,
	\end{align*}
	where $N=N(p,c,\mathrm{C}_0(\Omega))$.
\end{thm}
\begin{proof}
	Since $\lambda\geq 0$, Lemma~\ref{03.30} implies
	\begin{align}\label{22.04.18.1}
		\begin{split}
			\int_{\Omega}|u|^{p}\phi^c\rho^{-2}\dd x&\leq N\int_{\Omega}(-\Delta u)\cdot u |u|^{p-2}1_{\{u\neq 0\}}\phi^c\dd x\\
			&\leq N\int_{\Omega}(-\Delta u+\lambda u)\cdot u |u|^{p-2}1_{\{u\neq 0\}}\phi^c\dd x\,,
		\end{split}
	\end{align}
	where $N=N(p,c,\mathrm{C}_0(\Omega))>0$.
	Since $\phi^c\rho^{-2}$ is locally integrable on $\Omega$ (see Remark~\ref{2303081145}), the first term in \eqref{22.04.18.1} is finite. 
	By the H\"older inequality, the proof is completed.
\end{proof}

%

\begin{lemma}[Existence of a weak solution]\label{21.05.25.3}
	Suppose that \eqref{hardy} holds for $\Omega$.
	Then for any $\lambda\geq 0$ and $f\in C_c^{\infty}(\Omega)$, there exists a measurable function $u:\Omega\rightarrow \bR$ satisfying the following:
	\begin{enumerate}
		\item $u\in L_{1,\mathrm{loc}}(\Omega)$.
		
		\item $\Delta u-\lambda u=f$ in the sense of distrituion on $\Omega$, \textit{i.e.}, for any $\zeta\in C_c^{\infty}(\Omega)$,
		\begin{align}\label{230328849}
			\int_{\Omega}u\big(\Delta \zeta -\lambda \zeta\big)\dd x=\int_{\Omega} f \zeta \dd x\,.
		\end{align}
		
		\item For any $p\in(1,\infty)$, $\mu\in(-1/p,1-1/p)$ and positive superharmic function $\phi$ on $\Omega$,
		\begin{align}\label{220613103}
			\begin{split}
				\int_{\Omega}|u|^p\phi^{-\mu p}\rho^{-2}\dd x\leq N\int_{\Omega}|f|^p\phi^{-\mu p}\rho^{2p-2}\dd x
			\end{split}
		\end{align}
		where $N=N(p,c,\mathrm{C}_0(\Omega))>0$.
	\end{enumerate}
\end{lemma}

\begin{proof}
	Take infinitely smooth bounded open sets $\Omega_n$, $n\in\bN$, such that
	$$
	\text{supp}(f)\subset \Omega_1\,\,,\quad \overline{\Omega_n}\subset \Omega_{n+1}\,\,,\quad \bigcup_{n}\Omega_n=\Omega
	$$
	(see, \textit{e.g.}, \cite[Proposition 8.2.1]{DD_2008}).
	For $h\in C_c^{\infty}(\Omega_1)$ and $n\in\bN$, by $R_{\lambda,n}h$ we denote the classical solution $H\in C^{\infty}(\overline{\Omega_n})$ of the equation
	$$
	\Delta H-\lambda H=h1_{\Omega_1}\quad \text{on}\,\,\Omega_n\quad;\quad H|_{\partial\Omega_n}\equiv 0\,.
	$$
	Since
	$$
	\text{$\overline{\Omega_n}$ is a compact subset of $\Omega$}\,\,,\quad R_{\lambda,n}h\in C^{\infty}(\overline{\Omega_n})\,\,,\quad  R_{\lambda,n}h|_{\partial\Omega_n}\equiv 0,
	$$
	we obtain that $\big(R_{\lambda,n}h\big)1_{\Omega_n}\in C(\Omega)$ satisfies \eqref{22.01.25.2}.
	By Theorem~\ref{21.05.13.2}, for any $p\in (1,\infty)$, $\mu\in(-1/p,1-1/p)$ and positive superharmonic fucntions $\phi$ on $\Omega$, we have
	\begin{align}\label{220610424}
		\int_{\Omega}\big|\big(R_{\lambda,n}h\big)1_{\Omega_n}\big|^{p}\phi^{-\mu p}\rho^{-2}\dd x\leq N(p,c,\mathrm{C}_0(\Omega)) \int_{\Omega}|h|^p\phi^{-\mu p}\rho^{2p-2}\dd x\,.
	\end{align}
	Note that $N$ in \eqref{220610424} is independent of $n$.
	
	Take $F\in C_c^{\infty}(\Omega_1)$ such that $F\geq |f|$, and put
	\begin{align}\label{230308251}
		f_1=\frac{f-F}{2}\quad\text{and}\quad f_2=\frac{-f-F}{2}
	\end{align}
	so that $f_1,\,f_2\leq 0$, and $f_1-f_2=f$.
	
	For $v_n:=\big(R_{\lambda,n}f_1\big)1_{\Omega_n}$, the maximum principle implies that
	$$
	0\leq v_n\leq v_{n+1}\quad\text{on}\quad \Omega\,.
	$$
	We define $v(x):=\lim_{n\rightarrow \infty}v_n(x)$.
	By applying the monotone convergence theorem to \eqref{220610424} with $(h,\phi,p,c)=(f_1,1_{\Omega},2,0)$, we obtain
	\begin{align*}
	\int_{\Omega}|v|^2\rho^{-2}\dd x\lesssim \int_{\Omega}|f_1|^2\rho^{2}\dd x\,,
	\end{align*}
	which implies that $v\in L_{1,\mathrm{loc}}(\Omega)$.
	
	We next caim that for any $\zeta\in C_c^{\infty}(\Omega)$,
	\begin{align}\label{230125936}
		\int_{\Omega}v\big(\Delta \zeta -\lambda \zeta\big)\dd x=\int_{\Omega} f_1 \zeta \dd x\,.
	\end{align}
	Fix $\zeta\in C_c^{\infty}(\Omega)$, and take $N\in\bN$ such that $\mathrm{supp}(\zeta)\subset \Omega_N$.
	It follows from the definition of $v_n=R_{\lambda,n}f_1$ that for any $n\geq N$,
	$$
	\int_{\Omega}v_n\big(\Delta \zeta-\lambda \zeta\big)\dd x=\int_{\Omega} f_1\zeta \dd x.
	$$
	Since $0\leq v_n\leq v$ and $v\in L_{1,\mathrm{loc}}(\Omega)$, the Lebesgue dominated convergence theorem yields \eqref{230125936}.
	By the same argument, 
	$$
	w:=\lim\limits_{\substack{n\rightarrow \infty}}\big(R_{\lambda,n}f_2\big)1_{\Omega_n}
	$$
	belongs to $L_{1,\mathrm{loc}}(\Omega)$, and satisfies that for any $\zeta\in C_c^{\infty}(\Omega)$,
	\begin{align*}
		\int_{\Omega}w\big(\Delta \zeta -\lambda \zeta\big)\dd x=\int_{\Omega} f_2 \zeta \dd x\,.
	\end{align*}
	
	Put
	\begin{align*}
		u=v-w=\lim_{n\rightarrow \infty}\Big[\big(R_{\lambda,n}f\big)1_{\Omega_n}\Big]
	\end{align*}
	(the limit exists almost everywhere on $\Omega$).
	Then $u\in L_{1,\mathrm{loc}}(\Omega)$, and $u$ satisfies \eqref{230328849}.
	In addition, by applying Fatou's lemma to \eqref{220610424} with $h=f$, \eqref{220613103} is obtained.
\end{proof}

\begin{remark}\label{221022314}
	We discuss Lemma~\ref{21.05.25.3} and the Green functions for the Poisson equation.
	It follows from \cite[Theorem 4.1.2, Theorem 5.3.8]{AG} and \cite[Theorem 2]{AA} that if $\Omega$ admits the Hardy inequality, $\Omega$ also admits the Green function $G_{\Omega}:\Omega\times\Omega\rightarrow [0,\infty]$ for the equation 
	$$
	-\Delta u=f\,\,\,\text{on}\,\,\,\Omega\quad ;\quad u|_{\partial\Omega}=0\,
	$$
	(the definition of $G_{\Omega}$ can be found in \cite[Detinition 4.1.3]{AG}).
	For $\{\Omega_n\}_{n\in\bN}$ in the proof of Lemma~\ref{21.05.25.3}, $G_{\Omega_n}$ increases and converges to $G_{\Omega}$ on $\Omega\times\Omega$ (see \textit{e.g.} \cite[Theorem 4.1.10]{AG}).
	Since $f_1$ in \eqref{230308251} belongs to $C_c^{\infty}(\Omega_n)$ and $\Omega_n$ is a infinitely smooth domain, we have
	$$
	R_{0,n}f_1(x)=-\int_{\Omega_n}G_{\Omega_n}(x,y)f_1(y)\dd y\,.
	$$
	The monotone convergence theorem implies that
	$$
	v(x)=\lim_{n\rightarrow \infty}\big(R_{0,n}f_1(x)\big)1_{\Omega_n}(x)= -\int_{\Omega}G_{\Omega}(x,y)f_1(y)\dd y\,.
	$$
	By the same argument for $w$, we conclude that the function $u=v-w$ in Lemma~\ref{21.05.25.3} is representated by
	$$
	u(x)=-\int_{\Omega}G_{\Omega}(x,y)f(y)\dd y\,.
	$$
\end{remark}
\vspace{1mm}

We end this subsection providing the proof of Lemma~\ref{21.04.23.4}.

\begin{proof}[Proof of Lemma~\ref{21.04.23.4}]
	This proof is a variant of \cite[Lemma 2.17]{Krylov1999-1}.
	Take nonnegative functions $g_n\in C(\bR)$ such that
	\begin{align*}
		\begin{gathered}
			\text{$g_n=0$ on a neighborhood of $0$ for each $n\in\bN$, and}\\
			\text{$ g_n(s) \nearrow |s|^{p/2-1}1_{s\neq 0}$ for all $s\in\bR$.}
		\end{gathered}
	\end{align*}
	Recall the assumption \eqref{22.01.25.2}, and denote $A=\sup|u|$.
	Since $0\leq g_n(s)\leq |s|^{p/2-1}$, the Lebesgue dominated convergence theorem implies that
	\begin{alignat*}{2}
		F_n(t)\,&:=\,\,\int^t_0g_n(s)\dd s\,&&\rightarrow \,\frac{2}{p}\,|t|^{p/2-1}t\,,\\
		G_n(t)\,&:=\int_0^t\big(g_n(s)\big)^2\dd s\,&&\rightarrow  \,\frac{1}{p-1}|t|^{p-2}t
	\end{alignat*}
	uniformly for $t\in[-A,A]$.
	Furthermore, there absolute values increase as $n\rightarrow \infty$.
	Since $F_n(u)$ and $G_n(u)$ vanish on a neighborhood of $\{u=0\}$, these functions are supported on a compact subset of $\{u\neq 0\}$, and continuously differentiable with
	$$
	D_i\big(F_n(u)\big)=g_n(u)D_iu\,1_{\{u\neq 0\}}\quad\text{and}\quad D_i\big(G_n(u)\big)=\big(g_n(u)\big)^2D_iu\,1_{\{u\neq 0\}}\,.
	$$
	
	(1) Integrate by parts to obtain
	\begin{align*}
		\int_{\bR^d}|g_n(u)\nabla u \,1_{\{u\neq 0\}}|^2\dd x\,&=-\int_{\bR^d}G_n(u)\Delta u \,1_{\{u\neq 0\}}\dd x\\
		&\leq \frac{1}{p-1}\int_{\{u\neq 0\}}|u|^{p-1}|\Delta u|\dd x\,.
	\end{align*}
	Apply the monotone convergence theorem to obtain that 
	\begin{align}\label{230308116}
		|u|^{p/2-1}|\nabla u|\in L_2(\bR^d)\,.
	\end{align}
	We denote $v=\frac{2}{p}|u|^{p/2-1}u$.
	For any $\zeta\in C_c^{\infty}(\bR^d)$, we have
	\begin{align*}
		-\int_{\bR^d} v\cdot D_i\zeta \dd x&=-\lim_{n\rightarrow\infty}\int_{\bR^d} F_n(u)\cdot D_i\zeta\dd x\\
		&=\lim_{n\rightarrow\infty}\int_{\{u\neq 0\}} g_n(u)D_iu\cdot \zeta\dd x=\int_{\{u\neq 0\}} |u|^{p/2-1}D_iu\cdot \zeta \dd x\,.
	\end{align*}
	Here, the first and the last equalities follow from the Lebesgue dominated convergence theorem, because 
	$|F_n(u)|\leq |v|$ and $|g_n(u)|\leq |u|^{p/2-1}$ (recall \eqref{230308116}).
	Therefore $v\in W_2^1(\bR^d)$ and $D_i v=|u|^{p/2-1}D_i u\,1_{\{u\neq 0\}}$.
	
	(2) It follows from (1) of this lemma that $|u|^p\in W_1^1(\bR^d)$ with $D_i\big(|u|^p\big)=p|u|^{p-2}u D_iu 1_{u\neq 0}$.
	For any $\zeta\in C_c^{\infty}$, we have
	\begin{align*}
		&\,\frac{1}{p-1}\int_{\{u\neq 0\}}|u|^{p-2}uD_iu \cdot D_j\zeta \dd x\\
		=&\,\lim_{n\rightarrow \infty}\int_{\bR^d} G_n(u)D_iu \cdot D_j\zeta \dd x\\
		=&\,-\lim_{n\rightarrow \infty}\int_{\bR^d} \Big(|g_n(u)|^2 D_iu D_ju+G_n(u)D_{ij}u\Big)\zeta \dd x\\
		=&\,-\int_{\{u\neq 0\}} \Big(|u|^{p-2}D_i u D_ju+\frac{1}{p-1}|u|^{p-2}u D_{ij}u1_{\{u\neq 0\}}\Big)\zeta \dd x\,.
	\end{align*}
	Here, the first and last inequalities follow from the Lebesgue dominated convergence theorem, because $|G_n(u)|\leq \frac{1}{p-1}|u|^{p-1}$ and $|g_n(u)|\leq |u|^{p/2-1}$ (recall \eqref{230308116}). 
	Therefore $|u|^{p-2}uD_iu\in W_1^1(\bR^d)$ and 
	$$
	D_j\big(|u|^{p-2}uD_iu\big)=|u|^{p-2}D_i u D_ju+\frac{1}{p-1}|u|^{p-2}u D_{ij}u1_{\{u\neq 0\}}\,.
	$$
\end{proof}

\vspace{2mm}

\mysection{Weighted Sobolev spaces and solvability of the Poisson equation}\label{0040}
In this section, we focus on the Poisson equation
$$
\Delta u-\lambda u=f\quad(\lambda\geq 0)
$$
in an open set $\Omega\subset \bR^d$ admitting the Hardy inequality.
We use the weighted Sobolev $\Psi H_{p,\theta}^{\gamma}(\Omega)$ introduced in Definition~\ref{220610533}, for the classes of the solution $u$ and the force term $f$.
It is worth noting that the zero Dirichlet condition ($u|_{\partial\Omega}=0$) is implicitly considered in these Sobolev spaces, as $C_c^{\infty}(\Omega)$ is dense in $\Psi H_{p,\theta}^{\gamma}(\Omega)$ (see Lemma~\ref{21.09.29.4}).

We recall the organization of this section.
In Subsection~\ref{0041}, we present the notions of \textit{Harnack function} and \textit{regular Harnack function}.
Subsection~\ref{0042} introduces the weighted Sobolev spaces $\Psi H_{p,\theta}^{\gamma}(\Omega)$, which is a combination of regular Harnack functions $\Psi$ and the spaces $H_{p,\theta}^{\gamma}(\Omega)$; the spaces $H_{p,\theta}^{\gamma}(\Omega)$ was first introduced by Krylov (for $\Omega=\bR_+^d$, \cite{Krylov1999-1}) and Lototsky (for general $\Omega$, \cite{Lo1}). 
In Subsection~\ref{0043}, we prove the main theorem of this section (Theorem~\ref{21.09.29.1}), through Section~\ref{0030} and the localization argument used in \cite{Krylov1999-1}. 
The concept of regular Harnack functions helps us state the main theorem in a unified manner to obtain useful applications provided in Subsections~\ref{app.} and \ref{app2.}.

\vspace{2mm}

\subsection{Harnack function and regular Harnack function}\label{0041}

\begin{defn}\label{21.10.14.1}\,
	
	\begin{enumerate}
		\item We call a measurable function $\psi:\Omega\rightarrow \bR_+$ a \textit{Harnack function}, if there exists a constant $C=:\mathrm{C}_1(\psi)>0$ such that
		\begin{align*}
			\underset{B(x,\rho(x)/2)}{\mathrm{ess\,sup}}\,\psi\leq C\underset{B(x,\rho(x)/2)}{\mathrm{ess\,inf}}\,\psi\quad\text{for all}\,\,x\in\Omega\,.
		\end{align*}
		
		\item We call a function $\Psi\in C^{\infty}(\Omega)$ a \textit{regular Harnack function}, if $\Psi>0$ and there exists a seqeunce of constants $\{C^{(k)}\}_{k\in\bN}=:\mathrm{C}_2(\Psi)$ such that for every $k\in\bN$,
		\begin{align*}
			|D^k\Psi|\leq C^{(k)}\,\rho^{-k}\Psi\quad\text{on}\quad\Omega\,.
		\end{align*}
		
		\item Let $\psi$ be a measurable function and $\Psi$ be a regular Harnack function on $\Omega$. We say that $\Psi$ is a \textit{regularization} of $\psi$, if there exists a constant $C=:\mathrm{C}_3(\psi,\Psi)>0$ such that
		$$
		C^{-1}\Psi\leq\psi\leq C\,\Psi\quad\text{almost everywhere on}\,\,\Omega. 
		$$
	\end{enumerate}
\end{defn}
\vspace{1mm}

A relation between the notions of Harnack functions and regular Harnack functions is provided in Lemma~\ref{21.05.27.3}.

\begin{example}\label{21.05.18.2}\,
	
	\begin{enumerate}
		\item For any $E\subset \Omega^c$, the function $x\mapsto d(x,E)$ is a Harnack function on $\Omega$.
		Additionally, $\mathrm{C}_1\big(d(\,\cdot\,,E)\big)$ can be chosen as $3$.
		
		\item Let $\Psi\in C^{\infty}(\Omega)$ satisfy
		$$
		\Psi>0\quad\text{and}\quad \Delta\Psi=-\widetilde{\Lambda}\Psi
		$$ 
		for some constant $\widetilde{\Lambda}\geq 0$.
		We claim that $\Psi$ is a regular Harnack function on $\Omega$, and $\mathrm{C}_2(\Psi)$ can be chosen to depend only on $d$.
		To observe this, for a fixed $x_0\in\Omega$, put
		$$
		u(t,x):=e^{-\widetilde{\Lambda} \rho(x_0)^2 t}\Psi\big(x_0+\rho(x_0)x\big)
		$$
		so that $u_t=\Delta u$ on $\bR\times B_1(0)$.
		The interior estimates (see, \textit{e.g.}, \cite[Theorem 2.3.9]{Krylov2008}) and the parabolic Harnack inequality imply that for any $k\in\bR$,
		$$
		\rho(x_0)^k|D^k\Psi(x_0)|=|D^k_x u(0,0)|\lesssim_{k,d}\|u\|_{L_2((-1/4,0]\times B_{1/2}(0))}\lesssim_d u(1,0)\leq \Psi(x_0)\,.
		$$
		
		\item The multivariate Fa\'a di Bruno's formula (see, \textit{e.g.}, \cite[Theorem 2.1]{FDB}) implies the following:
		\begin{itemize}
		\item[] Let $U\subset \bR^d$ and  $V\subset\bR$ be open sets and $f:U\rightarrow V$ and $l:V\rightarrow \bR$ be smooth functions. For any multi-index $\alpha$,
		\begin{align*}
			\big|D^{\alpha}(l\circ f)\big|\leq N(d,\alpha)\sum_{k=1}^{|\alpha|}\Big(\big|\big(D^kl\big)\circ f\big|\sum_{\substack{\beta_1+\ldots+\beta_k=\alpha\\|\beta_i|\geq 1}}\,\prod_{i=1}^k|D^{\beta_k}f|\Big)\,.
		\end{align*}
	\end{itemize}
		This inequality implies that for any regular Harnack function $\Psi$ on $\Omega$, and $\sigma\in\bR$, $\Psi^\sigma$ is also a regular Harnack function on $\Omega$, and $\mathrm{C}_2(\Psi^\sigma)$ can be chosen to depend only on $d,\, \sigma,\,\mathrm{C}_2(\Psi)$.
		
		\item If $\Psi$ and $\Phi$ are regularizations of $\psi$ and $\phi$, respectively, then $\Psi\Phi$, $\Psi+\Phi$, and $\frac{\Phi\Psi}{\Phi+\Psi}$ are regularizations of $\psi\phi$, $\max (\psi,\phi)$, and $\min (\psi,\phi)$, respectively.
	\end{enumerate}
\end{example}

\begin{lemma}\label{21.11.16.1}
	A measurable function $\psi:\Omega\rightarrow \bR_+$ is a Harnack function if and only if there exists $r\in(0,1)$ and $N_{r}>0$ such that
	\begin{align*}
		\underset{B(x,r\rho(x))}{\mathrm{ess\,sup}}\,\psi\leq N_{r}\underset{B(x,r\rho(x))}{\mathrm{ess\,inf}}\,\psi\quad\text{for all}\,\,x\in\Omega.
	\end{align*}
	In this case, $\mathrm{C}_1(\psi)$ and $N_r$ depend only on each other and $r$.
\end{lemma}
\begin{proof}
	We only need to show that for fixed constants $r_0,\,r\in(0,1)$ and $\widetilde{N}\geq 1$,
	\begin{align}\label{220530319}
		\begin{split}
			\text{if}\quad \underset{B(x,r_0\rho(x))}{\mathrm{ess\,sup}}\,\psi\leq\,& \widetilde{N}\underset{B(x,r_0\rho(x))}{\mathrm{ess\,inf}}\,\psi\quad \forall\,\,x\in\Omega,\\
			\text{then}\quad \underset{B(x,r\rho(x))}{\mathrm{ess\,sup}}\,\psi\leq\,& \widetilde{N}^{2M+1}\underset{B(x,r\rho(x))}{\mathrm{ess\,inf}}\,\psi\quad \forall\,\,x\in\Omega\,,
		\end{split}
	\end{align}
	where $M$ is the smallest integer such that $M\geq \frac{r}{(1-r)r_0}$.
	
	If $r\leq r_0$, then there is nothing to prove.
	Consider the case $r>r_0$.
	For $x\in\Omega$ we denote $B(x)=B\big(x,r_0\rho(x)\big)$.
	For fixed $x_0\in\Omega$ and $y\in \overline{B}\big(x_0,r\rho(x_0)\big)$, put $x_{k}=(1-\frac{k}{M})x_0+\frac{k}{M}y$, $k=1,\,\ldots,\,M$.
	Since
	$$
	\rho(x_k)\geq \rho(x_0)-|x_0-x_k|\geq (1-r)\rho(x_0),
	$$
	we obtain that
	\begin{align*}
		|x_{k-1}-x_k|=\frac{|x_0-y|}{M}\leq (1-r)r_0\rho(x_0)\leq r_0\rho(x_k)\,.
	\end{align*}
	Therefore $x_{k-1}\in B(x_k)$, which implies $B(x_{k-1})\cap B(x_k)\neq \emptyset$, and hence
	\begin{align}\label{22.03.02.4}
		\underset{B(x_{k})}{\mathrm{ess\,sup}}\,\psi\leq \widetilde{N}\,\underset{B(x_{k})}{\mathrm{ess\,inf}}\,\psi\leq \widetilde{N}\underset{B(x_{k-1})\cap B(x_k)}{\mathrm{ess\,inf}}\psi\leq \widetilde{N}\,\underset{B(x_{k-1})}{\mathrm{ess\,sup}}\,\psi\,.
	\end{align}
	By applying \eqref{22.03.02.4} for $k=1,\,\ldots,\,M$, we have
	\begin{align}\label{2301261013}
		\underset{B(y)}{\mathrm{ess\,sup}}\,\psi\leq \widetilde{N}^{M}\underset{B(x_0)}{\mathrm{ess\,sup}}\,\psi\,.
	\end{align}
	Since $B(x_0,r\rho(x_0))$ is contained in a finite union of elements in 
	$$
	\big\{B(y)\,:\,y\in \overline{B}(x_0,r\rho(x_0))\big\}\,,
	$$
	\eqref{2301261013} implies 
	\begin{align}\label{22.03.02.6}
		\underset{B(x_0,r\rho(x_0))}{\mathrm{ess\,sup}}\,\psi\leq \widetilde{N}^{M}\underset{B(x_0,r_0\rho(x_0))}{\mathrm{ess\,sup}}\,\psi\,.
	\end{align}
	By the same argument, we obtain that
	\begin{align}\label{22.03.02.7}
		\underset{B(x_0,r_0\rho(x_0))}{\mathrm{ess\,inf}}\,\psi\leq \widetilde{N}^{M}\underset{B(x_0,r\rho(x_0))}{\mathrm{ess\,inf}}\,\psi\,.
	\end{align}
	By combining \eqref{22.03.02.6}, \eqref{22.03.02.7}, and the assumption in \eqref{220530319}, the proof is completed.
\end{proof}

\begin{remark}\label{22.02.17.5}
	Let $\psi$ be a Harnack function on $\Omega$. Since $\psi\in L_{1,\mathrm{loc}}(\Omega)$, almost every point in $\Omega$ is a Lebesgue point of $\psi$.
	If $x\in\Omega$ is a Lebesgue point of $\psi$, then for any $r\in(0,1)$,
	$$
	\underset{B(x,r\rho(x))}{\mathrm{ess\,inf}}\,\psi\leq \psi(x)\leq \underset{B(x,r\rho(x))}{\mathrm{ess\,sup}}\,\psi\,.
	$$
	By Lemma~\ref{21.11.16.1}, we obtain that for almost every $x\in\Omega$ and for any $r\in(0,1)$, there exists $N_r>0$ depending only on $\mathrm{C}_1(\psi)$ and $r$ such that
	$$
	N_r^{-1}\underset{B(x,r\rho(x))}{\mathrm{ess\,sup}}\,\psi\leq \psi(x)\leq N_r\underset{B(x,r\rho(x))}{\mathrm{ess\,inf}}\,\psi\,.
	$$
\end{remark}

\begin{lemma}\label{21.05.27.3}
	\,\,
	
	\begin{enumerate}
		\item If $\psi$ is a Harnack function, then there exists a regularization of $\psi$.
		For this regularization of $\psi$, denoted by $\widetilde{\psi}$, $\mathrm{C}_2(\widetilde{\psi})$ and $\mathrm{C}_3(\psi,\widetilde{\psi})$ can be chosen to depend only on $d$ and $\mathrm{C}_1(\psi)$.
		
		\item If $\Psi$ is a regular Harnack function, then it is also a Harnack function and $\mathrm{C}_1(\Psi)$ can be chosen to depend only on $d$ and $\mathrm{C}_2(\Psi)$.
	\end{enumerate}
\end{lemma}
This lemma implies that a measurable function is a Harnack function if and only if it has a regularization.
\begin{proof}[Proof of Lemma~\ref{21.05.27.3}]
	\,\,
	
	(1) Let $\psi$ be a Harnack function on $\Omega$.
	Take $\zeta\in C_c^{\infty}(\bR^d)$ such that
	$$
	\zeta\geq 0\,\,,\quad \text{supp}(\zeta)\subset B_1\,\,,\quad \int\zeta dx=1\,.
	$$
	For $i=1,\,2,\,3$ and $k\in\bZ$, put
	\begin{align*}
		U_{i,k}=\{x\in\Omega\,:\,2^{k-i}<\rho(x)<2^{k+i}\}\quad\text{and}\quad \zeta_k(x)=\frac{1}{2^{(k-4)d}}\zeta\Big(\frac{x}{2^{k-4}}\Big)\,.
	\end{align*}
	Note that for each $i$,
	\begin{align}\label{220604957}
		\text{$\big\{U_{i,k}\big\}_{k\in\bZ}$ is a locally finte cover of $\Omega$, and}\,\,\,\,\sum_{k\in\bZ}1_{U_{i,k}}\leq 2i\,.
	\end{align}
	For each $k\in\bZ$, put
	$$
	\Psi_k(x)=\big(\psi 1_{U_{2,k}}\big)\ast \zeta_k(x):=\int_{B(x,2^{k-4})}\big(\psi 1_{U_{2,k}}\big)(y)\zeta_k(x-y)\dd y\,,
	$$
	so that $\Psi_k\in C^{\infty}(\Omega)$.
	Since
	$$
	x\in U_{1,k}\quad \Longrightarrow \quad B(x,2^{k-4})\subset B(x,\rho(x)/8)\subset U_{2,k}\,,
	$$
	we have
	\begin{align}\label{22.02.17.3}
		\Big(\underset{B(x,\rho(x)/8)}{\mathrm{ess\,inf}}\psi\Big)1_{U_{1,k}}(x)\leq \Psi_k(x)\,.
	\end{align}
	Since
	\begin{align}\label{2303141439}
		\begin{split}
			&x\in U_{3,k}\quad\Longrightarrow \quad B(x,2^{k-4})\subset B(x,\rho(x)/2)\,;\\
			&x\notin U_{3,k}\quad\Longrightarrow \quad B(x,2^{k-4})\cap U_{2,k}=\emptyset\,,
		\end{split}
	\end{align}
	we have
	\begin{align}\label{22.02.17.4}
		\Psi_k(x)\leq \Big(\underset{B(x,\rho(x)/2)}{\mathrm{ess\,sup}}\psi\Big)1_{U_{3,k}}(x)\,.
	\end{align}
	By \eqref{22.02.17.3}, \eqref{22.02.17.4}, and Remark~\ref{22.02.17.5}, we obtain that
	\begin{align}\label{2206292521}
		N^{-1}\psi(x)1_{U_{1,k}}(x)\leq \Psi_k(x)\leq N\psi(x)1_{U_{3,k}}(x)
	\end{align}
	for almost every $x\in\Omega$, where $N=N(\mathrm{C}_1(\psi))$.
	Moreover,
	\begin{align}\label{2206292522}
		\begin{split}
			|D^{\alpha}\Psi_k(x)|\,&\leq \|D^{\alpha}\zeta_k\|_{\infty}\int_{B(x,2^{k-4})}\psi 1_{U_{2,k}}\dd y\\
			&\leq 2^{-|\alpha|k}\big(\underset{B(x,\rho(x)/2)}{\mathrm{ess\,sup}}\psi\big)1_{U_{3,k}}(x)\\
			&\leq N \rho(x)^{-|\alpha|}\psi(x)1_{U_{3,k}}(x)
		\end{split}
	\end{align}
	for almost every $x\in\Omega$, where $N=N(d,\alpha,\mathrm{C}_1(\psi))$ (see \eqref{2303141439} and  Remark~\ref{22.02.17.5}).
	Due to \eqref{220604957}, \eqref{2206292521}, and \eqref{2206292522}, $\Psi:=\sum_{k\in\bZ}\Psi_k$ belongs to $C^{\infty}(\Omega)$ and 
	\begin{align}\label{230328229}
		\Psi\simeq_{\mathrm{C}_1(\psi)}\psi\quad,\quad |D^\alpha\Psi|\leq\sum_{k\in\bZ}|D^{\alpha}\Psi_k|\lesssim_N\rho^{-|\alpha|}\psi
	\end{align}
	for almost every $x\in\Omega$, where $N=N(d,\alpha,\mathrm{C}_1(\psi))$.
	By \eqref{230328229}, the proof is completed.
	
	(2) Let $x,y\in\Omega$ satisfy $|x-y|<\rho(x)/2$.
	For $r\in[0,1]$, put $x_r=(1-r)x+ry$, so that 
	$$
	x_r\in B\big(x,\rho(x)/2\big)\quad\text{and}\quad \rho(x_r)\geq \rho(x)-|x-x_r|\geq|x-y|\,.
	$$
	Then we have
	\begin{align*}
		\Psi(x_r)\,&\leq \Psi(x_0)+|x-y|\int_0^r\big|(\nabla\Psi)(x_{t})\big|\dd t\\
		&\leq \Psi(x_0)+N_0|x-y|\int_0^r\rho(x_{t})^{-1}\Psi(x_t)\dd t\\
		&\leq \Psi(x_0)+N_0\int_0^r\Psi(x_t)\dd t\,,
	\end{align*}
	where $N_0=N(d,\mathrm{C}_2(\Psi))>0$.
	By Gr\"onwall's inequality, we obtain
	$$
	\Psi(y)=\Psi(x_1)\leq e^{N_0}\Psi(x_0)=e^{N_0}\Psi(x).
	$$
	If $x,\,y\in\Omega$ satisfy $|x-y|<\rho(x)/3$, then $|x-y|<\rho(x)/2$ and $|x-y|<\rho(y)/2$.
	Therefore we have
	$$
	e^{-N_0}\Psi(y)\leq \Psi(x)\leq e^{N_0}\Psi(y)\,.
	$$
	By Lemma~\ref{21.11.16.1}, the proof is completed.
\end{proof}

We end this subsection with the following remark, which describes the boundary behavior of regular Harnack functions on domains satisfying a certain geometric condition; this remark is used in Subsection~\ref{0074}

\begin{remark}\label{220819318}
	In \cite{VM}, the term `Harnack function' is used not for the Harnack function defined in Definition \ref{21.10.14.1} but for the continuous Harnack functions. It should be noted that regular Harnack functions are continuous Harnack functions.
	If a domain is a John domain (which will be introduced later), then we obtain the upper and lower bounds of the boundary behavior of regular Harnack functions.
	
	It follows from \cite[Corollary 3.4]{VM} that for any domain $\Omega$, if $\Psi$ is a regular Harnack function on $\Omega$, then
	\begin{align}\label{220818200}
		N_0^{-(k(x,x_0)+1)}\leq \frac{\Psi(x)}{\Psi(x_0)}\leq N_0^{k(x,x_0)+1}\qquad\text{for all}\quad x_0,\,x\in \Omega\,,
	\end{align}
	where $N_0\geq 1$ is a constant depending only on $\mathrm{C}_1(\Psi)$, and $k(x,x_0)\geq 0$ is the quasihyperbolic distance between $x$ and $x_0$ (see \cite[paragraph 2.5]{VM} for the definition). In addition, Gehring and Martio \cite[Theorem 3.11]{GO} proved that if $\Omega$ is a \textit{John domain}, then for any $x_0\in \Omega$, there exists $N,\,A>0$ depending only on $\Omega$ and $x_0$ such that
	\begin{align}\label{220818159}
		e^{k(x,x_0)}\leq N\rho(x)^{-A}\qquad\text{whenever}\quad x\in \Omega\,.
	\end{align}
	Here, $\Omega$ is called a John domain if the following conditions are satisfied:
	\begin{enumerate}
		\item $\Omega$ is a connected and bounded open set.
		
		\item There exist a point $x_0\in \Omega$ and a constant $L_0,\,\epsilon_0>0$ such that for any $x\in \Omega$, there exists a rectifiable path $\gamma:\big[0,L\big]\rightarrow \Omega$ parameterised by arclength such that $L\leq L_0$, $\gamma(0)=x$, $\gamma(L)=x_0$, and 
		$$
		d\big(\gamma(t),\partial \Omega\big)\geq \frac{\epsilon_0t}{L}\quad\text{for all}\,\,t\in[0,L]\,.
		$$
	\end{enumerate}
	Due to \eqref{220818200} and \eqref{220818159}, if $\Omega$ is a John domain, then for any $x_0\in \Omega$, there exist constants $N,\,A>0$ depending only on $\Omega$ and $x_0$ such that for any regular Harnack function $\Psi$ on $\Omega$ and $x\in\Omega$,
	$$
	N^{-1}\rho(x)^{A}\leq \frac{\Psi(x)}{\Psi(x_0)}\leq N\rho(x)^{-A}\,.
	$$
\end{remark}

\vspace{2mm}

\subsection{Weighted Sobolev spaces and regular Harnack functions}\label{0042}\,

In this subsection, we introduce the weighted Sobolev space $H_{p,\theta}^{\gamma}(\Omega)$ and generalize them through regular Harnack functions.

We first recall the definition of the Bessel potential space on $\bR^d$.
For $p\in(1,\infty)$ and $\gamma\in\bR$, $H_p^{\gamma}=H_p^{\gamma}(\bR^d)$ denotes the space of Bessel potential with the norm
\begin{align}\label{2207201137}
	\|f\|_{H_p^{\gamma}}:=\|(1-\Delta)^{\gamma/2}f\|_{L_p(\bR^d)}:=\big\|\cF^{-1}\big[(1+|\xi|^2)^{\gamma/2}\cF(f)(\xi)\big]\big\|_p\,,
\end{align}
where $\cF$ is the Fourier transform and $\cF^{-1}$ is the inverse Fourier transform.
If $\gamma\in\bN_0$, then $H_p^{\gamma}$ coincides with the Sobolev space 
\begin{align*}
	W_p^{\gamma}(\bR^d):=\left\{f\in\cD'(\bR^d)\,:\,\sum_{k=0}^\gamma\int_{\bR^d}|D^kf|^p\dd x<\infty\right\}
\end{align*}
(see, \textit{e.g.}, \cite[Theorem 2.5.6]{triebel2}).

We next introduce the weighted Sobolev spaces $H_{p,\theta}^{\gamma}(\Omega)$ and $\Psi H_{p,\theta}^{\gamma}(\Omega)$.
The space  $H_{p,\theta}^{\gamma}(\Omega)$ was first introduced by Krylov \cite{Krylov1999-1} for $\Omega=\bR_+^d$, and later generalized by Lototsky \cite{Lo1} for arbitrary domains $\Omega\subset \bR^d$.
It is worth mentioning in advance that for $p\in(1,\infty)$, $\theta\in\bR$ and $\gamma\in\bN_0$, the space $H_{p,\theta}^{\gamma}(\Omega)$ coincides with the space
\begin{align*}
	\left\{f\in\cD'(\Omega)\,:\,\sum_{k=0}^{\gamma}\int_{\Omega}|\rho^kD^kf|^p\rho^{\theta-d}\dd x<\infty\right\}
\end{align*}
(see \cite[Proposition 2.2.3]{Lo1} or Lemma~\ref{220512433} of this paper).

In the remainder of this subsection, we assume that
\begin{align}\label{22082801111}
	p\in(1,\infty)\,,\,\,\,\gamma,\,\theta\in\bR\,,\,\,\,\text{$\Psi$ is a regular Harnack function on $\Omega$}\,.
\end{align}
By $\trho$ we denote the regularization of $\rho=d(\,\cdot\,,\partial\Omega)$ constructed in Lemma~\ref{21.05.27.3}.(1).
Recall that for each $k\in\bN_0$, there exists a constant $N_k=N(d,k)>0$ such that
\begin{align}\label{230130541}
	\trho\simeq_{N_0}\rho\quad\text{and}\quad|D^k\trho\,|\leq N_k\trho^{\,1-k}\quad \text{on}\quad \Omega\,.
\end{align}
To define the weighted Sobolev spaces, fix a nonnegative function $\zeta_0\in C_c^{\infty}(\bR_+)$ such that
\begin{align*}
	\text{supp}(\zeta_0)\subset [e^{-1},e]\,\,,\,\,\text{and}\quad \sum_{n\in\bZ}\zeta_0(e^{n}t)=1\quad\text{for all}\,\,t\in\bR_+\,.
\end{align*}
For $x\in\bR^d$ and $n\in\bZ$, put
\begin{align}\label{230130543}
	\zeta_{0,(n)}(x)=\zeta_0\big(e^{-n}\trho(x)\big)1_{\Omega}(x)
\end{align}
so that
\begin{align}\label{230130542}
	\begin{split}
		&\sum_{n\in\bZ}\zeta_{0,(n)}\equiv 1\quad\text{on}\,\,\Omega\,,\\
		&\text{supp}(\zeta_{0,(n)})\subset \{x\in\Omega\,:\,e^{n-1}\leq \trho(x)\leq e^{n+1}\}\,,\\
		&\zeta_{0,(n)}\in C^{\infty}(\bR^d)\quad\text{and}\quad |D^{\alpha}\zeta_{0,(n)}|\leq N(d,\alpha,\zeta)\,e^{-n|\alpha|}\,.
	\end{split}
\end{align}

\begin{defn}\label{220610533}\,\,
	
	\begin{enumerate}
		\item By $H_p^{\gamma}(\Omega)$ we denote the class of all distributions $f\in\cD'(\Omega)$ such that
		\begin{align*}
			\|f\|^p_{H^{\gamma}_{p,\theta}(\Omega)}:=\sum_{n\in\bZ}e^{n\theta}\|\big(\zeta_{0,(n)}f\big)(e^n\cdot)\|_{H^{\gamma}_p(\bR^d)}^p<\infty\,.
		\end{align*}
		
		\item By $\Psi H_{p,\theta}^{\gamma}(\Omega)$ we denote the class of all distributions $f\in\cD'(\Omega)$ such that $f=\Psi g$ for some $g\in H_{p,\theta}^{\gamma}(\Omega)$. The norm in $\Psi H_{p,\theta}^{\gamma}(\Omega)$ is defined by
		$$
		\|f\|_{\Psi H_{p,\theta}^{\gamma}(\Omega)}:=\|\Psi^{-1}f\|_{H_{p,\theta}^{\gamma}(\Omega)}\,.
		$$
	\end{enumerate}
\end{defn}
\vspace{1mm}

We also denote
$$
L_{p,\theta}(\Omega)=H_{p,\theta}^{0}(\Omega)\quad\text{and}\quad \Psi L_{p,\theta}(\Omega)=\Psi H_{p,\theta}^{0}(\Omega)\,.
$$

The spaces $H_{p,\theta}^{\gamma}(\Omega)$ and $\Psi H_{p,\theta}^{\gamma}(\Omega)$ are independent of the choice of $\zeta_0$ (see \cite[Proposition 2.2.4]{Lo1} or Proposition~\ref{220527502}.(5) of this paper).
Therefore we ignore the dependence on $\zeta_0$.
Similar to $H_p^{\gamma}$ and $H_{p,\theta}^{\gamma}(\Omega)$, for $\gamma\in\bN_0$, the space $\Psi H_{p,\theta}^{\gamma}(\Omega)$ has the following equivalent norm:

\begin{lemma}[see Proposition~\ref{220528651}]\label{220512433}
	For any $k\in\bN_0$, 
	\begin{align*}
		\|f\|^p_{\Psi H_{p,\theta}^{k}(\Omega)}\simeq_N  \sum_{|\alpha|\leq k}\int_{\Omega}\big|\rho^{|\alpha|}D^\alpha f\big|^p\Psi^{-p}\rho^{\theta-d}\dd x\,,
	\end{align*}
	where $N=N(c,p,k,\theta,\mathrm{C}_2(\Psi))$.
\end{lemma}

For the case $-\gamma\in\bN$, an equivalent norm of $\Psi H_{p,\theta}^{\gamma}(\Omega)$ is introduced in Corollary~\ref{21.05.26.3}.

\begin{remark}
	If $\Psi$ is a regularization of a Harnack function $\psi$, then we have
	$$
	\|f\|^p_{\Psi H_{p,\theta}^{k}(\Omega)}\simeq_N  \sum_{|\alpha|\leq k}\int_{\Omega}\big|\rho^{|\alpha|}D^\alpha f\big|^p\psi^{-p}\rho^{\theta-d}\dd x
	$$
	where $N=N(d,p,k,\theta,\mathrm{C}_2(\Psi),\mathrm{C}_3(\psi,\Psi))$.
\end{remark}
\vspace{1mm}

The remainder of this subsection presents the properties of $\Psi H_{p,\theta}^{\gamma}(\Omega)$ that are used in Subsection~\ref{0043}.
Specifically, we focus on the generalization from $H_{p,\theta}^{\gamma}(\Omega)$ to $\Psi H_{p,\theta}^{\gamma}(\Omega)$.
While the properties of $H_{p,\theta}^{\gamma}(\Omega)$ are provided in Appendix~\ref{0082}, we list the properties of $H_{p,\theta}^{\gamma}(\Omega)$ in Lemma~\ref{21.05.20.3}, which are directly used in this subsection and Subsection~\ref{0043}.

We denote
\begin{align*}
	\cI=\{d,\,p,\,\gamma,\,\theta\}\quad\text{and}\quad \cI'=\{d,\,p,\,\gamma,\,\theta,\,\mathrm{C}_2(\Psi)\}\,.
\end{align*}
\begin{lemma}[see Proposition~\ref{220527502}]\label{21.05.20.3}\,\,
	
	\begin{enumerate}
		\item For any $s<\gamma$,
		$$
		\|f\|_{H_{p,\theta}^{s}(\Omega)}\lesssim_{\cI,s}\|f\|_{H_{p,\theta}^{\gamma}(\Omega)}\,.
		$$
		
		\item For any $\eta\in C_c^{\infty}(\bR_+)$,
		\begin{align*}
			\sum_{n\in\bZ}e^{n\theta}\|\eta\big(e^{-n}\trho(e^n\cdot)\big)f(e^n\cdot)\|^p_{H^{\gamma}_{p}}\lesssim_{\cI,\eta}\|f\|_{H^{\gamma}_{p,\theta}(\Omega)}^p.
		\end{align*}
		
		\item For any $s\in\bR$,
		$$
		\|\trho^{\,s} f\|_{H^{\gamma}_{p,\theta}(\Omega)}\simeq_{\cI,s} \|f\|_{H^{\gamma}_{p,\theta+sp}(\Omega)}\,.
		$$
		
		\item For any multi-index $\alpha$,
		$$
		\|D^{\alpha}f\|_{H^{\gamma}_{p,\theta}(\Omega)}\lesssim_{\cI,\alpha}\|f\|_{H^{\gamma+|\alpha|}_{p,\theta-|\alpha|p}(\Omega)}.
		$$
		
		\item Let $k\in\bN_0$ such that $|\gamma|\leq k$. If $a\in C^k_{\mathrm{loc}}(\Omega)$ satisfies
		\begin{align*}
			|a|_{k}^{(0)}:=\sup_{\Omega}\sum_{|\alpha|\leq k}\rho^{|\alpha|}|D^{\alpha}a|<\infty\,\,,
		\end{align*}
		then
		$$
		\|af\|_{H^{\gamma}_{p,\theta}(\Omega)}\lesssim_{\cI}|a|_{k}^{(0)}\|f\|_{H^{\gamma}_{p,\theta}(\Omega)}.
		$$
	\end{enumerate}
\end{lemma}

\begin{remark}
	Lemma~\ref{21.05.20.3} also holds if $f$ is replaced by $\Psi^{-1}f$.
	Therefore Lemmas~\ref{21.05.20.3}.(1), (3), and (5) remain valid when $H_{\ast,\ast}^{\ast}(\Omega)$ is replaced by $\Psi H_{\ast,\ast}^{\ast}(\Omega)$.
\end{remark}

\begin{lemma}\label{21.09.29.4}\,\,
	
	\begin{enumerate}
		\item $C_c^{\infty}(\Omega)$ is dense in $\Psi H_{p,\theta}^{\gamma}(\Omega)$.
		
		\item  $\Psi H_{p,\theta}^{\gamma}$ is a reflexive Banach space with the dual $\Psi^{-1}H_{p',\theta'}^{-\gamma}(\Omega)$, where
		$$
		\frac{1}{p}+\frac{1}{p'}=1\quad\text{and}\quad \frac{\theta}{p}+\frac{\theta'}{p'}=d\,.
		$$
		Moreover, for any $f\in\cD'(\Omega)$, we have
		$$
		N^{-1}\|f\|_{\Psi H_{p,\theta}^{\gamma}(\Omega)}\leq \sup_{g\in C_c^{\infty}(\Omega),g\neq 0}\frac{\langle f,g\rangle }{\|g\|_{\Psi^{-1}H_{p',\theta'}^{-\gamma}(\Omega)}}\leq N\|f\|_{\Psi H_{p,\theta}^{\gamma}(\Omega)}
		$$
		where $N=N(\cI')$.
		
		\item For any $k,\,l\in\bN_0$,
		$$
		\|\big(D^k\Psi\big) D^lf\|_{H^{\gamma}_{p,\theta}(\Omega)}\leq N\|\Psi f\|_{H^{\gamma+l}_{p,\theta-(k+l)p}(\Omega)}
		$$
		where $N=N(\cI',l,k)>0$.
		
		\item Let $\Phi$ be a regular Harnack function on $\Omega$, and there exist a constant $N_0>0$ such that
		$$
		\Psi\leq N_0\Phi \quad\text{on}\,\,\,\Omega\,.
		$$
		Then
		$$
		\|\Psi f\|_{H^{\gamma}_{p,\theta}(\Omega)}\leq N \|\Phi f\|_{H^{\gamma}_{p,\theta}(\Omega)}.
		$$
		where $N=N(\cI',\mathrm{C}_2(\Phi),N_0)>0$.
		
		\item Let $p'\in(1,\infty)$, $\gamma',\,\theta'\in\bR$, and $\Psi'$ be a regular Harnack function on $\Omega$, 
		if $f\in \Psi H_{p,\theta}^{\gamma}(\Omega)\cap \Psi'H_{p',\theta'}^{\gamma'}(\Omega)$, then there exists $\{f_n\}_{n\in\bN}\subset  C_c^{\infty}(\Omega)$ such that
		$$
		\|f-f_n\|_{\Psi H_{p,\theta}^{\gamma}(\Omega)}+\|f-f_n\|_{\Psi' H_{p',\theta'}^{\gamma'}(\Omega)}\rightarrow 0\quad\text{as}\,\,\,n\rightarrow\infty\,.
		$$
	\end{enumerate}
	
\end{lemma}
\begin{proof} 
	(1), (2) When $\Psi\equiv 1$, the results can be found in \cite{Lo1} (or see Proposition~\ref{220527502111} of this paper).
	Since the map $f\mapsto \Psi^{-1}f$ is an isometric isomorphism from $\Psi H_{p,\theta}^{\gamma}(\Omega)$ to $H_{p,\theta}^{\gamma}(\Omega)$, there is nothing to prove.
	
	(3) Since $\Psi$ and $\trho$ are regular Harnack functions, we obtain that for any $k,\,m\in\bN_0$,
	$$
	\Big|\frac{D^k\Psi}{\trho^{\,-k}\Psi}\Big|^{(0)}_m\leq N(d,k,m,\mathrm{C}_2(\Psi))\,.
	$$
	By Lemma~\ref{21.05.20.3}.(5) and (3), we have
	\begin{align}\label{21.11.16.2}
		\|\big(D^k\Psi\big) f\|_{H^{\gamma}_{p,\theta}(\Omega)}\lesssim_{\cI',k}\|\trho^{\,-k}\Psi f\|_{H^{\gamma}_{p,\theta}(\Omega)}\lesssim_{\cI',k}\|\Psi f\|_{H^{\gamma}_{p,\theta-kp}(\Omega)}\,.
	\end{align}
	Therefore we only need to prove that for any $l\in\bN$,
	$$
	\|\Psi D^lf\|_{H^{\gamma}_{p,\theta}(\Omega)}\lesssim_{\cI',l}\|\Psi f\|_{H^{\gamma+l}_{p,\theta-lp}(\Omega)}\,.
	$$
	Recall that $\Psi^{-1}$ is a regular Harnack function, and $\mathrm{C}_2(\Psi^{-1})$ can be chosen to depend only on $\mathrm{C}_2(\Psi)$ and $d$.
	It follows from Leibniz's rule, \eqref{21.11.16.2}, and Lemma~\ref{21.05.20.3}.(4) and (1) that
	\begin{alignat*}{2}
		\|\Psi D^l(\Psi^{-1}\Psi f)\|_{H^{\gamma}_{p,\theta}(\Omega)}&\lesssim_{d,l} &&\sum_{n=0}^l\|\Psi D^{l-n}(\Psi^{-1})\cdot D^n(\Psi f)\|_{H^{\gamma}_{p,\theta}(\Omega)}\\
		&\lesssim_N &&\sum_{n=0}^l\|\Psi^{-1} \Psi D^n(\Psi f)\|_{H^{\gamma}_{p,\theta-(l-n)p}(\Omega)}\\
		&\lesssim_N&&\|\Psi f\|_{H^{\gamma+l}_{p,\theta-lp}(\Omega)}\,.
	\end{alignat*}
	
	(4) For any $k\in\bN_0$, 
	$$
	|\Psi\Phi^{-1}|^{(0)}_k\leq N(d,k,\mathrm{C}_2(\Psi),\mathrm{C}_2(\Phi), N_0)\,.
	$$
	Therefore it follows from Lemma~\ref{21.05.20.3}.(5). that
	$$
	\|\Psi f\|_{H_{p,\theta}^{\gamma}(\Omega)}=\|\Psi \Phi^{-1}(\Phi f)\|_{H_{p,\theta}^{\gamma}(\Omega)}\lesssim_N \|\Phi f\|_{H_{p,\theta}^{\gamma}(\Omega)}\,.
	$$
	
	(5) This follows from Lemma~\ref{22.04.11.3}.
\end{proof}

\begin{remark}\label{2212141032}
	Lemmas~\ref{220512433} and \ref{21.05.20.3}.(1) imply that $C_c^{\infty}(\Omega)$ is continuously embedded in $\Psi H_{p,\theta}^{\gamma}(\Omega)$.
	Due to Lemmas~\ref{21.09.29.4}.(1), (2), and that $C_c^{\infty}(\Omega)$ is separable, we obtain that $\Psi H_{p,\theta}^{\gamma}(\Omega)$ is a separable and reflexive Banach space.
\end{remark}

\begin{remark}\label{21.10.05.2}
	From Lemma~\ref{21.09.29.4}.(4), it follows thatfor regular Harnack functions $\Psi$ and $\Phi$, if 
	$N^{-1}\Phi\leq \Psi\leq N\Phi$
	for some constant $N>0$,
	then $\Psi H_{p,\theta}^{\gamma}(\Omega)$ coincides with $\Phi H_{p,\theta}^{\gamma}(\Omega)$.
	Therefore, applying Lemma~\ref{21.05.20.3}.(3), we see that if $\Psi$ is a regularization of $\rho^\sigma$ ($\sigma\in\bR$), then $\Psi H_{p,\theta}^{\gamma}(\Omega)= H_{p,\theta-\sigma p}^{\gamma}(\Omega)$.
\end{remark}

\begin{lemma}\label{22.02.16.1}
	There exist linear maps 
	$$
	\Lambda_0\,:\,\Psi H_{p,\theta}^{\gamma}\rightarrow \Psi H_{p,\theta}^{\gamma+1}(\Omega)\quad\text{and}\quad \,\Lambda_1,\,\ldots,\,\Lambda_d:\Psi H_{p,\theta}^{\gamma}\rightarrow \Psi H_{p,\theta-p}^{\gamma+1}(\Omega)
	$$
	such that for any $f\in \Psi H_{p,\theta}^{\gamma}(\Omega)$, 
	\begin{align*}
		\begin{gathered}
			f=\Lambda_0 f+\sum_{i=1}^dD_i(\Lambda_if)\quad\text{and}\\
			\|\Lambda_0 f\|_{\Psi H_{p,\theta}^{\gamma+1}(\Omega)}+\sum_{i=1}^d\|\Lambda_i f\|_{\Psi H_{p,\theta-p}^{\gamma+1}(\Omega)}\lesssim_{\cI'}\|f\|_{\Psi H_{p,\theta}^{\gamma}(\Omega)}\,.
		\end{gathered}
	\end{align*}	
\end{lemma}
\begin{proof}
	It is provided in Lemma~\ref{220527700} that there exists linear maps
	$$
	\widetilde{\Lambda}_0,\,\ldots,\,\widetilde{\Lambda}_d\,:\,H_{p,\theta}^{\gamma}(\Omega)\rightarrow \cD'(\Omega)
	$$
	such that for any $g\in H_{p,\theta}^{\gamma}(\Omega)$, 
	\begin{equation}\label{2208280050}
		\begin{aligned}
			g=\widetilde{\Lambda}_0 g+\sum_{i=1}^dD_i(\widetilde{\Lambda}_ig)\quad\text{and}\quad \|\widetilde{\Lambda}_0 g\|_{H_{p,\theta}^{\gamma+1}(\Omega)}+\sum_{i=1}^d\|\widetilde{\Lambda}_i g\|_{H_{p,\theta-p}^{\gamma+1}(\Omega)}\lesssim_{\cI}\| g\|_{H_{p,\theta}^{\gamma}(\Omega)}\,.	
		\end{aligned}
	\end{equation}
	For $f\in \Psi H_{p,\theta}^{\gamma}$ ($\Leftrightarrow\Psi^{-1}f\in H_{p,\theta}^{\gamma}(\Omega)$), put
	\begin{align*}
		&\Lambda_0f=\Psi \widetilde{\Lambda}_0(\Psi^{-1}f)-\sum_{i=1}^d\big(D_i\Psi\big)\times \widetilde{\Lambda}_i(\Psi^{-1}f)\,;\\
		&\Lambda_if=\Psi \widetilde{\Lambda}_i(\Psi^{-1}f)\qquad \qquad \text{for}\,\,i=1,\,\ldots,\,d\,.
	\end{align*}
	Then we have
	$$
	\Big(\Lambda_0+\sum_{i=1}^dD_i\Lambda_i\Big)f=\Psi\Big(\widetilde{\Lambda}_0+\sum_{i=1}^dD_i\widetilde{\Lambda}_i\Big)\big(\Psi^{-1}f\big)=f\,.
	$$
	Moreover, Lemma~\ref{21.09.29.4}.(3) and \eqref{2208280050} imply that
	\begin{alignat*}{2}
		&&&\|\Psi^{-1}\Lambda_0f\|_{H_{p,\theta}^{\gamma+1}(\Omega)}+\sum_{i=1}^d\|\Psi^{-1}\Lambda_if\|_{H_{p,\theta-p}^{\gamma+1}(\Omega)}\\
		&\leq\,&&\|\widetilde{\Lambda}_0(\Psi^{-1} f)\|_{H_{p,\theta}^{\gamma+1}(\Omega)}+\sum_{i=1}^d\Big(\|\Psi^{-1}D\Psi\cdot \widetilde{\Lambda}_i(\Psi^{-1}f)\|_{H_{p,\theta}^{\gamma+1}(\Omega)}+\|\widetilde{\Lambda}_i(\Psi^{-1}f)\|_{H_{p,\theta-p}^{\gamma+1}(\Omega)}\Big)\\
		&\lesssim_{\cI'}\,&&\|\widetilde{\Lambda}_0(\Psi^{-1} f)\|_{H_{p,\theta}^{\gamma+1}(\Omega)}+\sum_{i=1}^d\|\widetilde{\Lambda}_i(\Psi^{-1}f)\|_{H_{p,\theta-p}^{\gamma+1}(\Omega)}\\
		&\lesssim_\cI\,&&\|\Psi^{-1}f\|_{H_{p,\theta}^{\gamma}(\Omega)}\,.
	\end{alignat*}
	Therefore the proof is completed.
\end{proof}

\begin{corollary}\label{21.05.26.3}
	For any $n\in\bN$ and $f\in\cD'(\Omega)$,
	\begin{align*}
		\|f\|_{\Psi H_{p,\theta}^{\gamma}(\Omega)}\simeq_{\cI',n} \inf\Big\{\sum_{|\alpha|\leq n}\|f_{\alpha}\|_{\Psi H_{p,\theta-|\alpha|p}^{\gamma+n}(\Omega)}:\,f=\sum_{|\alpha|\leq n}D^{\alpha}f_{\alpha}\Big\}\,.
	\end{align*}
\end{corollary}
\begin{proof}
	Repeatedly applying Lemma~\ref{22.02.16.1}, we obtain linear maps
	$$
	\Lambda_{n,\alpha}:\Psi H_{p,\theta}^{\gamma}(\Omega)\rightarrow \Psi H_{p,\theta-|\alpha|p}^{\gamma+n}(\Omega)\,,
	$$
	indexed by multi-indices $\alpha$ with $|\alpha|\leq n$, such that for any $f\in\Psi H_{p,\theta}^{\gamma}(\Omega)$,
	\begin{align*}
		f=\sum_{|\alpha|\leq n}D^{\alpha}\big(\Lambda_{n,\alpha}f\big)\quad \text{and}\quad \sum_{|\alpha|\leq n}\|\Lambda_{n,\alpha}f\|_{\Psi H_{p,\theta-|\alpha|p}^{\gamma+n}(\Omega)}\lesssim_{\cI',n}\|f\|_{\Psi H_{p,\theta}^{\gamma}(\Omega)}\,.
	\end{align*}
	Therefore we obtain that for any $f\in\cD'(\Omega)$,
	$$
	\inf\Big\{\sum_{|\alpha|\leq n}\| f_{\alpha}\|_{\Psi H_{p,\theta-|\alpha|p}^{\gamma+n}(\Omega)}:\,f=\sum_{|\alpha|\leq n}D^{\alpha}f_{\alpha}\Big\}\lesssim_{\cI',n}\| f\|_{\Psi H_{p,\theta}^{\gamma}(\Omega)}\,.
	$$
	
	For the inverse inequality, let $f=\sum_{|\alpha|\leq n}D^{\alpha}f_{\alpha}$ where $f_{\alpha}\in H^{\gamma+n}_{p,\theta-|\alpha|p}(\Omega)$.
	It follows from Lemma~\ref{21.09.29.4}.(2) and (3) that for any $g\in C_c^{\infty}(\Omega)$,
	\begin{alignat*}{2}
		|\langle f,g\rangle |\,&\leq &&\sum_{|\alpha|\leq n}\big|\big\langle \Psi^{-1}f_{\alpha},\Psi D^{\alpha}g\rangle\big|\\
		&\lesssim_{\cI',n} &&\sum_{|\alpha|\leq n}\Big(\|\Psi^{-1}f_{\alpha}\|_{H_{\theta-|\alpha|p}^{\gamma+n}(\Omega)}\|\Psi D^{\alpha}g\|_{H_{p',\theta'+|\alpha|p'}^{-\gamma-n}(\Omega)}\Big)\\
		&\lesssim_{\cI',n} \Big(&&\sum_{|\alpha|\leq n}\|\Psi^{-1}f_{\alpha}\|_{H_{\theta-|\alpha|p}^{\gamma+n}(\Omega)}\Big)\|\Psi g\|_{H^{-\gamma}_{p',\theta'}(\Omega)}\,,
	\end{alignat*}
	where $p'=p/(p-1)$ and $\theta'/p'=d-\theta/p$.
	By taking the infimum over $\{f_\alpha\}_{|\alpha|\leq n}$ and applying Lemma~\ref{21.09.29.4}.(2), we have
	$$
	\| f\|_{\Psi H_{p,\theta}^{\gamma}(\Omega)}\lesssim_{\cI',n}\inf\Big\{\sum_{|\alpha|\leq n}\|  f_{\alpha}\|_{\Psi H^{\gamma+n}_{p,\theta-|\alpha|p}(\Omega)}:\,f=\sum_{|\alpha|\leq n}D^{\alpha}f_{\alpha}\Big\}\,.
	$$
	Therefore the proof is completed.
\end{proof}

We end this subsection with Proposition~\ref{220512537}, which is a Sobolev-H\"older embedding theorem for the spaces $\Psi H_{p,\theta}^{\gamma}(\Omega)$.
This proposition is not used in Subsection~\ref{0043}.
However, it provides H\"older estimates for solutions obtained in Theorem~\ref{21.09.29.1}.
For $k\in\bN_0$, $\alpha\in(0,1]$ and $\delta\in\bR$, we define the weighted H\"older norm
\begin{align*}
	|f|^{(\delta)}_{k,\alpha}:=\sum_{i=0}^k\sup_{\Omega}\left|\rho^{\delta+i}D^if\right|+\sup_{x,y\in\Omega}\frac{\left|\big(\trho^{\delta+k+\alpha} D^kf\big)(x)-\big(\trho^{\delta+k+\alpha} D^kf\big)(y)\right|}{|x-y|^{\alpha}}\,.
\end{align*}

\begin{prop}\label{220512537}
	Let $k\in\bN_0$, $\alpha\in(0,1]$.
	\begin{enumerate}
		\item For any $\delta\in\bR$,
		\begin{align*}
			\left|\Psi^{-1} f\right|^{(\delta)}_{k,\alpha}\simeq_{N} &\,\sum_{i=0}^k\sup_{x\in\Omega}\left|\Psi(x)^{-1}\rho(x)^{\delta+i} D^if(x)\right|\\
			&+\sup_{x\in\Omega}\bigg(\Psi^{-1}(x)\rho^{\delta+k+\alpha}(x)\sup_{y:|y-x|<\frac{\rho(x)}{2}}\frac{\left|D^kf(x)-D^kf(y)\right|}{|x-y|^{\alpha}}\bigg)\,,
		\end{align*}
		where $N=N(d,k,\alpha,\delta,\mathrm{C}_2(\Psi))$.
		
		\item
		If $\alpha\in(0,1)$ and $k+\alpha\leq \gamma-d/p$,
		then for any $f\in \Psi H^{\gamma}_{p,\theta}(\Omega)$,
		\begin{align*}
			\left|\Psi^{-1} f\right|^{(\theta/p)}_{k,\alpha}\leq N \|f\|_{\Psi H^{\gamma}_{p,\theta}(\Omega)}\,,
		\end{align*}
		where $N=N(\cI',k,\alpha)$.
	\end{enumerate}
\end{prop}
\begin{proof}
	(1)
	This result from the direct calculation and the definition of regular Harnack functions.
	Therefore we leave the proof to the reader.
	
	(2)
	We only need prove for $\Psi\equiv 1$, and the result for this case is stated in \cite[Theorem 4.3]{Lo1}. 
	We give a proof for the convenience of the reader.
	
	For $f\in H_{p,\theta}^{\gamma}(\Omega)$, the Sobolev embedding theorem implies
	\begin{align}\label{220614409}
		\left\|\big(f\zeta_{0,(n)}\big)(e^n\,\cdot\,)\right\|_{C^{k,\alpha}}\leq N\left\|\big(f\zeta_{0,(n)}\big)(e^n\,\cdot\,)\right\|_{H_p^{\gamma}}<\infty\,,
	\end{align}
	where $N=N(d,p,\gamma,k,\delta)$.
	Hence $f$ belongs to $C^k_{\mathrm{loc}}(\Omega)$.
	For $x\in\Omega$, take $n_0\in\bZ$ such that $e^{n_0-1}\leq \rho(x)\leq e^{n_0}$.
	If $|x-y|<\frac{\rho(x)}{2}$, then $e^{n_0-2}\leq \rho(y)\leq e^{n_0+2}$.
	Take constants $A$ and $B$ depending only on $d$ such that
	\begin{align*}
		A^{-1}\rho\leq \trho \leq A\rho\,\,,\quad\text{and}\quad\sum_{|n|\leq B}\zeta_0\big(e^nt\big)\equiv 1\quad\text{for all}\quad t\in [(Ae^2)^{-1},Ae^2]\,.
	\end{align*}
	Then we have
	$$
	\sum_{|n-n_0|\leq B}\zeta_{0,(n)}\equiv 1\quad \text{on}\quad U_{n_0}:=\left\{y\,:\,e^{n_0-2}\leq \rho(y)\leq e^{n_0+2}\right\}.
	$$
	Due to $B(x,\rho(x)/2)\subset U_{n_0}$ and \eqref{220614409}, we have
	\begin{alignat*}{2}
		&&&\sum_{i=0}^k\left(\rho^{\theta/p+i}(x)\left|D^if(x)\right|\right)+\rho^{\theta/p+k+\alpha}(x)\sup_{y:|y-x|<\frac{\rho(x)}{2}}\frac{\left|D^kf(x)-D^kf(y)\right|}{|x-y|^{\alpha}}\\
		&\lesssim_N &&\,e^{n_0\theta/p}\bigg(\sum_{i=0}^k\left|D^i\big(f(e^{n_0}\,\cdot\,)\big)(x)\right|\\
		&&&\qquad\qquad  +\sup_{y:e^{-{n_0}}y\in U_{n_0}}\frac{\left|D^k\big(f(e^{n_0}\,\cdot\,)\big)(x)-D^k\big(f(e^{n_0}\,\cdot\,)\big)(y)\right|}{|x-y|^{\alpha}}\bigg)\\
		&\leq &&\sum_{|n-n_0|\leq B}e^{n_0\theta/p}\left\|(f\zeta_{0,(n)})(e^{n}\,\cdot\,)\right\|_{C^{k,\alpha}}\\
		&\lesssim_N &&\left(\sum_{n\in\bZ}e^{n\theta}\left\|(f\zeta_{0,(n)})(e^{n}\,\cdot\,)\right\|_{H_{p}^{\gamma}}^p\right)^{1/p}\,,
	\end{alignat*}
	where $N=N(d,p,\gamma,\theta,k,\delta)$.
	By (1) of this proposition, the proof is completed.
\end{proof}

\vspace{2mm}

\subsection{Solvability of the Poisson equation}\label{0043}\,

Throughout this subsection, we assume \eqref{22082801111}.
The goal of this subsection is to prove the following theorem:

\begin{thm}\label{21.09.29.1}
	Let
	\begin{itemize}
		\item[]$\Omega$ admit the Hardy inequality \eqref{hardy};\vspace{0.5mm}
		\item[]$\psi$ be a superharmonic Harnack function on $\Omega$;\vspace{0.5mm}
		\item[]$\mu\in (-1/p,1-1/p)$,\vspace{0.5mm}
	\end{itemize}
	and suppose that $\Psi$ is a regularization of $\psi$.
	For any $\lambda\geq 0$ and $f\in \Psi^{\mu}H_{p,d+2p-2}^{\gamma}(\Omega)$, the equation
	\begin{align}\label{220613956}
		\Delta u-\lambda u=f
	\end{align}
	has a unique solution $u$ in $\Psi^{\mu}H_{p,d-2}^{\gamma+2}(\Omega)$.
	Moreover, we have
	\begin{align}\label{220606433}
		\|u\|_{\Psi^{\mu}H_{p,d-2}^{\gamma+2}(\Omega)}+\lambda\|u\|_{\Psi^{\mu}H_{p,d+2p-2}^{\gamma}(\Omega)}\leq N\|f\|_{\Psi^{\mu}H_{p,d+2p-2}^{\gamma}(\Omega)},
	\end{align}
	where $N=N(d,p,\gamma,\mu,\mathrm{C}_0(\Omega),\mathrm{C}_2(\Psi),\mathrm{C}_3(\psi,\Psi))$.
\end{thm}
Recall that $\mathrm{C}_0(\Omega)$ is the constant in \eqref{hardy}, and $\mathrm{C}_2(\Psi)$ and $\mathrm{C}_3(\psi,\Psi)$ are the constants in Definition~\ref{21.10.14.1}.

In Theorem~\ref{21.09.29.1}, one can take $\psi=\Psi=1_{\Omega}$.
Another example of $\psi$ is introduced in Example~\ref{220912411} which is associated with the Green function, and valid for any domain admitting the Hardy inequality.

\begin{remark}\label{220617}
	The spaces $\Psi^{\mu} H_{p,d-2}^{\gamma+2}(\Omega)$ and $\Psi^{\mu} H_{p,d+2p-2}^{\gamma}(\Omega)$ in Theorem~\ref{21.09.29.1} do not depend on the specific choice of $\Psi$ among regularizations of $\psi$ (see Remark~\ref{21.10.05.2}).
	If we take $\Psi$ as $\widetilde{\psi}$ which is the regularization of $\psi$ provided in Lemma~\ref{21.05.27.3}.(1),  
	then Theorem~\ref{21.09.29.1} can be reformulated without including $\Psi$, by taking $\Psi=\widetilde{\psi}$.
	Indeed, $\mathrm{C}_2(\widetilde{\psi})$ and $\mathrm{C}_3(\psi,\widetilde{\psi})$ depend only on $d$ and $\mathrm{C}_1(\psi)$, the constant $N$ in \eqref{220606433} depends only on $d$, $p$, $\gamma$, $\mu$, $\mathrm{C}_0(\Omega)$, and $\mathrm{C}_1(\psi)$.
	Therefore,  
	Additionally, for the case $\gamma\in\bZ$, equivalent norms of $\widetilde{\psi}^{\,\mu}H_{p,d-2}^{\gamma+2}(\Omega)$ and $\widetilde{\psi}^{\,\mu}H_{p,d+2p-2}^{\gamma}(\Omega)$ are provided in Lemma~\ref{220512433} and Corollary~\ref{21.05.26.3}.
\end{remark}

\begin{remark}
	If $\mu\notin (-1/p,1-1/p)$, then Theorem~\ref{21.09.29.1} does not hold in general, as pointed out in \cite[Remark 4.3]{Krylov1999-1}.
	To observe this, consider $\Omega=(0,\pi)$, $\psi(x)=\Psi(x)=\sin x$, and $\gamma=0$, and refer Lemmas~\ref{220512433} and \ref{21.05.13.8}.
	
	For $\mu\geq 1-1/p$, we aim to prove the non-existence of solutions to the equation $\Delta u=f$ in $\Psi^{\mu}H_{p,d-2}^2(\Omega)$, for any fixed $f\in C_c^{\infty}(\Omega)$ such that $f\leq 0$ and $f\not\equiv 0$.
	Assume that there exists $u_1\in \Psi^{\mu}H_{p,d-2}^2(\Omega)$ such that $\Delta u_1=f$.
	Since $\Omega=(0,\pi)$ is bounded, we have
	\begin{align*}
		u_1\in\Psi^{\mu}H_{p,d-2}^2(\Omega)=H_{p,d-\mu p-2}^2(\Omega)\subset H_{p,d-p}^2(\Omega)=\Psi^{1-2/p}H_{p,d-2}^2(\Omega)\,.
	\end{align*}
	Let $u_0$ be the classical solution of the equation
	$$
	\Delta u=f\quad\text{on}\,\,\,(0,\pi)\quad ;\quad u(0)=u(\pi)=0.
	$$
	Then $u_0(x)\simeq \sin x$, which implies that
	$$
	u_0\in L_{p,d-\widetilde{\mu} p-2}(\Omega)\,(=\Psi^{\widetilde{\mu}}L_{p,d-2}(\Omega))\quad\Longleftrightarrow \quad \widetilde{\mu}<1-1/p\,.
	$$
	Due to Lemma~\ref{21.05.13.8}, we have 
	\begin{align}\label{230328442}
		u_0\in \Psi^{\widetilde{\mu}}H^2_{p,d-2}(\Omega)\quad\Longleftrightarrow \quad \widetilde{\mu}<1-1/p\,.
	\end{align}
	Since $u_0,\,u_1\in \Psi^{1-2/p} H_{p,d-2}^2(\Omega)$ and $\Delta u_1=\Delta u_0=f$, by the uniqueness of solutions in $\Psi^{1-2/p}H_{p,d-2}^2(\Omega)$ (see Theorem~\ref{21.09.29.1}), we conclude that $u_1=u_0$, which implies $u_1\in \Psi^{\mu} H_{p,d-2}^2(\Omega)$.
	However, this leads to a contradiction, since $\mu\geq 1-1/p$ (see \eqref{230328442}).
	Therefore, there are no solutions to the equation $\Delta u=f$ in $\Psi^{\mu}H_{p,d-2}^{2}(\Omega)$.
	
	If $\mu< -1/p$, then $1_{\Omega}$ belong to $\Psi^{\mu}H^2_{p,d-2}(\Omega)$ (see Lemma~\ref{220512433}).
	Therefore the equation $\Delta u=0$ has at least two solutions in $\Psi^{\mu}H^2_{p,d-2}(\Omega)$.
	
	Consider the case $\mu=-1/p$.
	For $n\in\bN$, take $\zeta_n\in C_c^{\infty}(\Omega)$ such that
	$$
	1_{\left[\frac{2}{n},\pi-\frac{2}{n}\right]}\leq \zeta_n \leq 1_{\left[\frac{1}{n},\pi-\frac{1}{n}\right]}\quad\text{and}\quad \left|D^k\zeta_n\right|\leq N(k)n^{k}\,.
	$$
	One can observe that
	$$
	\log n\lesssim\left\|\zeta_n\right\|_{H_{p,d-1}^2(\Omega)}^p\quad\text{and}\quad \left\|\Delta \zeta_n\right\|_{L_{p,d+2p-1}(\Omega)}^p\lesssim 1
	$$
	for all large enough $n$.
	Therefore there is no constant $N$ satisfying \eqref{220606433} for $\gamma=0$.
\end{remark}

\begin{example}\label{220912411}
	Let $\Omega\subset \bR^d$ be a domain admitting the Hardy inequality.
	We claim that $\phi_0:=G_{\Omega}(x_0,\,\cdot\,)\wedge 1$ is a superharmonic Harnack function on $\Omega$, where $G_{\Omega}$ is the Green function of the  Poisson equation in $\Omega$ (see Remark~\ref{221022314} for $G_{\Omega}$), and $x_0$ is an arbitrary fixed point of $\Omega$. 
	Note that
	\begin{itemize}
		\item[] $G_{\Omega}(x_0,\,\cdot\,)$ is a positive classical superharmonic function on $\Omega$;
		\item[] $G_{\Omega}(x_0,\,\cdot\,)$ is harmonic on $\Omega\setminus \{x_0\}$\,.
	\end{itemize}
	This implies that $\phi_0$ is a classical superharmonic function on $\Omega$ (see Proposition~\ref{21.05.18.1}.(1)).
	
	For $x\in\Omega$, denote $B(x)=B\big(x,\rho(x)/8\big)$.
	If $|x-x_0|>\rho(x)/4$, then $G_{\Omega}(x_0,\,\cdot\,)$ is harmonic on $B\big(x,\rho(x)/4\big)$.
	By the Harnack inequality, we have
	\begin{align*}
		\sup_{B(x)}\phi_0=\Big(\sup_{y\in B(x)}G_{\Omega}(x_0,y)\Big)\wedge 1\lesssim_d \Big(\inf_{y\in B(x)}G_{\Omega}(x_0,y)\Big)\wedge 1=\inf_{B(x)}\phi_0\,.
	\end{align*}
	If $|x-x_0|\leq \rho(x)/4$, then $\rho(x)\leq 4\rho(x_0)/3$, which implies $B(x)\subset B\big(x_0,\rho(x_0)/2\big)$.
	By Proposition \ref{21.04.23.3}.(3), there exists $\epsilon_0\in(0,1]$ such that $G(x_0,\,\cdot\,)\geq \epsilon_0$ on $B\big(x_0,\rho(x_0)/2\big)$.
	Therefore we have
	\begin{align*}
		\sup_{B(x)}\phi_0&\leq 1\leq \epsilon_0^{-1}\inf_{B(x)}\phi_0\,.
	\end{align*}
	Consequently, $\phi_0$ is a superharmonic Harnack function on $\Omega$.
	
	It is worth noting that $\phi_0$ is the smallest positive classical superharmonic function, up to constant multiples. That is, if $\phi$ is a positive classical superharmonic function on $\Omega$, then there exists $N_0=N(\phi,\Omega,x_0)>0$ such that $\phi_0\leq N_0\phi$ on $\Omega$.
	To prove this, we start by noting that $G_{\Omega}(x_0,\cdot)$ is continuous in $\partial B(x_0,\rho(x_0)/2)$ and $\phi$ has the minimum in $\partial B(x_0,\rho(x_0)/2)$. Indeed, $G_{\Omega}(x_0,\cdot)$ is harmonic and $\phi$ is superharmonic in $\Omega\setminus \{x_0\}$.
	Take $M\geq 1$ such that
	$$
	G_{\Omega}(x_0,\cdot)\leq M\phi\quad\text{on}\quad \partial B(x_0,\rho(x_0)/2)\,.
	$$
	Then we have
	$$
	\phi_0\leq G_{\Omega}(x_0,\cdot)\leq M\phi\quad\text{on}\quad \Omega\setminus B(x_0,\rho(x_0)/2)\,,
	$$
	which follows from properties of $G_{\Omega}$ (see \cite[Lemma 4.1.8]{AG}).
	In addition, we obtain that $\phi_0\leq 1\leq M_1\phi$ on $B\big(x_0,\rho(x_0)/2\big)$, where $M_1^{-1}:=\min_{\overline{B}(x_0,\rho(x_0)/2)}\phi>0$. 
\end{example}

To prove Theorem~\ref{21.09.29.1}, we need the following two lemmas; the proof of Theorem~\ref{21.09.29.1} is provided after the proof of Lemma~\ref{21.11.12.1}:

\begin{lemma}[Higher order estimates]\label{21.05.13.8}
	Let $\lambda\geq 0$, and suppose that $u,\,f\in\cD'(\Omega)$ satisfy
	$$
	\Delta u-\lambda u=f\,.
	$$
	Then for any $s\in\bR$,
	\begin{align}\label{21.09.30.1}
		\| u\|_{\Psi H^{\gamma+2}_{p,\theta}(\Omega)}+\lambda\| u\|_{\Psi  H^{\gamma}_{p,\theta+2p}(\Omega)}\leq N\left(\| u\|_{ \Psi H^{s}_{p,\theta}(\Omega)}+\|f\|_{\Psi H^{\gamma}_{p,\theta+2p}(\Omega)}\right)\,,
	\end{align}
	where $N=N(d,p,\theta,\gamma,C_2(\Psi),s)$.
\end{lemma}

\begin{proof}
	We denote $\Phi=\Psi^{-1}$ so that $\mathrm{C}_2(\Phi)$ depends only on $d$ and $\mathrm{C}_2(\Psi)$.
	
	\textbf{Step 1.} First, we consider the case $s\geq \gamma+1$. 
	We can certainly assume that 
	$$
	\|\Phi u\|_{ H^{s}_{p,\theta}(\Omega)}+\|\Phi f\|_{ H^{\gamma}_{p,\theta+2p}(\Omega)}<\infty\,,
	$$
	for if not, there is nothing to prove.
	Since
	\begin{align*}
		\|\Phi u\|_{H_{p,\theta}^{\gamma+1}(\Omega)}\lesssim_{d,p,s,\gamma} \|\Phi u\|_{H_{p,\theta}^{s}(\Omega)}
	\end{align*}
	(see Lemma~\ref{21.05.20.3}.(1)), we only need to prove for $s=\gamma+1$.
	Put
	\begin{align}\label{2210241127}
		v_n(x)=\zeta_0\big(e^{-n}\trho(e^nx)\big)\Phi(e^nx) u(e^nx)\,.
	\end{align}
	Since
	$$
	\sum_{n\in\bZ}e^{n\theta}\left\|v_n\right\|_{H_p^{\gamma+1}(\bR^d)}^p=\|\Phi u\|_{H_{p,\theta}^{\gamma+1}(\Omega)}^p<\infty\,,
	$$
	we have $v_n\in H_p^{\gamma+1}(\bR^d)$.
	Observe that
	\begin{align}\label{2301261222}
		\Delta v_n-e^{2n}\lambda v_n=\widetilde{f}_n\,,
	\end{align}
	where
	\begin{align*}
		\widetilde{f}_n(x)=\,&e^{2n}\zeta_0\big(e^{-n}\trho(e^nx)\big)\big(\Phi f\big)(e^nx)\\
		&-e^{2n}\zeta_0\big(e^{-n}\trho(e^nx)\big)\big(\Phi\Delta u\big)(e^nx)+\Delta v_n(t,x)\\
		=\,&e^{2n}\zeta_0\big(e^{-n}\trho(e^nx)\big)(\Phi f)(e^nx)\\
		&+2e^n\zeta_0'\big(e^{-n}\trho(e^nx)\big)\big(\nabla\trho\cdot \nabla(\Phi u)\big)(e^nx)\\
		&+2e^{2n}\zeta_0\big(e^{-n}\trho(e^nx)\big)\big(\nabla u\cdot\nabla\Phi\big)(e^nx)\\
		&+\zeta_0''\big(e^{-n}\trho(e^nx)\big)\big(|\nabla\trho|^2\Phi u\big)(e^nx)\\
		&+e^n\zeta_0'\big(e^{-n}\trho(e^nx)\big)\big((\Delta\trho)\Phi u\big)(e^nx)\\
		&+e^{2n}\zeta_0\big(e^{-n}\trho(e^nx)\big)\big((\Delta\Phi)u\big)(e^nx)\,.
	\end{align*}
	Make use of Lemmas~\ref{21.05.20.3}.(1) - (3) and \ref{21.09.29.4}.(3) to obtain
	\begin{align}\label{22.04.15.1}
		\begin{split}
			&\sum_{n\in\bZ}e^{n\theta}\big\|\widetilde{f}_n\big\|_{H^{\gamma}_p(\bR^d)}^p\\
			\lesssim_N &\left\|\Phi f\right\|^p_{H_{p,\theta+2p}^{\gamma}(\Omega)}
			+\left\|\trho_x(\Phi u)_x\right\|^p_{H_{p,\theta+p}^{\gamma}(\Omega)}
			+\left\|\Phi_x u_x\right\|^p_{H_{p,\theta+2p}^{\gamma}(\Omega)}\\
			&+\left\|\trho_x\trho_x\Phi u\right\|^p_{H_{p,\theta}^{\gamma}(\Omega)}
			+\left\|\trho_{xx}\Phi u\right\|^p_{H_{p,\theta+p}^{\gamma}(\Omega)}
			+\left\|\Phi_{xx} u\right\|^p_{H_{p,\theta+2p}^{\gamma}(\Omega)}\\
			\lesssim_N & \|\Phi f\|^p_{H_{p,\theta+2p}^{\gamma}(\Omega)}
			+\|\Phi u\|^p_{H^{\gamma+1}_{p,\theta}(\Omega)}<\infty\,,
		\end{split}
	\end{align}
	where $N=N(d,p,\gamma,\theta,\mathrm{C}_2(\Psi))$.
	Hence $\tilde{f}_n$ belongs to $H^{\gamma}_p(\bR^d)$, for all $n\in\bZ$.
	
	Due to \eqref{2301261222} and that $v_n\in H_p^{\gamma+1}(\bR^d)$ and $\widetilde{f}_n\in H_p^{\gamma}(\bR^d)$, we have 
	\begin{align*}
		\begin{gathered}
			V_n:=(1-\Delta)^{\gamma/2}v_n\in H_p^1(\bR^d)\,\,,\,\,\,\,F_n:=(1-\Delta)^{\gamma/2}\widetilde{f}_n\in L_p(\bR^d)\,;\\
			\Delta V_n-(e^{2n}\lambda+1)V_n=F_n-V_n\,.
		\end{gathered}
	\end{align*}
	It is implied by classical results for the Poisson equation in $\bR^d$ (see, \textit{e.g.}, \cite[Theorem 4.3.8, Theorem 4.3.9]{Krylov2008})
	that
	\begin{equation}\label{220506140}
		\begin{alignedat}{2}
			\|v_n\|_{H_p^{\gamma+2}(\bR^d)}+e^{2n}\lambda\|v_n\|_{H_p^{\gamma}(\bR^d)}&=&&\left\|V_n\right\|_{H_p^2(\bR^d)}+e^{2n}\lambda\left\|V_n\right\|_{L_p(\bR^d)}\\
			&\leq&& \left\|\Delta V_n\right\|_{L_p(\bR^d)}+(e^{2n}\lambda+1)\left\|V_n\right\|_{L_p(\bR^d)}\\
			&\lesssim_{d,p}&& \|F_n-V_n\|_{L_p(\bR^d)}\\
			&\leq &&\|\widetilde{f}_n\|_{H_p^{\gamma}(\bR^d)}+\|v_n\|_{H_p^{\gamma}(\bR^d)}\,.
		\end{alignedat}
	\end{equation}
	Combine \eqref{220506140} and \eqref{22.04.15.1} to obtain that
	\begin{align*}
		\begin{split}
			\|\Phi u\|_{H^{\gamma+2}_{p,\theta}(\Omega)}^p+\lambda^p\|\Phi u\|_{H^{\gamma}_{p,\theta+2p}(\Omega)}^p=\,\,\,\,& \sum_{n\in\bZ}e^{n\theta}\Big(\|v_n\|_{H^{\gamma+2}_p(\bR^d)}^p+(e^{2n}\lambda)^p\|v_n\|_{H^{\gamma}_p(\bR^d)}^p\Big)\\
			\lesssim_N& \sum_{n\in\bZ}e^{n\theta}\left(\|v_n\|_{H^{\gamma}_p(\bR^d)}^p+\|\tilde{f}_n\|_{H^{\gamma}_p(\bR^d)}^p\right)\\
			\lesssim_N&\|\Phi u\|_{H^{\gamma+1}_{p,\theta}(\Omega)}^p+\|\Phi f\|^p_{H_{p,\theta+2p}^{\gamma}(\Omega)}\,.
		\end{split}
	\end{align*}
	Therefore the case $s=\gamma+1$ is proved.
	Consequently, \eqref{21.09.30.1} holds for all $s\geq \gamma+1$.

	\textbf{Step 2.} Recall that for the case $s\geq\gamma+1$, \eqref{21.09.30.1} is proved in Step 1.
	For $s<\gamma+1$, take $k\in\bN$ such that
	$$
	\gamma+1-k\leq s<\gamma+2-k\,,
	$$
	and repeatedly apply \eqref{21.09.30.1} with $(\gamma,s)$ replaced by $(\gamma,\gamma+1)$, $(\gamma-1,\gamma)$, ..., $(\gamma-k,\gamma+1-k)$.
	Then we have
	\begin{align*}
		&\|\Phi u\|_{H^{\gamma+2}_{p,\theta}(\Omega)}+\lambda\|\Phi u\|_{H^{\gamma}_{p,\theta+2p}(\Omega)}\\
		\lesssim_N &\|\Phi u\|_{H^{\gamma+1}_{p,\theta}(\Omega)}+\|\Phi f\|_{H_{p,\theta+2p}^{\gamma}(\Omega)}\\
		\lesssim_N &\,\,\,\,\cdots\\
		\lesssim_N &\|\Phi u\|_{H^{\gamma-k+1}_{p,\theta}(\Omega)}+\|\Phi f\|_{H_{p,\theta+2p}^{\gamma}(\Omega)}\,.
	\end{align*}
	Since $\|\Phi u\|_{H^{\gamma-k+1}_{p,\theta}(\Omega)}\lesssim \|\Phi u\|_{H^{s}_{p,\theta}(\Omega)}$ (see Lemma~\ref{21.05.20.3}.(1)), the proof is completed.
\end{proof}

\begin{lemma}\label{21.11.12.1}
	Let $\lambda\geq 0$, and suppose the following: 
	\begin{itemize}
		\item[]For any $f\in \Psi H_{p,\theta+2p}^{\gamma}(\Omega)$, in $\Psi H_{p,\theta}^{\gamma+2}(\Omega)$ there exists a unique solution $u$ of equation \eqref{220613956}.
		For this solution, we have
		\begin{align}\label{220606837}
			\|u\|_{\Psi H_{p,\theta}^{\gamma+2}(\Omega)}+\lambda\|u\|_{\Psi H_{p,\theta+2p}^{\gamma}(\Omega)}\leq N_{\gamma}\|f\|_{\Psi H_{p,\theta+2p}^{\gamma}(\Omega)}\,,
		\end{align}
		where $N_{\gamma}$ is a constant independent of $f$ and $u$.
	\end{itemize}
	Then for all $s\in\bR$, the following holds:
	\begin{itemize}
		\item[]
		For any $f\in \Psi H_{p,\theta+2p}^{s}(\Omega)$, in $\Psi H_{p,\theta}^{s+2}(\Omega)$ there exists a unique solution $u$ of equation \eqref{220613956}.
		For this solution, we have
		\begin{align}\label{22.04.01.1}
			\|u\|_{\Psi H_{p,\theta}^{s+2}(\Omega)}+\lambda\|u\|_{\Psi H_{p,\theta+2p}^{s}(\Omega)}\leq N_s\|f\|_{\Psi H_{p,\theta+2p}^{s}(\Omega)}\,,
		\end{align}
		where $N_s$ is a constant depends only on $d,\,p,\,\gamma,\,\theta,\,\mathrm{C}_2(\Psi),\,N_{\gamma}$ and $s$.
	\end{itemize}
\end{lemma}

\begin{proof}
	To prove the uniqueness of solutions, let us assume that $\overline{u}\in \Psi H_{p,\theta}^{s+2}(\Omega)$ satisfies $\Delta \overline{u}-\lambda \overline{u}=0$.
	By Lemma~\ref{21.05.13.8}, $\overline{u}$ belongs to $\Psi H_{p,\theta}^{\gamma+2}(\Omega)$. 
	Due to the assumption of this lemma, in $\Psi H_{p,\theta}^{\gamma+2}(\Omega)$, the zero distribution is the unique solution for the equation $\Delta u-\lambda u=0$.
	Consequently, $\overline{u}$ is also the zero distribution, and the uniqueness of solutions is proved.
	Thus it remains to show the existence of solutions and estimate \eqref{22.04.01.1}.
	
	\textbf{Step 1.} We first consider the case $s> \gamma$.
	Let $f\in\Psi H^{s}_{p,\theta+2p}(\Omega)$.
	Due to $\Psi H^s_{p,\theta+2p}(\Omega)\subset \Psi H^{\gamma}_{p,\theta+2p}(\Omega)$, $f$ belongs to $\Psi H^{\gamma}_{p,\theta+2p}(\Omega)$, and hence there exists a solution $u\in \Psi H^{\gamma+2}_{p,\theta}(\Omega)$ of equation \eqref{220613956}.
	It follows from Lemma~\ref{21.05.13.8}, \eqref{220606837}, and Lemma~\ref{21.05.20.3}.(1) that
	\begin{alignat*}{2}
		\left\|u\right\|_{\Psi H^{s+2}_{p,\theta}(\Omega)}+\lambda\left\|u\right\|_{\Psi H^{s}_{p,\theta+2p}(\Omega)}\,&\lesssim_N&&\left\|u\right\|_{\Psi H^{\gamma+2}_{p,\theta}(\Omega)}+\left\|f\right\|_{\Psi H^{s}_{p,\theta+2p}(\Omega)}\\
		&\leq && N_{\gamma}\|f\|_{\Psi H_{p,\theta+2p}^{\gamma}(\Omega)}+\|f\|_{\Psi H_{p,\theta+2p}^s(\Omega)}\\
		&\lesssim_N &&(N_{\gamma}+1)\left\|f\right\|_{\Psi H^{s}_{p,\theta+2p}(\Omega)}\,,
	\end{alignat*}
	where $N=N(d,p,\theta,\gamma,\mathrm{C}_2(\Psi),s)$.
	Therefore $u$ belongs to $\Psi H_{p,\theta}^{s+2}(\Omega)$, and the proof is completed.
	
	\textbf{Step 2.} Consider the case $s<\gamma$.
	Since the case $s\geq\gamma$ is proved in Step 1, by mathematical induction, it is sufficient to show that if this lemma holds for $s=s_0+1$, then this also holds for $s=s_0$.
	
	Let us assume that this lemma holds for $s=s_0+1$. 
	For $f\in \Psi H_{p,\theta+2p}^{s_0}(\Omega)$, by Lemma~\ref{22.02.16.1}, there exists 
	$$
	f^0\in \Psi H_{p,\theta+2p}^{s_0+1}(\Omega)\quad \text{and}\quad f^1,\,\ldots,\,f^d\in \Psi H_{p,\theta+p}^{s_0+1}(\Omega)
	$$
	such that $f=f^0+\sum_{i=1}^dD_if^i$ and
	\begin{align}\label{220506335}
		\left\|f^0\right\|_{\Psi H_{p,\theta+2p}^{s_0+1}(\Omega)}+\sum_{i=1}^d\left\|\trho^{\,-1}f^i\right\|_{\Psi H_{p,\theta+2p}^{s_0+1}(\Omega)}
		\leq N\|f\|_{\Psi H_{p,\theta+2p}^{s_0}(\Omega)}\,,
	\end{align}
	where $N=N(d,p,\theta,s_0,\mathrm{C}_2(\Psi))$.
	Due to the assumption that this lemma holds for $s=s_0+1$, there exist $v^0,\,\cdots,\,v^d\in \Psi H_{p,d-2}^{s_0+3}(\Omega)$ such that
	\begin{align*}
		\Delta v^0-\lambda v^0=f^0\quad\text{and}\quad \Delta v^i-\lambda v^i=\trho^{\,-1} f^i\quad\text{for }i=1,\,\ldots,\,d\,,
	\end{align*}
	and
	\begin{alignat}{2}\label{21.09.30.100}
		&&&\left\|v^0\right\|_{\Psi H^{s_0+3}_{p,\theta}(\Omega)}+\lambda\left\|v^0\right\|_{\Psi H^{s_0+1}_{p,\theta+2p}(\Omega)}+\sum_{i=1}^d\left(\left\|v^i\right\|_{\Psi H^{s_0+3}_{p,\theta}(\Omega)}+\lambda\left\|v^i\right\|_{\Psi H^{s_0+1}_{p,\theta+2p}(\Omega)}\right)\nonumber\\
		&\leq&&N_{s_0+1}\bigg(\left\|f^0\right\|_{\Psi H^{s_0+1}_{p,\theta+2p}(\Omega)}+\sum_{i=1}^d\left\|\trho^{\,-1} f^i\right\|_{\Psi H^{s_0+1}_{p,\theta+2p}(\Omega)}\bigg)\\
		&\lesssim_N\,&&N_{s_0+1}\|f\|_{\Psi H_{p,\theta+2p}^{s_0}(\Omega)}\,,\nonumber
	\end{alignat}
	where the last inequality follows from \eqref{220506335}. 
	Put $v=v^0+\sum_{i=1}^dD_i\big(\trho v^i\big)$, and observe that
	$$
	\Delta v-\lambda v=f +\sum_{i=1}^dD_i\big(\Delta(\trho v^i)-\trho\Delta v^i\big)\,.
	$$
	By Lemmas~\ref{21.05.20.3} and \ref{21.09.29.4}.(3), we have
	\begin{alignat*}{2}
		&&&\left\|D_i\left(\Delta(\trho v^i)-\trho\Delta v^i\right)\right\|_{\Psi H^{s_0+1}_{p,\theta+2p}(\Omega)}\\
		&\lesssim_N\,&& \|\Delta(\trho v^i)-\trho\Delta v^i\|_{\Psi H^{s_0+2}_{p,\theta+p}(\Omega)}\\
		&\leq \,&& \sum_{k=1}^d\left(\|D_{kk}\trho\cdot v^i\|_{\Psi H^{s_0+2}_{p,\theta+p}(\Omega)}+\|D_k\trho\cdot D_kv^i\|_{\Psi H^{s_0+2}_{p,\theta+p}(\Omega)}\right)\\
		&\lesssim_N\,&& \|v^i\|_{\Psi H^{s_0+3}_{p,\theta}(\Omega)}<\infty\,,
	\end{alignat*}
	where $N=N(d,p,\theta,s_0,\mathrm{C}_2(\Psi))$.
	Due to the assumption that this lemma holds for $s=s_0+1$, there exists $w\in \Psi H^{s_0+3}_{p,\theta}(\Omega)$ such that
	$$
	\Delta w-\lambda w=\sum_{i=1}^dD_i\big(\Delta(\trho v^i)-\trho\Delta v^i\big)\quad(=\Delta v-\lambda v-f),
	$$
	and
	\begin{equation}\label{21.09.30.200}
		\begin{alignedat}{2}
			&&&\|w\|_{\Psi H^{s_0+3}_{p,\theta}(\Omega)}+\lambda\|w\|_{\Psi H^{s_0+1}_{p,\theta+2p}(\Omega)}\\
			&\leq&& N_{s_0+1}\sum_{i=1}^d\left\|D_i\left(\Delta(\trho v^i)-\trho\Delta v^i\right)\right\|_{\Psi H^{s_0+1}_{p,\theta+2p}(\Omega)}\\
			&\lesssim_N\,&& N_{s_0+1}\sum_{i=1}^d\|v^i\|_{\Psi H^{s_0+3}_{p,\theta}(\Omega)}\,.
		\end{alignedat}
	\end{equation}
	Put
	$$
	u=v-w=v^0+\sum_{i=1}^d D_i(\trho v^i)-w\,.
	$$
	Then $u$ satisfies $\Delta u-\lambda u = f$.
	Moreover, by \eqref{21.09.30.100} and \eqref{21.09.30.200}, we obtain \eqref{22.04.01.1} for $s=s_0$.
\end{proof}

\begin{proof}[Proof of Theorem~\ref{21.09.29.1}]
	By Lemma~\ref{21.11.12.1}, we only need to prove for $\gamma=0$.
	
	\textbf{\textit{A priori} estimates.} Assume that $u\in\Psi^{\mu}H_{p,d-2}^2(\Omega)$ and $\Delta u-\lambda u\in\Psi^{\mu}L_{p,d+2p-2}(\Omega)$.
	By Lemma~\ref{21.05.13.8}, we obtain
	\begin{align}\label{220506424}
		\begin{split}
			&\|u\|_{\Psi^{\mu}H_{p,d-2}^{2}(\Omega)}+\lambda\|u\|_{\Psi^{\mu}L_{p,d+2p-2}(\Omega)}\\
			\lesssim_N&\|u\|_{\Psi^{\mu}L_{p,d-2}(\Omega)}+\|\Delta u-\lambda u\|_{\Psi^{\mu}L_{p,d+2p-2}(\Omega)}<\infty\,,
		\end{split}
	\end{align}
	where $N=N(d,p,\mu,\mathrm{C}_2(\Psi))$.
	Due to \eqref{220506424} and Lemma~\ref{21.09.29.4}.(5), whether $\lambda=0$ or $\lambda>0$, there exists $u_n\in C_c^{\infty}(\Omega)$ such that
	$$
	\lim_{n\rightarrow\infty}\Big(\|u-u_n\|_{\Psi^{\mu}H_{p,d-2}^{2}(\Omega)}+\lambda\|u-u_n\|_{\Psi^{\mu} L_{p,d+2p-2}(\Omega)}\Big) =0\,.
	$$
	This implies
	$$
	\lim_{n\rightarrow \infty}\big\|\big(\Delta-\lambda\big)(u-u_n)\big\|_{\Psi^{\mu}L_{p,d+2p-2}(\Omega)}= 0\,.
	$$
	Since $\Psi$ is a regularization of the superharmonic Harnack function $\psi$, Theorem~\ref{21.05.13.2} and Lemma~\ref{220512433} imply
	\begin{equation}\label{221213517}
		\begin{aligned}
			\|u_n\|_{\Psi^{\mu}L_{p,d-2}(\Omega)}\,&\simeq_N\int_{\Omega}|u_n|^p\psi^{-\mu p}\rho^{-2}\dd x\\
			&\lesssim_{N}\int_{\Omega}|\Delta u_n-\lambda u_n|^p \psi^{-\mu p}\rho^{2p-2} \dd x\\
			&\simeq_N \|\Delta u_n-\lambda u_n\|_{\Psi^{\mu}L_{p,d+2p-2}(\Omega)}\,,
		\end{aligned}
	\end{equation}
	where $N=N(d,p,\mu,\mathrm{C}_0(\Omega),\mathrm{C}_2(\Psi),\mathrm{C}_3(\psi,\Psi))$.
	By letting $n\rightarrow \infty$, we obtain \eqref{221213517} for $u$ instead of $u_n$.
	By combining this with \eqref{220506424}, we obtain  the $\textit{a priori}$ estimates,
	\begin{align}\label{220506452}
		\begin{split}
			\,&\| u\|_{\Psi^{\mu}H_{p,d-2}^{2}(\Omega)}+\lambda\| u\|_{\Psi^{\mu}L_{p,d+2p-2}(\Omega)}\\
			\lesssim_N\,&\|u\|_{\Psi^{\mu}L_{p,d-2}(\Omega)}+\|\Delta u-\lambda u\|_{\Psi^{\mu}L_{p,d+2p-2}(\Omega)}\\
			\lesssim_N\,&\|\Delta u-\lambda u\|_{\Psi^{\mu}L_{p,d+2p-2}(\Omega)}\,.
		\end{split}
	\end{align}
	Note that \eqref{220506452} also implies the uniqueness of solutions.
	
	\textbf{Existence of solutions.}
	Since $C_c^{\infty}(\Omega)$ is dense in $\Psi^{\mu}L_{p,d+2p-2}(\Omega)$, for  $f\in \Psi^{\mu}L_{p,d+2p-2}(\Omega)$ there exists $f_n\in C_c^{\infty}(\Omega)$ such that $f_n\rightarrow f$ in $\Psi^{\mu}L_{d,d+2p-2}(\Omega)$.
	Lemmas~\ref{21.05.25.3} and \ref{220512433} yield that for each $n\in\bN$, there exists $u_n\in\Psi^{\mu}L_{p,d-2}^2(\Omega)$ such that 
	$$
	\Delta u_n-\lambda u_n=f_n\,.
	$$
	Due to Lemma~\ref{21.05.13.8}, $u_n\in\Psi^{\mu}H_{p,d-2}^2(\Omega)$.
	Since $f_n\rightarrow f$ in $\Psi^{\mu} L_{p,d+2p-2}(\Omega)$, it follows from \eqref{220506452} that
	$$
	\|u_n-u_m\|_{\Psi^{\mu}H_{p,d-2}^2(\Omega)}\leq N\|f_n-f_m\|_{\Psi^{\mu}L_{p,d+2p-2}}\rightarrow 0
	$$
	as $n,\,m\rightarrow \infty$.
	Therefore there exists $u\in\Psi^{\mu}H_{p,d-2}^2(\Omega)$ such that $u_n$ converges to $u$ in $\Psi^{\mu}H_{p,d-2}^2(\Omega)$.
	Since $u_n$ and $f_n$ converge to $u$ and $f$ in the sense of distribution, respectively (see Lemma~\ref{21.09.29.4}.(2)), $u$ is a solution of equation \eqref{220613956}.
\end{proof}

We end this subsection with a global uniqueness of solutions.

\begin{thm}[Global uniqueness]\label{220530526}
	Suppose that \eqref{hardy} holds for $\Omega$, and for each $i=1,\,2$, $\Psi_i$ is a regularlization of a superharmonic Harnack function, $p_i\in(1,\infty)$, $\gamma_i\in\bR$, and $\mu_i\in (-1/p_i,1-1/p_i)$.
	For $f\in \bigcap_{i=1,2}\Psi_i^{\mu_i} H_{p_i,d+2p_i-2}^{\gamma_i}(\Omega)$ and $i=1,\,2$, let $u^{(i)}\in\Psi_i^{\mu_i} H_{p_i,d-2}^{\gamma_i+2}(\Omega)$ be solutions of the equation
	\begin{align*}
		\Delta u-\lambda u=f\,.
	\end{align*}
	Then $u^{(1)}=u^{(2)}$ in $\cD'(\Omega)$.
\end{thm}

\begin{proof}
	By Lemma~\ref{21.09.29.4}.(5), there exist $\{f_n\}\subset  C_c^{\infty}(\Omega)$ such that
	$$
	f_n\rightarrow f\quad\text{in}\quad \bigcap_{i=1,2}\Psi_i^{\mu_i} H_{p_i,d+2p_i-2}^{\gamma_i}(\Omega)\,.
	$$
	By Lemmas~\ref{21.05.25.3} and \ref{220512433}, for each $n\in\bN$, there exists $u_n\in \bigcap\limits_{i=1,2}\Psi_i^{\mu_i} L_{p_i,d-2}(\Omega)$ such that
	$$
	\Delta u_n-\lambda u_n=f_n\,.
	$$
	Lemma~\ref{21.05.13.8} yields that $u_n\in \bigcap\limits_{i=1,2}\Psi_i^{\mu_i} H_{p_i,d-2}^{\gamma_i+2}(\Omega)$.
	Since
	$$
	(\Delta-\lambda)\big(u_n-u^{(1)}\big)=(\Delta-\lambda)\big(u_n-u^{(2)}\big)=f_n-f\,,
	$$
	Theorem~\ref{21.09.29.1} implies that
	$$
	u_n\rightarrow u^{(1)}\,\,\,\,\text{in}\,\,\,\, \Psi_1^{c_1} H_{p_1,d-2}^{\gamma_1+2}(\Omega)\,\,,\,\, \text{and}\quad u_n\rightarrow u^{(2)}\,\,\,\,\text{in}\,\,\,\, \Psi_2^{c_2} H_{p_2,d-2}^{\gamma_2+2}(\Omega)\,.
	$$
	Consequently, by Lemma~\ref{21.09.29.4}.(2),
	$$
	\langle u^{(1)},g\rangle =\lim_{n\rightarrow \infty}\langle u_n,g\rangle=\langle u^{(2)},g\rangle
	$$
	for all $g\in C_c^{\infty}(\Omega)$.
\end{proof}

\vspace{2mm}

\mysection{Parabolic equations}\label{0050}

For $0<\nu_1\leq \nu_2<\infty$ and $T\in (0,\infty]$, we denote
\begin{itemize}
	\item $\mathrm{M}(\nu_1,\nu_2)$ : the set of all $d\times d$ real-valued symmetric matrices $(\alpha^{ij})_{d\times d}$ satisfying
	$$
	\nu_1|\xi|^2\leq \sum_{i,j=1}^d\alpha^{ij}\xi_i\xi_j\leq \nu_2|\xi|^2\qquad\forall\,\, \xi\in\bR^d;
	$$
	
	\item $\cM_T(\nu_1,\nu_2)$ : the set of all $\cL:=\sum_{i,j=1}^da^{ij}(\cdot)D_{ij}$, where $\{a^{ij}(\cdot)\}_{i,j=1,....,d}$ is a family of time measurable function on $\bR_+$ such that $\big(a^{ij}(t)\big)_{d\times d}\in \mathrm{M}(\nu_1,\nu_2)$ for all $t\in(0,T]$.
\end{itemize}
Throughout this section we assume that $\Omega$ be a domain in $\bR^d$,
\begin{align}\label{2205241211}
	T\in(0,\infty]\,\,,\,\,\,\,0<\nu_1\leq \nu_2<\infty\,\,,\,\,\text{and}\,\,\,\,\cL\in\cM_T(\nu_1,\nu_2)\,,
\end{align}
and use the convention that $(0,T]=(0,\infty)$ and $[0,T]=[0,\infty)$ if $T=\infty$.
We deal with the equation
\begin{align}\label{2206201120}
	\partial_tu=\cL u+f:=\sum_{i,j=1}^da^{ij}(t)D_{ij} u + f\,\,,\quad t\in(0,T]\quad;\quad u(0,\cdot)=u_0\,,
\end{align}
repeating the arguments in Sections~\ref{0030} and \ref{0040}.

\vspace{2mm}

\subsection{Key estimates for parabolic equations}\label{0051}
Unlike $\Delta$, the operator $\cL$ in \eqref{2206201120} consists of variable coefficients. 
Hence Lemma~\ref{03.30}.(2) is not applied directly.
For this reason we introduce the following definition:

\begin{defn}\label{21.11.10.1}
Let $\phi$ be a positive superharmonic function on $\Omega$.
For $\delta\in (0,1]$ and $p\in(1,\infty)$, by $I(\phi,p,\delta)$ we denote the set of all constants $\mu\in (-\frac{1}{p},1-\frac{1}{p})$ satisfying the following: there exists a constant $\mathrm{C}_4>0$ such that the inequality 
\begin{align}\label{21.07.12.1}
	\int_{\Omega\cap\{u\neq 0\}}|u|^{p-2}|\nabla u|^2\phi^{-\mu p} \dd x\leq \mathrm{C}_4\int_{\Omega}\big(-\sum_{i,j=1}^d\alpha^{ij}D_{ij}u\big)\cdot u|u|^{p-2}\phi^{-\mu p} \dd x
\end{align}
holds for all $u\in C_c^{\infty}(\Omega)$ and $(\alpha^{ij})_{d\times d}\in \mathrm{M}(\delta,1)$.
\end{defn}
\vspace{1mm}

We employ the set $I(\phi,p,\delta)$ to state the main theorems in this section, specifically Theorems \ref{05.11.2} and \ref{21.05.13.9}. 
According to Lemma~\ref{03.30}, when $\delta=1$, $I(\phi,p,1)$ coincides with $(-1/p,1-1/p)$.
Notably, even for $\delta\in(0,1)$ and without additional assumptions on $\Omega$, $\phi$, and $p$, the following proposition guarantees the existence of a non-empty interval contained in $I(\phi,\delta,p)$.

\begin{prop}\label{05.11.1}
Let $\phi$ be a positive superharmonic function on $\Omega$, $p\in(1,\infty)$, and $\delta\in(0,1]$.
\begin{enumerate}
	\item If
	\begin{align}\label{22.04.12.1029}
		\mu\in \Big(-\frac{(p-1)/p}{p(\delta^{-1/2}+1)/2-1},\frac{(p-1)/p}{p(\delta^{-1/2}-1)/2+1}\Big)\,,
	\end{align}
	then $\mu\in I(\phi,p,\delta)$, and the constant $\mathrm{C}_4$  in \eqref{21.07.12.1} can be chosen to depend only on $\delta,p$ and $\mu$.
	In particular, $I(\phi,p,\delta)=(-1+1/p,1/p)$.
	
	\item Suppose that for any $(\alpha^{ij})_{d\times d}\in\mathrm{M}(\delta,1)$, $\sum\limits_{i,j=1}^d\alpha^{ij}D_{ij}\phi\leq 0$ in the sense of distribution.
	then
	$$
	I(\phi,\delta,p)=\big(-1/p,1-1/p\big)\,.
	$$
	Moreover for any $\mu\in \big(-1/p,1-1/p\big)$, the constant $\mathrm{C}_4$ in \eqref{21.07.12.1} can be chosen to depend only on $d,\delta,p$ and $\mu$.
\end{enumerate}
\end{prop}

Proposition~\ref{05.11.1}.(2) is used for results on convex domains and domains satisfying the totally vanishing exterior Reifenberg condition; see Subsections~\ref{convex} and \ref{ERD}, respectively.

\begin{proof}[Proof of Proposition~\ref{05.11.1}]
(1) Let $\mu$ satisfy \eqref{22.04.12.1029}. By the same argument as in the beginning of the proof of Lemma~\ref{03.30}.(2) and (3), it sufficies to prove \eqref{21.07.12.1} only for $u\in C_c^{\infty}(\Omega)$ and a positive smooth superharmonic function $\phi$ on a neighborhood of $\text{supp}(u)$.

Put $c=-\mu p\in(-p+1,1)$ and $v=u\phi^{c/(2p-2)}$. 
Due to Lemmas~\ref{21.04.23.4}, \ref{03.30}.(1), and that $(\alpha^{ij})\in\mathrm{M}(\delta,1)$, we have
\begin{align*}
	&-\sum_{i,j}\int_{\Omega}\alpha^{ij}u_{ij}|u|^{p-2}u\phi^c\dd x\\
	=\,&(p-1)\sum_{i,j}\int_{\Omega}\alpha^{ij}u_iu_j|u|^{p-2}\phi^cdx+c\sum_{i,j}\int_{\Omega}\alpha^{ij}|u|^{p-2}uu_i\phi_j\phi^{c-1}\dd x\\
	=\,&(p-1)\int_{\Omega}\Big(\sum_{i,j}\alpha^{ij}v_iv_j\Big)|v|^{p-2}\phi^{c'}\dd x-\frac{c^2}{4(p-1)}\int_{\Omega}|v|^p\Big(\sum_{i,j}\alpha^{ij}\phi_i\phi_j\Big)\phi^{c'-2}\dd x\\
	\geq \,&(p-1)\delta \int_{\Omega}|\nabla v|^2|v|^{p-2}\phi^{c'}\dd x-\frac{c^2}{4(p-1)}\int_{\Omega}|v|^p|\nabla \phi|^2\phi^{c'-2}\dd x\\
	\geq\,&\kappa'\int_{\Omega}|\nabla v|^2|v|^{p-2}\phi^{c'}\dd x\,,
\end{align*}
where $c':=\frac{(p-2)c}{2p-2}\in(-p+1,1)$ and 
$$
\kappa'=\delta(p-1)-\frac{1}{4(p-1)}\Big(\frac{pc}{1-c'}\Big)^2\,.
$$
One can observe that  $\kappa'>0$ if and only if $\mu$ satisfies \eqref{22.04.12.1029}. 
Therefore we only need to show that
\begin{align}\label{220606927}
	\int_{\Omega}|u|^{p-2}|\nabla u|^2\phi^{c} \dd x\leq N(p,\mu)\int_{\Omega}|v|^{p-2}|\nabla v|^2\phi^{c'} \dd x\,.
\end{align}
Note that
\begin{align*}
	&\int_{\Omega\cap\{u\neq 0\}}|v|^{p-2}|\nabla v|^2\phi^{c'} \dd x\\
	\geq\,&\int_{\Omega\cap\{u\neq 0\}}|u|^{p-2}|\nabla u|^2\phi^{c} \dd x+\frac{c}{p-1}\int_{\Omega\cap\{u\neq 0\}}|u|^{p-2}u(\nabla u \cdot \nabla \phi)\phi^{c-1} \dd x\,.
\end{align*}

If $c\in[0,1)$, then $\Delta(\phi^c)\leq 0$ on $\text{supp}(u)$ (see \eqref{22.04.12.1104}), which implies
$$
\frac{c}{p-1}\int_{\Omega\cap\{u\neq 0\}}|u|^{p-2}u(\nabla u \cdot \nabla \phi)\phi^{c-1} \dd x=-\frac{1}{p(p-1)}\int_{\Omega}|u|^{p}\Delta(\phi^c) \dd x\geq 0\,.
$$
Therefore \eqref{220606927} holds.

If $c\in(-p+1,0)$, then Lemma~\ref{03.30}.(1) implies
\begin{align*}
	&\int_{\Omega\cap\{u\neq 0\}}|u|^{p-2}|\nabla u|^2\phi^{c} \dd x+\frac{c}{p-1}\int_{\Omega\cap\{u\neq 0\}}|u|^{p-2}u(\nabla u \cdot \nabla \phi)\phi^{c-1} \dd x\\
	\geq\,& \int_{\Omega\cap\{u\neq 0\}}|u|^{p-2}|\nabla u|^2\phi^{c} \dd x\\
	&+\frac{c}{p-1}\Big(\int_{\Omega\cap\{u\neq 0\}}|u|^{p-2}|\nabla u|^2\phi^{c} \dd x\Big)^{1/2}\Big(\int_{\Omega}|u|^{p}\phi^{c-2}|\nabla \phi|^2 \dd x\Big)^{1/2}\\
	\geq\,&\frac{p-1+c}{(p-1)(1-c)}\int_{\Omega\cap\{u\neq 0\}}|u|^{p-2}|\nabla u|^2\phi^{c} \dd x\,.
\end{align*}
Since $p-1+c>0$, the proof is completed.

(2) For a fixed $\mathrm{A}=(\alpha^{ij})_{d\times d}\in \mathrm{M}(\delta,1)$, take $\mathrm{B}\in\mathrm{M}(\sqrt{\delta},1)$ such that $\mathrm{B}^2=\mathrm{A}$.
We denote $u_{\mathrm{B}}(y)=u(\mathrm{B}y)$ and $\phi_{\mathrm{B}}(y)=\phi(\mathrm{B}y)$.
Since 
$$
\Delta \phi_{\mathrm{B}}=\sum_{i,j=1}^d\alpha^{ij}\big(D_{ij}\phi\big)(\mathrm{B}\,\cdot\,)\leq 0
$$
on $\mathrm{B}^{-1}\Omega:=\{\mathrm{B}^{-1}x\,:\,x\in\Omega\}$ (in the sense of distribution),
Lemma~\ref{03.30}.(2) implies that for any $u\in C_c^{\infty}(\Omega)$,
\begin{equation}\label{230328715}
	\begin{alignedat}{2}
		&&&\int_{\Omega\cap\{u\neq 0\}}|u|^{p-2}|\nabla u|^2\phi^{-\mu p}\dd x\\
		&\leq &&\delta^{-1}\int_{\mathrm{B}^{-1}\Omega \cap \{u_\mathrm{B}\neq 0\}}|u_\mathrm{B}|^{p-2}|\nabla u_\mathrm{B}|^2 \phi_{\mathrm{B}}^{-\mu p}\dd y\\
		&\lesssim_{p,\mu,\delta}\,&&\int_{\mathrm{B}^{-1}\Omega}(-\Delta u_\mathrm{B})\cdot u_\mathrm{B}|u_\mathrm{B}|^{p-2}\phi_{\mathrm{B}}^{-\mu p}\dd y\\
		&= &&\det(\mathrm{B}^{-1})\int_{\Omega}\big(-\alpha^{ij}D_{ij}u\big)\cdot u|u|^{p-2}\phi^{-\mu p} \dd x\,.
	\end{alignedat}
\end{equation}
Since $\det(\mathrm{B}^{-1})=\big(\det(\mathrm{A})\big)^{-1/2}\in[1, \delta^{-d/2}]$, it follows that the last term in \eqref{230328715} is positive, and thus the proof is completed.
\end{proof}

Theorem~\ref{05.11.2} and Lemma~\ref{21.05.25.300} are counterparts of Theorem~\ref{21.05.13.2} and Lemma~\ref{21.05.25.3}, respectively.

\begin{thm}\label{05.11.2}
Suppose that
\begin{align*}
	&\text{$\Omega$ admits the Hardy inequality \eqref{hardy}}\,;\\
	&\text{$\phi$ is a positive superharmonic function on $\Omega$}\,;\\
	&\text{$p\in (1,\infty)$, and $\mu\in I(\phi,p,\nu_1/\nu_2)$}\,.
\end{align*}
Then for any $u\in C_c^{\infty}\big([0,T]\times \Omega\big)$ and $f:=\partial_tu-\cL u$, we have	\begin{align}\label{230328758}
	\begin{split}
		&\sup_{0\leq t\leq T}\int_{\Omega}|u(t,\,\cdot)|^p\phi^{-\mu p}\dd x+\int_0^T\int_{\Omega}|u|^{p}\phi^{-\mu p}\rho^{-2}\dd x\dd t\\
		\leq\,& N \left(\int_{\Omega}|u(0,\cdot)|^p\phi^{-\mu p}\dd x+\int_0^T\int_{\Omega}|f|^p\phi^{-\mu p}\rho^{2p-2}\dd x\dd t\right)\,,
	\end{split}
\end{align}
where $N=N(p,\mu,\mathrm{C}_0(\Omega),\mathrm{C}_4)$.
\end{thm}
\begin{proof}
For a fixed $t_0\in(0,T]$ and $\epsilon>0$, integrate
$$
p\,\big(\partial_tu\big) u|u|^{p-2}\,\phi^{-\mu p}-p\,\sum_{i,j=1}^da^{ij}(t)D_{ij} u\cdot u|u|^{p-2}\phi^{-\mu p}=pf\cdot u|u|^{p-2}\phi^{-\mu p}
$$
over $(0,t_0]\times \Omega$, and apply Young's inequality, to obtain
\begin{equation}\label{230328742}
	\begin{alignedat}{2}
		&\int_{\Omega}|u(t_0,\,\cdot)|^p\phi^{-\mu p}\dd x&&+p\int_0^{t_0}\int_{\Omega}\Big(-\sum_{i,j}a^{ij}(t)D_{ij} u\Big) u|u|^{p-2}\phi^{-\mu p}\dd x\dd t\\
		\leq\,& \int_{\Omega}|u(t_0,\,\cdot)|^p\phi^{-\mu p}\dd x&&+\epsilon^{-p+1}\int_0^{t_0}\int_{\Omega}|f|^p\phi^{-\mu p}\rho^{2p-2}\dd x\dd t\\
		&&&+(p-1)\epsilon\int_0^{t_0}\int_\Omega|u|^{p}\phi^{-\mu p} \rho^{-2}\dd x\dd t\,,
	\end{alignedat}
\end{equation}
for any $\epsilon>0$.
Due to Lemma~\ref{03.30}.(3) and that $\mu\in I(\phi,p,\nu_1/\nu_2)$, we have
\begin{equation}\label{230328743}
	\begin{alignedat}{2}
		\int_0^{t_0}\int_{\Omega}|u|^{p}\phi^{-\mu p}\rho^{-2}\dd x\dd t&\lesssim_{p,\mu,\mathrm{C}_1(\Omega)}&& \int_0^{t_0}\int_{\Omega}|\nabla u|^2|u|^{p-2}\phi^{-\mu p}\dd x\dd t\\
		&\leq_{\nu_2,\mathrm{C}_4} &&\int_0^{t_0}\int_{\Omega}\Big(-\sum_{i,j}a^{ij}(t)D_{ij} u\Big) u|u|^{p-2}\phi^{-\mu p}\dd x\dd t\,.
	\end{alignedat}
\end{equation}
By combining \eqref{230328742} and \eqref{230328743}, taking the supremum over $t_0\in(0,T]$, and choosing a small enough $\epsilon>0$, we obtain \eqref{230328758}; note that since $\phi^{-\mu p}$ is locally integrable (see Proposition~\ref{21.04.23.3}), the first term in \eqref{230328743} is finite.
\end{proof}

Recall that $\langle F,\zeta\rangle$ is the result of application of $F\in \cD'(\Omega)$ to $\zeta\in C_c^{\infty}(\Omega)$.
\begin{lemma}[Existence of a weak solution]\label{21.05.25.300}
Let $\Omega$ admit the Hardy inequality \eqref{hardy}. For any $u_0\in C_c^{\infty}(\Omega)$ and $f\in C_c^{\infty}\big([0,T]\times \Omega\big)$, there exists a measurable function $u:[0,T]\times \Omega\rightarrow \bR$ satisfying the following:
\begin{enumerate}
	\item $u(t,\,\cdot\,)\in L_{1,\mathrm{loc}}(\Omega)$ for each $t\in[0,T]$, and $u\in L_{1,\mathrm{loc}}([0,T]\times \Omega)$.
	
	\item For any $\zeta\in C_c^{\infty}(\Omega)$ and $t\in[0,T]$,
	\begin{align}\label{2206301251}
		\big\langle u(t,\cdot),\zeta\big\rangle=\langle u_0,\zeta\rangle+\int_0^t\big\langle\Delta u(s,\cdot)+f(s,\cdot),\zeta\big\rangle\,\dd s\,.
	\end{align}
	
	\item For any $p\in(1,\infty)$, $\mu\in(-1/p,1-1/p)$ and positive superharmic function $\phi$ on $\Omega$,
	\begin{align}\label{2206131031}
		\begin{split}
			&\sup_{t\in[0,T]}\int_{\Omega}|u(t,\,\cdot\,)|^p\phi^{-\mu p}\dd x+\int_0^T\int_{\Omega}|u|^p\phi^{-\mu p}\rho^{-2}\dd x\dd t\\
			&\leq N\Big(\int_0^{T}\int_{\Omega}|f|^p\phi^{-\mu p}\rho^{2p-2}\dd x\dd s+\int_{\Omega}|u_0|^p\phi^{-\mu p}\dd x\Big)
		\end{split}
	\end{align}
	where $N=N(p,c,\mathrm{C}_0(\Omega))$.
\end{enumerate}
\end{lemma}

\begin{proof}
We repeat the argument of Lemma~\ref{21.05.25.3}.
Take a sequence of infinitely smooth bounded domains $\{\Omega_n\}_{n\in\bN}$ such that
$$
\text{supp}(u_0)\subset \Omega_1\,\,,\quad \text{supp}(f)\subset [0,T]\times\Omega_1\,\,,\quad \overline{\Omega_n}\subset \Omega_{n+1}\,\,,\quad \bigcup_{n}\Omega_n=\Omega\,.
$$
For $h\in C_c^{\infty}(\Omega_1)$ and $H\in C_c^{\infty}([0,T]\times \Omega_1)$, by $R_{n}(h,H)$ we denote the classical solution $U\in C^{\infty}\big([0,T]\times \overline{\Omega_n}\big)$ of the equation
$$
\partial_tU=\Delta U+H1_{\Omega_1}\quad \text{on}\,\,(0,T]\times\Omega_n\quad; \quad U|_{[0,T]\times\partial\Omega_n}\equiv 0\quad\text{and}\quad U(0,\cdot\,)=h1_{\Omega_1}\,.
$$
We first claim that
\begin{align}\label{2206131250}
	\begin{split}
		&\sup_{t\in [0,T]}\int_{\Omega}|R_{n}(h,H)(t,\,\cdot\,)1_{\Omega_n}|^p\phi^c\dd x+\int_0^T\int_{\Omega}|R_{n}(h,H)1_{\Omega_n}|^{p}\rho^{-2}\phi^c\dd x\dd s\\
		&\leq N(p,c,\mathrm{C}_0(\Omega))\left(\int_0^T\int_{\Omega}|H|^p\phi^c\rho^{2p-2}\dd x\dd s+\int_{\Omega}|h|^p\phi^c\dd x\right)\,,
	\end{split}
\end{align}
for all $p\in(1,\infty)$, $c\in(-p+1,1)$ and positive superharmonic functions $\phi$.
Note that $\overline{\Omega_n}$ is a compact subset of $\Omega$, and for each $t\in[0,T]$,
\begin{align*}
	R_{n}(h,H)(t,\cdot\,)\in C^{\infty}(\overline{\Omega_n})\,\,,\,\, R_{n}(h,H)(t,\cdot\,)|_{\partial\Omega_n}\equiv 0\,,
\end{align*}
which implies that $R_{n}(h,H)(t,\,\cdot\,)1_{\Omega_n}$ satisfies condition \eqref{22.01.25.2}.
If $T<\infty$, then we can repeat the proof of Theorem~\ref{05.11.2} for $R_n(h,H)1_{\Omega_n}$ in place of $u$, using Lemma~\ref{03.30}.
This gives us \eqref{2206131250}.
For the case $T=\infty$, we first obtain \eqref{2206131250} for $K\in(0,\infty)$ instead of $T$.
Then, by letting $K\rightarrow \infty$, we obtain \eqref{2206131250} even for the case $T=\infty$.

Take $U_0\in C_c^{\infty}(\Omega)$ and $F\in C_c^{\infty}\big([0,T]\times \Omega\big)$ such that $|u_0|\leq U_0$ and $|f|\leq F$ (recall that $[0,T]:=[0,\infty)$ when $T=\infty$), and put
$$
u_0^1=\frac{U_0+u_0}{2}\,\,,\quad u_0^2=\frac{U_0-u_0}{2}\,\,,\quad  f^1=\frac{F+f}{2}\,\,,\quad f^2=\frac{F-f}{2}\,.
$$
so that these functions are nonnegative, $u_0=u_0^1-u_0^2$, and $f=f^1-f^2$.

For $v_n:=R_n(u_0^1,f^1)1_{\Omega_n}$, the maximum principle implies that
$$
0\leq v_n\leq v_{n+1}\quad\text{on}\quad [0,T]\times\Omega\,.
$$
We denote the pointwise limit of $v_n$ by $v$.
Apply the monotone convergence theorem to \eqref{2206131250} with $(h,H,\phi,p,c)=(u_0^1,f^1,1_{\Omega},2,0)$ to obtain
\begin{align*}
	\sup_{t\in [0,T]}\int_{\Omega}|v(t,\,\cdot\,)|^2\dd x+\int_0^T\int_{\Omega}|v|^{2}\rho^{-2}\dd x\dd t\lesssim \int_0^T\int_{\Omega}|f^1|^2\rho^{2}\dd x\dd t+\int_{\Omega}|u_0^1|^2\dd x\,.
\end{align*}
This implies that $v(t,\cdot\,)\in L_{1,\mathrm{loc}}(\Omega)$ for each  $t\in(0,T]$, and $v\in L_{1,\mathrm{loc}}\big([0,T]\times\Omega\big)$.

We next claim that for any $t\in(0,T]$,
\begin{align}\label{2302011002}
	\big\langle v(t,\cdot),\zeta\big\rangle =\langle u_0^1,\zeta\rangle +\int_0^t\big\langle \Delta v(s,\cdot)+f^1(s,\cdot),\zeta\big\rangle \dd s\,.
\end{align}
For a fixed $\zeta\in C_c^{\infty}(\Omega)$, take $N\in\bN$ such that $\mathrm{supp}(\zeta)\subset \Omega_N$.
Since $v_n:=R_n(u_0^1,f^1)1_{\Omega_n}$, we obtain that for any $n\geq N$,
$$
\int_{\Omega} v_n(t,\cdot)\zeta\dd x =\int_{\Omega} u_0^1\,\zeta\dd x +\int_0^t\int_{\Omega}\Big(v_n(s,\cdot)\Delta \zeta+f^1(s,\cdot)\zeta\Big)\dd x \dd s\,.
$$
Since $0\leq v_n\leq v$, $v(t,\cdot)\in L_{1,\mathrm{loc}}(\Omega)$ for each $t\in[0,T]$, and $v\in L_{1,\mathrm{loc}}([0,T]\times \Omega)$, the Lebesgue dominated convergence theorem implies \eqref{2302011002}.
By the same argument,
$$
w(t,x):=\lim\limits_{n\rightarrow\infty}R_n(u_0^2,f^2)(t,x)1_{\Omega_n}(x)
$$
satsifes that $w(t,\cdot\,)\in L_{1,\mathrm{loc}}(\Omega)$ for each $t\in (0,T]$, and $w\in L_{1,\mathrm{loc}}\big([0,T]\times\Omega\big)$.
In addition, for any $t\in(0,T]$ and $\zeta\in C_c^{\infty}(\Omega)$,
$$
\big\langle w(t,\cdot),\zeta\big\rangle =\langle u_0^2,\zeta\rangle +\int_0^t\big\langle \Delta w(s,\cdot)+f^2(s,\cdot),\zeta\big\rangle \dd s\,.
$$

Put
$$
u:=v-w=\lim\limits_{n\rightarrow\infty}R_{n}(u_0,f)1_{\Omega_n}\,.
$$
Then $u(t,\,\cdot\,)\in L_{1,\mathrm{loc}}(\Omega)$ for each $t\in [0,T]$, $u\in L_{1,\mathrm{loc}}\big([0,T]\times \Omega\big)$, and $u$ satisfies \eqref{2206301251}.
By applying Fatou's lemma to \eqref{2206131250} with $(h,H)=(u_0,f)$, \eqref{2206131031} is obtained.
\end{proof}

\vspace{2mm}

\subsection{Function spaces for parabolic equations}\label{0052}
Throughout this subsection, we assume \eqref{22082801111}.
This subsection introduces the function spaces $\Psi B_{p,\theta}^{\gamma}(\Omega)$, $\Psi\bH_{p,\theta}^{\gamma}(\Omega)$, and $\Psi\cH_{p,\theta}^{\gamma}(\Omega)$.
These spaces correspond to the initial data $u_0$, the force term $f$, and the solution $u$ for equation \eqref{2206201120}, respectively.

For $n\in\bZ$ and $s\in (0,1]$, by $B_p^{n+s}=B_p^{n+s}(\bR^d)$ we denote the Besov space whose norm is given by
\begin{align*}
\|f\|_{B_p^{n+s}(\bR^d)}:=\big\|(1-\Delta)^{n/2}f\big\|_{L_p(\bR^d)}+\big[(1-\Delta)^{n/2}f\big]_{B_p^s(\bR^d)}\,,
\end{align*}
where $(1-\Delta)^{n/2}f$ is introduced in \eqref{2207201137}, and
\begin{align*}
[f]_{B_p^s(\bR^d)}:=\bigg(\int_{\bR^d}\int_{\bR^d}\frac{|f(x+h)-2f(x)+f(x-h)|^p}{|h|^{d+s p}}\dd h\,\dd x\bigg)^{1/p}\,.
\end{align*}
Note that $B_p^{n+s}(\bR^d)$ coincides with $B_{p,p}^{n+s}(\bR^d)$ introduced in \cite[Definition 2.3.1/2]{triebel2}, and for any $\gamma,\,s\in\bR$, 
\begin{align*}
\|f\|_{B_p^{\gamma+s}(\bR^d)}\simeq_{d,p,\gamma,s} \|(1-\Delta)^{s/2}f\|_{B_p^{\gamma}(\bR^d)}
\end{align*}
(see, \textit{e.g.}, \cite[Theorem 2.3.8/(i), Remark 2.5.12/2]{triebel2}).
If $n\in\bN_0$ and  $s\in(0,1)$, then $B_{p}^{n+s}$ also coincides with the Sobolev-Slobodeckij space
\begin{align*}
W_p^{n+s}(\bR^d):=\,&\Big\{f\in W_p^{n}\,:\,\int_{\bR^d}\int_{\bR^d}\frac{|D^nf(x)-D^nf(y)|^p}{|x-y|^{d+sp}}\dd y\,\dd x<\infty\Big\}\\
=\,&\Big\{f\in W_p^{n}\,:\,\int_{\bR^d}\int_{\{y:|y-x|< 1\}}\frac{|D^nf(x)-D^nf(y)|^p}{|x-y|^{d+sp}}\dd y\,\dd x<\infty\Big\}
\end{align*}
(see, \textit{e.g.}, \cite[Theorem 2.5.7/(i)]{triebel2}).


Let $\zeta_0\in C_c^{\infty}(\bR_+)$, $\trho$, and $\{\zeta_{0,(n)}\}_{n\in\bN}$ be the functions used in Definition~\ref{220610533} (for $H_{p,\theta}^{\gamma}(\Omega)$); recall \eqref{230130541} - \eqref{230130542}.
Similar to the space $H_{p,\theta}^{\gamma}(\Omega)$, we define
\begin{align*}
\begin{gathered}
	B_{p,\theta}^{\gamma}(\Omega)=\Big\{f\in\cD'(\Omega)\,:\,\|f\|^p_{B^{\gamma}_{p,\theta}(\Omega)}:=\sum_{n\in\bZ}e^{n\theta}\|\big(\zeta_{0,(n)}f\big)(e^n\cdot)\|_{B^{\gamma}_p(\bR^d)}^p<\infty\Big\}\,.
\end{gathered}
\end{align*}
As mentioned in \cite[Remark 3.6]{Lo1}, $B_{p,\theta}^{\gamma}(\Omega)$ has properties similar to $H_{p,\theta}^{\gamma}(\Omega)$.
Properties of $B_{p,\theta}^{\gamma}(\Omega)$ are provided in Appendix~\ref{0082}.

For a regular Harnack function $\Psi$, we denote
\begin{align*}
\Psi B_{p,\theta}^{\gamma}(\Omega)=\big\{\Psi g\,:\,g\in B_{p,\theta}^{\gamma}(\Omega)\big\}\quad\text{and}\quad \|f\|_{\Psi B_{p,\theta}^{\gamma}(\Omega)}=\|\Psi^{-1}f\|_{B_{p,\theta}^{\gamma}(\Omega)}\,.
\end{align*}
The following equaivalent norm on $\Psi B_{p,\theta}^{\gamma}$ is provided in Proposition~\ref{220418435}:
if $k\in\bN_0$ and $\alpha\in (0,1)$, then
\begin{align}\label{2206171215}
\|f\|_{\Psi B_{p,\theta}^{k+\alpha}}^p\simeq_N\,&\sum_{i=0}^k\int_{\Omega}|\rho^i D^if|^p\Psi^{-p}\rho^{\theta-d}\dd x\\
& +\int_{\Omega}\Big(\int_{y:|y-x|<\frac{\rho(x)}{2}}\frac{|D^{k}f(x)-D^{k}f(y)|^p}{|x-y|^{d+\alpha p}}dy\Big)\Psi(x)^{-p}\rho(x)^{(k+\alpha)p +\theta-d}\dd x\,,\nonumber
\end{align}
where $N=N(d,p,k,\alpha,\mathrm{C}_2(\Psi))$.

\begin{lemma}\label{22.04.15.148}
\,\,
\begin{enumerate}
	\item Lemmas~\ref{21.05.20.3}.(1)-(4), \ref{21.09.29.4}, and \ref{22.02.16.1} hold with $B_{\ast,\ast}^{\ast}(\Omega)$ and $B_\ast^{\ast}$, instead of $H_{\ast,\ast}^{\ast}(\Omega)$ and $H_\ast^{\ast}$.
	
	\item Let $k\in\bN_0$ with $|\gamma|<k$. If $a\in C_{\mathrm{loc}}^{k}(\Omega)$ satisfies $|a|^{(0)}_k<\infty$, then
	$$
	\|af\|_{B_{p,\theta}^{\gamma}(\Omega)}\leq N |a|^{(0)}_k\|f\|_{B_{p,\theta}^{\gamma}(\Omega)}
	$$
	where $N=N(d,p,\gamma,\theta,k)$.
	
	\item If $\gamma'>\gamma$, then
	$$
	\|f\|_{\Psi H_{p,\theta}^{\gamma}(\Omega)}+\|f\|_{\Psi B_{p,\theta}^{\gamma}(\Omega)}\leq N\min\big(\|f\|_{\Psi H_{p,\theta}^{\gamma'}(\Omega)},\|f\|_{\Psi B_{p,\theta}^{\gamma'}(\Omega)}\big)\,.
	$$
	where $N=N(d,p,\gamma,\gamma',\theta)$.
	
	\item If $p\geq 2$, then
	$$
	\|f\|_{\Psi B_{p,\theta}^{\gamma}(\Omega)}\leq N\|f\|_{\Psi H_{p,\theta}^{\gamma}(\Omega)}\,,
	$$ 
	and if $1<p\leq 2$, then 
	$$
	\|f\|_{\Psi H_{p,\theta}^{\gamma}(\Omega)}\leq N\|f\|_{\Psi B_{p,\theta}^{\gamma}(\Omega)}\,.
	$$
	Here $N=N(d,p,\gamma,\theta)$.
\end{enumerate}
\end{lemma}

Lemma~\ref{22.04.15.148}, except for the counterparts of Lemmas~\ref{21.09.29.4} and  \ref{22.02.16.1} (in Lemma~\ref{22.04.15.148}.(1)), follows from Propositions~\ref{220527502111} and \ref{220527502}.
The excepted counterparts are proved by repeating the proofs of Lemmas~\ref{21.09.29.4} and \ref{22.02.16.1} with $H_{\ast,\ast}^{\ast}(\Omega)$ replaced by $B_{\ast,\ast}^{\ast}(\Omega)$; we left the proof to the reader.  

\begin{remark}\label{22.04.18.5}
By repeating the argument of Corollary~\ref{21.05.26.3} with using the counterpart of Lemma~\ref{22.02.16.1} in Lemma~\ref{22.04.15.148}.(1), we obtain that for any $n\in\bN$,
\begin{align*}
	\|f\|_{\Psi B_{p,\theta}^{\gamma}(\Omega)}\simeq_N \inf\Big\{\sum_{|\alpha|\leq n}\|f_{\alpha}\|_{\Psi B_{p,\theta-|\alpha|p}^{\gamma+n}(\Omega)}:\,f=\sum_{|\alpha|\leq n}D^{\alpha}f_{\alpha}\Big\}\,,
\end{align*}
where $N=N(d,p,\gamma,\theta,\mathrm{C}_2(\Psi),n)$.
\end{remark}
\vspace{1mm}

Next, we define function spaces for parabolic equations, following Krylov \cite{Krylov1999-1}.
Let $u$ be $\cD'(\Omega)$-valued function on $[0,T]$.
$\partial_tu$ denote a function $f:(0,T)\rightarrow \cD'(\Omega)$ satisfying the following condition: for any $\zeta\in C_c^{\infty}(\Omega)$,
\begin{align}\label{2208221058}
\begin{split}
	&\big\langle f(\,\cdot\,),\zeta\big\rangle\in L_{1,\mathrm{loc}}\big([0,T]\big)\,\,,\,\,\text{and}\\
	&\big\langle u(t),\zeta\big\rangle=\big\langle u(0),\zeta\big\rangle+\int_0^t\big\langle f(s),\zeta\big\rangle \,\dd s\quad\text{for all}\quad t\in (0,T].
\end{split}
\end{align}
In this situation, we also say that $\partial_tu=f$ in the sense of distribution (on $\Omega$).

Since $C_c^{\infty}(\Omega)$ is a separable topological vector space, if $\partial_tu=f$ and $\partial_tu=g$ in the sense of distribution, then $f(s)=g(s)$ for almost every $s\in[0,T]$.

We denote
\begin{equation}\label{221015645}
\begin{alignedat}{2}
	&\bH_{p}^{\gamma}(\bR^d,T)=L_p\big((0,T);H^{\gamma}_{p}(\bR^d)\big)\,\,,\quad &&\bL_{p}(\bR^d,T)=\bH_{p}^0(\bR^d,T)\,\,,\\
	&\bH^{\gamma}_{p,\theta}(\Omega,T)=L_p\big((0,T); H^{\gamma}_{p,\theta}(\Omega)\big)\,\,,\quad &&\bL_{p,\theta}(\Omega,T)=\bH_{p,\theta}^0(\Omega,T)\,\,,\\
	&\Psi\bH^{\gamma}_{p,\theta}(\Omega,T)=L_p\big((0,T);\Psi H^{\gamma}_{p,\theta}(\Omega)\big)\,\,,\quad &&\Psi\bL_{p,\theta}(\Omega,T)=\Psi\bH_{p,\theta}^0(\Omega,T)\,.
\end{alignedat}
\end{equation}

\begin{defn}
By $\Psi\cH^{\gamma+2}_{p,\theta}(\Omega,T)$ we denote the space of all functions $u:[0,T]\rightarrow \cD'(\Omega)$ satisfying the following condition:
$$
\text{$u\in \Psi\bH^{\gamma+2}_{p,\theta}(\Omega,T)$, $u(0)\in \Psi B^{\gamma+2-2/p}_{p,\theta+2}(\Omega)$, and there exists $\partial_tu$ in $\Psi\bH^{\gamma}_{p,\theta+2p}(\Omega,T)$.}
$$
The norm in $\Psi\cH^{\gamma}_{p,\theta}(\Omega,T)$ is defined by
\begin{align*}
	\|u\|_{\Psi\cH^{\gamma+2}_{p,\theta}(\Omega,T)}=\|u\|_{\Psi\bH^{\gamma+2}_{p,\theta}(\Omega,T)}+\|u(0)\|_{\Psi B^{\gamma+2-2/p}_{p,\theta+2}(\Omega)}+\|\partial_t u\|_{\Psi\bH^{\gamma}_{p,\theta+2p}(\Omega,T)}\,.
\end{align*}
\end{defn}
\vspace{1mm}

For the case $\Psi\equiv 1_{\Omega}$, we denote $\cH_{p,\theta}^{\gamma+2}(\Omega,T)=1_{\Omega}\,\cH_{p,\theta}^{\gamma+2}(\Omega,T)$.

\begin{remark}\label{2212141039}
The initial data space $\Psi B_{p,\theta+2}^{\gamma+2-2/p}(\Omega)$ coincides with
\begin{align*}
	\mathrm{Tr}_0:=\big\{u(0)\,|\,u:[0,\infty)\rightarrow \cD'&(\Omega)\,\,\,\text{satisfies that}\\
	&u\in \Psi \bH_{p,\theta}^{\gamma+2}(\Omega,\infty)\,\,\,\text{and}\,\,\, \partial_t u\in \Psi \bH_{p,\theta+2p}^{\gamma}(\Omega,\infty)\big\}\,.
\end{align*}
Note that for $u\in \Psi \bH_{p,\theta}^{\gamma+2}(\Omega,\infty)$, if there exists $f\in\Psi \bH_{p,\theta+2p}^{\gamma}(\Omega,\infty)$ such that
$$
\langle u(t)-u(s),\zeta \rangle=\int_s^t\langle f(r),\zeta\rangle\dd r\qquad \forall\,\,0<s<t<\infty\,\,,\,\,\zeta\in C_c^{\infty}(\Omega),
$$
then $u(0)\in \cD'(\Omega)$ is (uniquely) well defined to satisfy \eqref{2208221058}, by
\begin{align*}
	\langle u(0), \zeta\rangle:=\int_0^1\Big(\langle u(s),\zeta\rangle-\int_0^s\langle \partial_t u(r),\zeta \rangle \dd r\Big)\dd s\,.
\end{align*}
The space $\mathrm{Tr}_0$ is rewritten in the Bochner sense:
\begin{align*}
\mathrm{Tr}_0=\big\{u(0)\,|\,u:[0,\infty)\rightarrow X_0&+X_1\,\,\,\text{satisfies that}\\
&u\in L_p(\bR_+;X_0)\,\,,\partial_t u\in L_p(\bR_+;X_1)\big\}\,,
\end{align*}
where $X_0=\Psi H_{p,\theta}^{\gamma+2}(\Omega)$, $X_1=\Psi H_{p,\theta+2p}^{\gamma}(\Omega)$, and $\partial_t u$ is understood as the weak derivative of $u:\bR_+\rightarrow X_0+X_1$ in the Bochner sense.

It follows from the trace theorem (see, \textit{e.g.}, \cite[Theorem 1.8.2]{triebel}) and Proposition~\ref{220527502111}.(5) that
$$
\mathrm{Tr}_0=[X_0,X_1]_{1/p,p}=\Psi B_{p,\theta+2}^{\gamma+2-2/p}(\Omega)\,,
$$
where $[X_0,X_1]_{\nu,p}$ is the real interpolation space of $X_0$ and $X_1$.
Actually, the second equality is implied by Proposition~\ref{220527502111}.(5) and that the map $u\mapsto \Psi^{-1}u$ is isometric isomorphism from $\Psi H_{p,\theta}^{\gamma+2}(\Omega)$ (resp. $\Psi H_{p,\theta+2p}^{\gamma}(\Omega)$, $\Psi B_{p,\theta+2}^{\gamma+2-2/p}(\Omega)$) to  $H_{p,\theta}^{\gamma+2}(\Omega)$ (resp. $H_{p,\theta+2p}^{\gamma}(\Omega)$, $B_{p,\theta+2}^{\gamma+2-2/p}(\Omega)$).
In addition, we also obtain that 
\begin{alignat*}{2}
	&&&\|f\|_{\Psi B_{p,\theta+2}^{\gamma+2-2/p}(\Omega)}\\
	&\simeq_N\,&& \|f\|_{[X_0,X_1]_{1/p,p}}\\
	&\simeq_p\,&&\inf\Big\{\|u\|_{L_p(\bR_+;X_0)}+\|\partial_t u\|_{L_p(\bR_+;X_1)}\,\,\big|\,\,u:[0,\infty
	)\rightarrow X_0+X_1\,\,\,\,\text{satisfies}\\
	&&&\qquad\qquad\qquad\qquad\qquad\quad\, u\in L_p(\bR_+;X_0)\,\,,\,\, \partial_t u\in L_p(\bR_+;X_1)\,\,,\,\,u(0)=f\Big\}\\
	&=\,&&\inf\left\{\|u\|_{\Psi \bH_{p,\theta}^{\gamma+2}(\Omega,\infty)}+\|\partial_t u\|_{\Psi \bH_{p,\theta+2p}^{\gamma}(\Omega,\infty)}\,\,\big|\,\,u\in \Psi \cH_{p,\theta}^{\gamma+2}(\Omega,\infty)\,\,,\,\, u(0)=f\right\}\,,
\end{alignat*}
where $N=N(d,p,\theta,\gamma)$.
\end{remark}

\begin{prop}\label{2205241111}
\,\,

\begin{enumerate}
	\item $\Psi\cH^{\gamma+2}_{p,\theta}(\Omega,T)$ is a Banach space.
	
	\item $C_c^{\infty}\big([0,\infty)\times \Omega\big)$ is dense in $\Psi\cH^{\gamma+2}_{p,\theta}(\Omega,T)$.
\end{enumerate}
\end{prop}
\begin{proof}
The mapping $u\mapsto \Psi^{-1}u$ is an isometric isomorphism from $\Psi\cH^{\gamma+2}_{p,\theta}(\Omega,T)$ to $\cH^{\gamma+2}_{p,\theta}(\Omega,T)$.
Therefore, we only need to consider the case $\Psi\equiv 1_{\Omega}$.
In this case, (1) and (2) of this proposition are implied by the arguments presented in \cite[Remark 5.5]{Krylov1999-1} and \cite[Remark 3.8]{Krylov2001}, respectively.
We give proofs for the convenience of the reader.

(1) Since $\cH_{p,\theta}^{\gamma+2}(\Omega,T)$ is a normed vector space, we only need to prove the completeness.
By Lemma~\ref{21.09.29.4}.(2), for any $v\in\cH_{p,\theta}^{\gamma+2}(\Omega,T)$ and $S\in (0,T]$, we have $v-v(0)\in C\big([0,S];H_{p,\theta+2p}^{\gamma}(\Omega)\big)$ with
\begin{align}\label{2204160319}
	\sup_{t\in[0,S]}\|v(t)-v(0)\|_{H_{p,\theta+2p}^{\gamma}(\Omega)}^p\leq N\cdot S^{1-1/p}\|\partial_t v\|_{\bH_{p,\theta+2p}^{\gamma}(\Omega,T)}\,,
\end{align}
where $N$ is independent of $v$ and $S$.
Let $\{u^n\}_{n\in\bN}$ be a Cauchy sequence in $\cH_{p,\theta}^{\gamma+2}(\Omega,T)$.
Then there exists
\begin{align*}
	(u_0,f):=\lim_{n\rightarrow \infty}(u^n(0),\partial_t u^n)\quad\text{in}\,\,\,B_{p,\theta+2}^{\gamma+2-2/p}(\Omega)\times \bH_{p,\theta+2p}^{\gamma}(\Omega,T).
\end{align*}
Moreover, due to \eqref{2204160319}, there exists $u:[0,T]\rightarrow \cD'(\Omega)$ such that for any $K\in\bN$,
\begin{align*}
	u-u_0=\lim_{n\rightarrow\infty }\big(u^n-u^n(0)\big)\quad\text{in}\quad C\big([0,T\wedge K];H_{p,\theta+2p}^{\gamma}(\Omega)\big)\,.
\end{align*}
Therefore, by Lemma~\ref{21.09.29.4}.(2), we have
\begin{align*}
	\begin{split}
		\big\langle u(t),\zeta\big\rangle =\lim_{n\rightarrow \infty}\big\langle u^n(t),\zeta\big\rangle\,&=\lim_{n\rightarrow \infty}\bigg(\langle u_0^n,\zeta\rangle +\int_0^t\big\langle \partial_t u^n(s),\zeta\big\rangle \dd s\bigg)\\
		&=\langle u_0,\zeta\rangle +\int_0^t\big\langle f(s),\zeta\big\rangle \dd s
	\end{split}
\end{align*}
for all $t\in[0,T]$ and $\zeta\in C_c^{\infty}(\Omega)$.
Since $u_n$ is a Cauchy sequence in $\bH_{p,\theta}^{\gamma+2}(\Omega,T)$, we also obtain that
$$
u=\lim_{n\rightarrow \infty}u^n\quad\text{in}\quad \bH_{p,\theta}^{\gamma+2}(\Omega,T)\,.
$$
Consequently, $u\in \cH_{p,\theta}^{\gamma+2}(\Omega,T)$ with $\partial_t u=f$ and $u(0)=u_0$, and $u^n\rightarrow u$ in $\cH_{p,\theta}^{\gamma+2}(\Omega,T)$.

(2) In this proof, we use results in Appendix~\ref{0082}.
For $f\in \cD'(\Omega)$, we denote
$$
\Lambda_kf=\sum_{|n|\leq k}\zeta_{0,(n)}f\,.
$$
Due to \eqref{2209121038}, for $u\in\cH_{p,\theta}^{\gamma+2}(\Omega)$ we have
$$
\big\|\Lambda_ku-u\big\|_{\cH_{p,\theta}^{\gamma+2}(\Omega,T)}\rightarrow 0\quad\text{as}\,\,k\rightarrow \infty\,.
$$
Therefore we only need to show that each $\Lambda_ku$ belongs to the closure of $C_c^{\infty}\big([0,\infty)\times \Omega\big)$ in $\cH_{p,\theta}^{\gamma+2}(\Omega,T)$.
By Proposition~\ref{220527502}.(8), we obtain that
$$
\Lambda_ku\in\bH_p^{\gamma+2}(\bR^d,T)\,\,,\,\,\,\,\Lambda_k\big(\partial_t u\big)\in\bH_p^{\gamma}(\bR^d,T)\,\,,\,\,\text{and}\,\,\,\,\Lambda_ku(0)\in B_p^{\gamma+2-2/p}(\bR^d)\,.
$$
Note that $\partial_t \big(\Lambda_ku\big)=\Lambda_k\big(\partial_t u\big)$ in the sense of distribution on $\bR^d$.
By a standard mollification and cut-off argument, there exist $v_{k,m}\in C_c^{\infty}\big([0,\infty)\times \bR^d\big)$ such that
\begin{align*}
	I_{k,m}:=\,&\|v_{k,m}-\Lambda_ku\|_{\bH_p^{\gamma+2}(\bR^d,T)}\\
	&+\|\partial_t v_{k,m}-\partial_t \Lambda_k u\|_{\bH_p^{\gamma}(\bR^d,T)}+\|v_{k,m}(0,\cdot)-\Lambda_ku(0)\|_{B_p^{\gamma+2-2/p}(\bR^d)}
\end{align*}
converges to $0$ as $m\rightarrow \infty$.
Put
$$
u_{k,m}=v_{k,m}\sum_{|n|\leq k+1}\zeta_{0,(n)}\,.
$$
Since
$$
\sum_{|n|\leq k+1}\zeta_{0,(n)}\in \bigcap_{l\in\bN}C^l(\bR^d)\qquad\forall\,\,\,k\in\bN\,,
$$
Proposition~\ref{220527502}.(8) implies
\begin{align*}
	\|\Lambda_ku-u_{k,m}\|_{\cH_{p,\theta}^{\gamma+2}(\Omega,T)}=\Big\|(\Lambda_ku-v_{k,m})\sum_{|n|\leq k+1}\zeta_{0,(n)}\Big\|_{\cH_{p,\theta}^{\gamma+2}(\Omega,T)}\leq N I_{k,m}\,,
\end{align*}
where $N$ is independent of $m$.
Since the last term converges to $0$ as $m\rightarrow\infty$, the proof is completed.
\end{proof}

\begin{lemma}\label{2205241011}
Let $\Psi'$ be a regular Harnack function, $p'\in(1,\infty)$ and $\gamma',\,\theta'\in\bR$.
If $f\in \Psi \bH_{p,\theta}^{\gamma}(\Omega,T)\cap \Psi'\bH_{p',\theta'}^{\gamma'}(\Omega,T)$, then for any $\epsilon>0$, there exist $g\in C_c^{\infty}((0,T)\times\Omega)$ such that
$$
\|f-g\|_{\Psi \bH_{p,\theta}^{\gamma}(\Omega,T)}+\|f-g\|_{\Psi'\bH_{p',\theta'}^{\gamma'}(\Omega,T)}<\epsilon\,.
$$
\end{lemma}
\begin{proof}
We denote $X:=\Psi H_{p,\theta}^{\gamma}(\Omega)\cap \Psi'H_{p',\theta'}^{\gamma'}(\Omega)$, and
$$
\|g\|_X:=\|g\|_{\Psi H_{p,\theta}^{\gamma}(\Omega)}+\|g\|_{\Psi'H_{p',\theta'}^{\gamma'}(\Omega)}\,.
$$
By a standard molification and cut-off argument, for any $\epsilon>0$, there exist $F\in C_c^{\infty}\left((0,T);X\right)$ such that
$$
\|f-F\|_{\Psi \bH_{p,\theta}^{\gamma}(\Omega,T)}+\|f-F\|_{\Psi'\bH_{p',\theta'}^{\gamma'}(\Omega,T)}<\epsilon\,.
$$
This yields that for any $\epsilon>0$, there exists $\eta_1,\,\ldots,\,\eta_N\in C_c^{\infty}\big((0,T)\big)$ and $f_1,\,\ldots,\,f_N\in X$ such that
$$
\|\,f-\widetilde{f}\,\|_{L_p((0,T];X)}<\epsilon\,\,,\,\,\text{where}\,\,\,\, \widetilde{f}(t,\cdot)=\sum_{i=1}^N\eta_i(t)f_i(\cdot)\,.
$$
Due to Lemma~\ref{21.09.29.4}.(5), the proof is completed.
\end{proof}

We end this subsection with the following parabolic embedding theorem for the space $\Psi\cH_{p,\theta}^{\gamma+2}(\Omega)$, which is used in Subsections~\ref{fatex} and \ref{0074}:

\begin{prop}\label{2204160313}
Let $\beta\in\bR$ satisfy $1/p<\beta\leq 1$.
Then for any $u\in\Psi\cH^{\gamma+2}_{p,\theta}(\Omega,T)$, and $0\leq s< t\leq T$,
\begin{align*}
	\|u(t)-u(s)\|_{\Psi H_{p,\theta+2p\beta}^{\gamma+2-2\beta}(\Omega)}\leq  N |t-s|^{\beta-1/p}\Big(\|u\|_{\Psi \bH_{p,\theta}^{\gamma+2}(\Omega,T)}+\|\partial_t u\|_{\Psi \bH_{p,\theta+2p}^{\gamma}(\Omega,T)}\Big)
\end{align*}
where $N=N(d,p,\gamma,\theta,\beta)$.
\end{prop}

\begin{proof}
The map $f\mapsto \Psi^{-1}f$ is an isometric isomorphism from $\Psi H_{p,\theta'}^{\gamma'}(\Omega)$ (resp. $\Psi B_{p,\theta'}^{\gamma'}(\Omega)$, $\Psi \cH_{p,\theta'}^{\gamma'}(\Omega,T)$) to $H_{p,\theta'}^{\gamma'}(\Omega)$ (resp. $B_{p,\theta'}^{\gamma'}(\Omega)$, $\cH_{p,\theta'}^{\gamma'}(\Omega,T)$), for all $\gamma',\theta'\in\bR$.
Therefore we only need to prove this proposition for the case $\Psi\equiv 1$.
The proof of this case is provided in \cite{Krylov2001}, and we introduce this proof for reader's convenience.
Since $u\in\cH_{p,\theta}^{\gamma+2}(\Omega,T)$,
$$
u_n(t,x):=u(t,e^nx)\zeta_{0,(n)}(e^nx)\in \bH_{p}^{\gamma+2}(\bR^d,T)
$$
satisfies that 
$$
\big(\partial_tu_n\big)(s)=\partial_t u(s,e^n\,\cdot\,)\zeta_{0,(n)}(e^n\,\cdot)\quad;\quad u_n(0,\cdot\,)=u_0(e^n\,\cdot\,)\zeta_{0,(n)}(e^n\,\cdot\,)
$$
in the sense of distribution on $\bR^d$.
Since $u_n\in \bH_p^{\gamma+2}(\bR^d,T)$ and $\partial_t u_n\in\bH_p^{\gamma}(\bR^d,T)$, by \cite[Theorem 7.3]{Krylov2001} with $a=e^{-np}$, we obtain
\begin{align}\label{230219256}
	\begin{split}
		&e^{n(2p\beta)}\|u_n(t)-u_n(s)\|_{H_{p}^{\gamma+2-2\beta}(\bR^d)}^p\\
		\leq N &|t-s|^{\beta p-1}\int_0^T\Big(\|u_n(r,\cdot)\|_{H_p^{\gamma+2}(\bR^d)}^p+e^{2np}\|\partial_t u_n(r,\cdot\,)\|_{H_p^{\gamma}(\bR^d)}^p\Big)\dd r
	\end{split}
\end{align}
where $N=N(d,p,\gamma,\beta)$.
In fact, \cite[Theorem 7.3]{Krylov2001} considers the case that $v\,(=u_n)\in \bH_p^{\gamma+2}(\bR^d,T)$ satisfies $\partial_t v\in \bH_p^{\gamma}(\bR^d,T)$ and $v(0)\in H_p^{\gamma+2-2/p}(\bR^d)$.
However, \eqref{230219256} can be obtained without any additional assumptions on $u_n$.
This is because \cite[Theorem 7.3]{Krylov2001} is a consequence of \cite[Theorem 7.2]{aapp}, and the proof of \cite[Theorem 7.2]{aapp} only requires that $v\in \bH_{p}^{\gamma+2}(\bR^d,T)$ and $\partial_t v\in\bH_{p}^{\gamma}(\bR^d,T)$,  without assuming $v(0)\in H_p^{\gamma+2-2/p}(\bR^d)$.

Consequently we obtain
\begin{alignat*}{3}
	&&&\|u(t)-u(s)\|_{H_{p,\theta+2p\beta}^{\gamma+2-2\beta}(\Omega)}^p\\
	&=\,&&\sum_{n\in\bZ}e^{n(\theta+2p\beta)}\|u_n(t)-u_n(s)\|_{H_{p}^{\gamma+2-\beta}(\bR^d)}^p\\
	&\lesssim_N&&|t-s|^{\beta p-1}\int_0^T\sum_{n\in\bZ}e^{n\theta}\Big(\|u_n(r,\cdot)\|_{H_p^{\gamma+2}(\bR^d)}^p+e^{2np}\|\partial_t u_n(r,\cdot\,)\|_{H_p^{\gamma}(\bR^d)}^p\Big)\dd r\\
	&=&&|t-s|^{\beta p-1}\int_0^T\Big(\|u(r,\cdot)\|_{H_{p,\theta}^{\gamma+2}(\Omega)}^p+\|\partial_t u(r,\cdot)\|_{H_{p,\theta+2p}^{\gamma}(\Omega)}^p\Big)\dd r\,.
\end{alignat*}
\end{proof}

\vspace{2mm}

\subsection{Solvability of parabolic equations}\label{0053}
In this subsection, assuming \eqref{22082801111} and \eqref{2205241211}, we introduce the main theorem of this section.
\begin{thm}\label{22.02.18.6}
Let
\begin{itemize}
	\item[]$\Omega$ admit the Hardy inequality \eqref{hardy};\vspace{0.5mm}
	\item[]$\psi$ be a superharmonic Harnack function on $\Omega$;\vspace{0.5mm}
	\item[]$\mu\in I(\psi,p,\nu_1/\nu_2)$.\vspace{0.5mm}
\end{itemize}
and suppose that $\Psi$ is a regularization of $\psi$.
Then for any
$$
u_0\in \Psi^{\mu}B^{\gamma+2-2/p}_{p,d}(\Omega)\quad\text{and}\quad f\in \Psi^{\mu} \bH^{\gamma}_{p,d+2p-2}(\Omega,T)\,,
$$
the equation
\begin{align}\label{220616124}
	\partial_t u=\cL u + f \quad\text{in}\,\,\,(0,T]\quad;\quad u(0,\cdot)=u_0
\end{align}
has a unique solution $u$ in $\Psi^{\mu}\,\cH^{\gamma+2}_{p,d-2}(\Omega,T)$.
Moreover, for this solution $u$, we have
$$
\|u\|_{\Psi^{\mu}\cH^{\gamma+2}_{p,d-2}(\Omega,T)}\leq N\Big(\|u_0\|_{\Psi^{\mu}B^{\gamma}_{p,d}(\Omega)}+ \| f\|_{\Psi^{\mu}\bH^{\gamma}_{p,d+2p-2}(\Omega,T)}\Big)\,,
$$
where $N=N(d,p,\gamma,\mu,\mathrm{C}_0(\Omega),\mathrm{C}_2(\Psi),\mathrm{C}_3(\psi,\Psi), \mathrm{C}_4)$.
\end{thm}
Recall that $\mathrm{C}_0(\Omega)$ is the constant in \eqref{hardy}, $\mathrm{C}_2(\Psi)$ and $\mathrm{C}_3(\psi,\Psi)$ are the constants in Definition~\ref{21.10.14.1}, and $\mathrm{C}_4$ is the constant in Definition~\ref{21.11.10.1}.

\begin{remark}\label{230212657}
As mentioned in Remark~\ref{220617}, Theorem \ref{22.02.18.6} can be reformulated without including $\Psi$.
In addition, when considering the case $\gamma\in\bN_0$, an equivalent norm of $\Psi \bH_{p,\theta}^{\gamma}(\Omega,T)$ is implied by Lemma \ref{220512433}.
An equivalent norm of $\Psi^{\mu}B_{p,d}^{\gamma+2-2/p}(\Omega)$ is also provided by \eqref{2206171215} when $p\neq 2$, and Lemmas~\ref{22.04.15.148}.(4) and \ref{220512433} when $p=2$.
For equivalent norms in the case where $-\gamma\in \bN$, Corollary \ref{21.05.26.3} and Remark~\ref{22.04.18.5} can be used.
\end{remark}
\vspace{1mm}

The proof of Theorem~\ref{22.02.18.6} is parallel with the proof of  Theorem~\ref{21.09.29.1}.
We begin with introducing a well known counterpart of \eqref{220506140}.
\begin{lemma}\label{22.04.14.1}
Suppose that $u\in \bH^{\gamma+1}_p(\bR^d,T)$ and $f\in\bH^{\gamma}_p(\bR^d,T)$ satisfies $u(0,\cdot)\in B^{\gamma+2-2/p}_p(\bR^d)$, and
\begin{align}\label{220524332}
	\partial_t u=\cL u+f\quad\text{in}\,\, t\in (0,T]
\end{align}
in the sense of distributions on $\bR^d$. Then $u\in\bH^{\gamma+2}_p(\bR^d,T)$, and
\begin{align*}
	\|u\|_{\bH^{\gamma+2}_{p}(\bR^d,T)}\leq N\Big(\|u\|_{\bH^{\gamma+1}_{p}(\bR^d,T)}+\| f\|_{\bH^{\gamma}_{p}(\bR^d,T)}+\|u(0,\cdot)\|_{B^{\gamma+2-2/p}_{p}(\bR^d)}\Big)
\end{align*}
where $N=N(d,p,\nu_1,\nu_2)$.
\end{lemma}
\begin{proof}
We first consider the case $T<\infty$.
By applying the operator $(1-\Delta)^{\gamma/2}$ to both sides of \eqref{220524332}, we only need to prove for $\gamma=0$.
Since $u(0)\in B_p^{2-2/p}$, \cite[Section 4.3]{lady} yields that there exists $\overline{u}\in\bH_p^2(\bR^d,T)$ such that
$$
\partial_t \overline{u}=\Delta \overline {u}\quad\text;\quad \overline{u}(0)=u(0)\,,
$$
and
$$
\|\overline{u}_{xx}\|_{\bL_p(\bR^d,T)}\lesssim_{d,p}\|u_0\|_{B_p^{2-2/p}(\bR^d)}\,.
$$
Put $w=u-\overline{u}$ so that
$$
\partial_t w=\cL w+f+(\cL-\Delta)\overline{u}\quad;\quad w(0)=0\,.
$$
It is implied by \cite[Theorem 1.2]{Krylov2001-1} that
$$
\|w_{xx}\|_{\bL(\bR^d,T)}\lesssim_{d,p,\nu_1,\nu_2}\|f\|_{\bL_p(\bR^d,T)}+\|(\cL-\Delta)\overline{u}\|_{\bL(\bR^d,T)}.
$$ 
Therefore we have
\begin{equation}\label{220524304}
	\begin{alignedat}{3}
		\|u\|_{\bH_p^2(\bR^d,T)}&\lesssim_{d,p} &&\|u\|_{\bL_p(\bR^d,T)}+\|u_{xx}\|_{\bL_p(\bR^d,T)}\\
		&\leq&&\|u\|_{\bL_p(\bR^d,T)}+\|\overline{u}_{xx}\|_{\bL_p(\bR^d,T)}+\|w_{xx}\|_{\bL_p(\bR^d,T)}\\
		&\lesssim_{d,p,\nu_1,\nu_2}&&\|u\|_{\bL_p(\bR^d,T)}+\|u_0\|_{B_p^{2-2/p}(\bR^d,T)}+\|f\|_{\bL_p(\bR^d,T)}\,.
	\end{alignedat}
\end{equation}

For $T=\infty$, obtain \eqref{220524304} with $T$ replaced by $K\in\bN$, and let $K\rightarrow\infty$.
\end{proof}

The next two lemmas are driven along the same lines as the proofs of Lemmas~\ref{21.05.13.8} and \ref{21.11.12.1}, respectively.
We leave the proofs to the reader.
For proving Lemma~\ref{21.05.13.9}, put $v_n(t,x)=\zeta_{0,(n)}(e^nx)\Psi^{-1}(e^nx)u(e^{2n}t,e^nx)$ (\textit{cf}. \eqref{2210241127}).

\begin{lemma}[Higher order estimates]\label{21.05.13.9}
Let $s\in \bR$, and let $u\in\Psi\bH^{s+2}_{p,\theta}(\Omega,T)$ and  $f\in \Psi\bH^{\gamma}_{p,\theta+2p}(\Omega,T)$ satisfy $u(0,\,\cdot\,)\in B_{p,\theta+2}^{\gamma+2-2/p}(\Omega)$, and 
$$
\partial_t u=\cL u+f\quad \text{in}\,\, (0,T]
$$
in the sense of distributions.
Then
$u$ belongs to $\Psi\cH^{\gamma}_{p,\theta}(\Omega,T)$, and
\begin{align*}
	\begin{split}
		\|u\|_{\Psi \cH^{\gamma+2}_{p,\theta}(\Omega,T)}\leq N\Big(\|u\|_{\Psi \bH^{s}_{p,\theta}(\Omega,T)}+\| f\|_{\Psi \bH^{\gamma}_{p,\theta+2p}(\Omega,T)}+\|u(0)\|_{\Psi B^{\gamma+2-2/p}_{p,\theta+2}(\Omega)}\Big)
	\end{split}
\end{align*}
where $N=N(d,p,\gamma,\theta,\nu_1,\nu_2,\mathrm{C}_2(\Psi),s)$.
\end{lemma}

\begin{lemma}\label{22.04.15.0430}
Assume the following:
\begin{itemize}
	\item[]For any $u_0\in \Psi B_{p,\theta+2}^{\gamma+2-2/p}(\Omega)$ and $f\in \Psi \bH_{p,\theta+2p}^{\gamma}(\Omega)$,  \eqref{220616124} has a unique solution $u$ in $\Psi \cH_{p,\theta}^{\gamma+2}(\Omega)$.
	Moreover, we have
	\begin{align*}
		\|\Psi^{-1}u\|_{\cH_{p,\theta}^{\gamma+2}(\Omega)}\leq N_{\gamma}\Big(\|\Psi^{-1}f\|_{H_{p,\theta+2p}^{\gamma}(\Omega)}+\|\Psi^{-1}u_0\|_{B_{p,\theta+2}^{\gamma+2-2/p}(\Omega)}\Big)
	\end{align*}
	where $N_{\gamma}$ is a constant independent of $u_0$, $f$ and $u$.
\end{itemize}
Then for all $s\in\bR$, the following holds:
\begin{itemize}
	\item[]
	For any $u_0\in\Psi B_{p,\theta+2}^{s+2-2/p}(\Omega)$ and $f\in \Psi\bH_{p,\theta+2p}^{s}(\Omega)$,  \eqref{220616124} has a unique solution $u$ in $\Psi \cH_{p,\theta}^{s+2}(\Omega)$.
	Moreover, we have
	\begin{align*}
		\|\Psi^{-1}u\|_{\cH_{p,\theta}^{s+2}(\Omega)}\leq N_{s}\Big(\|\Psi^{-1}f\|_{H_{p,\theta+2p}^{s}(\Omega)}+\|\Psi^{-1}u_0\|_{B_{p,\theta+2}^{\gamma+2-2/p}(\Omega)}\Big)
	\end{align*}
	where $N_s$ is a constant depending only on $d,\,p,\,\gamma,\,\theta,\,\nu_1,\,\nu_2,\,\mathrm{C}_2(\Psi),\,N_{\gamma},\,s$.
\end{itemize}
\end{lemma}

\begin{proof}[Proof of Theorem~\ref{22.02.18.6}]
By Lemma~\ref{22.04.15.0430}, we only need to prove for $\gamma=0$.

\textbf{\textit{A priori} estimates.} Make use of Theorem~\ref{05.11.2} and Lemmas~\ref{220512433} and \ref{22.04.15.148}.(3) to obtain that for any $u\in C_c^{\infty}([0,T]\times\Omega)$, 
\begin{align*}
	\begin{split}
		\|u\|_{\Psi^{\mu}\bL_{p,d-2}(\Omega,T)}&\lesssim_N \|u(0,\cdot)\|_{\Psi^{\mu}L_{p,d}(\Omega)}+\|\partial_t u-\cL u\|_{\Psi^{\mu}\bL_{p,d+2p-2}(\Omega,T)}\\
		&\lesssim_N \|u(0,\cdot)\|_{\Psi^{\mu}B_{p,d}^{2-2/p}(\Omega)}+\|\partial_t u-\cL u\|_{\Psi^{\mu}\bL_{p,d+2p-2}(\Omega,T)}\,.
	\end{split}
\end{align*}
By combining this with Lemma~\ref{21.05.13.9}, we have
\begin{equation}\label{2206081117}
\begin{aligned}
	\|u\|_{\Psi^{\mu}\cH_{p,d-2}^2(\Omega,T)}\lesssim_N&\|u\|_{\Psi^{\mu}\bL_{p,d-2}(\Omega,T)}\\
	&+\|u(0,\cdot)\|_{\Psi^{\mu}B_{p,d}^{2-2/p}(\Omega)}+\|\partial_t u-\cL u\|_{\Psi^{\mu}\bL_{p,d+2p-2}(\Omega,T)}\\
	 \lesssim_N&\|u(0,\cdot)\|_{\Psi^{\mu}B_{p,d}^{2-2/p}(\Omega)}+\|\partial_t u-\cL u\|_{\Psi^{\mu}\bL_{p,d+2p-2}(\Omega,T)}\,,
\end{aligned}
\end{equation}
where $N=N(d,p,\mu,\mathrm{C}_0,\mathrm{C}_2(\Psi),\mathrm{C}_3(\psi,\Psi),\mathrm{C}_4)$.
By Proposition~\ref{2205241111}, \eqref{2206081117} also holds for all $u\in\Psi^{\mu} \cH_{p,d-2}^2(\Omega,T)$.
Therefore the \textit{a priori} estimates are obtained.
The uniqueness of solutions also follows from \eqref{2206081117}.

\textbf{Existence of solutions.}
We first consider the case $\cL=\nu_1\Delta$.
Let
$$
(f,u_0)\in\Psi^{\mu}\bL_{p,d+2p-2}(\Omega,T)\times \Psi^{\mu}B_{p,d}^{2-2/p}(\Omega)=:\cF\,.
$$
By Lemma~\ref{2205241011} and Lemma \ref{22.04.15.148}.(1) (the counterpart of Lemma~\ref{21.05.20.3}.(1), $C_c^{\infty}([0,T]\times\Omega)\times C_c^{\infty}(\Omega)$ is dense in $\cF$.
Therefore there exists $(f^{(n)},u^{(n)}_{0})\in C_c^{\infty}([0,T]\times\Omega)\times C_c^{\infty}(\Omega)$ such that $(f^{(n)},u_{0}^{(n)})\rightarrow (f,u_0)$ in $\cF$.
Make use of Lemmas~\ref{21.05.25.300} and \ref{21.05.13.9} to obtain that there exists a solution $u^{(n)}\in \Psi^{\mu}\cH_{p,d-2}^{2}(\Omega,T)$ of the equation
\begin{align*}
	\partial_t u^{(n)}=\nu_1\Delta u^{(n)}+f_n\quad;\quad u^{(n)}(0)=u^{(n)}_{0}\,.
\end{align*}
By \eqref{2206081117}, $u^{(n)}$ is a Cauchy sequence in $\Psi^{\mu}\cH_{p,d-2}^2(\Omega,T)$.
Since $\Psi^{\mu}\cH_{p,d-2}^{2}(\Omega,T)$ is a Banach space, there exists $u\in \Psi^{\mu}\cH_{p,d-2}^2(\Omega,T)$ such that $u^{(n)}\rightarrow u$ in $\Psi^{\mu}\cH_{p,d-2}^2(\Omega,T)$.
We also obtain that
$$
\lim_{n\rightarrow \infty}\big(f^{(n)},u^{(n)}_{0}\big)=\lim_{n\rightarrow \infty}\big(\partial_t u^{(n)}-\nu_1\Delta u^{(n)},u^{(n)}(0)\big)=\big(\partial_t u-\nu_1\Delta u,u(0)\big)\quad\text{in}\,\,\,\cF\,,
$$
which implies $\partial_t u-\nu_1\Delta u=f$ and $u(0)=u_0$.
Therefore, we have proven the theorem for the case $\cL=\nu_1\Delta$.

Let us consider a general $\cL:=\sum_{i,j}a^{ij}D_{ij}\in\cM_T(\nu_1,\nu_2)$, where $(a^{ij}(t))_{d\times d}\in\mathrm{M}(\nu_1,\nu_2)$ for all $t\in(0,T]$.
For $s\in[0,1]$, put
$$
\cL_s=\sum_{i,j=1}^d\big((1-s)\nu_1\delta^{ij}+sa^{ij}\big)D_{ij}\,.
$$
Since
\begin{align}\label{2206081243}
	\nu_1|\xi|^2\leq \sum_{i,j=1}^d\big((1-s)\nu_1\delta^{ij}+sa^{ij}(t)\big)\xi_i\xi_j\leq \nu_2|\xi|^2
\end{align}
for all $t\in(0,T]$, we have $\cL_s\in \cM_T(\nu_1,\nu_2)$.
It follows from \eqref{2206081117} that
\begin{align*}
	\|u\|_{\Psi^{\mu}\cH_{p,d-2}^2(\Omega,T)}\leq N\Big(\|u(0,\cdot)\|_{\Psi^{\mu}B_{p,d}^{2-2/p}(\Omega)}+\|\partial_t u-\cL_s u\|_{\Psi^{\mu}\bL_{p,d+2p-2}(\Omega,T)}\Big)
\end{align*}
for all $u\in \Psi^{\mu}\cH_{p,d-2}^{2}(\Omega,T)$ and $s\in[0,1]$, where $N$ is the constant in \eqref{2206081117}.
In particular, $N$ is independent of $s$.
Since the unique solvability for $\cL_0$, the method of continuity (see, \textit{e.g.}, \cite[Theorem 5.2]{GT}) yields the unique solvability for $\cL_1$.
\end{proof}

We end this subsection with the global uniqueness of solutions.
\begin{thm}[Global uniqueness]\label{220821002901}
Suppose that \eqref{hardy} holds for $\Omega$, and for $k=1,\,2$, 
\begin{equation*}
	\begin{gathered}
		\text{$\psi_k$ is a superharmonic Harnack functions on $\Omega$}\,\,,\quad\text{$\Psi_k$ is a regularization of $\psi_k$}\,\,,\\
		\gamma_k\in\bR\,\,,\,\,\,\,p_k\in(1,\infty)\quad \text{and}\quad \mu_k\in I(\psi_k,p_k,\nu_2/\nu_1)\,.
	\end{gathered}
\end{equation*}
Let
$$
f\in\bigcap_{k=1,2}\Psi_k^{\,\mu_k}\bH_{p_k,d+2p_k-2}^{\gamma_k}(\Omega,T)\quad\text{and}\quad u_0\in\bigcap_{k=1,2}B_{p_k,d}^{\gamma_k+2-2/p_k}(\Omega),
$$
and for each $k=1,\,2$,
$u^{(k)}\in\Psi_k^{\,\mu_k}\cH_{p_k,d-2}^{\gamma_k+2}(\Omega,T)$ be the solution to the equation
$$
\partial_t u^{(k)}=\cL u^{(k)}+f\quad;\quad u^{(k)}(0)=u_0\,.
$$
Then $u^{(1)}(t)=u^{(2)}(t)$ in the sense of distribution, for all $t\in[0,T]$.
\end{thm}
\begin{proof}
We denote 
\begin{align*}
	\begin{gathered}
		X_k=\Psi_k^{\,\mu_k}\cH_{p_k,d-2}^{\gamma_k+2}(\Omega,T)\,\,,\,\,\,\,X=X_1\cap X_2\\
		Y_k=\Psi_k^{\,\mu_k}\bH_{p_k,d+2p_k-2}^{\gamma_k}(\Omega,T)\times \Psi_k^{\,\mu_k}B_{p_k,d}^{\gamma_k+2-2/p_k}(\Omega)\,\,,\,\,\,\,Y=Y_1\cap Y_2\,.
	\end{gathered}
\end{align*}
\textbf{Step 1.} We first consider the case $\cL=\nu_1\Delta$.
For $(f,u_0)\in Y$, by Lemmas~\ref{2205241011} and ~\ref{22.04.15.148}.(1) (the counterpart of Lemma~\ref{21.05.20.3}.(1)), there exists $(f_n,u_{0,n})\in C_c^{\infty}\big((0,T]\times\Omega\big)\times C_c^{\infty}(\Omega)$ such that
$$
(f_n,u_{0,n})\rightarrow (f,u_0)\quad\text{in}\quad Y\,.
$$
Since $\mu_k\in(-1/p_k,1-1/p_k)$, it follows from Lemmas~\ref{21.05.25.300} and \ref{21.05.13.9} that there exists $u_n\in X=X_1\cap X_2$ such that
$$
\partial_t u_n=\nu_1 \Delta u_n+f_n\quad;\quad u_n(0)=u_{0,n}\,.
$$
By Theorem~\ref{22.02.18.6}, we have
$$
\lim_{n\rightarrow \infty}\|u_n-u^{(k)}\|_{X_k}\lesssim \lim_{n\rightarrow \infty}\|(f_n,u_{0,n})-(f,u_0)\|_{Y_k}=0
$$
for each $k=1,\,2$.
Due to \eqref{2204160319} and that $u_{0,n}\rightarrow u_0$ in $\cD'(\Omega)$ (see the counterpart of  Lemma~\ref{21.09.29.4}.(2)), we obtain that
$$
u^{(1)}(t)=u^{(2)}(t)\quad\text{in}\,\,\,\cD'(\Omega)\,\,,\,\,\,\,\text{for all}\,\,t\in[0,T]\,.
$$
Therefore the case $\cL=\nu_1\Delta$ is proved.

We can also observe that
$u^{(1)}(\,\cdot\,)$($=u^{(2)}(\,\cdot\,)$) is the unique solution of the equation
\begin{align}\label{2206211148}
	\partial_t u=\nu_1\Delta u+f\quad;\quad u(0)=u_0\,,
\end{align}
in the class $X$.
This is because $X_1\cap X_2$ and $X_1$ admits the unique solution to the equation \eqref{2206211148}, and $X_1\cap X_2\subset X_1$.

\textbf{Step 2.} Let $\cL\in\cM_T(\nu_1,\nu_2)$.
For $r\in[0,1]$, denote $\cL_r:=(1-r)\nu_1\Delta+r\cL$.
Due to \eqref{2206081243}, Theorem~\ref{22.02.18.6} implies that 
$$
\|u\|_X=\|u\|_{X_1}+\|u\|_{X_2}\leq N\|\big(\partial_t u-\cL_ru,u(0)\big)\|_Y\quad\text{for all}\,\,\,u\in X
$$
where $N$ is independent of $u$ and $r\in[0,1]$.
In addition, by the result in Step 1, the map $u\mapsto \big(\partial_t u-\cL_0 u, u(0)\big)$ is a bijective map from $X$ to $Y$.
Therefore the method of continuity yields that  for any $(f,u_0)\in Y$, there exists a unique solution $u\in X=X_1\cap X_2$ of the equation
\begin{align}\label{2204160344}
	\partial_t u=\cL u+f\quad;\quad u(0)=u_0\,.
\end{align}
For each $k=1,\,2$,  $u^{(k)}$ is the unique solution of equation \eqref{2204160344} in $X_k$, which implies $u=u^{(k)}$.
Consequently, $u^{(1)}(t)=u(t)=u^{(2)}(t)$ for all $t\in (0,T]$.
\end{proof}

\vspace{2mm}

\mysection{Application I - Domain with fat exterior or thin exterior}\label{app.}

In this section, we introduce applications of Sections~\ref{0040} and \ref{0050} to domains satisfying fat exterior or thin exterior conditions. 
The notions of fat exterior and thin exterior are closely related to the geometry of a domain $\Omega$, namely the Hausdorff dimension and Aikawa dimension of $\Omega^c$.

For a set $E\subset \bR^d$, the Hausdorff dimension of $E$ is defined by 
$$
\dim_{\cH}(E):=\inf\big\{\lambda\geq 0\,:\,H^{\lambda}_{\infty}(E)=0\big\}\,,
$$
where
$$
\cH_{\infty}^{\lambda}(E):=\inf \Big\{\sum_{i\in\bN}r_i^{\lambda}\,:\,E\subset \bigcup_{i\in\bN}B(x_i,r_i)\quad\text{where }x_i\in E\text{ and }r_i>0\Big\}\,.
$$
The Aikawa dimension of $E$, denoted by $\dim_{\cA}(E)$, is defined by the infimum of $\beta\geq 0$ for which
$$
\sup_{p\in E,\,r>0}\frac{1}{r^{\beta}}\int_{B_r(p)}\frac{1}{d(x,E)^{d-\beta}}\dd x< \infty\,,
$$
with considering $\frac{1}{0}=+\infty$.

\begin{remark}\label{22.02.24.1}\,\,
	
	(i) The Aikawa dimension is defined through integration.
	However, this dimension equals the Assouad dimension (see \cite[Theorem 1.1]{LT}).
	The Assouad dimension is defined in terms of a covering property, similar to the Hausdorff dimension and Minkowski dimension.
	Specifically, the Assouad dimension of a set $E$ is the infimum of $\beta\geq 0$ for which there exists $N_\beta>0$ such that, for any $\epsilon\in(0,1)$, each subset $F\subset E$ can be covered by at most $N_\beta\epsilon^{-\beta}$ balls of radius $r=\epsilon\cdot \mathrm{diam}(F)$.
	
	(ii) For any $E\subset \bR^d$,
	$$
	\dim_{\cH}(E)\leq \dim_{\cA}(E)
	$$
	and the equality does not hold in general (see \cite[Section 2.2]{lehr}).
	However, if $E$ is Alfors regular, for example, if $E$ has a self-similar property such as Cantor set or Koch snowflake set, then $\dim_{\cH}(E)$ equals $\dim_{\cA}(E)$; see \cite[Lemma 2.1]{lehr} and \cite[Theorem 4.14]{Mattila}.
\end{remark}
\vspace{1mm}

Koskela and Zhong \cite{KZ} established the dimensional dichotomy results for domains admitting the Hardy inequality,  using the Hausdorff and Minkowski dimension.
Their result can be expressed through Hausdorff and Aikawa dimension, as shown in \cite[Theorem 5.3]{lehr}.

\begin{prop}[see Theorem 5.3 of \cite{lehr}]\label{22.02.07.1}
	Suppose a domain $\Omega\subset \bR^d$ admits the Hardy inequality.
	Then there is a constant $\epsilon>0$ such that for each $p\in\partial\Omega$ and $r>0$, either
	$$
	\dim_{\cH}\big(\Omega^c\cap \overline{B}(p,4r)\big)\geq d-2+\epsilon\quad\text{or}\quad \dim_{\cA}\big(\Omega^c\cap \overline{B}(p,r)\big)\leq d-2-\epsilon\,.
	$$
\end{prop}
\vspace{1mm}

For a deeper discussion of the dimensional dichotomy, we refer the reader to \cite{ward}. 
In virtue of Proposition~\ref{22.02.07.1}, we consider domains $\Omega\subset \bR^d$ which satisfy one of the following situations:
\begin{enumerate}
	\item(Fat exterior) There exists $\epsilon\in(0,1)$ and $c>0$ such that
	\begin{align}\label{220617253}
		\cH^{d-2+\epsilon}_{\infty}\big(\Omega^c\cap \overline{B}(p,r)\big)\geq cr^{d-2+\epsilon}\quad\text{for all }p\in \partial\Omega\,\,,\,\,\,\,r>0\,.
	\end{align}
	
	\item(Thin exterior) $\dim_{\cA}(\Omega^c)<d-2$.
\end{enumerate}
It is mentioned in detail in Subsections~\ref{fatex} and \ref{0062} that if a domain satisfyies one of these situations, then this domain admits the Hardy inequality.

In this section and Section~\ref{app2.}, for various domains $\Omega\subset \bR^d$, we construct superharmonic functions equivalent to powers of boundary distance functions $\rho^{\alpha}:=d(\,\cdot\,,\partial\Omega)^{\alpha}$.
It is provided in Remark~\ref{220617557} that for each $p\in(1,\infty)$, these superharmonic functions imply ranges of $\theta\in\bR$ for which the following statement holds:
\begin{statement}[$\Omega,p,\theta$]\label{22.02.19.1}
	For any $\gamma\in\bR$, the following hold:
	\begin{enumerate}
		\item For any $\lambda \geq 0$ and  $f\in H_{p,\theta+2p}^{\gamma}(\Omega)$, the equation
		$$
		\Delta u-\lambda u=f
		$$
		has a unique solution $u$ in $H^{\gamma+2}_{p,\theta}(\Omega)$.
		Moreover, we have
		\begin{align}\label{2205241155}
			\|u\|_{H^{\gamma+2}_{p,\theta}(\Omega)}+\lambda\|u\|_{H_{p,\theta+2p}^{\gamma}(\Omega)}\leq N_1\|f\|_{H_{p,\theta+2p}^{\gamma}(\Omega)}\,,
		\end{align}
		where $N_1$ is a constant independent of $u$, $f$, and $\lambda$.
		
		\item Let $T\in(0,\infty]$.
		For any $u_0\in B^{\gamma+2-2/p}_{p,\theta+2}(\Omega)$ and $f\in \bH_{p,\theta+2p}^{\gamma}(\Omega,T)$, the equation
		\begin{align*}
			u_t=\Delta u+f\quad\text{on }\Omega\times (0,T]\quad;\quad u(0,\cdot)=u_0\,.
		\end{align*}
		has a unique solution $u$ in $\cH^{\gamma+2}_{p,\theta}(\Omega)$.
		Morever, we have
		\begin{align}\label{2205241156}
			\|u\|_{\cH^{\gamma+2}_{p,\theta}(\Omega)}\leq N_2\big(\|u_0\|_{B^{\gamma+2-2/p}_{p,\theta+2}(\Omega)}+\|f\|_{H_{p,\theta+2p}^{\gamma}(\Omega)}\big)\,,
		\end{align}
		where $N_2$ is a constant independent of $u$, $f$, and $T$.
	\end{enumerate}
\end{statement}

\begin{remark}\label{220617557}
	Let $\Omega$ admit the Hardy inequality \eqref{hardy} and suppose that for a fixed $\alpha\in\bR\setminus\{0\}$, there exists a superharmonic function $\psi$ and a constant $M>0$ such that
	$$
	M^{-1}\rho^{\alpha}\leq \psi\leq M\rho^{\alpha}\,.
	$$
	Then $\psi$ is a superharmonic Harnack function, and $\Psi:=\trho^{\,\alpha}$ is a regularization of $\psi$.
	Furthermore, the constants $\mathrm{C}_2(\Psi)$ and $\mathrm{C}_3(\Psi,\psi)$ can be chosen to depend only on $d$, $\alpha$ and $M$.
	In this case, Lemmas~\ref{21.05.20.3}.(3) and \ref{22.04.15.148}.(1) (the counterpart of Lemma~\ref{21.05.20.3}.(3)) imply that for any $p\in(1,\infty)$ and $\gamma,\,\theta\in\bR$, there exists $N=N(d,p,\gamma,\alpha,\mu,M)$ such that
	\begin{align*}
		\|f\|_{\Psi^{\mu}H_{p,\theta}^{\gamma}(\Omega)}\simeq_N \|f\|_{H_{p,\theta-\alpha\mu}^{\gamma}(\Omega)}\quad\text{and}\quad \|f\|_{\Psi^{\mu}B_{p,\theta}^{\gamma}(\Omega)}\simeq_N \|f\|_{B_{p,\theta-\alpha\mu}^{\gamma}(\Omega)}\,.
	\end{align*}
	Therefore, due to Theorems~\ref{21.09.29.1} and \ref{22.02.18.6} (with Proposition~\ref{05.11.1}.(1)), we conclude that Statement~\ref{22.02.19.1} $(\Omega,p,\theta)$ holds for all $p\in(1,\infty)$ and 
	\begin{align*}
		&\theta\in\big(d-2-(p-1)\alpha,d-2+\alpha\big)\quad\text{if}\,\,\alpha>0\,;\\
		&\theta\in\big(d-2+\alpha,d-2-(p-1)\alpha\big)\quad\text{if}\,\,\alpha<0\,.
	\end{align*}
	Moreover, $N_1$ (in \eqref{2205241155}) and $N_2$ (in \eqref{2205241156}) depend only $d,\,p,\,\gamma,\,\theta,\,\mathrm{C}_0(\Omega),\,\alpha$ and $M$.
\end{remark}
\vspace{1mm}

We collect basic properties of classical superharmonic functions, which are used in this section and Section \ref{app2.}. 

\begin{prop}\label{21.05.18.1}\
	Let $\Omega$ be an open set in $\bR^d$.
	\begin{enumerate}
		\item Let $\phi_1,\,\phi_2$ be classical superharmonic functions on $\Omega$. Then $\phi_1\wedge \phi_2$ is also a classical superharmonic function on $\Omega$.
		
		\item Let $\{\phi_{\alpha}\}$ be a family of positive classical superharmonic functions on $\Omega$.
		Then $\phi:=\inf_{\alpha}\phi_{\alpha}$ is a superharmonic function on $\Omega$.
		
		\item Let $\phi_1,\,\phi_2$ be positive classical superharmonic functions on $\Omega$. For any $\alpha\in(0,1)$, $\phi_1^{\alpha}\phi_2^{1-\alpha}$ is also a classical superharmonic fucntion on $\Omega$; in particular, $\phi_1^{\alpha}$ is a classical superharmonic function for all $\alpha\in(0,1)$.
		
		
		\item Let $\Omega_1$ and $\Omega_2$ be open sets in $\bR^d$ and $\phi_i$ be a classical superharmonic function on $\Omega_i$, for $i=1,\,2$.
		Suppose that 
		\begin{alignat*}{2}
			&\liminf_{x\rightarrow x_1,x\in\Omega_2}\phi_2(x)\geq \phi_1(x_1)\quad &&\text{for all}\quad  x_1\in \Omega_1\cap \partial\Omega_2\,;\\
			&\liminf_{x\rightarrow x_2,x\in\Omega_1}\phi_1(x)\geq \phi_2(x_2)\quad &&\text{for all}\quad  x_2\in \Omega_2\cap \partial\Omega_1\,.
		\end{alignat*}
		Then the function
		\begin{align*}
			\phi(x):=
			\begin{cases}
				\phi_1(x)&\quad x\in\Omega_1\setminus\Omega_2\\
				\phi_1(x)\wedge \phi_2(x) &\quad x\in \Omega_1\cap \Omega_2\\
				\phi_2(x)&\quad x\in\Omega_2\setminus\Omega_1\\
			\end{cases}
		\end{align*}
		is also a classical superharmonic function on $\Omega$.
		
	\end{enumerate}
\end{prop}

For the proof of Proposition~\ref{21.05.18.1}, (1) follows from the definition of classical superharmonic functions, (2) and (3) can be found in \cite[Theorem 3.7.5, Corollary 3.4.4]{AG}, respectively, and (4) is implied by \cite[Corollary 3.2.4]{AG}.

\vspace{2mm}

\subsection{Domain with fat exterior : Harmonic measure decay property}\label{fatex}\,\,

This subsection begins by introducing a relation among the condition \eqref{220617253}, classical potential theory, and the Hardy inequality; see the paragraph below Remark~\ref{21.07.06.1}.

We first recall notions in classical potential theory.
For a bounded open set $U\subset\bR^d$ ($d\geq 2$) and a bounded Borel function $f$ on $\partial U$, the Perron-Wiener-Brelot solution (abbreviated to `PWB solution') of the equation
\begin{align}\label{2208021143}
	\Delta u=0\quad\text{in}\,\,\,U\quad;\quad u=f\quad\text{on}\,\,\,\partial U
\end{align}
is defined by
\begin{equation}\label{2301051056}
	\begin{aligned}
		u(x):=\inf\big\{\phi(x)\,:\,\text{$\phi$ is a superharmonic function on $U$ and}&\\
		\text{$\underset{y\rightarrow z}{\liminf}\,\phi(y)\geq f(z)$ for all $z\in\partial U$}&\,\big\}\,.
	\end{aligned}
\end{equation}
For a Borel set $E\subset \partial U$, $w(\,\cdot\,,U,E)$ denotes the PWB solution $u$ of the equation 
\begin{align*}
	\Delta u=0\quad\text{in}\,\,\,U\quad;\quad u=1_E\quad\text{on}\,\,\,\partial U\,,
\end{align*}
which is also called the \textit{harmonic measure} of $E$ over $U$.

\begin{remark}\label{220802309}
	A bounded open set $U$ is said to be regular if, for any $f\in C(\partial U)$, the PWB solution of equation \eqref{2208021143} belongs to $C(\overline{U})$ and satisfies \eqref{2208021143} pointwisely.
	One of the equivalent conditions for $U$ to be regular is provided by N. Wiener \cite{NW} (see with \cite[Theorem 7.7.2]{AG}), which is called the Wiener criterion.
\end{remark}
\vspace{1mm}

We fix an arbitrary open set $\Omega\subset \bR^d$, $d\geq 2$ (not necessarily bounded).
For $p\in\partial\Omega$ and $r>0$, we denote
$$
w(\,\cdot\,,p,r)=w\big(\,\cdot\,,\Omega\cap B_r(p),\Omega\cap \partial B_r(p)\big)
$$ 
(see Figure~\ref{230113736} below); note that $\Omega\cap \partial B_r(p)$ is a relatively open subset of $\partial\big(\Omega\cap B_r(p)\big)$.
\begin{figure}[h]
	\begin{tikzpicture}[> = Latex]
		\begin{scope}
			\begin{scope}[scale=0.8]
				\clip (-4,-1.8) rectangle (4,2.8);
				
				\begin{scope}[shift={(0.5,0)}]
					\fill[gray!10] (-0.5,-0.5) .. controls +(0,1) and +(0.3,-0.6) ..(-1.5,1.9) .. controls +(-0.3,0.6) and +(0.2,-0.2) .. (-1.7,4) -- (4,4) -- (4,0) .. controls +(-0.2,0) and +(0.5,0.5) .. (2.1,-0.5) .. controls +(-0.5,-0.5) and +(1.2,0.1)..(-0.5,-0.5);
					
				\end{scope}

				\begin{scope}[shift={(0.5,0)}]
					\clip (-0.5,-0.5) .. controls +(0,1) and +(0.3,-0.6) ..(-1.5,1.9) arc (acos(-5/13):0:2.6)  .. controls +(-0.5,-0.5) and +(1.2,0.1)..(-0.5,-0.5);
					\foreach \i in {-4,-3.6,...,3}
					{\draw (\i,-2.8)--(\i+1.5,2.8);}
					\path[fill=gray!10] (-0.35,0.05) rectangle (1.35,0.85);
					
					\draw[gray!10,line width=6pt] (2.1,-0.5) arc (0:acos(-5/13):2.6);
					\draw[gray!10,line width=6pt] (2.1,-0.5) .. controls +(-0.5,-0.5) and  +(1.2,0.1)..	(-0.5,-0.5) .. controls +(0,1) and +(0.3,-0.6) ..(-1.5,1.9);
					
				\end{scope}
				
				\begin{scope}[shift={(0.5,0)}]
					\draw[dashed] (-0.5,-0.5) circle (2.6);
					\draw
					(-0.5,-0.5) .. controls +(0,1) and +(0.3,-0.6) ..(-1.5,1.9) .. controls +(-0.3,0.6) and +(0.2,-0.2) .. (-1.7,4);
					\draw
					(4,0) .. controls +(-0.2,0) and +(0.5,0.5) .. (2.1,-0.5) .. controls +(-0.5,-0.5) and  +(1.2,0.1)..	(-0.5,-0.5);

					\draw[ line width=0.8pt] (2.1,-0.5) arc (0:acos(-5/13):2.6);
					\draw[line width=0.8pt] (2.1,-0.5) .. controls +(-0.5,-0.5) and  +(1.2,0.1)..	(-0.5,-0.5) .. controls +(0,1) and +(0.3,-0.6) ..(-1.5,1.9);
					\draw[fill=gray!10] (2.1,-0.5) circle (0.08);
					\draw[fill=gray!10] (-1.5,1.9) circle (0.08);
					\draw[fill=black] (-0.5,-0.5) circle (0.05);
				\end{scope}

			\end{scope}
			
			\begin{scope}[shift={(0.4,0)}]
				\draw (0.42,0.35) node  {$\Delta u=0$} ;
				\draw (1.8,1.1) node (1) {$u=1$};
				\draw (0.42,-0.9) node (0) {$u=0$};
				
			\end{scope}

		\end{scope}

	\end{tikzpicture}
	\caption{$u:=w(\,\cdot\,,p,r)$}\label{230113736}
\end{figure}
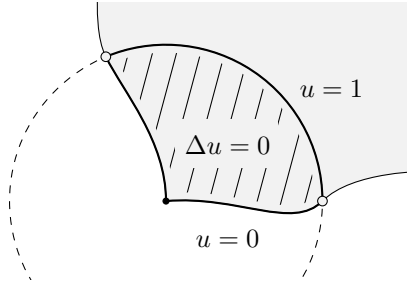

Here are basic properties of $w(\,\cdot\,,p,r)$ which can be found in \cite[Chapter 6]{AG}.

\begin{prop}\label{220621237}\,\,
	
	\begin{enumerate}
		\item $w(\,\cdot\,,p,r)$ is harmonic on $\Omega\cap B_r(p)$ with values in $[0,1]$.\vspace{0.5mm}
		\item For any $x_0\in \Omega\cap \partial B_r(p)$, $\underset{\substack{x\rightarrow x_0\\x\in\Omega\cap B_r(p)}}{\lim}w(x,p,r)=1$.
	\end{enumerate}
\end{prop}

For convenience, based on Proposition~\ref{220621237}, we consider $w(\,\cdot\,,p,r)$ to be continuous on $\Omega\cap \overline{B}(p,r)$ with $w(x,p,r)=1$ for $x\in\Omega\cap \partial B(p,r)$.

\begin{defn}\label{2301021117}
	A domain $\Omega$ is said to satisfy the \textit{local harmonic measure decay property} with exponent $\alpha>0$ (abbreviated to `$\mathbf{LHMD}(\alpha)$'), if there exists a constant $M_{\alpha}>0$ depending only on $\Omega$ and $\alpha$ such that
	\begin{align}\label{21.08.03.1}
		w(x,p,r)\leq M_{\alpha}\left(\frac{|x-p|}{r}\right)^{\alpha}\quad\text{for all}\,\,\, x\in \Omega\cap B(p,r)
	\end{align}
	whenever $p\in\partial \Omega$ and $r>0$.
\end{defn}
\vspace{1mm}

It is worth noting that if $\Omega$ satisfies $\mathbf{LHMD}(\alpha)$ for some $\alpha>0$, then $\Omega$ is regular (see, \textit{e.g.}, \cite[Theorem 6.6.4]{AG}).

\begin{remark}
	$\mathbf{LHMD}$ is closely related to the H\"older continuity of the PWB solutions.
	We temporarily assume that $\Omega$ is a bounded regular domain (see Remark~\ref{220802309}).
	For $\alpha\in(0,1]$ and $f\in C(\partial \Omega)$, by $H_\Omega f$ we denote the PWB solution $u$ of the equation
	$$
	\Delta u=0\quad\text{on}\,\,\,\Omega\,\,\, ;\quad u|_{\partial \Omega}\equiv f\,,
	$$
	and denote
	$$
	\|H_\Omega\|_{\alpha}:=\sup_{\substack{f\in C^{0,\alpha}(\partial \Omega)\\f\not\equiv 0}}\frac{\|H_\Omega f\|_{C^{0,\alpha}(\Omega)}}{\|f\|_{C^{0,\alpha}(\partial\Omega)}}\,.
	$$
	The following are provided in \cite[Theorem 2, Theorem 3]{aikawa2002}:
	\begin{enumerate}
		\item $\|H_\Omega\|_1=\infty$.
		
		\item For $\alpha\in(0,1)$, if $\|H_\Omega\|_{\alpha}<\infty$ then $\Omega$ satisfies $\mathbf{LHMD}(\alpha)$.
		Conversely, if $\Omega$ satisfies $\mathbf{LHMD}(\alpha')$ for some $\alpha'>\alpha$, then $\|H_\Omega\|_{\alpha}<\infty$.
	\end{enumerate}
\end{remark}

\begin{remark}\label{22.02.18.4}
	Let $\Omega$ be a bounded domain, and suppose that for a constant $\alpha>0$, there exist constants $r_0,\,\widetilde{M}\in(0,\infty)$ such that
	\begin{align*}
		w(x,p,r)\leq \widetilde{M}\left(\frac{|x-p|}{r}\right)^{\alpha}\quad\text{for all}\quad x\in \Omega\cap B(p,r)
	\end{align*}
	whenever $p\in\partial\Omega$ and $r\in (0, r_0]$.
	Then $\Omega$ satisfies $\mathbf{LHMD}(\alpha)$, where $M_{\alpha}$ in \eqref{21.08.03.1} depends only on $\alpha,\,\widetilde{M}$ and $\mathrm{diam}(\Omega)/r_0$.
	Indeed, for a fixed $p\in\partial\Omega$, the function 
	$$w(x,p,r_0)1_{\Omega\cap B(p,r_0)}+1_{\Omega\setminus B(p,r_0)}$$
	is a classical superharmonic function on $\Omega$ (see Propositions~\ref{21.05.18.1}.(4) and \ref{220621237}). 
	In addition, the definition of harmonic measures implies the following:
	\begin{align*}
		&\text{If $r>r_0$ then }w(\,\cdot\,,p,r)\leq w(\,\cdot\,,p,r_0)\text{ on $\Omega\cap B(p,r_0)$,; }\\
		&\text{If  $r\geq \text{diam}(\Omega)$, then $w(\,\cdot\,,p,r)\equiv 0$}.
	\end{align*}
	Therefore for any $r>0$,
	$$
	w(x,p,r)\leq \widetilde{M}\left(\frac{\mathrm{diam}(\Omega)}{r_0}\vee 1\right)^{\alpha}\left(\frac{|x-p|}{r}\right)^{\alpha}\quad\text{for all}\,\,\,x\in\Omega\cap B(p,r)\,.
	$$
\end{remark}

\begin{remark}\label{21.07.06.1}
	For an open ball $B\subset \bR^d$ and compact set $K\subset B$,
	\begin{align}\label{230324942}
		\mathrm{Cap}(K,B):=\inf\left\{\|\nabla f\|_2^2\,:\,f\in C_c^{\infty}(B)\,\,\,,\,\,\, f\geq 1\,\,\text{on}\,\,K\right\},
	\end{align}
	denotes the capacity of $K$ relative to $B$.
	Ancona establishes the following in \cite[Lemma 3, Theorem 1, Theorem 2]{AA}:
	\begin{enumerate}
		\item $\Omega$ satisfies $\mathbf{LHMD}(\alpha)$ for some $\alpha\in(0,1)$ if and only if there exists $\epsilon_0$ such that
		\begin{align}\label{22.02.22.1}
			\inf_{p\in\partial\Omega,r>0}\frac{\mathrm{Cap}(\Omega^c\cap \overline{B}(p,r),B(p,2r))}{r^{d-2}}\geq \epsilon_0>0\,.
		\end{align}
		Here, $\epsilon_0$ and $(\alpha,M_{\alpha})$ depend only on each other and $d$.
		
		\item If $\Omega$ satisfies \eqref{22.02.22.1}, then the Hardy inequality \eqref{hardy} holds for $\Omega$, where $\mathrm{C}_0(\Omega)$ depends only on $d$ and $\epsilon_0$.
		
		\item If $\Omega\subset \bR^2$ is a planar domain and admits the Hardy inequality, then \eqref{22.02.22.1} holds for some $\epsilon_0$.
	\end{enumerate}
	The condition \eqref{22.02.22.1} is also called the capacity density condition or uniformly fat exterior condition.
	A well-known sufficient condition to satisfy \eqref{22.02.22.1} is
	\begin{align}\label{Acondition}
		\inf_{p\in\partial\Omega,r>0}\frac{m\big(\Omega^c\cap \overline{B}(p,r)\big)}{r^d}\geq \epsilon_1>0\,,
	\end{align}
	where $m$ is the Lebesgue measure on $\bR^d$.
	Indeed, if $f\in C_c^{\infty}\big(B(p,2r)\big)$ satisfies $f\geq 1$ on $\Omega^c\cap \overline{B}(p,r)$, then  the Poincar\'e inequality implies
	$$
	r^{-d+2}\int_{B(p,2r)}|\nabla f|^2\dd x\gtrsim_d r^{-d}\int_{B(p,2r)}|f|^2\dd x\geq \frac{m\big(\Omega^c\cap \overline{B}(p,r)\big)}{r^d}\,.
	$$
	Therefore, \eqref{Acondition} implies \eqref{22.02.22.1}, where $\epsilon_0$ depends only on $d$ and $\epsilon_1$.
	%
	
	For a deeper discussion of the capacity density condition, we refer the reader to \cite{Kinhardy,lewis} and the references given therein.
\end{remark}
\vspace{1mm}

We finally introduce the relation between \eqref{220617253} and the local harmonic measure decay property.
It was established by Lewis \cite[Theorem 1]{lewis} that if $\Omega$ satisfies the capacity density condition \eqref{22.02.22.1}, then there exist constants $c,\,\epsilon>0$ depending only on $d,\,\epsilon_0$ such that
$$
\cH_{\infty}^{d-2+\epsilon}\big(\Omega^c\cap \overline{B}(p,r)\big)\geq c\, r^{d-2+\epsilon}
$$
for all $p\in\partial\Omega$ and $r>0$.
Conversely, it is well known (see, \textit{e.g.}, \cite[Theorem B]{aikawa1997} or \cite[Theorem 5.9.6]{AG}) that 
for any $\epsilon>0$ and a compact set $E\subset B_1(0)$, we have
$$
\cH_{\infty}^{d-2+\epsilon}(E)\leq N(d,\epsilon)\cdot \mathrm{Cap}\big(E,B_2(0)\big)\,.
$$
Therefore, due to Remark~\ref{21.07.06.1}.(1), \eqref{220617253} holds for some $\epsilon,\,c>0$ if and only if $\Omega$ satisfies $\mathbf{LHMD}(\alpha)$ for some $\alpha>0$.
%
%

Based on this discussion, we consider domains satisfying $\mathbf{LHMD}(\alpha)$ for some $\alpha>0$, instead of \eqref{220617253}.
This condition is implied by geometric conditions introduced in Section~\ref{app2.}, and the value of $\alpha$ reflects each geometric condition; see Theorem~\ref{22.02.18.3}, Remark~\ref{2209071100}, and Corollary~\ref{22.07.16}.
In the rest of this subsection, we construct appropriate superharmonic functions related to $\alpha$ (see Remark~\ref{220617557}).
The results in this subsection are crucially used in Subsection~\ref{app2.}.




\begin{thm}\label{21.11.08.1}
	Let $\Omega$ satisfy $\mathbf{LHMD}(\alpha)$, $\alpha>0$.
	Then for any $\beta\in(0,\alpha)$, there exists a superharmonic function $\phi$ on $\Omega$ satisfying
	$$
	N^{-1}\rho^{\beta}\leq \phi\leq N\rho^{\beta}
	$$
	where $N=N(\alpha,\beta,M_{\alpha})>0$.
\end{thm}

Before proving Theorem~\ref{21.11.08.1}, we look at the following corollary:

\begin{corollary}\label{22.02.19.3}
	Let $\Omega\subset\bR^d$ satisfy $\mathbf{LHMD}(\alpha)$, $\alpha>0$.
	For any $p\in(1,\infty)$ and $\theta\in\bR$ satisfying
	$$
	d-2-(p-1)\alpha<\theta<d-2+\alpha\,,
	$$
	Statement~\ref{22.02.19.1} $(\Omega,p,\theta)$ holds.
	In addition, $N_0$ (in \eqref{2205241155}) and $N_1$ (in \eqref{2205241156}) depends only on $d$, $p$, $\gamma$, $\theta$, $\alpha$, $M_{\alpha}$.
\end{corollary}
\begin{proof}[Proof of Corollary \ref{22.02.19.3}]
	Take $\beta\in(0,\alpha)$ such that
	$$
	d-2-(p-1)\beta<\theta<d-2+\beta.
	$$
	It follows from Theorem~\ref{21.11.08.1} that there exists a superharmonic function $\phi$ such that
	$$
	N^{-1}\rho^{\beta}\leq \phi\leq N\rho^{\beta}\,,
	$$
	where $N=N(\alpha,\,\beta,\,M_{\alpha})$.
	Remarks~\ref{21.07.06.1}.(1) and (2) yield that $\Omega$ admits the Hardy inequality \eqref{hardy}, where $\mathrm{C}_0(\Omega)$ can be chosen to depend only on $d,\alpha$ and $M_{\alpha}$ (in \eqref{21.08.03.1}). 
	Therefore, by Remark~\ref{220617557}, the proof is completed.
\end{proof}

\begin{proof}[Proof of Theorem~\ref{21.11.08.1}]
	The following construction is a combination of \cite[Theorem 1]{AA} and \cite[Lemma 2.1]{KenigToro}.
	Recall that $M_{\alpha}$ is the constant in \eqref{21.08.03.1}, and $\beta<\alpha$.
	Take $r_0\in(0,1)$ small enough to satisfy $M_{\alpha} r_0^{\alpha} < r_0^{\beta}$,  and take $\eta\in(0,1)$ small enough to satisfy
	$$
	(1-\eta)M_{\alpha} r_0^{\alpha}+\eta\leq r_0^{\beta}\,.
	$$
	For $w(x,p,r)$, we shall need only the following properties (see Proposition~\ref{220621237} and Definition~\ref{2301021117}):
	\begin{align*}
		\begin{split}
			&\text{$w(\cdot,p,r)$ is a classical superharmonic function on $\Omega\cap B(p,r)$}\,;\\
			&\text{$w(\cdot,p,r)=1$  on $\Omega\cap \partial B(p,r)$}\,;\\
			&\text{$0\leq w(\cdot,p,r)\leq M_{\alpha}r_0^{\alpha}$  on $\Omega\cap B(p,r_0r)$}\,.
		\end{split}
	\end{align*}
	For $p\in\partial\Omega$ and $k\in\bZ$, put
	$$
	\phi_{p,k}(x)=r_0^{k\beta}\big((1-\eta)\, w(x,p,r_0^k)+\eta\big).
	$$
	Then $\phi_{p,k}$ is a classical superharmonic function on $\Omega\cap B(p,r_0^k)$,
	\begin{alignat*}{2}
		&\quad\,\, \phi_{p,k}\leq  r_0^{(k+1)\beta}\qquad\,\,&&\text{on}\quad \Omega\cap \overline{B}(p,r_0^{k+1})\,,\\
		&\quad\,\, \phi_{p,k}=r_0^{k\beta}\qquad\qquad &&\text{on}\quad \Omega\cap \partial B(p,r_0^k)\,,\\
		&\eta\cdot r_0^{k\beta}\leq \phi_{p,k}\leq r_0^{k\beta}\quad\,\,&&\text{on}\quad \Omega\cap B(p,r_0^k)\,.
	\end{alignat*}
	For $p\in\partial\Omega$ and $x\in\Omega$, we denote
	$$
	\phi_p(x)=\inf\{\phi_{p,k}(x)\,:\,|x-p|< r_0^{k}\}.
	$$
	
	If we prove the following:
	\begin{align}
		&\text{$\phi_p$ is a classical superharmonic function on $\Omega$ ;}\label{221230124}\\
		&\eta|x-p|^{\beta}\leq \phi_p(x)\leq r_0^{-\beta}|x-p|^{\beta}\,,\label{221230125}
	\end{align}
	then $\phi:=\inf_{p\in\partial\Omega}\phi_p$ is superharmonic on $\Omega$ (see Proposition~\ref{21.05.18.1}.(2)) and satisfies
	$$
	\eta \rho(x)^{\beta}\leq \phi(x) \leq r_0^{-\beta}\rho(x)^{\beta}\,.
	$$
	Therefore the proof is completed.
	
	To obtain \eqref{221230124} and \eqref{221230125}, we only need to prove each of the following, respectively: for each $k_0\in\bZ$,
	\begin{align*}
		&\text{$\phi_p$ is a classical superharmonic function on $\{x\in\Omega\,:\,r_0^{k_0+2}<|x-p|<r_0^{k_0}\}$;}\\
		&\text{$\eta\,r_0^{k_0\beta}\leq \phi_p\leq r_0^{k_0\beta}$ on $\{x\in\Omega\,:\,r_0^{k_0+1}\leq |x-p|<r_0^{k_0}\}.$}
	\end{align*}
	
	\noindent
	\textbf{-} \eqref{221230124} : 
	For $x\in\Omega\cap B(p,r_0^{k_0})$, put
	\begin{align*}
		v_{p,k_0}(x)=
		\begin{cases}
			\phi_{p,k_0}(x) &\text{if}\quad r_0^{k_0+1}\leq |x-p|< r_0^{k_0}\\
			\phi_{p,k_0}(x)\wedge \phi_{p,k_0+1}(x)&\text{if}\quad |x-p|< r_0^{k_0+1}\,.
		\end{cases}
	\end{align*}
	Since $\phi_{p,k_0}\leq \phi_{p,k_0+1}$ on $\Omega\cap \partial B(p,r_0^{k_0+1})$, Proposition~\ref{21.05.18.1}.(4) implies that $v_{p,k_0}$ is a classical superharmonic function on $\Omega\cap B(p,r_0^{k_0})$.
	We denote
	$$
	U_{k_0}=\{x\in\Omega\,:\,r_0^{k_0+2}<|x-p|<r_0^{k_0}\}\,.
	$$
	For $x\in U_{k_0}$, we have
	$$
	\phi_p(x)=v_{p,k_0}(x)\wedge \inf\{\phi_{p,k}(x)\,:\,k\leq k_0-1\}.
	$$
	Moreover, if $\eta\,r_0^{k\beta}\geq r_0^{k_0\beta}$ then
	$$
	v_{p,k_0}(x)\leq \phi_{p,k_0}(x)\leq r_0^{k_0\beta}\leq \eta\,r_0^{k\beta}\leq \phi_{p,k}(x)\,.
	$$
	Therefore
	\begin{align*}
		\phi_p(x)=v_{p,k_0}(x)\wedge \inf\{\phi_{p,k}(x)\,:\, k\leq k_0-1\quad\text{and}\quad \eta\, r_0^{k\beta}\leq r_0^{k_0\beta}\}\,,
	\end{align*}
	which implies that $\phi_p$ is the minimum of finitely many classical superharmonic functions, on $U_{k_0}$.
	Consequently, by Proposition~\ref{21.05.18.1}.(1), $\phi_p$ is a classical superharmonic function on $U_{k_0}$.
	
	\noindent
	\textbf{-} \eqref{221230125} : 
	Let $x\in\Omega$ satisfy $r_0^{k_0+1}\leq |x-p|< r_0^{k_0}$.
	Since
	$$
	\phi_{p,k_0}(x)\leq r_0^{k_0\beta}\,\,,\quad\text{and}\quad \phi_{p,k}(x)\geq \eta r_0^{k\beta}\geq \eta r_0^{k_0\beta}\quad\text{for all}\quad k\leq k_0\,,
	$$
	we obtain that $\eta\,r_0^{k_0\beta}\leq \phi_p(x)\leq r_0^{k_0\beta}$.
\end{proof}

%
%

\vspace{2mm}

\subsection{Further results for domains with fat exterior}\label{fatex2}\,

In this subsection, we introduce a unweighted solvablity results and embedding theorems for the Poisson and heat equations in domains satisfying the capacity density condition \eqref{22.02.22.1}.

Recall that $\mathring{W}_p^1(\Omega)$ denotes the closure of $C_c^{\infty}(\Omega)$ in
$$
W_p^1(\Omega):=\big\{f\in\cD'(\Omega)\,:\,\|f\|_{W_p^1(\Omega)}:=\|f\|_p+\|\nabla f\|_p<\infty\big\}\,.
$$
Note that $W_p^1(\Omega)$ is a Banach space, and therefore $\mathring{W}_p^1(\Omega)$ is also a Banach space.

\begin{thm}\label{230210356}
	Let $\Omega$ satisfy the capacity density condition \eqref{22.02.22.1} and 
	\begin{equation}\label{2304044112}
		\begin{aligned}
			\lambda\geq 0\quad \text{if}\quad d_{\Omega}<\infty\quad\text{and}\quad  	\lambda> 0\quad \text{if}\quad d_{\Omega}=\infty\,,
		\end{aligned}	
	\end{equation}
	where $d_{\Omega}:=\sup_{x\in\Omega}d(x,\partial\Omega)$.
	Then there exists $\epsilon\in(0,1)$ depending only on $d$, $\epsilon_0$ (in \eqref{22.02.22.1}) such that for any $p\in(2-\epsilon,2+\epsilon)$, the following holds:
	\begin{itemize}
		\item[] For any $f^0,\,\ldots,\,f^d\in L_p(\Omega)$, the equation
		\begin{align}\label{230203750}
			\Delta u-\lambda u=f^0+\sum_{i=1}^dD_if^i
		\end{align}
		has a unique solution $u$ in $\mathring{W}^{1}_{p}(\Omega)$.
		Moreover, we have
		\begin{align}\label{230203814}
			\begin{split}
				&\|\nabla u\|_{L_p(\Omega)}+\big(\lambda^{1/2}+d_\Omega^{-1}\big)\|u\|_{L_p(\Omega)}\\
				\leq N(d,p,\epsilon_0)\,& \Big(\min\big(\lambda^{-1/2},d_\Omega\big)\|f^0\|_{L_p(\Omega)}+\sum_{i=1}^d\|f^i\|_{L_p(\Omega)}\Big)\,.
			\end{split}
		\end{align}
	\end{itemize}
\end{thm}

\begin{proof}
	We first note the following two result which follows from \eqref{22.02.22.1}:
	\begin{itemize}
		\item[(a)] By Remark~\ref{21.07.06.1}.(1), there exists $\alpha\in(0,1)$ such that $\Omega$ satisfies $\mathbf{LHMD}(\alpha)$.
		Due to Corollary~\ref{22.02.19.3}, Statement~\ref{22.02.19.1} $(\Omega,p,d-p)$ holds for
		\begin{align*}
			2-\alpha<p<2+\frac{\alpha}{1-\alpha}\,,
		\end{align*}
		and $N_1$ (in \eqref{2205241155}) depends only on $d,p,\gamma,\epsilon_1$.

		\item[(b)] It is implied by \cite[Theorem 1, Theorem 2]{lewis} (or see \cite[Theorem 3.7, Corollary 3.11]{Kinhardy}) that there exists $p_0\in(1,2)$ depending only on $d$ and $\epsilon_0$ such that for any $p>p_0$,
		\begin{align}\label{230213147}
			\int_{\Omega}\Big|\frac{u(x)}{\rho(x)}\Big|^p\dd x\leq N(d,p,\epsilon_0)\int_{\Omega}|\nabla u|^p\dd x\quad \forall\,\, u\in C_c^{\infty}(\Omega)\,.
		\end{align}
		Due to Lemma~\ref{21.09.29.4}.(1) and the definition of $\mathring{W}_p^1(\Omega)$, $C_c^{\infty}(\Omega)$ is dense in $\mathring{W}_p^1(\Omega)$ and $H_{p,d-p}^1(\Omega)$, separately. Therefore \eqref{230213147} implies that $\mathring{W}_p^1(\Omega)\subset H_{p,d-p}^1(\Omega)$.
	\end{itemize}
	Take $\epsilon\in(0,1)$ such that $\epsilon\leq \alpha$ and $\epsilon\leq 2-p_0$.
	We consider a fixed $p\in (2-\epsilon,2+\epsilon)$.
	
	We will use Lemma~\ref{220512433}, Corollary~\ref{21.05.26.3}, $d_{\Omega}^{\,-1}\|u\|_p\leq \|\rho^{-1}u\|_p$, and $\|\rho f\|_p\leq d_{\Omega}\|f\|_p$, without mentioning.
	
	\vspace{2mm}
	
	\textbf{Step 1. Uniqueness of solutions.}
	
	Suppose that $u\in \mathring{W}_p^1(\Omega)$ satisfies $\Delta u-\lambda u=0$.
	By (a) in this proof, $u$ belongs to $H_{p,d-p}^1(\Omega)$, which implies $u\equiv 0$.
	Therefore, the uniqueness of solutions is proved.
	
	%

	\vspace{2mm}
	\textbf{Step 2. Existence of solutions and estimate \eqref{230203814}.}
	
	To prove the existence of solutions, it is enough to find a solution in $L_{p,d}(\Omega)\cap H_{p,d-p}^1(\Omega)$.
	Indeed, if $u\in L_{p,d}(\Omega)\cap H_{p,d-p}^1(\Omega)$, then there exists $u_n\in C_c^{\infty}(\Omega)$ such that $u_n\rightarrow u$ in $L_{p,d}(\Omega)\cap H_{p,d-p}^1(\Omega)$ (see Lemma~\ref{21.09.29.4}.(5)).
	Therefore,
	$$
	\|u_n-u\|_{W_p^1(\Omega)}\lesssim \|u_n-u\|_{L_{p,d}(\Omega)}+\|u_n-u\|_{H_{p,d-p}^1(\Omega)}\rightarrow 0\,,
	$$
	which implies that $u\in \mathring{W}_p^1(\Omega)$.
	
	Without loss of generality, we can assume that $\lambda=0$ or $\lambda=1$ by dilation.
	Note that $\epsilon_0$ in \eqref{22.02.22.1} is invariant even if $\Omega$ is replaced by $r\Omega=\{rx\,:\,x\in\Omega\}$, for any $r>0$.

	\textbf{- Step 2.1.} Consider the case $\lambda=1$. 
	Since $\trho^{-1}f^0\in L_{d+p}(\Omega)$ and Statement~\ref{22.02.19.1} $(\Omega,p,d-p)$ holds, there exists $v\in H_{p,d-p}^2(\Omega)$ such that 
	$$
	\Delta v- v=\trho^{\,-1}f^0
	$$
	and $\|v\|_{H_{p,d-p}^2(\Omega)}+\|v\|_{L_{p,d+p}(\Omega)} \lesssim_{d,p,\epsilon_0}\left\|\trho^{-1} f^0\right\|_{L_{p,d+p}(\Omega)}$
	(see (a) in this proof)
	By Proposition~\ref{220527502}.(9) and Lemma~\ref{21.05.20.3}.(1), we have
	\begin{equation}\label{230403637}
		\begin{alignedat}{2}
			&&&\|v\|_{L_{p,d}(\Omega)}+\left(\|v\|_{L_{p,d+p}(\Omega)}+\|v\|_{H_{p,d}^1(\Omega)}\right)\\
			&\lesssim_{d,p}\,&& \left(\|v\|_{H_{p,d-p}^1(\Omega)}+\|v\|_{H_{p,d+p}^{-1}(\Omega)}\right)+\left(\|v\|_{L_{p,d+p}(\Omega)}+\|v\|_{H_{p,d-p}^2(\Omega)}\right)\\
			&\lesssim_{d,p}&&\|v\|_{H_{p,d-p}^2(\Omega)}+\|v\|_{L_{p,d+p}(\Omega)}\\
			&\lesssim_{d,p,\epsilon_0}&& \left\|\trho^{-1} f^0\right\|_{L_{p,d+p}(\Omega)}\\
			&\simeq_{d,p}&& \|f^0\|_p\,.
		\end{alignedat}
	\end{equation}
	Observe that $\widetilde{f}:=f^0-\Delta(\trho v)+\trho v$ satisfies
	\begin{align*}
		\widetilde{f}=-2\Big[\sum_{i= 1}^dD_i\big(vD_i\trho )\Big]+v\Delta \trho
	\end{align*}
	and therefore
	\begin{align}\label{230213214}
		\begin{split}
			\big\|\widetilde{f}\big\|_{H_{p,d+p}^{-1}(\Omega)}\,&\lesssim_{d,p} \left\|v \trho_x\right\|_{L_{p,d}(\Omega)}+\left\|v\trho_{xx}\right\|_{L_{p,d+p}(\Omega)}\\
			&\lesssim_{d,p,\epsilon_0} \|v\|_{L_{p,d}(\Omega)}\lesssim_{d,p,\epsilon_0}\left\|f^0\right\|_p\,,
		\end{split}
	\end{align}
	where the last inequality follows from \eqref{230403637}.
	Since Statement~\ref{22.02.19.1} $(\Omega,p,d-p)$ holds, there exists $w\in H_{p,d-p}^1(\Omega)$ such that
	$$
	\Delta w- w=\sum_{i= 1}^dD_if^i+\widetilde{f}
	$$
	and $$\|w\|_{H_{p,d-p}^1(\Omega)}+\|w\|_{H^{-1}_{p,d+p}(\Omega)}\lesssim_{d,p,\epsilon_0}\sum_{i= 1}^d\|f^i\|_{L_{p,d}(\Omega)}+\big\|\widetilde{f}\big\|_{H_{p,d+p}^{-1}(\Omega)}\lesssim \sum_{i=0}^d\|f^i\|_{p}
	$$
	(see (a) in this proof).
	Therefore, by Proposition~\ref{220527502}.(9), Lemma~\ref{21.05.20.3}.(1), and \eqref{230213214}, we have
	\begin{equation}\label{2304041137}
		\begin{alignedat}{2}
			&&&\|w\|_{L_{p,d-p}(\Omega)}+\big(\|w\|_{L_{p,d}(\Omega)}+\|w\|_{H_{p,d-p}^1(\Omega)}\big)\\
			&\lesssim_{d,p}&&\|w\|_{H_{p,d-p}^1(\Omega)}+\|w\|_{H^{-1}_{p,d+p}(\Omega)} \lesssim_{d,p,\epsilon_0} \sum_{i\geq 0}\left\|f^i\right\|_p\,.
		\end{alignedat}
	\end{equation}
	Put $u=v\trho+w$.
	Then $u$ is a solution of equation \eqref{230203750} and satisfies
	\begin{equation}\label{230204149}
		\begin{alignedat}{2}
			&&&\|u_x\|_p+(1+d_\Omega^{-1})\|u\|_p\\
			&\lesssim_{d,p} &&\|u\|_{L_{p,d}(\Omega)}+\|u\|_{H_{p,d-p}^1(\Omega)}\\
			&\lesssim_{d,p}&& \|w\|_{L_{p,d}(\Omega)}+\|w\|_{H_{p,d-p}^1(\Omega)}+\|v\|_{L_{p,d+p}(\Omega)}+\|v\|_{H_{p,d}^1(\Omega)}\\
			&\lesssim_{d,p,\epsilon_0} &&\sum_{i\geq 0}\|f^i\|_p\,,
		\end{alignedat}
	\end{equation}
	Here, the second inequality follows from Lemma~\ref{21.05.20.3}.(3), and the last inequality follows from \eqref{230403637} and \eqref{2304041137}.
	
	Note that \eqref{230204149} also implies that $u\in L_{p,d}(\Omega)\cap H_{p,d-p}^1(\Omega)$.
	
	\textbf{- Step 2.2.} Consider the case $d_{\Omega}<\infty$, and observe that
	\begin{equation}\label{230404924}
		\begin{aligned}
			\|f^0+\sum_{i\geq 1}D_if^i\|_{H_{p,d+p}^{-1}(\Omega)}&\lesssim_{d,p}\|f^0\|_{L_{p,d+p}(\Omega)}+\sum_{i\geq 1}\|f^i\|_{L_{p,d}(\Omega)}\\
			&\leq d_{\Omega}\|f^0\|_p+\sum_{i\geq 1}\|f^i\|_p<\infty\,,
		\end{aligned}
	\end{equation}
	Since Statement~\ref{22.02.19.1} $(\Omega,p,d-p)$ holds, either $\lambda=0$ or $\lambda=1$, there exists $\widetilde{u}\in H_{p,d-p}^1(\Omega)$ such that
	$$
	\Delta \widetilde{u}-\lambda \widetilde{u}=f^0+\sum_{i\geq 1}D_if^i\,,
	$$
	and
	\begin{align}\label{230404925}
		\|\widetilde{u}\|_{H_{p,d-p}^1(\Omega)}+\lambda \|\widetilde{u}\|_{H_{p,d+p}^{-1}(\Omega)}\lesssim \|f^0+\sum_{i\geq 1}D_if^i\|_{H_{p,d+p}^{-1}(\Omega)}
	\end{align}
	(see (a) in this proof).
	By Proposition~\ref{220527502}.(9), \eqref{230404924}, and \eqref{230404925}, we obtain that
	\begin{equation}\label{230204150}
		\begin{alignedat}{2}
			&&&\|\nabla\widetilde{u}\|_{L_p(\Omega)}+d_\Omega^{-1}\|\widetilde{u}\|_{L_p(\Omega)}+\lambda^{1/2}\|\widetilde{u}\|_{L_p(\Omega)}\\
			&\lesssim_{d,p}\,&& \|\widetilde{u}\|_{H_{p,d-p}^1(\Omega)}+\lambda \|\widetilde{u}\|_{H_{p,d+p}^{-1}(\Omega)}\\
			&\lesssim_{d,p,\epsilon_0}\,&& d_{\Omega}\|f^0\|_p+\sum_{i\geq 1}\|f^i\|_p\,.
		\end{alignedat}
	\end{equation}
	Due to \eqref{230204150}, we have $\widetilde{u}\in L_{p,d}(\Omega)\cap H_{p,d-p}^1(\Omega)$.
	
	\textbf{- Step 2.3.}
	The existence of solutions is proved in Steps 2.1 and 2.2, for all $\lambda$ and $d_{\Omega}$ satisfying \eqref{2304044112}.
	For the cases where $d_{\Omega}=\infty$ and $\lambda=1$, and $d_{\Omega}<\infty$ and$\lambda=0$, estimate \eqref{230203814} is proved in \eqref{230204149} and \eqref{230204150}, respectively.
	Therefore, we only need prove estimate \eqref{230203814} in the remaining case where $d_{\Omega}<\infty$ and $\lambda=1$.
	Since $u$ in Step 2.1 and $\widetilde{u}$ in Step 2.2 are the same (due to the result in Step 1), \eqref{2208131023} can be obtained by combining \eqref{230204149} and \eqref{230204150}.
\end{proof}

\begin{thm}\label{230210357}
	Let $\Omega$ satisfy the capacity density condition \eqref{22.02.22.1} and 
	\begin{equation*}
		\begin{aligned}
			T\leq \infty\quad \text{if}\quad d_{\Omega}<\infty\quad\text{and}\quad  	T<\infty\quad \text{if}\quad d_{\Omega}=\infty\,,
		\end{aligned}	
	\end{equation*}
	where $d_{\Omega}:=\sup_{x\in\Omega}d(x,\partial\Omega)$.
	Then for any $\nu_1,\,\nu_2\in\bR$ with $0<\nu_1\leq \nu_2<\infty$, there exists $\epsilon\in(0,1)$ depending only on $d$, $\epsilon_0$ (in \eqref{22.02.22.1}), $\nu_1$, $\nu_2$ such that the following holds:
	\begin{itemize}
		\item[]
		Suppose that $p\in(2-\epsilon,2+\epsilon)$ and $\cL\in \cM_T(\nu_1,\nu_2)$.
		Then for any $f^0,\,\ldots,\,f^d\in L_p((0,T]\times \Omega)$, the equation
		\begin{align}\label{2208131023}
			\partial_tu=\cL u+f^0+\sum_{i=1}^dD_if^i\quad\text{in}\quad (0,T]\quad;\quad u(0,\cdot)=0
		\end{align}
		has a unique solution $u$ in $L_p\big((0,T];\mathring{W}^{1}_{p}(\Omega)\big)$ (see \eqref{2208221058} for the definition of equation \eqref{2208131023}).
		Moreover, we have
		\begin{align}\label{230204228}
			\begin{split}
				&\|\nabla u\|_{L_p((0,T]\times \Omega)}+\big(T^{-1/2}+(d_\Omega)^{-1}\big)\|u\|_{L_p((0,T]\times \Omega)}\\
				\leq N(d,p,\epsilon_0) \,&\Big(\min(T^{1/2},d_\Omega)\|f^0\|_{L_p((0,T]\times \Omega)}+\sum_{i=1}^d\|f^i\|_{L_p((0,T]\times \Omega)}\Big)\,.
			\end{split}
		\end{align}
	\end{itemize}
\end{thm}

\begin{proof}
	We introduce the expression `$\mathrm{Statement}_{\nu_1,\nu_2}\,(\Omega,p,\theta)$ holds' to indicate that
	\begin{align*}
		\begin{gathered}
			\text{` Statement~\ref{22.02.19.1}}\,(\Omega,p,\theta).(2)\,\,\text{holds for $\Delta$ replaced by arbitary}\\
		\text{$\cL\in\cM_T(\nu_1,\nu_2)$. In addition, $N_2$ (in \eqref{2205241156}) depends only}\\
		\text{on $d$, $p$, $\theta$, $\epsilon_0$, $\nu_1$, $\nu_2$. '}
	\end{gathered}	
\end{align*}

Remarks~\ref{21.07.06.1}.(1), (2) and Theorem~\ref{21.11.08.1} imply the following:
\begin{itemize}
	\item[-] $\Omega$ admits the Hardy inequality \eqref{hardy}, where $\mathrm{C}_0(\Omega)$ can be chosen to depend only on $d$ and $\epsilon_0$.
	
	\item[-] There exists $\alpha>0$ and a superharmonic function $\phi$ on $\Omega$ such that
	$$
	N^{-1}\rho^\alpha\leq \phi\leq N\rho^\alpha\,,
	$$
	where $\alpha$ and $N$ depend only on $d$ and $\epsilon_0$.
\end{itemize}
Therefore, due to Theorem~\ref{22.02.18.6} (with $\Psi=\trho^{\alpha}$) and Proposition~\ref{05.11.1}.(1), if $\theta\in\bR$ satisfies
\begin{align}\label{230404901}
	-\frac{(p-1)\alpha}{p(\sqrt{\nu_2/\nu_1}-1)/2+1}<\theta-d+2<\frac{(p-1)\alpha}{p(\sqrt{\nu_2/\nu_1}+1)/2-1}\,,
\end{align}
then Statement$_{\nu_1,\nu_2}(\Omega,p,\theta)$ holds.
The first term in \eqref{230404901} goes to $-\alpha\sqrt{\nu_1/\nu_2}$ as $p\rightarrow 2$, while the second term in \eqref{230404901} goes to $\alpha\sqrt{\nu_1/\nu_2}$ as $p\rightarrow 2$.
Therefore, there exists $\epsilon_1>0$ (which depends only on $\nu_1$, $\nu_2$, and $\alpha$) such that if $p\in(2-\epsilon_1,2+\epsilon_1)$, then $\theta:=d-p$ satisfy \eqref{230404901}, and thus $\mathrm{Statement}_{\nu_1,\nu_2}\,(\Omega,p,d-p)$ holds.

By (a) in the proof of Theorem \ref{230210356}, there exists $p_0\in(1,2)$ such that for any $p>p_0$, $\mathring{W}_p^1(\Omega)\subset H_{p,d-p}^1(\Omega)$.

Take $\epsilon\in(0,\epsilon_1)$ such that $2-\epsilon>p_0$.
Then for any $p\in(2-\epsilon,2+\epsilon)$, $\mathrm{Statement}_{\nu_1,\nu_2}\,(\Omega,p,d-p)$ holds and $\mathring{W}_p^1(\Omega)\subset H_{p,d-p}^1(\Omega)$.
\vspace{1mm}

\textbf{Step 1. Uniqueness of solutions.}
Suppose that $u\in L_p((0,T];\mathring{W}_p^1(\Omega))$ satisfies
$$
\partial_tu=\Delta u\quad;\quad u(0,\cdot)\equiv 0\,.
$$
Since $\mathring{W}_p^1(\Omega)\subset H_{p,d-p}^1(\Omega)$, we have $L_p((0,T];\mathring{W}_p^1(\Omega))\subset \bH_{p,d-p}^1(\Omega,T)$.
Therefore, by Lemma~\ref{21.05.13.9}, $u\in \cH_{p,d-p}^1(\Omega,T)$.
Since Statement$_{\nu_1,\nu_2}\,(\Omega,p,d-p)$ holds, $u\equiv 0$.

\vspace{1mm}
\textbf{Step 2. Existence of solutions and estimate \eqref{230204228}.}
Proof of the existence of solutions and estimate \eqref{230204228} is left to the reader, as it can be shown in a similar way by following Steps 2.1 - 2.3 in the proof of Theorem~\ref{230210356}, with the following details:
\begin{itemize}
	\item[-] To prove the existence of solutions, it is enough to find a solution in $\bL_{p,d}(\Omega,T)\cap \bH_{p,d-p}^1(\Omega,T)$.
	It is because if $u\in \bL_{p,d}(\Omega,T)\cap \bH_{p,d-p}^1(\Omega,T)$, there exists $u_n\in C_c^{\infty}((0,T)\times \Omega)$ such that $u_n\rightarrow u$ in $\bL_{p,d}(\Omega,T)$ and $ \bH_{p,d-p}^1(\Omega,T)$, separately (see Lemma~\ref{2205241011}).
	Since $L_{p,d}(\Omega)\cap H_{p,d-p}^1(\Omega)\subset \mathring{W}_p^1(\Omega)$ (see Step 2 in the proof of Theorem~\ref{230210356}), $\{u_n\}_{n\in\bN}$ is a Cauchy sequence in $L_p\big((0,T];\mathring{W}_p^1(\Omega)\big)$.
	Therefore
	$$
	u=\lim_{n\rightarrow \infty}u_n\quad\text{in}\,\,\,\,L_p\big((0,T];\mathring{W}_p^1(\Omega)\big)\,.
	$$
	
	\item[-] Without loss of generality, we can assume that $T=1$ or $T=\infty$ by dilation.
	
	\item[-] For the case $T=1$, note that if $v\in\cH_{p,d-p}^{n+2}(\Omega,1)$ satisfies $v(0)\equiv 0$, then
	\begin{align*}
		\begin{split}
			\|v\|_{\bH_{p,d+p}^n(\Omega,1)}^p&\,= \int_0^1\|v(t)\|_{H_{p,d+p}^n(\Omega)}^p\dd t\leq \int_0^1\Big(\int_0^t\|v_t(s)\|_{H_{p,d+p}^n(\Omega)}\dd s\Big)^p\dd t\\
			&\,\leq \|v_t\|_{\bH_{p,d+p}^n(\Omega,1)}^p\leq \|v\|_{\cH_{p,d-p}^{n+2}(\Omega,1)}\,.
		\end{split}
	\end{align*}
\end{itemize}
\end{proof}

\begin{remark}
Actually, from the proofs Theorem~\ref{230210356} and Theorem \ref{230210357}, it can observed that for a fixed $p\in(1,\infty)$, the assertion in Theorem~\ref{230210356} (resp. Theorem \ref{230210357}) holds if Statement~\ref{22.02.19.1} $(\Omega,p,d-p)$ holds (resp. $\mathrm{Statement}_{\nu_1,\nu_2}\,(\Omega,p,d-p)$ holds) and $\mathring{W}_p^1(\Omega)\subset H_{p,d-p}^1(\Omega)$.
Note that if
$$
\inf_{\substack{p\in\partial\Omega\\r>0}}\frac{m(\Omega^c \cap B_r(p))}{m(B_r(p))}>0
$$
(where $m$ is the Lebesgue measure on $\bR^d$),
then the $L_p$-Hardy inequality holds (see \cite[Example 3.6, Corollary 3.11]{Kinhardy}), and therefore we hve $\mathring{W}_p^1(\Omega)\subset H_{p,d-p}^1(\Omega)$.
\end{remark}

In the next theorems, we discuss the embedding theorems, Propositions~\ref{220512537} and \ref{2204160313}.
For a fixed $\epsilon\in(0,1]$, let $p$ be large enough such that $p>d$ and $\epsilon>1/p$.
Then it follows from Proposition~\ref{220512537} that if $f\in \Psi^{1-\epsilon} H_{p,d+2p-2}^{-1}(\Omega)$ and $u\in \Psi^{1-\epsilon}H_{p,d-2}^1(\Omega)$ satisfy $\Delta u=f$, then
\begin{align*}
u(x)\lesssim \|f\|_{\Psi^{1-\epsilon}H_{p,2p-2}^{-1}(\Omega)}\cdot \rho(x)^{-(d-2)/p}\Psi^{1-\epsilon}(x)\,.
\end{align*}
In Theorems \ref{220602323} and \ref{220602322}, we modify this type of estimates to delete the term $\rho^{-(d-2)/p}$ using Theorem~\ref{21.11.08.1}.

\begin{thm}\label{220602323}
Let $\Omega$ satisfy $\mathbf{LHMD}(\alpha)$, $\lambda\geq 0$, and $\psi$ be a superharmonic Harnack function on $\Omega$.
Suppose that $\delta$ and $\epsilon$ are positive constants such that
\begin{align}
	0<\delta<\frac{\alpha d}{\alpha+d-2}\wedge 1\qquad \text{and}\qquad 
	\epsilon\in
	\begin{cases}\label{21.05.17.2}
		(\delta/2,1+\delta/2)&\text{if}\quad d=2\,;\\[1mm]
		(\frac{\alpha+d-2}{\alpha d}\delta,1] &\text{if}\quad d\geq 3\,.
	\end{cases}
\end{align}
If $f^0,\,f^i,\,\ldots,\,f^d$ are measurable functions on $\Omega$ with
\begin{align}\label{2304121049}
	F:=\Big\||\psi^{-1+\epsilon}\rho^{2-\delta}f^0|+\sum_{i=1}^d|\psi^{-1+\epsilon}\rho^{1-\delta} f^i|\Big\|_{L_{d/\delta}(\Omega,\dd x)}<\infty\,,
\end{align}
then the equation
\begin{align}\label{220621310}
	\Delta u-\lambda u=f^0+\sum_{i=1}^dD_if^i
\end{align}
has a unique solution $u$ in $\widetilde{\psi}^{1-\epsilon}H_{d/\delta,0}^1(\Omega)$,
where $\widetilde{\psi}$ is the regularization of $\psi$ in Lemma~\ref{21.05.27.3}.(1).
Moreover, we have
\begin{align}\label{220906540}
	|u(x)|+\rho(x)^{1-\delta}\sup_{y\in B_{\rho(x)/2}(x)}\frac{|u(x)-u(y)|}{|x-y|^{1-\delta}}\leq N  F\cdot\big(\psi(x)\big)^{1-\epsilon}
\end{align}
for all $x\in\Omega$, where $N=N(d,p,\mathrm{C}_1(\psi),\alpha,M_{\alpha},\delta,\epsilon)$.
\end{thm}

\begin{proof}
Put $p=d/\delta$, and note that Corollary~\ref{21.05.26.3} implies
\begin{align}\label{220707248}
	\|f^0+\sum_{i=1}^dD_i f^i\|_{\widetilde{\psi}^{1-\epsilon}H^{-1}_{p,2p}(\Omega)}\lesssim F\,.
\end{align}

\textbf{Case 1.} $d=2$.

Observe that $\delta=2/p$ and $1-\epsilon\in(-1/p,1-1/p)$.
Due to \eqref{220707248} and $d-2=0$,
this corollary is implied by Theorem~\ref{21.09.29.1} and Proposition~\ref{220512537}.

\textbf{Case 2.} $d\geq 3$.

Take $\alpha_1\in(0,\alpha)$ such that
$$
\epsilon>\frac{\alpha_1+d-2}{\alpha_1d}\delta\,,
$$
and put
$$
\mu=\frac{d-2}{\alpha_1\,p}+(1-\epsilon)\,\,,\,\,\text{and}\,\,\,\,t=\mu^{-1}(1-\epsilon)\,,
$$
so that
$$
\mu\in(0,1-1/p)\,\,,\,\,\,\,t\in[0,1]\,\,,\,\,\,\,\mu t=1-\epsilon\,\,,\,\,\,\,\alpha_1\mu(1-t)=(d-2)/p.
$$
By Theorem~\ref{21.11.08.1}, there exists a superharmonic function $\phi_0$ such that $\phi_0\simeq_{N} \rho^{\alpha_1}$ where $N=N(\alpha,M_{\alpha},\alpha_1)$.
Put $\Psi=\trho^{\,\alpha_1(1-t)}\big(\widetilde{\psi}\big)^t$ which is a regularization of the superharmonic function $\phi_0^{1-t}\psi^{t}$ (see Proposition~\ref{21.05.18.1}.(3)).
Note that, by Lemma~\ref{21.05.20.3}.(3), $$
\Psi^{\mu}H_{p,\theta+d-2}^{\gamma}(\Omega)=\widetilde{\psi}^{1-\epsilon}H_{p,\theta}^{\gamma}(\Omega)
$$
for all $\gamma,\,\theta\in\bR$.
By \eqref{220707248} and Theorem~\ref{21.09.29.1}, equation \eqref{220621310} has a unique solution $u$ in the class $\Psi^{\mu}H_{p,d-2}^1(\Omega)=\widetilde{\psi}^{1-\epsilon}H_{p,0}^{1}(\Omega)$.
Furthermore,  Proposition~\ref{220512537} implies \eqref{220906540} for this $u$.
\end{proof}

\begin{thm}\label{220602322}
Let $\Omega$ satisfy $\mathbf{LHMD}(\alpha)$, $T\in(0,\infty)$, and $\psi$ be a superharmonic Harnack function on $\Omega$. Suppose that $\beta_x,\,\beta_t,\,\delta$ and $\epsilon$ are constants in $(0,1)$ such that
\begin{align}\label{220602312}
	\beta_x+2\beta_t\leq 1-\delta\qquad\text{and}\qquad \frac{\delta}{d+2}+\alpha^{-1}\Big(\frac{d}{d+2}\delta+2\beta_t\Big)<\epsilon\leq 1\,.
\end{align}
If $f^0,\,\ldots,\,f^d:(0,T]\times\Omega\rightarrow \bR$ and $u_0:\Omega\rightarrow \bR$ are measurable functions satisfy
\begin{align*}
	&\Big\||\psi^{-1+\epsilon}\rho^{2-2\beta_t-\delta}f^0|+\sum_{i=1}^d|\psi^{-1+\epsilon}\rho^{1-2\beta_t-\delta} f^i|\Big\|_{L_{(d+2)/\delta}((0,T]\times\Omega,\dd x\dd t)}\\
	&+\Big\|\psi^{-1+\epsilon}\rho^{-2\beta_t-\delta}|u_0|+\psi^{-1+\epsilon}\rho^{1-2\beta_t-\delta}|\nabla u_0|\Big\|_{L_{d/\delta}(\Omega,\dd x)}\\
	=:\,&F+I<\infty\,,
\end{align*}
then the equation
\begin{align}\label{2304201115}
u_t=\Delta u+f^0+D_if^i\quad;\quad u(0)=u_0
\end{align}
has a unique solution $u$ in $\widetilde{\psi}^{1-\epsilon}\cH_{p,-2-2\beta_tp}^{1}(\Omega,T)$,
where $\widetilde{\psi}$ is the regularization of $\psi$ in Lemma~\ref{21.05.27.3}.(1).
Moreover, we have
\begin{align*}
	\frac{\big|\widetilde{\psi}^{-1+\epsilon}\big(u(t,\cdot)-u(s,\cdot)\big)\big|_{\beta_x}^{(0)}}{|t-s|^{\beta_t}}\leq N(F+I)
\end{align*}
for all $s,\,t\in[0,T]$ with $t>s$, where $N=N(d,p,\mathrm{C}_1(\psi),\alpha,M_{\alpha},\beta_x,\beta_t,\delta,\epsilon)$ (see Proposition \eqref{220512537} for the definition of $|\cdot|_{\beta_x}^{(0)}$).
%
\end{thm}
\begin{proof}
Take $\alpha_1\in(0,\alpha)$ such that
$$
\epsilon>\frac{\delta}{d+2}+\alpha_1^{-1}\Big(\frac{d}{d+2}\delta+2\beta_t\Big)\,,
$$
and put
$$
p=\frac{d+2}{\delta}\,\,,\,\,\,\,\mu=\alpha_1^{-1}\Big(\,\frac{d}{p}+2\beta_t\,\Big)+(1-\epsilon)\,\,,\,\,\,\,t=\mu^{-1}(1-\epsilon)\,,
$$
so that
$$
\mu\in(0,1-1/p\big)\,\,,\,\,\,\,t\in[0,1]\,\,,\,\,\,\,\mu t=1-\epsilon\,\,,\,\,\,\, \alpha_1\mu(1-t)=\frac{d}{p}+2\beta_t.
$$
By Theorem~\ref{21.11.08.1}, there exists a superharmonic function  $s$ satisfying $s\simeq \rho^{\alpha_1}$.
Put
$$
\Psi=\trho^{\,\alpha_1(1-t)}\widetilde{\psi}^{t}
$$
which is a regularization of the superharmonic function $s^{1-t}\psi^t$(see  Proposition~\ref{21.05.18.1}.(3)).
Note that
\begin{align*}
	\Psi^{\mu}\cH_{p,d-2}^{1}(\Omega,T)\,&=\widetilde{\psi}^{1-\epsilon}\cH_{p,-2-2\beta_tp}^{1}(\Omega,T)\,,\\
	\Psi^{\mu}\bH_{p,d+2p-2}^{-1}(\Omega,T)\,&=\widetilde{\psi}^{1-\epsilon}\bH_{p,-2+2(1-\beta_t)p}^{-1}(\Omega,T)\,,\\
	\quad \Psi^{\mu}B_{p,d}^{1-2/p}(\Omega)\,&=\widetilde{\psi}^{1-\epsilon}B_{p,-2\beta_t p}^{1-2/p}(\Omega)\,.
\end{align*}
Corollary~\ref{21.05.26.3} implies
$$
\big\|f^0+\sum_{i=1}^dD_if^i\big\|_{\Psi^{\mu}\bH_{p,d+2p-2}^{-1}(\Omega,T)}\simeq \big\|f^0+\sum_{i=1}^dD_if^i\big\|_{\widetilde{\psi}^{\,1-\epsilon}\bH_{p,2p-2-2\beta_tp}^{-1}(\Omega,T)}\leq  NF\,,
$$
and Proposition~\ref{220527502}.(3) implies
$$
\big\|u_0\big\|_{\Psi^{\mu}B_{p,d}^{1-2/p}(\Omega)}\simeq \big\|u_0\big\|_{\widetilde{\psi}^{1-\epsilon}B_{p,-2\beta_tp}^{1-2/p}(\Omega)}\lesssim N\big\|u_0\big\|_{\widetilde{\psi}^{1-\epsilon}H_{d/\delta,-2\beta_td/\delta}^{1}(\Omega)}\simeq I\,.
$$
Thererfore Theorem \ref{22.02.18.6} implies that there exists a unique solution $u\in \Psi^{\mu}\cH_{p,d-2}^1(\Omega,T)=\widetilde{\psi}^{1-\epsilon}\cH_{p,-2-2\beta_tp}^1(\Omega,T)$ of equation \eqref{2304201115}.
Moreover, by Propositions~\ref{220512537} and \ref{2204160313} (with $\beta=\beta_t+1/p$), we obtain
\begin{align*}
	\big|\widetilde{\psi}^{-1+\epsilon}\big(u(t,\cdot)-u(s,\cdot)\big)\big|_{\beta_x}^{(0)}\lesssim\,&\|u(t)-u(s)\|_{\widetilde{\psi}^{1-\epsilon}H_{p,0}^{d/p+\beta_x}}\\
	\lesssim\,&|t-s|^{\beta_t}\|u\|_{\widetilde{\psi}^{1-\epsilon}\cH_{p,-2-2\beta_tp}^1}\\
	\simeq\,&|t-s|^{\beta_t}\|u\|_{\Psi^{\mu}\cH_{p,d-2}^1}\\
	\lesssim\,&|t-s|^{\beta_t}\big(F+I\big)\,.
\end{align*}
\end{proof}

\begin{remark}
Assume that $\Omega$ is a bounded domain satisfying the capacity density condition \eqref{22.02.22.1}.
By Remark~\ref{21.07.06.1}, $\Omega$ satisfies $\mathbf{LHMD}(\alpha)$ for some $\alpha>0$.
Let $\epsilon,\,\delta_0\in(0,1]$ and $f^0$, $f^1$, $\ldots$, $f^d$ be measurable functions on $\Omega$ such that
\begin{align*}
	\widetilde{F}_{\delta_0}:=\Big\|\psi^{-1+\epsilon}|\rho^{2-\delta_0}f^0|+\sum_i\psi^{-1+\epsilon}|\rho^{1-\delta_0}f^i|\Big\|_{L_{\infty}(\Omega)}<\infty\,.
\end{align*}
Take small enough $\delta\in (0,\delta_0]$ such that \eqref{21.05.17.2} holds.
Then, since $\Omega$ is bounded, $F$ in \eqref{2304121049} satisfies
$$
F\lesssim_d \mathrm{diam}(\Omega)^{\delta_0}\widetilde{F}_{\delta_0}\,,
$$
where $\mathrm{diam}(\Omega)$ is the diameter of $\Omega$.
By Theorem~\ref{220602323}, the solution $u$ of the equation
$$
\Delta u=f^0+D_if^i
$$
satisfies
\begin{align*}
	|u(x)|+\rho(x)^{1-\delta_0}\sup_{y\in B_{\rho(x)/2}(x)}\frac{|u(x)-u(y)|}{|x-y|^{1-\delta_0}}\lesssim \mathrm{diam}(\Omega)^{\delta_0}F_{\delta_0}\cdot\big(\psi(x)\big)^{1-\epsilon}.
\end{align*}
\end{remark}

\begin{remark}\label{230112512}
The existences of solutions in Theorem~\ref{220602323} (resp. \ref{220602322}) follows from Theorem~\ref{21.09.29.1} (resp. \ref{22.02.18.6}).
Therefore, the global uniqueness theorem, Theorem~\ref{220530526} (resp. \ref{220821002901}), also holds for the solutions in this corollary.

For example, suppose that $f^0\in L_{2,d+2}(\Omega)$ and $f^1,\,\ldots,\,f^d\in L_{2,d}(\Omega)$ under the same assumption as in Theorem \ref{220602323}.
Then Theorem~\ref{21.09.29.1} implies a solution $u\in H_{2,d-2}^1(\Omega)$ of equation \eqref{220621310}.
Furthermore, due to Theorem~\ref{220602323} and Theorem~\ref{220530526}, this $u$ satisfies estimate \eqref{220906540}.
\end{remark}

\vspace{2mm}

\subsection{Domain with thin exterior : Aikawa dimension}\label{0062}\,

The notion of the Aikawa dimension was first introduced by Aikawa \cite{aikawa1991} to observe the quasiadditivity of the Riesz capacity.
We recall the definition of the Aikawa dimension.
For a set $E\subset\bR^d$, the Aikawa dimension of $E$, denoted by $\dim_{\cA}(E)$, is defined by
$$
\dim_{\cA}(E)=\inf\Big\{\beta\geq 0\,:\,\sup_{p\in E,\,r>0}\frac{1}{r^\beta}\int_{B_{r}(p)}\frac{1}{d(y,E)^{d-\beta}}\dd y<\infty\Big\}
$$
with considering $\frac{1}{0}=\infty$.

In this subsection, we assume that $d\geq 3$, and $\Omega$ satisfies
$$
\beta_0:=\dim_{\cA}\Omega^c<d-2\,.
$$

%

\begin{thm}\label{21.10.18.1}
For a constant $\beta<d-2$, if there exists a constant $A_{\beta}$ such that
\begin{align}\label{22.02.08.2}
	\sup_{p\in \Omega^c,\,r>0}\frac{1}{r^\beta}\int_{B_{r}(p)}\frac{1}{d(y,\Omega^c)^{d-\beta}}\dd y\leq A_{\beta}<\infty\,,
\end{align}
then the function
$$
\phi(x):=\int_{\bR^d}|x-y|^{-d+2}\rho(y)^{-d+\beta}\dd y
$$
is a superharmonic function on $\bR^d$ with $-\Delta\phi=N(d)\rho^{-d+\beta}$.
Moreover, we have
\begin{align}\label{21.11.10.4}
	N^{-1}\rho(x)^{-d+2+\beta}\leq \phi(x)\leq N \rho(x)^{-d+2+\beta}\,.
\end{align}
where $N=N(d,\beta,A_\beta)$.
\end{thm}
Before proving Theorem~\ref{21.10.18.1}, we first look at corollaries of this theorem.

\begin{corollary}\label{22.02.24.2}
The Hardy inequality \eqref{hardy} holds on $\Omega$, where $\mathrm{C}_0(\Omega)$ depends only on $d,\beta_0$ and $\{A_{\beta}\}_{\beta> \beta_0}$.
\end{corollary}
\begin{proof}
We first note that this corollary is implied by \cite[Theorem 3]{aikawa1991}, which establishes that if $p\in (1,\infty)$, $\alpha>0$, $\gamma\in\bR$ satisfy
$$
-(p-1)\big(d-\dim_{\cA}(\Omega^c)\big)<\gamma<d-\dim_{\cA}(\Omega^c)-\alpha p\,,
$$
then
$$
\int_{\bR^d}\frac{|u(x)|^p}{\rho(x)^{\alpha p+\gamma}}\dd x\leq N\int_{\bR^d}\frac{\big|(-\Delta)^{\alpha/2}u(x)\big|^p}{\rho(x)^{\gamma}}\dd x\qquad \text{for all}\quad u\in C_c^{\infty}(\bR^d),
$$
where $N=N(d,\{A_{\beta}\}_{\beta>\dim_{\cA}(\Omega^c)},p,\alpha,\gamma)$, and $(-\Delta)^{\alpha/2}u:=\cF^{-1}\big(\,|\cdot|^{\alpha}\cF(u)\big)$.
Actually, \cite[Theorem 3]{aikawa1991} is more general than this corollary, and the proof is based on Muckenhoupt's $A_p$ weight theory.

Considering only Corollary~\ref{22.02.24.2}, this result can be proved differently.
We first note the following inequality provided in \cite[Lemma 3.5.1]{BEL}: if $f\in C_c^{\infty}(\bR^d)$ and $s>0$ is a smooth superharmonic function on a neighborhood of $\text{supp}(f)$, then
\begin{align}\label{220829005500}
	\int_{\bR^d}\frac{-\Delta s}{s}|f|^2 \dd x\leq \int_{\bR^d} |\nabla f|^2\dd x \quad\text{for all}\quad f\in C_c^{\infty}(\bR^d)
\end{align}
(the proof of this inequality is based on integrating $\big|\nabla f -(f/s)\nabla s\big|^2$ and performing integration by parts).
Take any $\beta\in(\beta_0,d-2)$, and let $\phi$ be the function in Theorem~\ref{21.10.18.1}, so that 
\begin{align}\label{220913400}
	-\Delta \phi\geq N_1\rho^{-2}\phi>0
\end{align}
where $N_1=N(d,\beta,A_{\beta})>0$.
Fix $f\in C_c^{\infty}(\Omega)$.
For $0<\epsilon<d\big(\text{supp}(f),\partial\Omega\big)$, let $\phi^{(\epsilon)}$ be the mollification of $\phi$ in \eqref{21.04.23.1}.
Observe that
$$
-\Delta\big(\phi^{(\epsilon)}\big)\geq N_1^{-1}\big(\rho^{-2}\phi\big)^{(\epsilon)}\geq N_1^{-1}(\rho+\epsilon)^{-2}\phi^{(\epsilon)}\quad\text{on}\quad \bR^d\,,
$$
where $N_1$ is in \eqref{220913400}.
By appling the monotone convergence theorem to \eqref{220829005500} with $s=\phi^{(\epsilon)}$(see Lemma~\ref{21.04.23.5}), we obtain \eqref{hardy} with $\mathrm{C}_0(\Omega)=N_1$.
\end{proof}



\begin{corollary}\label{22.02.19.300}
For any $p\in(1,\infty)$ and $\theta\in\bR$ satisfying
$$
\beta_0<\theta<(d-2-\beta_0)p+\beta_0\,,
$$
Statement~\ref{22.02.19.1} $(\Omega,p,\theta)$ holds.
In addition, $N_1$ in \eqref{2205241155} and $N_2$ in \eqref{2205241156} depend only on $d$, $p$, $\gamma$, $\theta$, $\beta_0$, $\{\cA_{\beta}\}_{\beta>\beta_0}$.
\end{corollary}
\begin{proof}[Proof of Corollary \ref{22.02.19.300}]
Choose $\beta\in(\beta_0,d-2)$ satisfying
$$
\beta<\theta<(d-2-\beta)p+\beta\,.
$$
By Theorem~\ref{21.10.18.1}, there exists a superharmonic function $\phi$ satisfying $\phi\sim \rho^{-d+2+\beta}$, and therefore 
by Remark~\ref{220617557}, the proof is completed.
\end{proof}

\begin{proof}[Proof of Theorem~\ref{21.10.18.1}]
We first prove \eqref{21.11.10.4}.
For a fixed $x\in \bR^d$, there exists $p_x\in \partial\Omega$ such that $|x-p_x|=\rho(x)=:\rho_x$.
Put
$$
E_0=B(x,2^{-1}\rho_x) \quad\text{and}\quad E_j=B(x,2^{j-1}\rho_x)\setminus B(x,2^{j-2}\rho_x)
$$
for $j=1,\,2,\,\ldots$, and put
$$
I_j=\int_{E_j}|x-y|^{-d+2}\rho(y)^{-d+\beta}dy
$$
for $j=0,\,1,\,2,\,\ldots$, so that $\phi(x)=\sum\limits_{j\in\bN_0}I_j$.
If $y\in E_0$ then $\frac{1}{2}\rho_x\leq \rho(y)\leq 2 \rho_x$, which implies 
$$
I_0\simeq_{d,\beta} \rho_x^{-d+\beta}\int_{B(x,\rho_x/2)}|x-y|^{-d+2}\dd y\simeq_d \rho_x^{-d+2+\beta}\,.
$$
For $I_j$, $j\geq 1$, observe that
\begin{align*}
	0\leq \sum_{j=1}^{\infty}I_j\,&\leq\sum_{j=1}^{\infty}(2^{j-2}\rho_x)^{-d+2}\int_{B(x,2^{j-1}\rho_x)}\rho(y)^{-d+\beta}\dd y\\
	&\leq \sum_{j=1}^{\infty}2^{-j(d-2)}\rho_x^{-d+2}\int_{B(p_x,2^j\rho_x)}\rho(y)^{-d+\beta}\dd y\\
	&\leq N\Big(\sum_{j=1}^{\infty}2^{-j(d-2-\beta)}\Big)\rho_x^{-d+2+\beta}\,.
\end{align*}
where $N=N(d,\beta,A_{\beta})$.
Since the summation in the last term is finit, \eqref{21.11.10.4} is proved.

To prove that $-\Delta \phi=N(d)\phi$ in the sense of distribution, recall that 
$$
-\Delta_x\Big(|x-y|^{-d+2}\Big)=N(d)\,\delta_0(x-y)
$$
in the sense of distribution, where $\delta_0(\cdot)$ is the Dirac delta distribution.
Due to \eqref{22.02.08.2} and $\phi\simeq \rho^{-d+2+\beta}$, $\phi$ is locally integrable.
Therefore, by the Fubini theorem, for any $\zeta\in C_c^{\infty}(\bR^d)$ we obtain
\begin{align*}
	\int_{\bR^d}\phi(x)\big(-\Delta\zeta)(x)\dd x\,&=\int_{\bR^d}\Big(\int_{\bR^d}|x-y|^{-d+2}(-\Delta\zeta)(x)\dd x\Big)\rho(y)^{-d+\beta}\dd y\\
	&=N(d) \int_{\bR^d}\zeta(y)\rho(y)^{-d+\beta}\dd y\,.
\end{align*}
\end{proof}

\vspace{2mm}

\mysection{Application II - Various domains with fat exterior}\label{app2.}

In this section, we present results for the exterior cone condition, convex domains, the exterior Reifenberg condition, and Lipschitz cones.
These domains and conditions imply the fat exterior condition.

Throughout this section, we consider a domain $\Omega\subsetneq \bR^d$, $d\geq 2$.

\vspace{2mm}

\subsection{Exterior cone condition and exterior line segment condition}\label{0071}

\begin{defn}[Exterior cone condition]
	For $\delta\in [0,\frac{\pi}2)$ and $R\in (0,\infty]$, a domain $\Omega\subset\bR^d$ is said to satisfy the \textit{exterior $(\delta,R)$-cone condition} if for every $p\in\partial\Omega$, there exists a unit vector $e_p\in\bR^d$ such that
	\begin{align}\label{21.08.03.4}
		\{x\in B_R(p)\,:\,(x-p)\cdot e_p\geq |x-p|\cos\delta\}\subset \Omega^c\,. 
	\end{align}
\end{defn}
\vspace{1mm}

Note that the left hand side of \eqref{21.08.03.4} is obtained by applying a translation and a rotation to the set
$$
\{x=(x_1,\ldots,x_d)\in B_R(0)\,:\,x_1\geq|x|\cos\delta\}\,.
$$

The exterior $(0,R)$-cone condition can be called the \textit{exterior $R$-line segment condition}, due to
$$
\{x\in B_R(p)\,:\,(x-p)\cdot e_p\geq |x-p|\}=\{p+re_p\,:\,r\in[0,R)\}\,.
$$
For examples of the exterior cone condition and exterior line segment condition, see Figure~\ref{230212745} below.
\begin{figure}[h]\centering
	\begin{tikzpicture}[>=Latex]
		\begin{scope}[shift={(-4.5,0)}]
			\clip (-1.5,-1.5) rectangle (1.5,1.5);
			\begin{scope}[name prefix = p-]
				\coordinate (A) at (0,0.7);
				\coordinate (B) at (-1.4,0.35);
				\coordinate (C) at (0,-1.4);
				\coordinate (D) at (1.4,0.35);
			\end{scope}
			
			\begin{scope}
				\draw[fill=gray!20]
				(p-A) .. controls +(-0.5,{sqrt(3)*0.5}) and +(0,0.7)..
				(p-B) .. controls +(0,-0.7) and +(-0.5,0.5)..
				(p-C) .. controls +(0.5,0.5) and +(0,-0.7) ..
				(p-D) .. controls +(0,0.7) and +(0.5,{sqrt(3)*0.5}) .. (p-A);
			\end{scope}
			
		\end{scope}
		\node[align=center] at (-4.5,-2.34) {A. Lipschitz boundary\\condition};

		\begin{scope}
			\clip (-1.5,-1.7) rectangle (1.5,1.5);
			\begin{scope}[name prefix = p-]
				\coordinate (A) at (0,0.7);
				\coordinate (B) at (-1.4,0.35);
				\coordinate (C) at (0,-1.4);
				\coordinate (D) at (1.4,0.35);
			\end{scope}
			
			\begin{scope}
				\draw[fill=gray!20]
				(p-A) .. controls +(-0.5,{sqrt(3)*0.5}) and +(0,0.7)..
				(p-B) .. controls +(0,-1) and +(0,1.4)..
				(p-C) .. controls +(0,1.4) and +(0,-1) ..
				(p-D) .. controls +(0,0.7) and +(0.5,{sqrt(3)*0.5}) .. (p-A);
			\end{scope}
			\draw[dashed] (p-C) circle (0.25);
			
		\end{scope}
		\node[align=center] at (0,-2.4) {B. Exterior\\$(\frac{\pi}{3},\infty)$-cone condition};
		\node[align=center] at (0, -3.5) {(does not satisfy Lipschitz\\boundary conddition)};
		
		\begin{scope}[shift={(4.5,0)}]
			\clip (-1.5,-1.7) rectangle (1.5,1.5);
			\begin{scope}[name prefix = p-]
				\coordinate (A) at (0,0.35);
				\coordinate (B) at (-1.4,0.45);
				\coordinate (C) at (0,-1.4);
				\coordinate (D) at (1.4,0.45);
			\end{scope}
			
			\begin{scope}
				\draw[fill=gray!20]
				(p-A) ..controls +(0,1.1) and +(0,0.9) ..
				(p-B) .. controls +(0,-1.1) and +(0,1.4)..
				(p-C) .. controls +(0,1.4) and +(0,-1.1) ..
				(p-D) .. controls +(0,0.9) and +(0,1.1) .. (p-A);
			\end{scope}
			\draw[dashed] (p-A) circle (0.25);
			\draw[dashed] (p-C) circle (0.25);
		\end{scope}
		\node[align=center] at (4.5,-2.37) {C. Exterior $\infty$-line\\segment condition};
		\node[align=center] at (4.5, -3.5) {(does not satisfy $(\delta,R)$-cone \\condition, $\forall\,\,\,\delta,\,R>0$)};
		
		%
		%
		%
		%
		
	\end{tikzpicture}
	\caption{Examples for exterior cone condition}\label{230212745}
\end{figure}
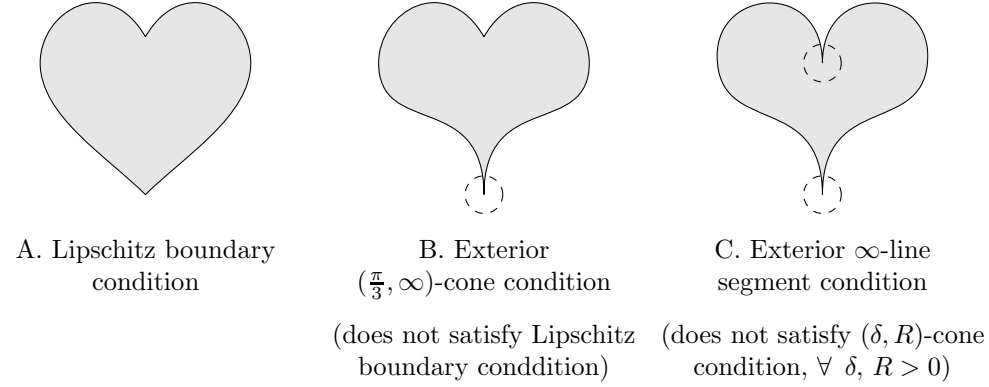

\begin{example}\label{220816845}
	Suppose that  there exists $K,\,R\in(0,\infty]$ such that for any $p\in\partial\Omega$, there exists a function $f_p\in C(\bR^{d-1})$ such that
	\begin{align}
		|f_p(y')-f_p(z')|\leq K|y'-z'|\quad\text{for all}\quad y',\,z'\in\bR^{d-1}\,\,\,,\,\,\,\,\,\text{and} \qquad\,\,\, \label{209021056}\\
		\Omega\cap B_R(p)=\big\{y=(y',y_d)\in\bR^{d-1}\times \bR\,:\,y_d>f_p(y')\quad\text{and}\quad |y|<R\big\}\label{209021055}
	\end{align}
	where $(y',y_d)=(y_1,\cdots,y_d)$ in \eqref{209021055} is an orthonormal coordinate system centered at $p$.
	Then $\Omega$ satisfies the exterior $(\delta,R)$-cone condition, where $\delta=\textmd{arctan}(1/K)\in[0,\pi/2)$.
	
	Moreover, if $f\in C(\bR^{d-1})$ satisfies \eqref{209021056} with $f_p=f$, then a domain 
	$$
	\big\{(x',x_n)\in\bR^{d-1}\times \bR\,:\,x_n>f(x')\big\}
	$$
	satisfies the exterior $(\delta,\infty)$-cone condition, where $\delta=\arctan(1/K)$.
\end{example}

For $\delta\in(0,\pi)$, 
let 
$$
E_{\delta}:=\{\sigma\in \partial B_1(0)\,:\,\sigma_1>-\cos\delta\}
$$
(see Figure \ref{2304201145} below).
\begin{figure}[h]
	\begin{tikzpicture}[> = Latex]

		\begin{scope}
			\begin{scope}
				\clip (-0.95,-1.6)--(1.6,-1.6)--(1.6,1.6)--(-0.95,1.6);
				\draw[line width=0.6,fill=gray!40] (0,0) circle (1.5);
			\end{scope}
			\draw[line width=0.6,fill=gray!15] (-0.9*40/41,0) circle [x radius=1.2*9/41, y radius=1.2];
			
			\draw[fill=black] (0,0) circle (0.03);
			\draw[->,line width=0.4] (-2,0)--(2,0);
		\end{scope}

		\begin{scope}
			\draw[dashed] (0,0) -- (-1.2,1.6);
			\draw (-0.4,0.2) node  {$\delta$};
			\draw (-0.3,0) arc(180:180-atan(4/3):0.3);
		\end{scope}

		%
		%
		%
		%
		%
		%
		
	\end{tikzpicture}
	\caption{$E_{\delta}$}\label{2304201145}
\end{figure}
\noindent
We denote the first eigenvalue of the Dirichlet spherical Laplacian on $E_{\delta}$ as $\Lambda_{\delta}$  (see Proposition~\ref{230413525}.(1)).
Alternatively, we can express $\Lambda_{\delta}$ as follows: 
\begin{align}\label{230212803}
	\Lambda_{\delta}=\inf_{f\in F_{\pi-\delta}}\frac{\int_0^{\pi-\delta}|f'(\theta)|^2(\sin\theta)^{d-2}  d\theta}{\int_0^{\pi-\delta}|f(\theta)|^2(\sin\theta)^{d-2} d\theta}\,,
\end{align}
where $F_{\pi-\delta}$ is the set of all non-zero Lipschitz continuous function $f:[0,\pi-\delta]\rightarrow \bR$ such that $f(\pi-\delta)=0$ (see \cite{FH}).

We also define
$$
\lambda_{\delta}:=-\frac{d-2}{2}+\sqrt{\Big(\frac{d-2}{2}\Big)^2+\Lambda_{\delta}}\,,
$$
and when $d=2$, we define $\lambda_0=\frac{1}{2}$.

The following quantitative information of $\Lambda_{\delta}$ and $\lambda_{\delta}$ is provided in \cite{BCG}:
\begin{prop}\label{230212757}
	Let $\delta\in(0,\pi)$.
	
	\begin{enumerate}
		\item If $d=2$ then $\lambda_{\delta}=\sqrt{\Lambda_{\delta}}=\frac{\pi}{2(\pi-\delta)}>\frac{1}{2}$.
		\item If $d=4$ then $\lambda_{\delta}=-1+\sqrt{1+\Lambda_{\delta}}=\frac{\delta}{\pi-\delta}$.
		\item For $d\geq 3$,
		$$
		\Lambda_{\delta}\geq \left(\int_0^{\pi-\delta}(\sin t)^{-d+2}\Big(\int_0^t(\sin r)^{d-2}\dd r\Big)\,\dd t\right)^{-1}.
		$$
		Moreover, $\Lambda_{\pi/2}= d-1$, $\lim\limits_{\delta\searrow 0}\Lambda_{\delta}=0$, and $\lim\limits_{\delta\nearrow\pi}\Lambda_{\delta}=+\infty$.
	\end{enumerate}
\end{prop}

Note that when $d=3$, $\Lambda_{\delta}\geq \frac{1}{2}|\log\,\sin\frac{\delta}{2}|^{-1}$.

\begin{remark}\label{220716637}
	For each $\delta>0$, there is a function $F\in C\big(\overline{E_{\delta}}\big)\cap C^{\infty}(E_{\delta})$ such that
	$$
	F>0\quad\text{and}\quad   \Delta_\bS F+\Lambda_{\delta}F=0\quad\text{on}\,\,\,\,E_{\delta}\,\,\,,\,\,\,\text{and}\,\,\,  F|_{\overline{E_{\delta}}\setminus E_{\delta}}\equiv 0
	$$
	(see, \textit{e.g.}, \cite[Section 5]{FH}), where $C^\infty(E_{\delta})$ and $\Delta_\bS$ are introduced in Subsection~\ref{0074}.
	It follows from \eqref{220716632} that the function
	$$
	v_{\delta}(x):=|x|^{\lambda_{\delta}}F(x/|x|)
	$$
	is harmonic on
	$$
	U_{\delta}:=\big\{y\in B_1(0)\,:\,y_1>-|y|\cos\delta \big\}\,,
	$$
	and vanishes on $\partial U_{\delta}\cap B_1(0)$.
\end{remark}
\vspace{1mm}

With help of $\lambda_\delta$, we state main results of this subsection.
\begin{thm}\label{22.02.18.3}
	For
	\begin{align*}
		\delta\in[0,\pi/2)\quad\text{if}\,\,\,\, d=2\,\,,\,\,\,\,\text{and}\,\,\quad\delta\in (0,\pi/2)\quad\text{if}\,\,\,\, d\geq 3\,,
	\end{align*}
	let $\Omega\subset \bR^d$ satisfy the exterior $(\delta,R)$-cone condition, where
	$$
	R\in (0,\infty]\quad\text{if}\,\,\,\,\text{$\Omega$\, is\, bounded}\,\,,\,\,\,\, \text{and}\,\,\quad R=\infty\quad\text{if}\,\,\,\,\text{$\Omega$\, is\, unbounded.}
	$$
	Then $\Omega$ satisfies $\mathbf{LHMD}(\lambda_{\delta})$, where $M_{\lambda_{\delta}}$ in \eqref{21.08.03.1} depends only on $d$, $\delta$ and $\mathrm{diam}(\Omega)/R$.
	If $\Omega$ is unbounded (and $R=\infty$), then for we can drop the dependence of $M_{\lambda_{\delta}}$ on $\mathrm{diam}(\Omega)/R$ .
\end{thm}

Before the proof of Theorem~\ref{22.02.18.3}, we state a corollary which follows from Theorem~\ref{22.02.18.3} and Corollary~\ref{22.02.19.3}.

\begin{corollary}\label{221026914}
	Let $p\in(1,\infty)$.
	Under the same assumption of Theorem \ref{22.02.18.3}, if $\theta\in\bR$ satisfies
	$$
	-2-(p-1)\lambda_{\delta}<\theta-d<-2+\lambda_{\delta}\,,
	$$
	then Statement~\ref{22.02.19.1} $(\Omega,p,\theta)$ holds.
	In addition, $N_1$ in \eqref{2205241155} and $N_2$ in \eqref{2205241156} depend only on $d,\,p,\,\theta,\,\gamma,\,\delta,\,\mathrm{diam}(\Omega)/R$, and if $\Omega$ is unbounded (and $R=\infty$) then we can drop the dependence of $N_1$ and $N_2$ on $\mathrm{diam}(\Omega)/R$.
\end{corollary}

To prove Theorem~\ref{22.02.18.3} we use the boundary Harnack principle on Lipschitz domains.

\begin{prop}[see Theorem 1 of \cite{Wu}]\label{21.10.18.4}
	Let $D$ be a bounded Lipschitz domain, $A$ be a relatively open subset of $\partial D$, and $U$ be a subdomain of $D$ with $\partial U\cap \partial D\subset A$ (see Figure~\ref{2304131241} below).
	Then there exists $N=N(D,A,U)>0$ such that if $u, v$ are positive harmonic funtion on $D$, and vanish on $E$, then
	$$
	\frac{u(x)}{v(x)}\leq N\frac{u(x_0)}{v(x_0)}\quad\text{for any}\quad x_0,\,x\in U\,.
	$$
\end{prop}
\begin{figure}[h]\centering
	\begin{tikzpicture}[> = Latex]
		\begin{scope}[scale=1]
			\begin{scope}
				\draw[fill=gray!10] (-2,0.3)..controls +(-0.3,-0.3) and +(-0,0.3).. (-3,-0.5).. controls +(0,-0.5) and +(-0.3,-0.1) .. (-2,-1.2)..controls +(0.3,0.1) and +(-0.3,0) ..(-1.5,-1.2) .. controls +(0.2,0) and +(-0.6,-0.36).. (0.5,-1.3)..controls +(0.2,-0.12) and +(-0.5,-0.2) .. (1.5,-1.2).. controls +(0.25,0.1) and +(-0.2,-0.1) .. (2,-1.2).. controls +(0.2,-0.1) and +(0,-0.5) .. (3,-0.5)..controls +(0,0.5) and +(0.2,-0.1)..(2,0.3) .. controls +(-0.8,0.4) and +(0.8,0).. (-2,0.3);
				%
			\end{scope}
			
			\begin{scope}
				\fill[gray!20] (-1,-0.5)..controls +(-0.3,-0.15) and +(-0,0.3)..(-1.5,-1.2) .. controls +(0.2,0) and +(-0.6,-0.36).. (0.5,-1.3)..controls +(0.2,-0.12) and +(-0.5,-0.2) .. (1.5,-1.2)..  controls +(0,0.3) and +(0.1,-0.1) .. (1,-0.5)..controls +(-0.2,0.2) and +(0.3,0.15) .. (-1,-0.5);
				\draw[dashed] (1.5,-1.2).. controls +(0,0.3) and +(0.1,-0.1)..(1,-0.5)..controls +(-0.2,0.2) and +(0.3,0.15) .. (-1,-0.5).. controls +(-0.3,-0.15) and +(-0,0.3)..(-1.5,-1.2);
			\end{scope}
			
			\begin{scope}
				\draw[line width=0.7,Parenthesis-Parenthesis] (-2,-1.2)..controls +(0.3,0.1) and +(-0.3,0) ..(-1.5,-1.2) .. controls +(0.2,0) and +(-0.6,-0.36).. (0.5,-1.3)..controls +(0.2,-0.12) and +(-0.5,-0.2) .. (1.5,-1.2).. controls +(0.25,0.1) and +(-0.2,-0.1) .. (2,-1.2);
			\end{scope}
			
		\end{scope}
		\begin{scope}
			\draw (-1,0) node {$D$};
			\draw (0,-1) node {$U$};
			\draw (1.5,-1.6) node {$A$};
		\end{scope}
	\end{tikzpicture}
	\caption{$D$, $A$, and $U$ in Proposition~\ref{21.10.18.4}}\label{2304131241}
\end{figure}
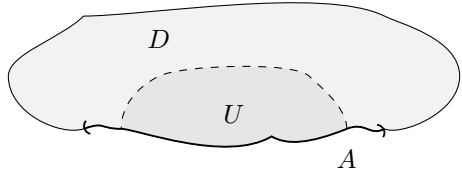

The boundary Harnack principle has also been established for a more general class of domains, so-called  non-tangentially accessible domains, by Jerison and Kenig \cite{Kenig2}.  

\begin{proof}[Proof of Theorem~\ref{22.02.18.3}]
	
	By Remark~\ref{22.02.18.4}, it is sufficient to prove that there exists a constant $M>0$ such that 
	\begin{align*}
		w(x,p,r)\leq M\Big(\frac{|x-p|}{r}\Big)^{\lambda_{\delta}}\quad\text{for all}\,\,\,x\in\Omega\cap B(p,r)
	\end{align*}
	whenever $p\in\partial\Omega$ and $r\in(0,R)$.
	For any $p\in\partial\Omega$, there exists a unit vector $e_p\in\bR^d$ such that
	$$
	C_{p}:=\{y\in B_R(p)\,:\,(y-p)\cdot e_p\geq |y-p|\cos\delta\}\subset \Omega^c\,.
	$$
	Since
	$$
	\Omega\cap B_r(p)\subset B_r(p)\setminus C_{p}\quad\text{and}\quad \Omega\cap \partial B_r(p)\subset \partial B_r(p)\setminus C_{p}\,,
	$$
	we have
	\begin{align}\label{2301111002}
		w(x,p,r)\leq w\big(x,\,B_r(p)\setminus C_{p}\,,\,\partial B_r(p)\setminus C_{p}\,\big),
	\end{align}
	by directly applying the definition of $w(\cdot,p,r)$ (see \eqref{2301051056}).
	Consider a rotation map $T$ such that $T(e_p)=(-1,0,\ldots,0)$, and put $T_0(x)=r^{-1}T(x-p)$. 
	Then
	\begin{align}\label{2301111003}
		w\big(x,\,B_r(p)\setminus C_{p}\,,\,\partial B_r(p)\setminus C_{p}\,\big)=w\big(T_0(x),U_{\delta},E_{\delta}\big),
	\end{align}
	where
	$$
	U_{\delta}=\{y\in B_1(0)\,:\,y_1> -|y|\cos\delta\}\,\,\,\,\text{and}\,\,\,\, E_{\delta}=\{y\in \partial B_1(0)\,:\,y_1>- |y|\cos\delta\}\,.
	$$
	Due to \eqref{2301111002} and \eqref{2301111003}, it is sufficient to show that there exists a constant $M>0$ depending only on $d$ and $\delta$ such that 
	\begin{align}\label{2304131242}
		w(x,U_{\delta},E_{\delta})\leq M|x|^{\lambda_{\delta}}\quad\text{for all}\quad x\in U_{\delta}\,,
	\end{align}

	- \textbf{Case 1.} $\delta>0$.
	
	Put $v(x)=|x|^{\lambda_{\delta}}F_0(x/|x|)$
	where $F_0$ is the first Dirichlet eigenfunction of spherical laplacian on $E_{\delta}\subset \partial B_1(0)$, with $\sup_{E_{\delta}} F_0=1$ (see Remark~\ref{220716637}).
	Note that $U_{\delta}$ is a bounded Lipschitz domain, and $w(\,\cdot\,,U_{\delta},E_{\delta})$ and $v$ are positive harmonic functions on $U_{\delta}$, and vanish on $\partial U_{\delta}\cap B_{1}$.
	By applying Proposition~\ref{21.10.18.4} for $D=U_{\delta}$, $A=(\partial U_{\delta})\cap B_{1}(0)$, and $U=U_{\delta}\cap B_{1/2}(0)$, we obtain that there exists a constant $N_0=N_0(d,\delta)>0$ such that 
	$$
	w(x,U_{\delta},E_{\delta})\leq N_0 v(x)\leq N_0|x|^{\lambda_{\delta}}\quad\text{for}\quad x\in U_{\delta}\cap B_{1/2}(0).
	$$
	Therefore \eqref{2304131242} is obtained, where $M_0=N_0\vee 2^{\lambda_0}$.
	
	- \textbf{Case 2.} $\delta=0$ and $d=2$.
	
	We consider $\bR^2$ as $\bC$.
	Note 
	$$
	U_0=\{re^{i\theta}\,:\,r\in(0,1),\,\theta\in(-\pi,\pi)\}\,\,,\,\,\,\, E_0=\{e^{i\theta}\,:\,\theta\in(-\pi,\pi)\}\,.
	$$
	Observe that a function $s$ is a classical superharmonic function on $U_0$ if and only if $s(z^2)$ is a classical superharmonic function on $B_{1}(0)\cap \bR_+^2$ (use Proposition~\ref{21.04.23.3}).
	It is implied by the definition of PWB solutions (see \eqref{2301051056}) that
	$$
	w(z^2,U_0,E_0)=w\big(z,B_1(0)\cap \bR_+^2,\partial B_1(0)\cap \bR_+^2\big)\,.
	$$
	Since the map $z=(z_1,z_2)\mapsto z_1$ is harmonic on $B_1(0)\cap \bR^2_+$, by Proposition~\ref{21.10.18.4} with $D=B_1(0)\cap \bR^2_+$, we obtain that
	\begin{align}\label{220906722}
		w\big(z,B_1(0)\cap \bR_+^2,\big(\partial B_1(0)\big)\cap \bR_+^2\big)\leq N |z|\quad\text{for}\,\,\,z\in B_{1/2}(0)\cap \bR_+^2\,,
	\end{align}
	where $N$ depends on nothing.
	Therefore the proof is completed.
\end{proof}

\vspace{2mm}

\subsection{Convex domains}\label{convex}\,

We recall the definition of convex set.
A set $E\subset \bR^d$ is said to be \textit{convex} if $(1-t)x+ty\in E$ for any $x,\,y\in E$ and $t\in [0,1]$.

\begin{remark}\label{2209071233}
	We claim that for an open set $\Omega\subset \bR^d$, $\Omega$ is convex if and only if for any $p\in\partial\Omega$, there exists a unit vector $e_p\in \bR^d$ such that
	\begin{align}\label{21.08.17.1}
		\Omega\subset \{x\,:\,(x-p)\cdot e_p<0\}=:U_p
	\end{align}
(see Figure \ref{2304201156} below).
	\begin{figure}[h]\centering
		\begin{tikzpicture}[>=Latex]
			\begin{scope}
				\clip (3,-2) -- (3,1.8) -- (-3,1.8) -- (-3,-2);

				\begin{scope}
					\draw[fill=black] (2,-0.3) circle (0.05);
				\end{scope}

				\begin{scope}
					\fill[gray!30] (2,-0.3) -- (1, 1) -- (-1.5,0.7) arc(120:240:1) -- (-1,-1.2) -- (1,-1) -- (2,-0.3); 
				\end{scope}

				\begin{scope}
					\draw (2,-0.3) -- (1, 1) -- (-1.5,0.7) arc(120:240:1) -- (-1,-1.2) -- (1,-1) -- (2,-0.3); 
					\draw[line width=0.8] ($(2,-0.3)+3*(0.195,0.65)$) -- ($(2,-0.3)+{-3}*(0.24,0.8)$);
				\end{scope}
				
			\end{scope}
			
			\begin{scope}
				\draw[->] (2,-0.3) -- +(${0.5}*(1,-0.3)$);
				\draw (2.6,-0.75) node {$e_p$};
				\draw (2.8,1.1) node {$F_p$};
				\draw[white](3,0) circle(0.1);
			\end{scope}
			
		\end{tikzpicture}
		\caption{$e_p$ in \eqref{21.08.17.1}, and $F_p$ in \eqref{230114328}}\label{2304201156}
	\end{figure}

	Let $\Omega$ be a convex domain, and fix $p\in \partial\Omega$.
	Since the set $\{p\}$ is convex and disjoint from $\Omega$, the hyperplane separation theorem (see, \textit{e.g.}, \cite[Theorem 3.4.(a)]{rudin}) implies that there exists a unit vector $e_p\in\bR^d$ such that \eqref{21.08.17.1} holds.
	
	Conversely, suppose that for every $p\in\partial\Omega$, there exist a unit vector $e_p$ satisfying \eqref{21.08.17.1}.
	Then $E:=\bigcap_{p\in\partial\Omega}U_p$ is convex, $\Omega\subset E$, and $E\cap \partial\Omega=\emptyset$. These imply $E=\Omega$. Therefore our claim is proved.
\end{remark}

\begin{remark}\label{2209071100}
	The argument to obtain \eqref{220906722} also implies that for any $d\in\bN$,
	$$
	w\big(x,B_1(0)\cap \bR_+^d,(\partial B_1(0))\cap \bR_+^d\big)\leq N(d) |x|\quad \text{for all}\,\,\,x\in B_1(0)\cap \bR^d_+\,.
	$$
	By translation, dilation and rotation, we obtain that for a convex domain $\Omega$ and $p\in\partial\Omega$, 
	$$
	w(x,p,r)\leq w(x,B_r(p)\cap U_p,\big(\partial B_r(p)\big)\cap U_p)\leq N(d)\frac{|x-p|}{r}
	$$
	for all $x\in B_r(p)\cap \Omega$, where $U_p$ is the set on the right-hand side of \eqref{21.08.17.1}.
	Consequently, $\Omega$ satisfies $\mathbf{LHMD}(1)$, where $M_1$ in \eqref{21.08.03.1} depends only on $d$.
	
	This result also implies that the Hardy inequality \eqref{hardy} holds on $\Omega$, where $\mathrm{C}_0(\Omega)$ depends only on $d$ (see Remark~\ref{21.07.06.1}).
	However, it is worth noting that Marcus, Mizel and Pinchover \cite[Theorem 11]{MMP} provided that for a convex domain $\Omega$, \eqref{hardy} holds where $\mathrm{C}_0(\Omega)=4$, and $\mathrm{C}_0(\Omega)$ cannot be chosen less than 4.
\end{remark}
\vspace{1mm}

Krylov \cite{Krylov1999-1} provided results for the Poisson equation and parabolic equations in $\bR_+^d$.
In this subsection, we extend this result for convex domains; see Corollary~\ref{2208131026}.
Recall the definitions of $\mathrm{M}(\nu_1,\nu_2)$ and $\cM_T(\nu_1,\nu_2)$ in the front of Section~\ref{0050}.

\begin{thm}\label{22.02.19.5} Let $\Omega$ be a convex domain.
	For any $(\alpha^{ij})_{d\times d}\in\bigcup\limits_{0<\nu\leq 1}\mathrm{M}(\nu^2,1)$,
	$$
	\sum_{i,j=1}^d\alpha^{ij}D_{ij}\rho\leq 0
	$$
	in the sense of distribution.
\end{thm}
We temporarily assume Theorem~\ref{22.02.19.5} and prove Corollary~\ref{2208131026}.
\begin{corollary}\label{2208131026}
	Let $\Omega\subset \bR^d$ be convex, $p\in(1,\infty)$, $\gamma\in\bR$, and $\theta\in\bR$ with
	$$
	-p-1<\theta-d<-1\,.
	$$
	\begin{enumerate}
		\item For any $\lambda \geq 0$ and $f\in H_{p,\theta+2p}^{\gamma}(\Omega)$, 
		the equation
		$$
		\Delta u-\lambda u=f\,.
		$$
		has a unique solution $u$ in $H^{\gamma+2}_{p,\theta}(\Omega)$.
		Moreover, we have
		\begin{align}\label{220920553}
			\|u\|_{H^{\gamma+2}_{p,\theta}(\Omega)}+\lambda \|u\|_{H^{\gamma}_{p,\theta+2p}(\Omega)}\leq N_1\|f\|_{H_{p,\theta+2p}^{\gamma}(\Omega)}\,,
		\end{align}
		where $N_1=N(d,p,\gamma,\theta)$.
		
		\item Let $T\in(0,\infty]$ and $\cL\in\cM_T(\nu,\nu^{-1})$ for some $\nu\in(0,1]$.
		For any $u_0\in B^{\gamma+2-2/p}_{p,\theta+2}(\Omega)$ and $f\in \bH_{p,\theta+2p}^{\gamma}(\Omega,T)$, the equation
		$$
		\partial_t u=a^{ij}D_{ij} u+f\quad\text{on }\Omega\times (0,T]\quad;\quad u(0,\cdot)=u_0
		$$
		has a unique solution $u$ in $\cH^{\gamma+2}_{p,\theta}(\Omega)$.
		Moreover, we have
		\begin{align}\label{220920554}
			\|u\|_{\cH^{\gamma+2}_{p,\theta}(\Omega)}\leq N_2\big(\|u_0\|_{B^{\gamma+2-2/p}_{p,\theta+2}(\Omega)}+\|f\|_{H_{p,\theta+2p}^{\gamma}(\Omega)}\big)\,,
		\end{align}
		where $N_2=N(d,p,\theta,\gamma,\nu)$.
	\end{enumerate}
	In particular, $\Omega$ is not necessarily bounded, and $N_1$ and $N_2$ are independent of $\Omega$.
\end{corollary}
\begin{proof}[Proof of Corollary \ref{2208131026}]
	Since $\Omega$ is convex, \eqref{hardy} holds on $\Omega$ where $\mathrm{C}_0(\Omega)=4$.
	Put $\Psi=\trho$ which is the regularization of $\rho$ in Lemma~\ref{21.05.27.3}.(1) so that constants $\mathrm{C}_2(\trho)$ and $\mathrm{C}_3(\rho,\trho)$(in Definition~\ref{21.10.14.1}) can be chosen to depend only on $d$.
	It follows from Proposition~\ref{05.11.1}.(2) that for any $\mu\in (-1/p,1-1/p)$, $\mu\in I(\rho,\nu^2,p)$, and the constant $\mathrm{C}_4$ in \eqref{21.07.12.1} can be chosen to depend only on $\mu$, $p$ and $\nu$.
	Putting
	$$
	\mu=-\frac{\theta-d+2}{p}\in\Big(-\frac{1}{p},1-\frac{1}{p}\,\Big)
	$$
	and applying Theorems~\ref{21.09.29.1} and \ref{22.02.18.6}, we finish the proof.
\end{proof}

\begin{proof}[Proof of Theorem~\ref{22.02.19.5}]
	For $p\in\partial\Omega$, put $W_p(x)=(p-x)\cdot e_p$, where $e_p$ is a unit vector satisfying \eqref{21.08.17.1}.
	We first claim that
	\begin{align}\label{22.02.17.10}
		\rho(x)=\inf_{p\in\partial\Omega}W_p(x)\quad\text{for all}\,\,\,x\in\Omega\,.
	\end{align}
	For a fixed $x\in\Omega$, we have
	$$
	\inf_{p\in\partial\Omega}W_p(x)= \inf_{p\in\partial\Omega}d(x,F_p)\geq \rho(x)\,,
	$$ 
	where 
	\begin{align}\label{230114328}
		F_p:=\{y\in\bR^d:(y-p)\cdot e_p=0\}\subset \Omega^c
	\end{align}
(see Figure \ref{2304201156} above).
	For the inverse inequality, take $p_x\in\partial\Omega$ such that $|x-p_x|=\rho(x)$.
	Since 
	$$
	B\big(x,\rho(x)\big)\subset \Omega\quad\text{and}\quad p_x\in\partial B\big(x,\rho(x)\big)\,,
	$$
	we obtain that $e_{p_x}=(p_x-x)/|p_x-x|$.
	Therefore
	$$
	\inf_{p\in\partial\Omega}W_p(x)\leq W_{p_x}(x)=|p_x-x|=\rho(x).
	$$
	
	Let $\mathrm{A}=(\alpha^{ij})_{d\times d}\in \mathrm{M}(\nu^2,1)$, $\nu\in(0,1]$, and take $\mathrm{B}\in \mathrm{M}(\nu,1)$ such that $\mathrm{B}^2=\mathrm{A}$.
	For any $p\in\partial\Omega$, 
	$$
	\Delta \big(W_{p}(\mathrm{B}\,\cdot\,)\big)\equiv 0\quad\text{on}\,\,\mathrm{B}^{-1}\Omega\,.
	$$ 
	Due to \eqref{22.02.17.10} and Proposition~\ref{21.05.18.1}.(2), we obtain that $\rho(\mathrm{B}\,\cdot\,)$ is a infimum of classical superharmonic functions, and therefore $\rho(\mathrm{B}\,\cdot\,)$ is a superharmonic function.
	Consequently we have
	$$
	\langle\alpha^{ij}D_{ij}\rho,\zeta\rangle=\det(\mathrm{A})^{1/2}\big\langle\Delta (\rho(\mathrm{B}\,\cdot\,)),\zeta(\mathrm{B}\,\cdot\,)\big\rangle\leq 0
	$$
	for any $\zeta\in C_c^{\infty}(\Omega)$ with $\zeta\geq 0$.
\end{proof}

\vspace{2mm}

\subsection{Exterior Reifenberg condition}\label{ERD}\,

The notion of the vanishing Reifenberg condition was introduced by Reifenberg \cite{Reifcondition}, and has been extensively studied in the literature (see, \textit{e.g.}, \cite{Relliptic, CKL, KenigToro3, TT}).
The following definition can be found in \cite{Relliptic, KenigToro3}:
For $\delta\in(0,1)$ and $R>0$, a domain $\Omega\subset \bR^d$ is said to satisfy the $(\delta,R)$-\textit{Reifenberg condition}, if for every $p\in\partial\Omega$ and $r\in(0,R]$, there exists a unit vector $e_{p,r}\in\bR^d$ such that
\begin{align}\label{22.02.26.41}
	\begin{split}
		&\Omega\cap B_r(p)\subset \{x\in B_r(p)\,:\,(x-p)\cdot e_{p,r}<\delta r\}\quad\text{and}\\
		&\Omega\cap B_r(p)\supset \{x\in B_r(p)\,:\,(x-p)\cdot e_{p,r}>-\delta r\}\,.
	\end{split}
\end{align}
In addition, $\Omega$ is said to satisfy the \textit{vanishing Reifenberg condition} if for any $\delta\in(0,1)$, there exists $R_{\delta}>0$ such that $\Omega$ satisfies the $(\delta,R_{\delta})$-Reifenberg condition.
Note that the vanishing Reifenberg condition is weaker than the $C^1$-boundary condition; 
see Example~\ref{220910305}.(2) and (3).

It was established by Kenig and Toro \cite[Lemma 2.1]{KenigToro} that if a bounded domain satisfies the vanishing Reifenberg condition, then this domain also satisfies $\mathbf{LHMD}(1-\epsilon)$ for all $\epsilon\in(0,1)$.
Combining this with Corollary~\ref{22.02.19.3}, we obtain that Statement~\ref{22.02.19.1} $(\Omega,p,\theta)$ holds for all $\theta\in(d-p-1,d-1)$.
Furthermore, in addition to the Poisson and heat equations, there have been studies on elliptic and parabolic equations with variable coefficients on domains satisfying the vanishing Reifenberg condition (see, \textit{e.g.}, \cite{Relliptic, Rparabolic, Reifweight2, DongKim})

In this subsection, we present the totally vanishing exterior Reifenberg condition which is a generalization of the Reifenberg condition, and we obtain a result similar to Corollary~\ref{2208131026} for domains satisfying the totally vanishing exterior Reifenberg condition; see Definition~\ref{2209151117} and Corollary~\ref{22.07.17.109}.

\begin{defn}[Exterior Reifenberg condition]\label{2209151117}\,\,
	
	\begin{enumerate}
		
		\item
		By $\mathbf{ER}_{\Omega}$ we denote the set of all $(\delta,R)\in[0,1]\times\bR_+$ satisfying the following: for each $p\in\partial\Omega$, and each connected component $\Omega_{p,R}^{(i)}$ of $\Omega\cap B(p,R)$, there exists a unit vector $e_{p,R}^{(i)}\in\bR^d$ such that
		\begin{align}\label{22.02.26.4}
			\Omega_{p,R}^{(i)}\subset \{x\in B_R(p)\,:\,(x-p)\cdot e_{p,R}^{(i)}<\delta R\}\,.
		\end{align}
		By $\delta(R):=\delta_{\Omega}(R)$ we denote the infimum of $\delta$ such that $(\delta,R)\in \mathbf{ER}_{\Omega}$.
		
		\item For $\delta\in[0,1]$, we say that $\Omega$ satisfies the \textit{totally $\delta$-exterior Reifenberg condition} (abbreviate to `$\langle\mathrm{TER}_{\delta}\rangle$'), if there exist $0<R_0\leq R_\infty<\infty$ such that
		\begin{align}\label{220916111}
			\delta_{\Omega}(R)\leq \delta\quad\text{whenever}\quad R\leq R_0\,\,\,\text{or}\,\,\,R\geq R_\infty\,.
		\end{align}
		
		\item We say that $\Omega$ satisfies the \textit{totally vanishing exterior Reifenberg condition} (abbreviate to `$\langle\mathrm{TVER}\rangle$'), if
		$\Omega$ satisfies the $\delta$-condition for all $\delta\in(0,1]$. In other word,
		$$
		\lim_{R\rightarrow 0}\delta_{\Omega}(R)=\lim_{R\rightarrow \infty}\delta_{\Omega}(R)=0\,.
		$$
	\end{enumerate}
\end{defn}
\vspace{1mm}
For a comparison between the Refenberg condition and $\langle\mathrm{TVER}\rangle$, see Figure~\ref{230212856} and Example~\ref{220910305} below.

In this subsection, we provide results on domains satisfying $\langle\mathrm{TER}_{\delta}\rangle$ for sufficiently small $\delta>0$.
However, our main interest is thd condition $\langle\mathrm{TVER}\rangle$.

\begin{figure}[h]\centering
	\begin{tikzpicture}
		\begin{scope}[shift={(0,4.5)}]
			\begin{scope}
				\clip (-5,-1.5) rectangle (5,1);
				\draw[decoration={Koch, Koch angle=24, Koch order=4}] 
				decorate {(-5,-0.6) -- (-3,-0.3)};
				\draw[decoration={Koch, Koch angle=20, Koch order=4}] 
				decorate {(-1.8,-0.1)--(-3,-0.3)};
				\draw[decoration={Koch, Koch angle=15, Koch order=4}] 
				decorate {(-1.8,-0.1) -- (-0.3,0)};
				\draw[decoration={Koch, Koch angle=16, Koch order=4}] 
				decorate {(1.5,-0.1)--(-0.3,0)};
				\draw[decoration={Koch, Koch angle=23, Koch order=4}] 
				decorate {(3.5,-0.3)--(1.5,-0.1)};
				\draw[decoration={Koch, Koch angle=19, Koch order=4}] 
				decorate {(3.5,-0.3) -- (5,-0.6)};
			\end{scope}
			\draw (-5,-1.5) rectangle (5,1);
			
		\end{scope}
		
		\begin{scope}[scale=0.7, shift={(-6.5,0)}]
			\draw[decorate, decoration={random steps,segment length=2, amplitude=0.5}, fill=gray!20]
			(2,0) arc(0:180:2) .. controls +(0,-1) and +(-0.5,0) .. (-1,-1.5)..controls +(0.5,0) and +(-0.5,0) .. (0,-1).. controls +(0.5,0) and +(-0.5,0) .. (1,-1.5)..controls +(0.5,0) and +(0,-1)..(2,0);
			
			\begin{scope}
				\draw[dashed] (0,2) circle (0.5);
				\draw[->] (0.4,2.6) -- (1,3.8);
					%
					%
					%
			\end{scope}
			
			\node[align=center] at (0,-3.5) {Vanishing\\Reifenberg condition};
		\end{scope}

		\begin{scope}[scale=0.7]
			\begin{scope}
				\draw[decorate, decoration={random steps,segment length=2, amplitude=0.5}, fill=gray!20]
				(2,0) arc(0:180:2);
				\fill[gray!20] (-2,0)--(0,-2) -- (2,0)--(0,2);
				\fill[white]
				(0,-2) ..controls +(0,2) and +(-0.5,0) .. (2,0)--(2,-2);
				\draw (0,-2) ..controls +(0,2) and +(-0.5,0) .. (2,0);
				\fill[white]
				(0,-2) ..controls +(0,2) and +(0.5,0) .. (-2,0)--(-2,-2);
				\draw (0,-2) (0,-2) ..controls +(0,2) and +(0.5,0) .. (-2,0);
			\end{scope}
			\draw[dashed] (0,2) circle (0.5);
			\draw[->] (0,2.8) -- (0,3.8); 
			
			\node[align=center] at (0,-3.5) {Totally vanishing exterior\\Reifenberg condition\\(Definition~\ref{221013228})};
			
		\end{scope}

		\begin{scope}[scale=0.7, shift={(6.5,0)}]
			\begin{scope}
				\draw[decorate, decoration={random steps,segment length=2, amplitude=0.5}, fill=gray!20]
				(2,0) arc(0:180:2);
				\fill[gray!20] (-2,0)--(0,-2) -- (2,0)--(0,2);
				\fill[white]
				(0,-2) ..controls +(0,2) and +(-0.5,0) .. (2,0)--(2,-2);
				\draw (0,-2) ..controls +(0,2) and +(-0.5,0) .. (2,0);
				\fill[white]
				(0,-2) ..controls +(0,2) and +(0.5,0) .. (-2,0)--(-2,-2);
				\draw (0,-2) (0,-2) ..controls +(0,2) and +(0.5,0) .. (-2,0);

				\draw[fill=white, shift={(0,0.5)}] (0,0) arc(0:180:0.5) .. controls +(0,-0.5) and +(-0.5,0) .. (0,0) .. controls +(0.5,0) and +(0,-0.5) .. (1,0) arc(0:180:0.5);
			\end{scope}
			\draw[dashed] (0,2) circle (0.5);
			\draw[->] (-0.4,2.6) -- (-1,3.8); 
			
			\node[align=center] at (0,-3.5) {Totally vanishing exterior\\Reifenberg condition\\(Definition~\ref{2209151117})};
			
		\end{scope}
	\end{tikzpicture}
	
	\caption{Totally vanishing exterior Reifenberg condition}\label{230212856}
\end{figure}
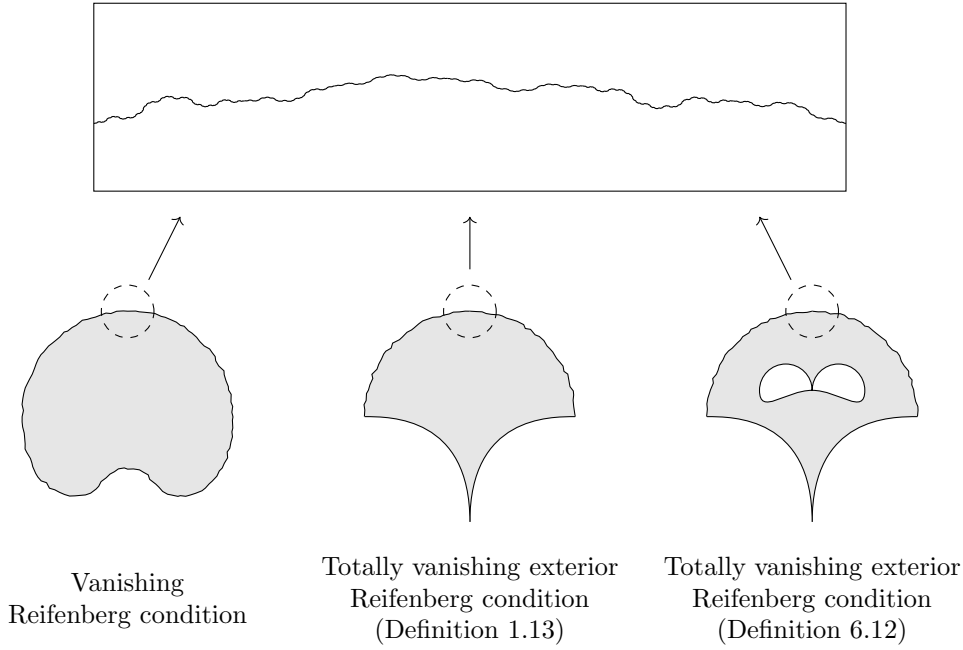

\begin{remark}\label{220918306}\,\,
	We claim that for any $R>0$, $\big(\delta(R),R\big)\in\mathbf{ER}_\Omega$. Take a sequence $\{\delta_n\}_{n\in\bN}$ such that $(\delta_n,R)\in \mathbf{ER}_\Omega$ and $\delta_n\rightarrow \delta(R)$.
	Let $p\in\partial\Omega$, and let $\Omega^{(i)}$ be a connected component of $\Omega\cap B(p,R)$.
	There exists a unit vector $e_n$ such that 
	\begin{align}\label{230115252}
		\Omega^{(i)}\subset \{x\in B_R(p)\,:\,(x-p)\cdot e_n<\delta_n R\}\,.
	\end{align}
	Since $\{e_n\}_{n\in\bN}\subset \partial B(0,1)$, there exists a subsequence $\{e_{n_k}\}_{k\in\bN}$ such that $e_p:=\lim\limits_{k\rightarrow \infty}e_{n_k}$ exists in $\partial B(0,1)$.
	It is impliled by \eqref{230115252} that 
	\begin{align*}
		\Omega^{(i)}\subset \{x\in B_R(p)\,:\,(x-p)\cdot e_p<\delta(R) R\}\,.
	\end{align*}
	Therefore $(\delta(R),R)\in \mathbf{ER}_\Omega$.
\end{remark}

%
%

\begin{example}\label{220910305}
\,\,

\begin{enumerate}
	\item
	If $\Omega$ satisfies the $(\delta,R_1)$-Reifenberg condition,
	then $\delta(R)\leq \delta$ for all $R\leq R_1$, indeed the first line of \eqref{22.02.26.41} implies \eqref{22.02.26.4} with $e_{p,r}^{(i)}=e_{p,r}$.
	Moreover, if $\Omega$ is bounded, then Propositoin~\ref{220918248} implies $\delta(R)\leq \frac{\mathrm{diam}(\Omega)}{R}$.
	Therefore, if $\Omega$ is a bounded domain satisfying the vanishing Reifenberg condition, then $\Omega$ also satisfies $\langle\mathrm{TVER}\rangle$.
	
	\item 
	By $\lambda_*(\bR^{d-1})$ we denote the little Zygmund class which is the set of all $f\in C(\bR^{d-1})$ such that
	$$
	\lim_{h\rightarrow 0}\sup_{x\in\bR^{d-1}}\frac{|f(x+h)-2f(x)+f(x-h)|}{|h|}=0\,.
	$$
	For $f\in\lambda_*(\bR^{d-1})$, put
	$$
	\Omega=\{(x',x_d)\in\bR^{d-1}\times \bR\,:\,x_d>f(x')\}\,.
	$$
	Then, as mentioned in \cite[Example 1.4.3]{CKL} (see also \cite[Theorem 6.3]{RA}), $\Omega$ satisfies the vanishing Reifenberg condition, which implies $\lim\limits_{R\rightarrow 0}\delta_{\Omega}(R)=0$.
	Moreover, since $A:=\|f\|_{C(\bR^{d-1})}<\infty$, Proposition~\ref{220918248} implies that $\delta(R)\leq \frac{2\|f\|_{C(\bR^{d-1})}}{R}$.
	Therefore $\Omega$ satisfies $\langle\mathrm{TVER}\rangle$.
	
	\item Suppose that $\Omega$ is bounded, and for any $p\in\partial\Omega$ there exists $R>0$ and $f\in\lambda_{\ast}(\bR^{d-1})$ such that
	$$
	\Omega\cap B(p,R)=\big\{y=(y',y_n)\in\bR^{d-1}\times \bR\,:\,|y|<R\,\,\,\text{and}\,\,\,y_n>f(y')\big\}\,,
	$$
	where $(y',y_n)=(y_1,\,\ldots\,,y_n)$ is an orthonormal coordinate system centered at $p$. Then $\Omega$ satisfies the vanishing Reifenberg condition, and therefore $\Omega$ satisfies $\langle\mathrm{TVER}\rangle$.
	
	\item Let $\Omega$ satisfy the exterior $R_0$-ball condition, \textit{i.e.}, there exists $R_0>0$ such that for any $p\in\partial\Omega$, there exists $q\in\bR^d$ satisfying $|p-q|=R_0$ and $B(q,R_0)\subset \Omega^c$.
	Then $\delta(R)\leq R/(2R_0)$, and therefore $\lim\limits_{R\rightarrow 0}\delta(R)=0$.
	
	\item If a domain $\Omega$ is an intersection of domains satisfying the totally vanishing Reifenberg condition, then $\Omega$ satisfies $\langle\mathrm{TVER}\rangle$.
\end{enumerate}
\end{example}

A sufficient condition for $\lim\limits_{R\rightarrow \infty}\delta_{\Omega}(R)=0$ is that $\delta_{\Omega}(R)\lesssim 1/R$.
We now provide an equivalent condition for $\Omega$ to satisfy $\delta_{\Omega}(R)\lesssim 1/R$.
Note that the definition of $\delta_{\Omega}(R)$ implies that $R\delta_{\Omega}(R)$ increases as $R\rightarrow \infty$, and therefore if $\delta_{\Omega}(r_0)>0$ for some $r_0>0$, then $\delta_{\Omega}(R)\gtrsim 1/R$ as $R\rightarrow \infty$.

\begin{prop}\label{220918248}
\begin{align*}
	\sup_{R>0}R\,\delta_{\Omega}(R)=\sup_{p\in\partial\Omega}d\big(p,\partial(\Omega_{\mathrm{c.h.}})\big)\,,
\end{align*}
where $\Omega_{\mathrm{c.h.}}$ is the convex hull of $\Omega$, \textit{i.e.},
$$
\Omega_{\mathrm{c.h.}}:=\big\{(1-t)x+ty\,:\,x,\,y\in\Omega\,\,,\,\,t\in[0,1]\big\}\,.
$$
\end{prop}

\begin{remark}\label{230104535}\,\,

\begin{enumerate}
	\item $\Omega_{\mathrm{c.h.}}$ is an open set, and the smallest convex set containing $\Omega$.
	
	\item Proposition~\ref{220918248} implies that $\delta_{\Omega}(\,\cdot\,)\equiv 0$ if and only if $\Omega$ is convex.
\end{enumerate}
\end{remark}

\begin{proof}[Proof of Proposition~\ref{220918248}]\,\,
We only need to prove that for $0<N_0<\infty$,
\begin{align}\label{230104440}
	\sup_{R>0}R\,\delta_{\Omega}(R)\leq N_0\quad  \Longleftrightarrow \quad \sup_{p\in\partial\Omega}d\big(p,\partial(\Omega_{\mathrm{c.h.}})\big)\leq N_0\,.
\end{align}

\textbf{Step 1.}
We first claim that the LHS of \eqref{230104440} holds
if and only if for any $p\in\partial\Omega$, there exists a unit vector $e_p$ such that
\begin{align}\label{220918300}
	\Omega\subset \{x\in\bR^d\,:\,(x-p)\cdot e_p<N_0\}\,.
\end{align}
The `if' part is obvious, and therefore we only need to prove the `only if' part.
Therefore we assume that the LHS of \eqref{230104440} holds. 
Fix $p\in\partial\Omega$, and take $\{\widetilde{\Omega}_n\}_{n\in\bN}$ satisfying the following:
\begin{enumerate}
	\item $\widetilde{\Omega}_n$ is a connceted component of $\Omega\cap B_n(p)$;
	\item $\widetilde{\Omega}_1\subset \widetilde{\Omega}_2\subset \widetilde{\Omega}_3\subset \cdots\,.$
\end{enumerate}
Since $\Omega$ is a domain, $\Omega$ is path connected, which implies
\begin{align}\label{2209201131}
	\bigcup_{n\in\bN}\widetilde{\Omega}_n=\Omega\,.
\end{align}
Since $R\delta(R)\leq N_0$, for each $n\in\bN$ there exists $e_n\in\partial B_1(0)$ such that
\begin{align}\label{2209201230}
	\widetilde{\Omega}_n\subset \{x\in\bR^d\,:\,(x-p)\cdot e_n<N_0\}\,.
\end{align}
Since $\partial B_1(0)$ is compact, there exists a subsequence $\{e_{n_k}\}$ such that
$$
\exists \lim_{k\rightarrow \infty}e_{n_k}=:e_p\in\partial B_1(0).
$$
Due to \eqref{2209201131} and \eqref{2209201230}, we obtain that \eqref{220918300} holds for this $e_p$.

\textbf{Step 2.}
Due to \eqref{230104440}, we only need to prove the following: for $p\in\partial\Omega$,
$$
\text{\eqref{220918300} holds for some $e_p\in\partial B_1(0)$} \quad \Longleftrightarrow \quad d\big(p,\partial(\Omega_{\mathrm{c.h.}})\big)\leq N_0\,.
$$

($\Rightarrow$) Observe that
$$
p\in \partial\Omega\subset\overline{\Omega_{\mathrm{c.h.}}}\subset \{x\in\bR^d\,:\,(x-p)\cdot e_p\leq N_0\}\,.
$$
Put
$$
\alpha_0=\sup\{\alpha\geq 0\,:\,p+\alpha e_p\in\overline{\Omega_{\mathrm{c.h.}}}\}\,.
$$
Then $p+\alpha_0e_p\in \partial(\Omega_{\mathrm{c.h.}})$, and therefore $d\big(p,\partial(\Omega_{\mathrm{c.h.}})\big)\leq \alpha_0\leq N_0$.

($\Leftarrow$) Take $q\in\partial(\Omega_{\mathrm{c.h.}})$ such that
$$
|p-q|=d\big(p,\partial(\Omega_{\mathrm{c.h.}})\big)\leq N_0\,.
$$
Due to Remarks~\ref{2209071233} and \ref{230104535}.(1), there exists a unit vector $\widetilde{e}_q$ such that
$$
\Omega_{\mathrm{c.h.}}\subset \{x\in\bR^d\,:\,(x-q)\cdot \widetilde{e}_q<0\}\,.
$$
This implies that for any $x\in\Omega\subset \Omega_{\mathrm{c.h.}}$, 
$$
(x-p)\cdot \widetilde{e}_q<(q-p)\cdot\widetilde{e}_q\leq |p-q|\leq  N_0\,.
$$
Therefore \eqref{220918300} holds for $e_p:=\widetilde{e}_q$.
\end{proof}

\begin{remark}
From Step 1 in the proof of Proposition~\ref{220918248}, it can be observed that this proposition remains valid even if the definition of $\delta_{\Omega}(R)$ is replaced by the infimum of $\delta>0$ such that, for any $p\in\partial\Omega$, there exists a unit vector $e_{p,R}$ satisfying \eqref{230203624} with $r=R$.
\end{remark}
\vspace{1mm}

Now we state the main result of this subsection.
We temporarily assume Theorem~\ref{22.02.19.6} and Corollary~\ref{22.07.16} (they are proved in the end of this subsection), and prove Corollary~\ref{22.07.17.109}.

\begin{thm}\label{22.02.19.6}
For any $\nu\in(0,1]$ and $\epsilon\in(0,1)$, there exists $\delta_1>0$ depending only on $d,\,\epsilon,\,\nu$ such that if 
$\Omega$ satisfies $\langle\mathrm{TER}_{\delta}\rangle$,
then there exists a measurable function $\phi:\Omega\rightarrow \bR$ satisfying the following:
\begin{enumerate}
	\item For any $(\alpha^{ij})_{d\times d}\in\mathrm{M}(\nu^2,1)$, $\alpha^{ij}D_{ij}\phi\leq 0$ in the sense of distribution
	
	\item There exists $N=N(d,\nu,\epsilon,R_0/R_\infty)>0$ such that
	$$
	N^{-1}\rho(x)^{1-\epsilon}\leq \phi(x)\leq N\rho(x)^{1-\epsilon}\quad\text{for all}\,\,\,x\in\Omega\,,
	$$
	where $R_0$ and $R_\infty$ are constants in \eqref{220916111}.
\end{enumerate}
\end{thm}

\begin{corollary}\label{22.07.16}
For any $\epsilon\in(0,1)$, there exists $\delta_2>0$ depending only on $d,\,\epsilon$ such that if 
$\Omega$ satisfies $\langle\mathrm{TER}_{\delta}\rangle$,
then $\Omega$ satisfies $\mathbf{LHMD}(1-\epsilon)$.
Moreover, $M_{1-\epsilon}$ in \eqref{21.08.03.1} depends only on $d,\,\epsilon$ and $R_0/R_\infty$, where $R_0$ and $R_\infty$ are constants in \eqref{220916111}.
\end{corollary}

\begin{corollary}\label{22.07.17.109}
Let $p\in(1,\infty)$, $\theta\in\bR$, $\gamma\in\bR$, $\nu\in(0,1]$ with
$$
-p-1<\theta-d<-1\,.
$$
Then there exists $\delta>0$ depending only on $d,p,\epsilon,\nu$ such that if
$\Omega$ satisfies $\langle\mathrm{TER}_{\delta}\rangle$,
then the assertions of (1) and (2) in Corollary~\ref{2208131026} hold for this $\Omega$, where $N_0$ in \eqref{220920553} depends only on $d$, $p$, $\gamma$, $\theta$, $R_0/R_\infty$, and $N_1$ in \eqref{220920554} depends only on $d$, $p$, $\gamma$, $\theta$, $\nu$, $R_0/R_\infty$.
Here, $R_0$ and $R_\infty$ are constants in \eqref{220916111}.

%
\end{corollary}
\begin{proof}[Proof of Corollary \ref{22.07.17.109}]
Take $\epsilon\in(0,1)$ such that
$$
-p-1+(p-1)\epsilon<\theta-d<-1-\epsilon
$$
and put
$$
\mu=-\frac{\theta-d+2}{p(1-\epsilon)}\in\big(-\frac{1}{p},1-\frac{1}{p}\big)\,.
$$
Put $\delta=\delta_1\wedge \delta_2>0$, where $\delta_1$ and $\delta_2$ are constants in Theorem~\ref{22.02.19.6} and Corollary~\ref{22.07.16}, respectively, for given $\epsilon$ and $\nu$.

By Corollary~\ref{22.07.16} and Remark~\ref{21.07.06.1}, the Hardy inequality \eqref{hardy} holds on $\Omega$ where $\mathrm{C}_0(\Omega)$ depends only on $d, \epsilon, R_{\infty}/R_0$.
Let $\phi$ be the function in Theorem~\ref{22.02.19.6}.
Due to Proposition~\ref{05.11.1}.(2), we obtain that $\mu\in I(\phi,\nu^2,p)$ and $\mathrm{C}_4$ in \eqref{21.07.12.1} can be chosen to depend only on $\mu$, $\nu$, and $p$.
Put $\Psi=\trho^{\,1-\epsilon}$ which is a regularization of $\phi$.
Then $\mathrm{C}_2(\Psi)$ and $\mathrm{C}_3(\phi,\Psi)$ can be chosen to depend only on $d$, $\epsilon$, $\nu$ and $R_0/R_{\infty}$.
By applying Theorem~\ref{21.09.29.1} and Theorem~\ref{22.02.18.6}, the proof is completed.
\end{proof}

To prove Theorem~\ref{22.02.19.6}, we need to construct functions used instead of the harmonic measure.

\begin{lemma}\label{21.08.24.1}
Suppose that $(\delta,R)\in\mathbf{ER}_{\Omega}$.
For any $\nu\in(0,1)$ and $p\in\partial\Omega$, there exists a continuous function $w_{p,R}:\Omega\rightarrow (0,1]$ satisfying the following:
\begin{enumerate}
	\item For any $\mathrm{B}\in\mathrm{M}(\nu,1)$, $w_{p,R}(\mathrm{B}\,\cdot\,)$ is a classical superharmonic function on $\mathrm{B}^{-1}\Omega$\,.
	
	\item $w_{p,R}=1\text{ on }\{x\in\Omega\,:\,|x-p|>(1-\delta)R\}$\,.
	\item $w_{p,R}\leq M\delta$ on $\Omega\cap B(p,\delta R)$\,.
\end{enumerate}
Here, $M$ depends only on $\nu$ and $d$.
In particular, $M$ is independent of $\delta$.
\end{lemma}
\begin{proof}[Proof of Lemma~\ref{21.08.24.1}]
If $\delta>1/8$, then by putting $w_{p,R}\equiv 1$ and $M=8$, this lemma is proved.
Therefore we only need to consider the case $\delta\leq 1/8$.
For a fixed $p\in\partial\Omega$, let $\big\{\Omega^{(i)}_{p,R}\big\}$ be the set of all connected components of $\Omega\cap B(p,R)$.
For each $i$, take a unit vector $e_{p,R}^{(i)}$ satisfying \eqref{22.02.26.4}.
Put
\begin{align}\label{230116526}
	q=p+R(\delta+1/4)e^{(i)}_{p,R}
\end{align}
so that 
\begin{align}\label{220920414}
	|p-q|=R(\delta+1/4)\quad\text{and}\quad \Omega_{p,R}^{(i)}\cap B(q,R/4)\neq \emptyset
\end{align}
(see Figure~\ref{230116528} below).

\begin{figure}[h]
	\begin{tikzpicture}[>=Latex]
		
		\begin{scope}[scale=0.7]
			\clip (0,0) circle (3.3);		
			
			\begin{scope}[rotate=90]	
				\begin{scope}
					\clip[shift={(-0.15,0)}](-0.2,-{sqrt(10)}) .. controls +(0.1,0.5) and +(0,-.5) .. (0.3,-1.5) .. controls +(0,0.5) and +(0.1,-0.5) .. (0.06,-0.6) .. controls +(-0.1,0.5) and +(0,-0.5) .. (0,0) .. controls +(0,0.5) and +(-0.05,-0.5) .. (0.06,0.6).. controls +(0.05,0.5) and +(0,-0.4) .. (0.25,1.5).. controls +(0,+0.4) and +(-0.2,-0.75) .. (.25,{sqrt(10)}) arc({atan(3)}:270:{sqrt(11)});
					\clip (0,0) circle (3);
					\foreach \i in {-7,-6.5,...,0}	
					{\draw (\i,-3.5)--(\i+4,3.5);}
				\end{scope}
				
				\draw[decoration={Koch, Koch angle=20, Koch order=3}] 
				decorate {(0,0) -- (0.06,0.6)};
				\draw[decoration={Koch, Koch angle=15, Koch order=3}] 
				decorate {(0.25,1.5)--(0.06,0.6)};
				\draw[decoration={Koch, Koch angle=20, Koch order=3}] 
				decorate {(0.25,1.5) -- (0.25,{sqrt(10)})};
				\draw[decoration={Koch, Koch angle=24, Koch order=3}] 
				decorate {(0.06,-0.6)--(0,0)};
				\draw[decoration={Koch, Koch angle=20, Koch order=3}] 
				decorate {(0.06,-0.6) -- (0.3,-1.5)};
				\draw[decoration={Koch, Koch angle=15, Koch order=3}] 
				decorate {(0.3,-1.5) -- (-0.2,-{sqrt(10)})};
				
				\draw (0,0) circle ({sqrt(10)});

				\draw[fill=black] (0,0) circle (0.06);
				
				\draw (1.15,0) circle (0.75);
				\draw[fill=black] (1.15,0) circle (0.06);	

			\end{scope}
			

			
		\end{scope}
		
		\begin{scope}

			\fill[white] (-1,-1.2) circle (0.5);
			\draw (-1,-1.2) node {$\Omega_{p,R}^{(i)}$};
			
			\fill[white] (-0.14,-0.25) circle (0.23);
			
			\draw (-0.1,-0.25) node {$p$};
			
			\draw (-0.14,1) node {$q$};
			
			\draw[dashed] (-4,0.28) -- (0,0.28);
			\draw[dashed] (4,0.28) -- (0,0.28);	
			\draw[dashed] (4,0)-- (0,0) ;	
			\draw[->] (3,0)--(3,1.7);
			\draw (3.6,1) node {$e_{p,R}^{(i)}$};
		\end{scope}
	\end{tikzpicture}
	\caption{$q$ and $B(q,R/4)$  in \eqref{230116526}, \eqref{220920414}}\label{230116528}
\end{figure}

%
%

Put
\begin{align}\label{230104814}
W^{(i)}(x)=\frac{1-(4R^{-1}|x-q|)^{2-\nu^{-2}d}}{1-2^{2-\nu^{-2}d}}\,.
\end{align}
Then we have
$$
\sum_{k,l}\alpha^{kl}D_{kl}\,W^{(i)}\leq 0\quad \text{on}\,\,\,\, \bR^d\setminus\{q\},\quad \text{for all }(\alpha^{kl})_{d\times d}\in\mathrm{M}(\nu^2,1)\,.
$$
Indeed, for $f\in C^2(\bR_+)$, if $f'\geq 0$ and $f''\leq 0$ then
\begin{align}\label{2209161108}
\begin{split}
	&\sum_{k,l=1}^d\alpha^{kl}D_{kl}\big(f(|x|)\big)\\
	=\,\,&\frac{\sum_{i,j}\alpha^{kl}x_kx_l}{|x|^2}f''(|x|)+\left(\frac{\sum_{k}\alpha^{kk}}{|x|}-\frac{\sum_{k,l}\alpha^{kl}x_kx_l}{|x|^3}\right)f'(|x|)\\
	\leq \,\,&\nu^2 f''(|x|)+\frac{d-\nu^2}{|x|}f'(|x|)\,.
\end{split}
\end{align}
Observe that
\begin{alignat*}{2}
0\leq W^{(i)}(x)\,&\leq M_0\big(4R^{-1}|x-q|-1\big)\quad&&\text{if}\,\,\,|x-q|\geq R/4\,;\\
W^{(i)}(x)\,&\geq 1\quad&&\text{if}\,\,\,|x-q|\geq R/2\,,
\end{alignat*}
where $M_0$ is a constant depends only on $\nu$ and $d$.
Due to \eqref{220920414} and that $\delta<\frac{1}{8}$, for $x\in\Omega_{p,R}^{(i)}$,
\begin{alignat*}{2}
&\text{if\quad $|x-p|\leq\delta R$}\,\,,\,\,\,&&\text{then\quad$\frac{R}{4}\leq |x-q|\leq \frac{R}{4}+2\delta R$}\,;\\
&\text{if\quad$|x-p|\geq (1-\delta)R$}\,\,,\,\,\,&&\text{then\quad$|x-q|\geq \frac{(3-8\delta)R}{4}\geq \frac{R}{2}$}\,.
\end{alignat*}
Therefore we obtain that
\begin{alignat*}{2}
\quad 0\leq &W^{(i)}(x)\leq 8M_0\delta\qquad&&\text{if}\,\,\,|x-p|\leq \delta R\\
\quad &W^{(i)}(x)\geq 1\qquad&&\text{if}\,\,\,|x-p|\geq (1-\delta)R\,.
\end{alignat*}
Put
\begin{align*}
w_{p,R}(x)=
\begin{cases}
	W^{(i)}(x)\wedge 1&\text{if}\,\,\,x\in\Omega^{(i)}_{p,R}\\
	1&\text{if}\,\,\,x\in\Omega\setminus B(p,R)\,.
\end{cases}
\end{align*}
Then $w_{p,R}$ is continuous on $\Omega$, and satisfies (2) and (3) of this lemma.
(1) of this lemma follows from \eqref{230104814} and Proposition~\ref{21.05.18.1}.
\end{proof}

\begin{proof}[Proof of Theorem~\ref{22.02.19.6}]
We only need to prove for $\nu\in(0,1)$.
Let $M>0$ be the constant in Lemma~\ref{21.08.24.1}.
For a fixed $\epsilon\in (0,1)$, take small enough $\delta\in(0,1)$ such that $M\delta< \delta^{1-\epsilon}$, and take small enough $\eta\in(0,1)$ such that
$$
(1-\eta)M\delta+\eta\leq \delta^{1-\epsilon}\,.
$$
We assume that $\Omega$ satisfies \eqref{220916111} for this $\delta$.
By using dilation and Remark~\ref{220918306}, without lose of generality, we assume that $(\delta,R)\in\mathbf{ER}_{\Omega}$ whenever $R\leq \widetilde{R}_0:=R_0/R_\infty\,\,(\leq 1)$ or $R\geq 1$.

\textbf{Step 1.}
Put
$$
k_0=\min\big\{k\in\bN\,:\,\delta^k\leq \widetilde{R}_0\big\}\quad\text{and}\quad \cI=\{k\in\bZ\,:\,k\leq 0\,\,\,\text{or}\,\,\, k\geq k_0\}\,,
$$
so that $(\delta,\delta^k)\in\mathbf{ER}_\Omega$ for every $k\in\cI$.
For each $p\in\partial\Omega$ and $k\in\cI$, put
\begin{align*}
\phi_{p,k}=\delta^{k(1-\epsilon)}\Big((1-\eta)\, w_{p,\delta^k}+\eta\Big)\,,
\end{align*}
where $w_{p,\delta^k}$ is the function $w_{p,R}$ in Lemma~\ref{21.08.24.1} with $R=\delta^k$.
Note that
\begin{alignat*}{2}
&\qquad \phi_{p,k}(x)\leq \delta^{(k+1)(1-\epsilon)}&&\text{on}\quad \Omega\cap \overline{B}(p,\delta^{k+1})\,;\\
&\qquad \phi_{p,k}(x)=\delta^{k(1-\epsilon)} &&\text{on}\quad \Omega\cap \partial B(p,\delta^k)\,;\\
&\eta\cdot \delta^{k(1-\epsilon)}\leq \phi_{p,k}\leq \delta^{k(1-\epsilon)}\qquad&&\text{on}\quad \Omega\cap B(p,\delta^k)\,.
\end{alignat*}
Put
\begin{alignat*}{2}
&\phi_p^{(1)}(x):=\inf\{\phi_{p,k}(x)\,:\,k\geq k_0\,\,,\,\,|x-p|< \delta^k\}\quad&&\text{for}\quad |x-p|< \delta^{k_0}\,;\\
&\phi_p^{(2)}(x):=\inf\{\phi_{p,k}(x)\,:\,k\leq 0\,\,,\,\,|x-p|< \delta^k\}\quad&&\text{for}\quad |x-p|>\delta\,.
\end{alignat*}
The similar argument with the proof of Theorem~\ref{21.11.08.1} implies that for any $\mathrm{B}\in\mathrm{M}(\nu,1)$, $\phi_p^{(1)}(\mathrm{B}\,\cdot\,)$ and $\phi_p^{(2)}(\mathrm{B}\,\cdot\,)$ are classical superharmonic functions on
$$
\{\mathrm{B}^{-1}x\,:\,x\in\Omega\cap B(p,\delta^k)\}\quad\text{and}\quad \{\mathrm{B}^{-1}x\,:\,x\in \Omega\setminus \overline{B}(p,\delta)\}\,,
$$
respectively.
Moreover, for each $i\in\{1,2\}$, $\phi_p^{(i)}(x)$ satisfies
\begin{align}\label{2209161228}
\eta |x-p|^{1-\epsilon}\leq \phi_p^{(i)}(x)\leq \delta^{-1+\epsilon}|x-p|^{1-\epsilon}
\end{align}
on its domain.
\vspace{1mm}

\textbf{Step 2.}
Observe that
\begin{equation}\label{230105945}
\begin{alignedat}{2}
	&\phi_p^{(1)}(x)\leq \phi_{p,k_0}(x)\leq \delta^{(k_0+1)(1-\epsilon)}\quad&&\text{if}\,\,\,|x-p|=\delta^{k_0+1}\,\,;\\
	&\phi_p^{(1)}(x)=\phi_{p,k_0}(x)=\delta^{k_0(1-\epsilon)}\quad &&\text{if}\,\,\,|x-p|=\delta^{k_0}\,.
\end{alignedat}
\end{equation}
Put $\gamma=-\nu^{-2}d+2<0$ and take $\alpha_1,\,\beta_1\in\bR$ such that $f(t):=\alpha_1-\beta_1t^{\gamma}$ satisfies
\begin{align}\label{23010594511}
f(\delta^{k_0+1})=\delta^{(k_0+1)(1-\epsilon)}\quad\text{and}\quad f(\delta^{k_0})=\delta^{k_0(1-\epsilon)}\,.
\end{align}
Since $f(\delta^{k_0+1})<f(\delta^{k_0})$, we have $\beta_1>0$, which implies the following:
\begin{itemize}
\item Due to \eqref{2209161108}, for any $(\alpha^{ij})_{d\times d}\in\mathrm{M}(\nu^2,1)$,
$$
\sum_{i,\,j}\alpha^{ij}D_{ij}\Big(f\big(|\cdot-p|\big)\Big)\leq 0\,;
$$

\item $f(t)$ increases as $t\rightarrow \infty$.
In particular, $f(t)\geq \delta^{(k_0+1)(1-\epsilon)}$ for all $t\geq \delta^{k_0+1}$.
\end{itemize}
Put
\begin{align*}
\widetilde{\phi}_p^{(1)}:=
\begin{cases}
	\phi_p^{(1)}&\quad \text{on}\,\,\,\big\{x\in\Omega\,:\,|x-p|\leq \delta^{k_0+1}\big\}\\
	\phi_p^{(1)}\wedge f(|\,\cdot\,-p|) &\quad \text{on}\,\,\,\big\{x\in\Omega\,:\,\delta^{k_0+1}<|x-p|< \delta^{k_0}\big\}\\
	f(|\,\cdot\,-p|)&\quad \text{on}\,\,\,\big\{x\in\Omega\,:\,|x-p|\geq \delta^{k_0}\big\}\,.
\end{cases}
\end{align*}
Due to Proposition~\ref{21.05.18.1}.(4), \eqref{230105945}, and \eqref{23010594511}, we obtain that for any $\mathrm{B}\in\mathrm{M}(\nu,1)$, $\widetilde{\phi}_p^{(1)}(\mathrm{B}\,\cdot\,)$ is a classical superharmonic function on $\mathrm{B}^{-1}\Omega$.

Take $\alpha_2>0,\,\beta_2\in\bR$ such that
\begin{align}\label{2301051008}
\alpha_2\eta\delta^{1-\epsilon}+\beta_2=f(\delta)\quad\text{and}\quad \alpha_2\delta^{-1+\epsilon}+\beta_2=f(1)\,.
\end{align}
Then, due to \eqref{2209161228} and \eqref{2301051008}, $\widetilde{\phi}_p^{(2)}:=\alpha_2\phi_p^{(2)}+\beta_2$ satisfies that
\begin{equation}\label{230105950}
\begin{alignedat}{2}
	&\widetilde{\phi}_p^{(1)}(x)=f(\delta)\leq \widetilde{\phi}_p^{(2)}(x)\quad&&\text{on}\,\,\,\{x\in\Omega\,:\,|x-p|=\delta\}\quad(\text{in the sense of limit});\\
	&\widetilde{\phi}_p^{(1)}(x)=f(1)\geq \widetilde{\phi}_p^{(2)}(x)\quad &&\text{on}\,\,\,\{x\in\Omega\,:\,|x-p|=1\}
\end{alignedat}
\end{equation}
(see \eqref{2209161228}).
%
Due to Proposition~\ref{21.05.18.1}.(4), \eqref{230105950}, and that $\alpha_2>0$, the function
\begin{align*}
\phi_p(x):=
\begin{cases}
	\widetilde{\phi}_p^{(1)}&\quad \text{on}\,\,\,\big\{x\in\Omega\,:\,|x-p|\leq \delta\big\}\,;\\
	\widetilde{\phi}_p^{(1)}\wedge \widetilde{\phi}_p^{(2)} &\quad \text{on}\,\,\,\big\{x\in\Omega\,:\,\delta<|x-p|< 1\big\}\,;\\
	\widetilde{\phi}_p^{(2)}&\quad \text{on}\,\,\,\big\{x\in\Omega\,:\,|x-p|\geq 1\big\}\,,
\end{cases}
\end{align*}
satisfies that for any $\mathrm{B}\in\mathrm{M}(\nu,1)$,
\begin{align}\label{230413212}
\text{$\phi_p(\mathrm{B}\,\cdot\,)$ is a classical superharmonic function on $\mathrm{B}^{-1}\Omega$.}
\end{align}
\vspace{1mm}

\textbf{Step 3.}
We claim that for every $x\in\Omega$,
\begin{align}\label{2209161229}
N^{-1}|x-p|^{1-\epsilon}\leq \phi_p(x)\leq N|x-p|^{1-\epsilon}\,,
\end{align}
where $N=N(d,\epsilon,\nu,\widetilde{R}_1)>0$.
Note that
\begin{align}\label{230105954}
\phi_p=
\begin{cases}
	\vspace{1mm}
	\,\,\, \phi_p^{(1)}&\quad \text{on}\,\,\,\big\{x\in\Omega\,:\,|x-p|\leq \delta^{k_0+1}\big\}\\
	\vspace{1mm}
	\,\,\,\phi_p^{(1)}\wedge f(|\,\cdot\,-p|)&\quad \text{on}\,\,\,\big\{x\in\Omega\,:\,\delta^{k_0+1}<|x-p|< \delta^{k_0}\big\}\\
	\vspace{1mm}
	\,\,\,f(|\,\cdot\,-p|)&\quad \text{on}\,\,\,\big\{x\in\Omega\,:\,\delta^{k_0}\leq |x-p|\leq \delta\big\}\\
	\vspace{1mm}
	\,\,\,f(|\,\cdot\,-p|)\wedge \widetilde{\phi}_p^{(2)} &\quad \text{on}\,\,\,\big\{x\in\Omega\,:\,\delta<|x-p|< 1\big\}\\
	\vspace{1mm}
	\,\,\,\widetilde{\phi}_p^{(2)}&\quad \text{on}\,\,\,\big\{x\in\Omega\,:\,|x-p|\geq 1\big\}\,.
\end{cases}
\end{align}

\textbf{Step 3.1)} It is provided in \eqref{2209161228} that 
\begin{align}\label{230105957}
\eta|x-p|^{1-\epsilon}\leq \phi_p^{(1)}(x)\leq \delta^{-1+\epsilon}|x-p|^{1-\epsilon}\quad\text{on}\quad \big\{x\in\Omega\,:\,|x-p|<\delta^{k_0}\big\}\,.
\end{align}

\textbf{Step 3.2)} Since 
$$
f(t)=\alpha_1-\beta_1t^{\gamma}\,\,,\,\,\,\,\beta_1>0\,\,,\,\,\text{and}\,\,\,\,\gamma<-d+2\leq 0\,,
$$
we have
\begin{align}\label{2301059571}
\delta^{(k_0+1)(1-\epsilon)}=f(\delta^{k_0+1})\leq f\big(|x-p|\big)\leq f(1)\quad\text{if}\quad \delta^{k_0+1}<|x-p|< 1
\end{align}
(this implies that $f(1)=\alpha_1-\beta_1>0$).

\textbf{Step 3.3)} Note that $\widetilde{\phi}_p^{(2)}=\alpha_2\phi_p^{(2)}+\beta_2$ and $\alpha_2>0$. 
Take $K\geq 1$ such that 
$$
\alpha_2\eta K^{1-\epsilon}\geq 2|\beta_2|\,.
$$
For $x\in \Omega$ satisfying $\delta<|x-p|<K$, it follows from \eqref{2209161228}, \eqref{2301051008}, and \eqref{2301059571} that
\begin{equation*}
\begin{aligned}
	&\widetilde{\phi}_p^{(2)}(x)\geq \alpha_2\eta\delta^{1-\epsilon}+\beta_2=f(\delta)\geq \delta^{(k_0+1)(1-\epsilon)}\quad\text{and}\\
	&\widetilde{\phi}_p^{(2)}(x)\leq \alpha_2\delta^{-1+\epsilon}K^{1-\epsilon}+\beta_2\,.
\end{aligned}
\end{equation*}
Therefore there exists $N=N(\delta,\alpha_2,\beta_2,K)$ such that
\begin{align}\label{2301048512}
N^{-1}|x-p|^{1-\epsilon}\leq \widetilde{\phi}^{(2)}_p(x)\leq N|x-p|^{1-\epsilon}\,.
\end{align}

\textbf{Step 3.4)} If $|x-p|\geq K$, then 
$$
2|\beta_2|\leq \alpha_2\eta K^{1-\epsilon}\leq \alpha_2\eta|x-p|^{1-\epsilon}\,.
$$
Due to \eqref{2209161228}, we have
\begin{align}\label{2301048513}
\frac{\eta\alpha_2}{2}|x-p|^{1-\epsilon}\leq \widetilde{\phi}_p^{(2)}(x)\leq \alpha_2\big(\delta^{-1+\epsilon}+\frac{\eta}{2}\,\big)|x-p|^{1-\epsilon}\,.
\end{align}

Since $k_0,\eta,\,\alpha_1\,\beta_1,\,\alpha_2,\,\beta_2,\,K$ depend only on $d,\,\nu,\,\epsilon,\,\delta,\,\widetilde{R}_0$, 
\eqref{230105954} - \eqref{2301048513} imply \eqref{2209161229}.

\vspace{1mm}

\textbf{Step 4.}
Put $\phi(x):=\inf\limits_{p\in\partial\Omega}\phi_{p}(x)$.
Then 
$$
N^{-1}\rho(x)\leq \phi(x)\leq N\rho(x)\,,
$$
where $N$ is the same constant as in \eqref{2209161229}.
For any fixed $\mathrm{B}\in \mathrm{M}(\nu,1)$, due to \eqref{230413212} and Proposition~\ref{21.05.18.1}.(2), $\phi(\mathrm{B}\,\cdot\,)$ is superharmonic on $\mathrm{B}^{-1}\Omega$.
\end{proof}

\begin{proof}[Proof of Corollary~\ref{22.07.16}]
For a given $\epsilon>0$, let $\delta$ be the constant in Theorem~\ref{22.02.19.6} with $\nu=1/2$, and supposes that \eqref{220916111} holds for this $\delta$.
The proof of Theorem~\ref{22.02.19.6} (see \eqref{230413212} and \eqref{2209161229}) implies that for each $p\in\partial\Omega$, there exists a classical superharmonic function $\phi_p$ such that
$$
N_0^{-1} |x-p|^{1-\epsilon}\leq \phi_p(x)\leq N_0 |x-p|^{1-\epsilon}\quad\text{for any}\,\,\,x\in\Omega\,,
$$
where $N_0=N(d,\epsilon,R_\infty/R_0)>0$.
Note that 
$$
N_0r^{-1+\epsilon}\phi_p\geq 1\quad\text{on}\quad \Omega\cap \partial B_r(p)\,.
$$
From the definition of the harmonic measure $w(\,\cdot\,,p,r)$ (see \eqref{2301051056}), we obtain that if $r>0$ and $x\in \Omega\cap B_r(p)$, then
$$
w(x,p,r)\leq N_0 r^{-1+\epsilon}\phi_p(x)\leq N_0^2\Big(\frac{|x-p|}{r}\Big)^{1-\epsilon}\,.
$$
Therefore we obtain \eqref{21.08.03.1} with $\alpha=1-\epsilon$ and $M_{\alpha}=N_0^2$.
\end{proof}

\vspace{2mm}

\subsection{Conic domains}\label{0074}\,

$\bS^{d-1}$ denotes the set $\{x\in\bR^d\,:\,|x|=1\}$, and $A_{\bS}$ denotes the surface measure on $\bS^{d-1}$. 
Note that for any nonnegative Borel function $F$ on $\bR^{d}\setminus\{0\}$,
$$
\int_{\bR^{d}\setminus\{0\}} F(x)\dd x=\int_0^{\infty}\Big(\int_{\bS^{d-1}}F(r\sigma)\dd A_{\bS}(\sigma)\Big)r^{d-1}\dd r\,.
$$
Let $\cM$ be a relatively open set of $\bS^{d-1}$, and define
\begin{align*}
\Omega=\big\{x\in\bR^d\setminus\{0\}\,:\,\frac{x}{|x|}\in\cM\big\}
\end{align*}
which is the conic domain generated by $\cM$ (see Figure \ref{230222651} below). 
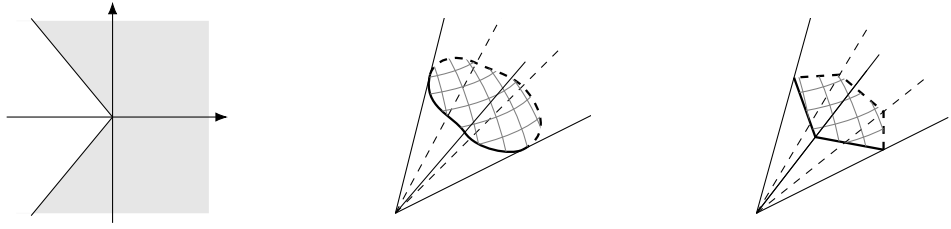
\begin{figure}[ht]
\begin{tikzpicture}[> = Latex, scale=0.8]

\begin{scope}
	
	\begin{scope}[shift={(-6,0)}, scale=0.8]
		\fill[gray!20] (0,0)--(-1.65,-2.0)--(-2.0,-2.0)--(2.0,-2.0)--(2.0,2.0)--(-2.0,2.0)--(-1.65,2.0)--(0,0);
		
		\clip (-2.2,-2.25)--(2.4,-2.25)--(2.4,2.4)--(-2.4,2.4);	
		\draw[->] (-2.2,0) -- (2.4,0);
		\draw[->] (0,-2.2) -- (0,2.4);
		\draw (0,0) -- (-1.69125,2.05);
		\draw (0,0) -- (-1.69125,-2.05);
		\begin{scope}
		\end{scope}
	\end{scope}

	\begin{scope}[shift={(0,-.2)}]
		\clip (-2.2,-2)--(2.4,-2)--(2.4,2.4)--(-2.4,2.4);	
		\begin{scope}[shift={(-1.3,-1.4)}, scale=1.2]
			
			\begin{scope}[scale=0.9]
				\clip (0,0) circle (3.6) ;
				\clip (-3,-3) -- (3,-3) -- (3,3) -- (-3,3);
				\draw[dashed] (0,0) -- (4.4,4.4);
				\draw[dashed] (0,0) -- (2.8,5.2);
				\draw (0,0) -- (4.4,2.2);
				\draw (0,0) -- (1.2,4.8);
			\end{scope}


			\begin{scope}[scale=1.6]
				\draw[dashed,line width=0.3mm] (1.1,0.55) .. controls (1.4,0.7) and (1.2,1) .. (1.1,1.1) .. controls (1.0,1.2) and (0.8,1.25) .. (0.7,1.3) .. controls (0.5,1.4) and (0.33,1.32) .. (0.3,1.2);
			\end{scope}
			
			\begin{scope}[scale=1.6]
				\draw[line width=0.3mm] (0.3,1.2) .. controls (0.23,0.92) and (0.5,0.85) .. (0.6,0.7) .. controls (0.7,0.55) and (1,0.5) .. (1.1,0.55);
			\end{scope}
			
			\begin{scope}[scale=1.6]
				\clip (1.1,0.55) .. controls (1.4,0.7) and (1.2,1) .. (1.1,1.1) .. controls (1.0,1.2) and (0.8,1.25) .. (0.7,1.3) .. controls (0.5,1.4) and (0.33,1.32) .. (0.3,1.2) .. controls (0.23,0.92) and (0.5,0.85) .. (0.6,0.7) .. controls (0.7,0.55) and (1,0.5) .. (1.1,0.55);

				%
				\draw[gray] (0,1.7) arc(90:0:0.54 and 1.95);
				\draw[gray] (0,1.7) arc(90:0:0.795 and 1.92);
				\draw[gray] (0,1.7) arc(90:0:1.03 and 1.85);
				\draw[gray] (0,1.7) arc(90:0:1.23 and 1.787);
				\draw[gray] (0,1.7) arc(90:0:1.4 and 1.71);
				
				\draw[gray] (-0.15,1.14) arc(-90:0:1.7*0.565 and 0.6*0.565);
				\draw[gray] (-0.15,1.00) arc(-90:0:1.7*0.66 and 0.6*0.66);
				\draw[gray] (-0.15,0.84) arc(-90:0:1.7*0.745 and 0.6*0.745);
				\draw[gray] (-0.15,0.675) arc(-90:0:1.7*0.82 and 0.6*0.82);
				\draw[gray] (-0.15,0.5) arc(-90:0:1.7*0.89 and 0.6*0.89);
			\end{scope}
			
			\draw[thin] (0,0) -- (1.8,2.1);
		\end{scope}
	\end{scope}
	
	\begin{scope}[shift={(6,-.2)}]
		\clip (-2.2,-2)--(2.4,-2)--(2.4,2.4)--(-2.4,2.4);		
		\begin{scope}[shift={(-1.3,-1.4)}, scale=1.2]
			
			\begin{scope}[scale=0.9]
				\clip (0,0) circle (3.3) ;
				\clip (-3,-3) -- (3,-3) -- (3,3) -- (-3,3);
				\draw (0,0) -- (4.4,2.2); 
				\draw[dashed] (0,0) -- (5,4);
				\draw[dashed] (0,0) -- (3,5);
				\draw (0,0) -- (1.3,4.7);
				\draw (0,0) -- (1.70,2.20);
			\end{scope}
			
			
			\begin{scope}[scale=1.6]
				\draw[dashed,line width=0.3mm] (1.1,0.55) -- (1.1, 0.88) -- (0.72, 1.2) -- (0.325, 1.175);
			\end{scope}
			
			\begin{scope}[scale=1.6]
				\draw[line width=0.3mm] (0.325, 1.175) -- (0.51, 0.66) -- (1.1,0.55);
			\end{scope}

			\begin{scope}[scale=1.6]
				\clip (1.1,0.55)-- (1.1, 0.88) -- (0.72, 1.2) -- (0.325, 1.175) -- (0.51, 0.66) -- (1.1,0.55);
				
				%
				
				\draw[gray] (0,1.7) arc(90:0:0.54 and 1.95);
				\draw[gray] (0,1.7) arc(90:0:0.795 and 1.92);
				\draw[gray] (0,1.7) arc(90:0:1.03 and 1.85);
				\draw[gray] (0,1.7) arc(90:0:1.23 and 1.787);
				\draw[gray] (0,1.7) arc(90:0:1.4 and 1.71);
				
				\draw[gray] (-0.15,1.14) arc(-90:0:1.7*0.565 and 0.6*0.565);
				\draw[gray] (-0.15,1.00) arc(-90:0:1.7*0.66 and 0.6*0.66);
				\draw[gray] (-0.15,0.84) arc(-90:0:1.7*0.745 and 0.6*0.745);
				\draw[gray] (-0.15,0.675) arc(-90:0:1.7*0.82 and 0.6*0.82);
				\draw[gray] (-0.15,0.5) arc(-90:0:1.7*0.89 and 0.6*0.89);
			\end{scope}

			\draw (0,0) -- (1.70,2.20);
		\end{scope}

	\end{scope}
	

\end{scope}
\end{tikzpicture}
\caption{Conic domains}\label{230222651}
\end{figure}

\noindent
We denote
$$
B_R^{\Omega}=\Omega\cap B_R(0)\quad\text{and}\quad Q_R^{\Omega}=(1-R^2,1]\times B_R^{\Omega}\,.
$$

In this subsection, we suppose that $\cM$ satisfies Assumption \ref{2207111151}; this assumption is satisfied if $\cM$ is a Lipschitz domain in $\bS^{d-1}$.
We prove that if $u$ satisfies
\begin{align*}
\begin{cases}
u_t=\Delta u\quad&\text{in}\quad Q_1^{\Omega}\,;\\
u=0\quad&\text{on}\quad (0,1]\times\big((\partial \Omega)\cap B_1(0)\big)\,,
\end{cases}
\end{align*}
then for any $\lambda\in(0,\lambda_0)$ and  $0<R<1$,
\begin{align}\label{2301111236}
|u(t,x)|\lesssim_{\cM,\lambda,r} |x|^{\lambda}\sup_{Q_1^{\Omega}}|u|\quad\text{whenever}\quad (t,x)\in Q_R^{\Omega}
\end{align}
(see Remark~\ref{2301111204}), where $\lambda_0$ is the constant defined in \eqref{220818848}.

\begin{remark}
As shown in \cite{Kozlov}, estimate \eqref{2301111236} is closely related to Heat kernel estimates.
In \cite[Lemma 3.9]{Kozlov}, Kozlov and Nazarov used the type of estimate \eqref{2301111236} to obtain estimates for the kernel $G$ of parabolic equations in $C^{1,1}$-cones.

%
\end{remark}
\vspace{1mm}

Before state the main result of this subsection, Theorem~\ref{2208221223}, we introduce spherical gradient and spherical Laplacian, avoding notions of differential geometry.
For a function $f$ on $\cM$, we denote $F_f(x)=f(x/|x|)$.
We denote
\begin{align*}
C^{\infty}(\cM)&=\text{the set of all $f:\cM\rightarrow\bR$ for which $F_f\in C^{\infty}(\Omega)$}\,;\\
C_c^{\infty}(\cM)&=\{f\in C^{\infty}(\cM)\,:\,\mathrm{supp}(f)\subset \cM\}\,.
\end{align*}
The spherical gradient and spherical Laplacian of $f\in C^{\infty}(\cM)$, denoted by $\nabla_\bS f$ and $\Delta_\bS f$, are defined by
\begin{align}\label{230203650}
\nabla_\bS f=\nabla F_f|_{\cM}\quad\text{and}\quad \Delta_\bS f=\Delta F_f|_{\cM}\,.
\end{align}
A direct calculation gives the following:
\begin{itemize}
\item For any $f\in C_c^{\infty}(\cM)$ and $g\in C^{\infty}(\cM)$\,,
\begin{align*}
\int_{\cM}\big(\nabla_\bS f,\nabla_\bS g\big)_d\dd A_\bS=-\int_{\cM}(\Delta_\bS f)\cdot g\,\dd A_\bS\,,
\end{align*}
where $(\,\cdot\,,\,\cdot\,)_d$ is the inner product on $\bR^d$.

\item For any $F\in C^{\infty}(\Omega)$,
\begin{align}\label{220818255}
|\nabla F|^2=|D_r F|^2+\frac{1}{r^2}|\nabla_\bS F|^2\,.
\end{align}

\item For a function $F\in C^{\infty}(\Omega)$,
\begin{align}\label{220716632}
\Delta F=D_{rr}F+\frac{d-1}{r}D_r F+\frac{1}{r^2}\Delta_\bS F\,.
\end{align}
\end{itemize}
In \eqref{220818255} and \eqref{220716632}, $F$ is also considered a function on $\bR_+\times\cM$ defined as $(r,\sigma)\mapsto F(r\sigma)$.
We leave it to the reader to verify that $\nabla_{\bS}$ (resp. $\Delta_{\bS}$) is equivalent with the gradient (resp. Laplace-Beltrami) operator implied by standard differential structure on $\bS^{d-1}$; see \cite{JJ} for the standard differential structure on $\bS^{d-1}$.

We make certain assumption about $\cM$ to applicate the results in Subsections~\ref{0040} and \ref{0050}.

\begin{assumption}\label{2207111151}
We denote $\partial_\bS \cM=\overline{\cM}\setminus\cM$.
\begin{enumerate}
\item $\cM$ is a connected (relatively) open set of $\bS^{d-1}$ with $\overline{\cM}\neq \bS^{d-1}$.
\item
\begin{align}\label{220820634}
	\inf_{\substack{p\in\partial_\bS\cM \\ r\in(0,1]}}\frac{A_\bS\big(\{\sigma\in \bS^{d-1}\setminus\cM\,:\,|\sigma-p|<r\}\big)}{r^{d-1}}>0\,,
\end{align}

\item Let $w_0(\sigma)$ be the first (positive) Dirichlet eigenfunction of the spherical Laplacian $\Delta_\bS$ on $\cM$ (see Proposition~\ref{230413525}.(1)).
There exist constants $A,\,N>0$ such that
\begin{align}\label{2208181123}
	w_0(\sigma)\geq N^{-1}d(\sigma,\partial_\bS\cM)^A\,.
\end{align}
\end{enumerate}
\end{assumption}
%

By $\mathring{W}_2^1(\cM)$, we denotes the closure of $C_c^{\infty}(\cM)$ in
$$
W_2^1(\cM):=\{f\in\cD'(\cM)\,:\,\|f\|_{L_2(\cM)}+\|\nabla_{\bS}f\|_{L_2(\cM)}<\infty\}\,.
$$

\begin{prop}\label{230413525}\,

\begin{enumerate}
\item If Assumption~\ref{2207111151}.(1) holds, then 
\begin{align}\label{230107830}
	\Lambda_0:=\inf_{\substack{w\in C_c^{\infty}(\cM)\\w\not\equiv 0}}\frac{\int_{\cM}|\nabla_\bS w|^2\dd A_\bS}{\int_{\cM}|w|^2\dd A_\bS}>0\,,
\end{align}
and there exists a unique $w_0$ in $C^{\infty}(\cM)\cap \mathring{W}_2^1(\cM)$ such that
\begin{align}\label{230413250}
	w_0>0\quad,\quad \int_{\cM}|w_0|^2\dd A_\bS=1\quad,\quad\Delta_\bS w_0+\Lambda_0w_0=0\,.
\end{align}
Moreover, $w_0$ is bounded on $\cM$.	
Furthermore, the function
\begin{align}\label{220818845}
	W_0(x):=|x|^{\lambda_0}w_0(x/|x|)
\end{align}
is a positive harmonic function on $\Omega$, where
\begin{align}\label{220818848}
	\lambda_0=-\frac{d-2}{2}+\sqrt{\Lambda_0+\Big(\frac{d-2}{2}\Big)^2}>0\,.
\end{align}

\item If Assumption~\ref{2207111151}.(2) holds, then 
\begin{align*}
	\inf_{p\in\partial\Omega,\, r>0}\frac{m\big(\Omega^c\cap B_r(p)\big)}{r^d}>0\,.
\end{align*}

\item Let $-e_d\notin \overline{\cM}$ and define $\phi_d$ to be the stereographic projection from $\bS^{d-1}\setminus\{-e_d\}$ to $\bR^{d-1}$ given by
\begin{align}\label{2208221047}
	\phi_d(\sigma_1,\ldots,\sigma_{d-1},\sigma_d)=\Big(\frac{\sigma_1}{1+\sigma_d},\ldots,\frac{\sigma_{d-1}}{1+\sigma_d}\Big)\,.
\end{align}
If $\phi_d(\cM)$ is a John domain in $\bR^{d-1}$ (see Remark~\ref{220819318} for the definition of a John domain), then Assumption~\ref{2207111151}.(3) holds.
\end{enumerate}
\end{prop}
\begin{proof}
(1) \eqref{230107830} follows from \cite[Theorems 10.13, 10.18, 10.22]{AGriog}.
It is provided in \cite[Theorem 10.11, Corollary 10.12]{AGriog} that there exists a unique $w_0\in C^{\infty}(\cM)\cap \mathring{W}_2^1(\cM)$ satisfying \eqref{230413250}.

To prove the boundedness of $w_0$, without lose of generality, we assume that $-e_d:=(0,\ldots,0,-1)\notin\cM$.
By $\phi_d$ we denote the stereographic projection from $\bS^{d-1}\setminus\{-e_d\}$ to $\bR^{d-1}$ defined by \eqref{2208221047}.
Then $\phi_d(\cM)$ is a bounded domain in $\bR^{d-1}$.
Consider the function $\widetilde{w}_0:\phi_d(\cM)\rightarrow \bR$ defined as $\widetilde{w}_0(\phi_d(x)):=w_0(x)$.
Then $\widetilde{w}_0$ belongs to $\mathring{W}_2^1\big(\phi_d(\cM)\big)$ and satisfies
$$
\sum_{i,j=1}^da^{ij}D_{ij}\widetilde{w}_0+\sum_{i=1}^db^iD_i\widetilde{w}_0+\Lambda_0\widetilde{w}_0=0\quad\text{in}\quad \phi_d(\cM)\subset \bR^{d-1}\,.
$$
Here, $a^{ij},\,b^{i}\in C^{\infty}(\bR^{d-1})$ ($i,\,j=1,\,\ldots,\,d$) are smooth functions on $\bR^{d-1}$ such that  there exists $\nu>0$ satisfying
$$
\nu|\xi|^2\leq \sum_{i,j=1}^da^{ij}(x)\xi_i\xi_j\leq \nu^{-1}|\xi|^2\qquad\forall\,\,\xi=(\xi_1,\ldots,\xi_d)\in\bR^{d-1}\,\,,\,\,x\in\phi_d(\cM)\,.
$$
The boundedness of $\widetilde{w}_0$ follows from classical results for elliptic equations (see, \textit{e.g.}, \cite[Theorem 3.13.1]{lady2}), and this implies that $w_0$ is bounded.

It directly follows from \eqref{220716632} that the function $W_0$ in \eqref{220818845} is harmonic on $\Omega$.

(2) 
For any $p\in\partial_\bS\cM$ and $r\in(0,1)$, we have
\begin{align*}
&\big\{s\sigma\in\bR^{d}\,:\,s\in(1-r/2,1+r/2),\,\sigma\in \bS^{d-1}\cap B_{r/2}(p)\big\}\\
\subset&\,\,B_r(p)\\
\subset&\,\big\{s\sigma\in\bR^{d}\,:\,s\in(1-r,1+r),\,\sigma\in \bS^{d-1}\cap B_{2r}(p)\big\}\,.
\end{align*}
Therfore \eqref{220820634} holds if and only if
\begin{align}\label{2208201149}
\inf_{\substack{p\in\partial_\bS\cM\\ r\in(0,1]}}\frac{m\big(\Omega^c\cap B_r(p)\big)}{r^d}>0\,,
\end{align}
where $m$ is the Lebesgue measure on $\bR^d$.
Ir $r\geq 2$, then $B_r(p)\supset B_{r/2}(0)$.
Therefore 
\begin{align}\label{2208206491}
\inf_{p\in\partial_\bS\cM,\, r\geq 2}\frac{m\big(\Omega^c\cap B_r(p)\big)}{r^d}\geq\inf_{r\geq 1}\frac{m\big(\Omega^c\cap B_{r}(0)\big)}{(2r)^d}=\frac{A_{\bS}\left(\bS^{d-1}\setminus\cM\right)}{2^dd}>0\,.
\end{align}
Consequently, it is implied by \eqref{2208201149} and \eqref{2208206491} that
\begin{align*}
\inf_{p\in\partial\Omega,\, r>0}\frac{m\big(\Omega\cap B_r(p)\big)}{r^d}>0\,.
\end{align*}

(3) We denote $U_{\cM}:=\phi_d(\cM)$.
It follows from Example~\ref{21.05.18.2}.(2) that $W_0$ (in \eqref{220818845}) is a Harnack function on $\Omega$.
Since $W_0$ is a Harnack function, and $\phi_d$ (resp. $\phi_d^{-1}$) is Lipschitz continuous on $\cM$ (resp. $U_\cM$), we obtain that $\widetilde{w}_0:=w_0\circ \phi_d^{-1}$ is a Harnack on $U_{\cM}$ (see Lemma~\ref{21.11.16.1}).
In addition, $d(\sigma,\partial_\bS\cM)\simeq d\big(\phi_d(\sigma),\partial U_{\cM}\big)$.
By Remark~\ref{220819318}, if $\Omega$ is a John domain, then
$$
\widetilde{w}_0(x')\gtrsim d(x',\partial U_{\cM})^A\qquad\text{for all}\quad x'\in \phi_d(\cM)\,,
$$
and therefore \eqref{2208181123} is proved.
\end{proof}

\begin{thm}\label{2208221223}
Let $\cM\subset \bS^{d-1}\,(d\geq 2)$ satisfy Assuption~\ref{2207111151}, and suppose that $u\in C^{\infty}(Q_1^{\Omega})$ satisfies that
\begin{align*}
&u_t=\Delta u\quad\text{in}\quad Q_1^{\Omega}\,;\\
&\lim_{(t,x)\rightarrow (t_0,x_0)}u(t,x)=0\quad\text{whenever}\quad 0<t_0\leq 1\,\,,\,\,x_0\in(\partial\Omega)\cap B_1.
\end{align*}
Then for any $\epsilon\in(0,1)$ and $R \in(0,1)$,
\begin{align*}
|u(t,x)|\leq N\Big(\sup_{Q_1^{\Omega}}|u|\Big)W_0(x)^{1-\epsilon}\qquad \forall\,\, (t,x)\in Q_R^{\Omega}\,,
\end{align*}
where $W_0$ is the function defined in \eqref{220818845} and $N=N(\cM,\epsilon,R)>0$.
\end{thm}
Recall that $\mathring{W}_2^1\big(B_1^{\Omega}\big)$ is the closure of $C_c^{\infty}\big(B_1^{\Omega}\big)$ in $W_2^1\big(B_1^{\Omega}\big)$.


\begin{remark}\label{2301111204}
Theorem~\ref{2208221223} implies that if $u$ satisfies the assumptions in Theorem~\ref{2208221223} and $\lambda\in(0,\lambda_0)$, where $\lambda_0$ is in \eqref{220818848}, then
\begin{align}\label{2208221229}
|u(t,x)|\leq N\Big(\sup_{Q_1^{\cD}}|u|\Big)|x|^{\lambda}\quad\text{on}\,\,\,Q_{R}^{\Omega}
\end{align}
where $N=N(\cM,\lambda,R)$.
We note that for $\lambda>\lambda_0$, \eqref{2208221229} does not hold in general.
Observe that $u(t,x):=W_0(x)$ satisfies assumptions in Theorem~\ref{2208221223}.
Due to \eqref{220818845}, there is no constant $N$ satisfying \eqref{2208221229} with $u(t,x)=W_0(x)$ and $\lambda>\lambda_0$.
\end{remark}

\begin{proof}[Proof of Theorem~\ref{2208221223}]\,

\textbf{Step 1.}
Put $K=A\vee \lambda_0$ where $A$ and $\lambda_0$ are the constants in \eqref{2208181123} and \eqref{220818848}, respectively.
From direct calculation (see, \textit{e.g.}, \cite[Lemma 3.4.(1)]{ConicPDE}) we obtain that
\begin{align*}
d(\sigma,\partial\Omega)\leq d(\sigma,\partial_\bS\cM)\leq 2d(\sigma,\partial\Omega)\quad \text{for all}\quad \sigma\in\cM\,.
\end{align*}
Therefore for $x\in\Omega\cap B_1(0)$, we have
$$
\rho(x)^K=d(x,\partial\Omega)^{K}\simeq  |x|^{K}d\big(x/|x|,\partial_\bS\cM\big)^{K}\leq |x|^{\lambda_0}d(x/|x|,\partial_\bS\cM)^A\lesssim W_0(x)\,.
$$
Due to Proposition~\ref{230413525}.(2) and Remark~\ref{21.07.06.1}, $\Omega$ satisfies $\mathbf{LHMD}(\alpha)$ for some $\alpha\in(0,1)$.
Take small enough $\delta\in(0,1)$ such that
$$
\big((d+2)^{-1}+\alpha^{-1}\big)\delta<\epsilon\quad\text{and}\quad 1-\frac{d+4}{d+2}\delta>0\,,
$$
and put 
$$
\beta_t=\frac{\delta}{d+2}\quad \text{and}\quad \beta_x=1-\delta-2\beta_t\,.
$$
Then $\epsilon$, $\delta$, $\beta_t$, $\beta_x$ satisfy \eqref{220602312}.
Put
$$
\epsilon_i=\epsilon+\frac{\beta_x}{K}i\quad \text{for}\quad i\in\bN_0\,\,,\,\,\,\,\text{and}\qquad i_0=\Big[\frac{1-\epsilon}{\beta_x}K\Big]
$$
such that $\epsilon_{i_0}\leq 1<\epsilon_{i_0+1}$.
Since $W_0$ is bounded on $B_1^\Omega$ (see Proposition~\ref{230413525}), we have
$$
\sup_{Q_{1}^\Omega}|W_0^{-1+\epsilon_{i_0+1}}u|\lesssim_{\Omega,\epsilon}\sup_{Q_{1}^{\Omega}}|u|\,,
$$
and therefore we only need to prove that for any $i\in\{0,1,\ldots,i_0\}$ and $0<R_1<R_2\leq 1$,
\begin{align}\label{2208210059}
\sup_{Q_{R_1}^\Omega}|W_0^{-1+\epsilon_{i}}u|\lesssim N(\cD,\epsilon,R,r)\sup_{Q_{R_2}^\Omega}|W_0^{-1+\epsilon_{i+1}}u|\,.
\end{align}

\textbf{Step 2.}
Take $\eta_0\in C^{\infty}\big(\bR\big)$ such that $\eta_0(s)=1$ if $s<R_1^2$ and $\eta_0(s)=0$ if $s>R_2^2$, and put $\eta(t,x):=\eta_0(1-t)\eta_0\big(|x|^2\big)$. Note that
\begin{align*}
\eta(t,x)=
\begin{cases}
	\vspace{0.5mm}
	1\quad&\text{if}\quad t>1-R_1^2\quad\text{and}\quad|x|<R_1\,;\\
	0\quad&\text{if}\quad t<1-R_2^2\quad\text{or}\quad|x|>R_2\,.
\end{cases}
\end{align*}
Put 
\begin{align}\label{2304141114}
v=u\eta\,\,,\,\,\,\,f^0:=\big(\partial_t\eta+\Delta \eta\big)u\,\,,\,\,\,\,f^i:=-2u D_i\eta\quad (i=1,\,\ldots,\,d)\,,
\end{align}
so that $v\in C\big(\overline{Q_1^{\Omega}}\big)\cap C^{\infty}(Q_1^{\Omega})$ satisfies 
\begin{align*}
\partial_tv=\Delta v+f_0+\sum_{i=1}^d D_if^i\quad\text{in}\quad Q_1^{\Omega}\quad ;\quad v\equiv 0\quad\text{on}\quad  \overline{Q_1^{\Omega}}\setminus Q_1^{\Omega}\,.
\end{align*}

\textbf{Step 2.1)} We first claim that $v\in\cH_{2,d-2}^1(\Omega,1)$.
Since
\begin{align*}
\Big\|f^0+\sum_{i=1}^dD_if^i\Big\|_{\bH_{2,d+2}^{-1}(B_1^{\Omega},1)}\lesssim \left\|f^0\right\|_{\bL_{2,d+2}(B_1^{\Omega},1)}+\sum_{i=1}^d\left\|f^i\right\|_{\bL_{2,d}(B_1^{\Omega},1)}\lesssim \sup_{Q_1^{\Omega}}|u|
\end{align*}
(see Lemmas~\ref{220512433} and \ref{21.09.29.4}),
there exists $\widetilde{v}\in \cH_{2,d-2}^1(\Omega,1)$ such that 
\begin{align*}
\partial_t\widetilde{v}=\Delta \widetilde{v}+f^0+\sum_{i=1}^d D_if^i\quad\text{and}\quad \widetilde{v}(0\,\cdot\,)=0\,.
\end{align*}
For the claim in this step, we only need to prove that
\begin{align}\label{230413658}
\widetilde{v}\in C^{\infty}(\Omega)\cap C\big(\overline{Q_1^{\Omega}}\big)\,\,,\,\,\text{and}\quad \widetilde{v}\equiv 0\quad\text{on}\quad \overline{Q_1^{\Omega}}\setminus Q_1^{\Omega}\,.
\end{align}
Indeed, if \eqref{230413658} holds, then the maximum principle yields that $v\equiv \widetilde{v}\in \cH_{2,d-2}^1(B_1^{\Omega},1)$.
Since $\mathrm{supp}\big(v(t,\cdot)\big)\subset \overline{B_{R_2}^{\Omega}}$ for each $t\in[0,1]$, $v$ belongs to $\cH_{2,d-2}^1(\Omega,1)$.

Let us prove \eqref{230413658}.
Since $f^0,f^i\in C^{\infty}(Q_1^{\Omega})$, we obtain that $\widetilde{v}\in C^{\infty}(Q_1^{\Omega})$.
Note that $B_1^{\Omega}$ satisfies \textbf{LHMD}($\alpha'$) for some $\alpha'\in(0,1)$ (see Proposition~\ref{230413525}.(2) and Remark~\ref{21.07.06.1}), and therefore there exists a superharmonic function $\psi$ such that $\psi\simeq (\rho_{B_1^{\Omega}})^{\alpha'/2}$ (see Theorem~\ref{21.11.08.1}), where $\rho_{B_1^{\Omega}}(x):=d(x,\partial B_1^{\Omega})$.
Take $\beta_x',\,\beta_t',\,\delta',\,\epsilon'>0$ such that \eqref{220602312} holds (for $\alpha'$ instead of $\alpha$), and $2\beta_t'+\delta'<1/2$.
Then we have
\begin{align*}
&\Big\||\psi^{-1+\epsilon'}(\rho_{B_1^{\Omega}})^{2-2\beta_t'-\delta'}f^0|+\sum_{i=1}^d|\psi^{-1+\epsilon'}(\rho_{B_1^{\Omega}})^{1-2\beta_t'-\delta'} f^i|\Big\|_{L_{(d+2)/\delta}(Q_1^{\Omega},\dd x\dd t)}\\
\lesssim\,\, &\Big\||(\rho_{B_1^{\Omega}})^{(3-\alpha')/2}f^0|+\sum_{i=1}^d|(\rho_{B_1^{\Omega}})^{(1-\alpha')/2} f^i|\Big\|_{L_{(d+2)/\delta}(Q_1^{\Omega},\dd x\dd t)}\\
\lesssim\,\, & \sup_{Q_1^{\Omega}}|u|<\infty\,.
\end{align*}
Therefore, Theorem~\ref{220602322} and Remark~\ref{230112512} yield that
\begin{align}\label{230413652}
\sup_{Q_1^{\Omega}}t^{-\beta_t'}(\rho_{B_1^{\Omega}})^{-(1-\epsilon')\alpha'/2}|\widetilde{v}|\lesssim \sup_{0<t\leq 1}\frac{\big|\widetilde{\psi}^{-1+\epsilon'}\big(\widetilde{v}(t,\cdot)-\widetilde{v}(0,\cdot)\big)\big|_{\beta_x'}^{(0)}}{|t-0|^{\beta_t'}}<\infty
\end{align}
(for the first inequality, see Proposition~\ref{220512537}).
Since $\widetilde{v}\in C^{\infty}(Q_1^{\Omega})$, \eqref{230413652} implies that $\widetilde{v}\in C\big(\overline{Q_1^{\Omega}}\big)$ and $\widetilde{v}\equiv 0$ on $\overline{Q_1^{\Omega}}\setminus Q_1^{\Omega}$.
Therefore \eqref{230413658} is proved.

\textbf{Step 2.2)}
To prove \eqref{2208210059}, assume that the left hand side of \eqref{2208210059} is finite.

Recall that $\Omega$ admits the Hardy inequality (see Proposition~\ref{230413525}.(2) and Remark~\ref{21.07.06.1}), $v\in\cH_{2,d-2}^1(\Omega,1)$ (in \eqref{2304141114}) is a solution of the equation
$$
\partial_tv=\Delta v+f_0+\sum_{i=1}^d D_if^i\quad;\quad v(0,\cdot)\equiv 0
$$
(see Step 2.1), and that $W_0$ is a regular Harnack function (see Example~\ref{21.05.18.2}.(2)).
Since
\begin{align*}
&\Big\||W_0^{-1+\epsilon_i}\rho^{\beta_x+1}f^0|+|W_0^{-1+\epsilon_i}\rho^{\beta_x}f^i|\Big\|_{L_{(d+2)/\delta}((0,1]\times \cD;\dd t\dd x)}\\
\lesssim_N\,&\sup_{Q_{R_2}^{\Omega}}\big|W_0^{-1+\epsilon_i}\rho^{\beta_x}u\big|\\
\lesssim_N\,&\sup_{Q_{R_2}^{\Omega}}\big|W_0^{-1+\epsilon_{i}+\beta_x/K}u\big|=\sup_{Q_{R_2}^{\Omega}}\big|W_0^{-1+\epsilon_{i+1}}u\big|
\end{align*}
(where $N=N(r,R_1,R_2,\cM)$), Theorem~\ref{220602322} (with Remark~\ref{230112512}) implies that $v\in W_0^{1-\epsilon_i}\cH_{p,-2-2\beta_t p}^1(\Omega,1)$ and
\begin{equation}\label{230413551} 
\begin{alignedat}{2}
	&&&\sup_{(0,1]\times\Omega}t^{-\beta_t}\big|W_0^{-1+\epsilon_i}v|\\
	&\lesssim_{R_1}\,&& \sup_{t\in(0,1]}\frac{\big|W_0^{-1+\epsilon_i}v(t,\cdot)-W_0^{-1+\epsilon_i}v(0,\cdot)\big|_{\beta_x}^{(0)}}{|t-0|^{\beta_t}}\\
	&\lesssim_N\,&&\Big\||W_0^{-1+\epsilon_i}\rho^{\beta_x+1}f^0|+|W_0^{-1+\epsilon_i}\rho^{\beta_x}f^i|\Big\|_{L_{(d+2)/\delta}((1-R_2^2,1]\times \Omega;\dd t\dd x)}\\
	&\lesssim_N\,&&\sup_{Q_{R_2}^{\Omega}}\big|W_0^{-1+\epsilon_{i+1}}u\big|\,.
\end{alignedat}
\end{equation}
Since $v\equiv u$ in $Q_{R_1}^{\Omega}$, \eqref{230413551} implies \eqref{2208210059}.
\end{proof}

\vspace{2mm}

\appendix

\mysection{Weighted Sobolev/Besov spaces}\label{008}

\subsection{Weighted Sobolev/Besov spaces without regular Harnack functions}\label{0082}\,

The spaces $H_{p,\theta}^{\gamma}(\Omega)$ were initially developed for studying partial differential equations in domains, as demonstrated in \cite{KK2004,Krylov1999-1, Lo1}.
Moreover, these spaces, along with similar function spaces like $B_{p,\theta}^{\gamma}(\Omega)$, have also been found in studies on Fourier multipliers arising in harmonic analysis, as seen in works such as \cite{CGHS, GHS, LSJFA}.

In this subsection, we introduce the properties of the spaces $H_{p,\theta}^{\gamma}(\Omega)$ and $B_{p,\theta}^{\gamma}(\Omega)$, which are independent of the previous contents of this paper, except for Subsection~\ref{0041} which is used only for specifying $\trho$ satisfying \eqref{220719515}.
The contents of this subsection are based on the properties of $H_p^{\gamma}(\bR^d)$ and $B_p^{\gamma}(\bR^d)$.

In this section, we assume that
$$
d\in\bN\,\,,\,\,\,\,p\in(1,\infty)\,\,,\,\,\,\,\gamma,\,\theta\in\bR\,\,,\,\,\,\,\text{$\Omega$ is an open set in $\bR^d$}\,,
$$
and denote
$$
\cI=\{d,\,p,\,\gamma,\,\theta\}\,.
$$
By $X_p^{\gamma}$ and $X_{p,\theta}^{\gamma}(\Omega)$, we denote either $H_p^{\gamma}\,(\,=H_p^{\gamma}(\bR^d))$ and $H_{p,\theta}^{\gamma}(\Omega)$, or $B_{p}^{\gamma}\,(\,=B_p^{\gamma}(\bR^d))$ and $B_{p,\theta}^{\gamma}(\Omega)$.

The spaces $H_p^{\gamma}$ and $B_p^{\gamma}$ are introduced in Subsections~\ref{0042} and \ref{0052}, respectively.
Recall the following elementary properties of $X_p^{\gamma}$, which can be found in \cite[Corollary 2.8.2, Theorem 2.10.2]{triebel2}: 
\begin{align}\label{2205070028}
	\|af\|_{X_p^{\gamma}}\lesssim_{d,p,\gamma}\|a\|_{C^{[|\gamma|]+1}}\|f\|_{X_p^{\gamma}}\quad\text{and}\quad \|f(A,\cdot )\|_{X_p^{\gamma}}\lesssim_{d,p,\gamma,A}\|f\|_{X_p^{\gamma}}\,.
\end{align}
We also recall the definitions of $X_{p,\theta}^{\gamma}(\Omega)$.
Fix $\zeta_0\in C_c^{\infty}(\bR_+)$ such that 
\begin{align}\label{220526211}
	\zeta_0\geq 0\,\,,\,\,\,\
	\text{supp}(\zeta_0)\subset [e^{-1},e]\quad\text{and}\quad \sum_{n\in\bZ}\zeta_0(e^n\,\cdot\,)\equiv 1\quad\text{on}\,\,\bR_+\,.
\end{align}
Put $\zeta_1(t)=\zeta_0(e^{-1}t)+\zeta_0(t)+\zeta_0(e t)$, so that
\begin{align}\label{230120312}
	\zeta_1\cdot \zeta_0\equiv \zeta_0\quad\text{on}\quad \bR_+\,.
\end{align}
For $\xi\in C_c^{\infty}(\bR_+)$, we denote
$$
\xi_{(n)}(x)=\xi(e^{-n}\trho(x))\,,
$$
where $\trho(x)$ is the regularization of $\rho(x)$ constructed in Lemma~\ref{21.05.27.3}.(1).
Note the following properties of $\trho$ and $\xi_n(n)$:
\begin{itemize}
	\item For each $k\in\bN_0$, there exists $N_k=N_k(d,k)>0$ such that
	\begin{align}\label{220719515}
		N_0\rho(x)\leq \trho(x)\leq N_0\rho(x)\quad\text{and}\quad |D^k\trho(x)|\leq N_k\trho(x)^{-k+1}
	\end{align}
	for all $k\in\bN_0$ and $x\in\Omega$.
	
	\item Let $\xi\in C_c^{\infty}(\bR_+)$ be supported on $[e^{-K},e^K]\subset \bR_+$, $K\in\bN$.
	For each $n\in\bZ$,
	\begin{equation}\label{220526148}
		\begin{gathered}
			\text{supp}(\xi_{(n)})\subset \{x\in\Omega\,:\,e^{n-K}\leq \trho(x)\leq e^{n+K}\}\,;\\
			\xi_{(n)}\in C^{\infty}(\bR^d)\quad\text{with}\quad |D^{\alpha}\xi_{(n)}|\leq N(\alpha,\xi)\,e^{-n|\alpha|}\,.
		\end{gathered}
	\end{equation}
	In addition, since $\sum_{|n|\leq K+1}\zeta_0(e^n\cdot)\equiv 1$ on $[e^{-K},e^K]$, we have
	\begin{align*}
		\xi_{(n)}\equiv \xi_{(n)}\sum_{|k|\leq N}\zeta_{0,{(n+k)}}\,.
	\end{align*}
\end{itemize}
We denote
\begin{alignat*}{2}
	&X_{p,\theta}^{\gamma}(\Omega)&&=\Big\{f\in\cD'(\Omega)\,:\,\|f\|_{X_{p,\theta}^{\gamma}(\Omega)}^p:=\sum_{n\in\bZ}e^{n\theta}\|\big(\zeta_{(n)}f\big)(e^n\,\cdot\,)\|_{X_p^{\gamma}}^p<\infty\Big\}\,,\\
	&l^{\theta/p}_p(X_p^{\gamma})&&=\Big\{\{f_n\}_{n\in\bZ}\subset X_p^{\gamma}\,:\,\|\{f_n\}\|_{l^{\theta/p}_p(X_p^{\gamma})}^p:=\sum_{n\in\bZ}e^{n\theta}\|f_n\|_{X_p^{\gamma}}^p<\infty\Big\}\,.
\end{alignat*}

For $\xi\in C_c^{\infty}(\bR_+)$, we define the maps
\begin{alignat*}{2}
	&S_{\xi}:\cD'(\Omega)&&\rightarrow \cD'(\bR^d)^{\bZ}:=\big\{\{f_n\}_{n\in\bZ}\,:\,f_n\in\cD'(\bR^d)\big\}\,;\\
	&R_{\xi}:\cD'(\bR^d)^{\bZ}&&\rightarrow \cD'(\Omega)
\end{alignat*}
as
\begin{equation*}
	\begin{alignedat}{2}
		&S_{\xi}f&&:=\big\{(S_{\xi}f)_n\big\}_{n\in\bZ}:=\big\{\big(f\xi_{(n)}\big)(e^n\,\cdot\,)\big\}_{n\in\bZ}\,,\\
		&R_{\xi}\{f_n\}&&:=\sum_{n\in\bZ}\xi_{(n)}(\,\cdot\,)f_n(e^{-n}\,\cdot\,)\,.
	\end{alignedat}
\end{equation*}
Note that, since $\zeta_1\zeta_0\equiv \zeta_0$, $R_{\zeta_1}\circ S_{\zeta_0}$ is the identity map on $\cD'(\Omega)$.
Following \cite{Lo1}, we use the maps $S_{\xi}$ and $R_{\xi}$ to obtain properties of $X_{p,\theta}^{\gamma}(\Omega)$ from the properties of $l_p^{\theta/p}(X_p^{\gamma})$.

We now introduce the properties of $X_{p,\theta}^{\gamma}(\Omega)$. 
Since $\zeta_0$ is fixed and the spaces $X_{p,\theta}^{\gamma}(\Omega)$ are independent of choice of $\zeta_0$ (see Proposition~\ref{220527502}.(5)), the dependence on $\zeta_0$ will be ignored.
For the case $X=H$, Propositions \ref{230129407} - \ref{220527502} follow from \cite[Section 2, 3]{Lo1} and elementary properies of $H_p^{\gamma}$.
Corresponding results for the case $X=B$ can also be obtained in a similar way.
However, it needs to be clearly stated in \cite{Lo1} that the constants in the inequalities in Propositions \ref{230129407} - \ref{220527502} are independent of $\Omega$.
Therefore we provide proof of these propositions to verify the case $X=B$ and to investigate the dependence of the constants in each inequality.

\begin{prop}\label{230129407}
	Let $\xi\in C_c^{\infty}(\bR_+)$.
	For any $f\in X_{p,\theta}^{\gamma}(\Omega)$ and $\{f_n\}_{n\in\bN}\in X_{p}^{\gamma}(\Omega)$,
	\begin{equation}\label{220526250}
		\|S_{\xi}f\|_{l_p^{\theta/p}(X_p^{\gamma})}\leq N\|f\|_{X_{p,\theta}^{\gamma}(\Omega)}\quad\text{and}\quad \|R_{\xi}\{f_n\}\|_{X_{p,\theta}^{\gamma}(\Omega)}\leq N\|\{f_n\}\|_{l_p^{\theta/p}(X_p^{\gamma})}\,,	\end{equation}
	where $N=N(\cI,\zeta,\xi)$.
\end{prop}
\begin{proof}
	Take $K\in\bN$ such that $\mathrm{supp}(\xi)\subset \left[e^{-K},e^{K}\right]$ so that
	\begin{align}\label{230129311}
		|n-k|>K\quad\Longrightarrow \quad \zeta_{0,(n)} \xi_{(k)}\equiv 0\,.
	\end{align}
	
	Due to \eqref{220526211} and \eqref{230129311}, we have
	\begin{align}\label{230129339}
		\xi_{(n)}=\sum_{|k|\leq K}\xi_{(n)}\zeta_{0,(n+k)}\,.
	\end{align}
	From \eqref{230129339}, \eqref{220526148}, and \eqref{2205070028}, we have
	\begin{alignat*}{2}
		\|S_{\xi}f\|_{l^{\theta/p}_p(X_p^{\gamma})}^p\,&=&&\sum_{n\in\bZ}\|\big(\xi_{(n)}f\big)(e^n\,\cdot\,)\|_{X_p^{\gamma}}^p\\
		&\lesssim_{K,p}&&\sum_{|k|\leq K}\sum_{n\in\bZ}e^{n\theta}\|\xi_{(n)}(e^n\,\cdot\,)\big(\zeta_{0,(n+k)}f\big)(e^n\,\cdot\,)\|_{X_p^{\gamma}}^p\\
		&\lesssim_{\cI,\xi}&&\sum_{|k|\leq K}\sum_{n\in\bZ}e^{n\theta}\|\big(\zeta_{0,(n+k)}f\big)(e^n\,\cdot\,)\|_{X_p^{\gamma}}^p\\
		&\lesssim_{\cI,K}&&\sum_{|k|\leq K}\sum_{n\in\bZ}e^{n\theta}\|\big(\zeta_{0,(n+k)}f\big)(e^{n+k}\,\cdot\,)\|_{X_p^{\gamma}}^p\\
		&\leq &&e^{K|\theta|}\sum_{n\in\bZ}e^{n\theta}\|\big(\zeta_{0,(n)}f\big)(e^{n}\,\cdot\,)\|_{X_p^{\gamma}}^p\,.
	\end{alignat*}
	Therefore the first inequality in \eqref{220526250} is proved.
	
	Due to \eqref{230129339}, \eqref{220526148}, and \eqref{2205070028}, we have
	\begin{alignat*}{2}
		\|R_{\xi}\{f_n\}\|_{X_{p,\theta}^{\gamma}(\Omega)}&=&&\sum_{n\in\bZ}\Big\|\zeta_{0,(n)}\sum_{k\in\bZ}\xi_{(k)}(e^n\,\cdot\,)f_k(e^{n-k}\,\cdot\,)\Big\|_{X_p^{\gamma}(\Omega)}^p\\
		&\lesssim_{K,p}&&\sum_{|k|\leq K}\sum_{n\in\bZ}\Big\|\zeta_{0,(n)}(e^n\,\cdot\,)\xi_{(n+k)}(e^n\,\cdot\,)f_{n+k}(e^n\,\cdot\,)\Big\|_{X_p^{\gamma}(\Omega)}^p\\
		&\lesssim_{\cI,\zeta,\xi,K}&&\sum_{n\in\bZ}\|f_{n+k}(e^n\,\cdot\,)\|_{X_{p}^{\gamma}}^p\\
		&\lesssim_{\cI,K}&& \|\{f_n\}\|_{l_p^{\theta/p}(X_p^{\gamma})}^p\,.
	\end{alignat*}
	Therefore the second inequality in \eqref{220526250} is proved.
\end{proof}

\begin{prop}[Properties of weighted Sobolev/Besov spaces - I]\label{220527502111}
	\,\,
	
	\begin{enumerate}
		\item $X_{p,\theta}^{\gamma}$ is a Banach space.
		
		\item $C_c^{\infty}(\Omega)$ is dense in $X_{p,\theta}^{\gamma}(\Omega)$.
		
		\item $X_{p,\theta}^{\gamma}(\Omega)$ is the dual of $X_{p',\theta'}^{-\gamma}(\Omega)$, where 
		$$
		\frac{1}{p}+\frac{1}{p'}=1\quad\text{and}\quad \frac{\theta}{p}+\frac{\theta'}{p'}=d\,.
		$$
		Furthermore, we have
		\begin{align}\label{220526214}
			\sup_{g\in C_c^{\infty}(\Omega),\,g\neq 0}\frac{\big|(f,g)\big|}{\| g\|_{X_{p',\theta'}^{-\gamma}(\Omega)}}\simeq_{\cI}\|f\|_{X_{p,\theta}^{\gamma}(\Omega)}\,.
		\end{align}
		In particular, $X_{p,\theta}^{\gamma}(\Omega)$ is reflexive.
		
		\item Let $p_i\in(1,\infty)$ and $\gamma_i$, $\theta_i\in\bR$ for $i=0,\,1$.
		For any $t\in (0,1)$,
		\begin{align*}
			&\big[ X^{\gamma_0}_{p_0,\theta_0}(\Omega), X^{\gamma_1}_{p_1,\theta_1}(\Omega)\big]_t\simeq_N X^{\gamma_t}_{p_t,\theta_t}(\Omega)
		\end{align*}
		where $N=N(d,p_i,\theta_i,\gamma_i,t;i=1,2)$.
		Here, $[Y_0,Y_1]_t$ is the complex interpolation space of $Y_0$ and $Y_1$ (see \cite[Section 1.9]{triebel} for the definition and properties of the complex interpolation spaces), and $p_t\in(1,\infty)$ and $\gamma_t,\,\theta_t\in\bR$ are constants satisfying
		\begin{align}\label{220526223}
			\frac{1}{p_t}=\frac{1-t}{p_0}+\frac{t}{p_1}\,\,,\,\,\,\,\gamma_t=(1-t)\gamma_0+t\gamma_1\,\,\,,\,\,\,\,\frac{\theta_t}{p_t}=(1-t)\frac{\theta_0}{p_0}+t\frac{\theta_1}{p_1}\,.
		\end{align}
		
		\item Let $p_i\in(1,\infty)$ and $\gamma_i$, $\theta_i\in\bR$ for $i=0,\,1$, with $\gamma_0\neq \gamma_1$.
		For any $t\in (0,1)$,
		\begin{align*}
			&\big( H^{\gamma_0}_{p_0,\theta_0}(\Omega), H^{\gamma_1}_{p_1,\theta_1}(\Omega)\big)_{t,p_t}\simeq_NB^{\gamma_t}_{p_t,\theta_t}(\Omega)\simeq_N \big(B^{\gamma_0}_{p_0,\theta_0}(\Omega), B^{\gamma_1}_{p_1,\theta_1}(\Omega)\big)_{t,p_t}
		\end{align*}
		where $N=N(d,p_i,\theta_i,\gamma_i,t;i=1,2)$.
		Here, $(Y_0,Y_1)_{t,p_t}$ is the real interpolation space of $Y_0$ and $Y_1$ (see \cite[Section 1.3]{triebel} for the definition and properties of the real interpolation spaces), and $p_t\in(1,\infty)$ and $\gamma_t,\,\theta_t\in\bR$ are constants satisfying \eqref{220526223}.
	\end{enumerate}
\end{prop}
\begin{proof}
	
	(1) We only need to prove that if $\{f^{(n)}\}_{n\in\bN}$ is a Cauchy sequence  in $X_{p,\theta}^{\gamma}(\Omega)$, then this sequence converges in $X_{p,\theta}^{\gamma}(\Omega)$.
	Due to \eqref{220526250}, $S_{\zeta_0}f^{(n)}$ is a Cauchy sequence  in $l_p^{\theta/p}(X_{p}^{\gamma})$, and therefore there exists $\lim\limits_{n\rightarrow \infty}S_{\zeta_0}f^{(n)}=:F$ in $l_p^{\theta/p}(X_p^{\gamma})$.
	Put $f=R_{\zeta_1}F\in X_{p,\theta}^{\gamma}(\Omega)$, so that 
	$$
	\|f-f^{(n)}\|_{X_{p,\theta}^{\gamma}(\Omega)}=\|R_{\zeta_1}\big(F-S_{\zeta_0}f^{(n)}\big)\|_{X_{p,\theta}^{\gamma}(\Omega)}\lesssim \|F-S_{\zeta_0}f^{(n)}\|_{l_p^{\theta/p}(X_{p}^{\gamma})}\rightarrow 0
	$$
	as $n\rightarrow \infty$.
	The proof is completed.

	(2) If $f\in C_c^{\infty}(\Omega)$, then $\|\big(\zeta_{0,(n)}f\big)(e^n\,\cdot\,)\|_{X_{p}^{\gamma}}=0$ for all but finitely many $n\in\bZ$. Therefore $C_c^{\infty}(\Omega)\subset X_{p,\theta}^{\gamma}(\Omega)$.
	To prove that $C_c^{\infty}(\Omega)$ is dense in $X_{p,\theta}^{\gamma}(\Omega)$, note that $C_c^{\infty}(\bR^d)$ is dense in $X_p^{\gamma}$.
	For any $f\in X_{p,\theta}^{\gamma}(\Omega)$ and $\epsilon>0$, since $S_{\zeta_0}{f}\in l_p^{\theta/p}(X_p^{\gamma})$, there exists $\{g_n\}_{n\in\bZ}\subset C_c^{\infty}(\bR^d)$ such that $g_n\equiv 0$ for all but finitely many $n$, and 
	$$
	\big\|S_{\zeta_0}f-\{g_n\}\big\|_{l_p^{\theta/p}(X_p^{\gamma})}<\epsilon\,.
	$$
	Since $g_n\equiv 0$ for all but finitely many $n$, $g:=R_{\zeta_1}\{g_n\}$ belongs to $C_c^{\infty}(\Omega)$.
	Due to
	$$
	f-g=R_{\zeta_1}\big(S_{\zeta_0}f-\{g_n\}\big)
	$$
	and \eqref{220526250}, we obtain
	$$
	\|f-g\|_{X_{p,\theta}^{\gamma}(\Omega)}\leq N\big\|S_{\zeta_0}f-\{g_n\}\big\|_{l_p^{\theta/p}(X_p^{\gamma})}\leq N\epsilon\,,
	$$
	where $N=N(\cI)$. Since $N$ is independent of $\epsilon$, the proof is completed.
	
	(3)	Observe that for any $g\in X_{p',\theta'}^{-\gamma}(\Omega)$ and $f\in C_c^{\infty}(\Omega)$,
	\begin{align}\label{220904441}
		\begin{split}
			|\langle g,f\rangle|\,&\leq \sum_{n\in\bZ}\big|\big\langle g,\zeta_{0,(n)} f\big\rangle\big|=\sum_{n\in\bZ}\big|\big\langle \zeta_{1,(n)} g,\zeta_{0,(n)} f\big\rangle\big|\\
			&= \sum_{n\in\bZ}\big|\big\langle(e^{nd}S_{\zeta_1}g)_n,(S_{\zeta_0}f)_n\big\rangle\big|\\
			&\lesssim_{\cI} \|S_{\zeta_1}g\|_{l^{-\theta/p+d}_{p'}(X_{p'}^{-\gamma})}\|S_{\zeta_0}f\|_{l^{\theta/p}_p(X_p^{\gamma})}\\
			&\lesssim_{\cI}\|g\|_{X_{p',\theta'}^{-\gamma}(\Omega)}\|f\|_{X_{p,\theta}^{\gamma}(\Omega)}\,.
		\end{split}
	\end{align}
	For $g\in X_{p',\theta'}^{-\gamma}(\Omega)$, let $\mathbf{L}_g$ be the linear map from $C_c^{\infty}(\Omega)$ to $\bR$ defined by 
	$$
	\mathbf{L}_gf= \langle g,f\rangle\,.
	$$
	Then \eqref{220904441} and (1) of this proposition imply that $\mathbf{L}_g\in \big(X_{p,\theta}^{\gamma}(\Omega)\big)^{\ast}$ with
	\begin{align*}
		\|\mathbf{L}_g\|_{\left(X_{p,\theta}^{\gamma}(\Omega)\right)^{\ast}}=\sup_{f\in C_c^{\infty}(\Omega),\,f\neq 0}\frac{|\langle g,f\rangle |}{\| f\|_{X_{p,\theta}^{\gamma}(\Omega)}}\lesssim_{\cI} \|g\|_{X_{p',\theta'}^{-\gamma}(\Omega)}\,.
	\end{align*}
	In other words, $\mathbf{L}:g\mapsto \mathbf{L}_g$ is a bounded linear operator from $X_{p',\theta'}^{-\gamma}(\Omega)$ to $\big(X_{p,\theta}^{\gamma}(\Omega)\big)^{\ast}$.
	We claim that $\mathbf{L}$ is bijective and for any $g\in X_{p',\theta'}^{-\gamma}(\Omega)$,
	\begin{align}\label{220904854}
		\|g\|_{X_{p',\theta'}^{-\gamma}(\Omega)}\lesssim_{\cI}\|\mathbf{L}_g\|_{\left(X_{p,\theta}^{\gamma}(\Omega)\right)^{\ast}}\,.
	\end{align}
	
	- Injectivity : 
	If $\mathbf{L}_g\equiv 0$, then $\mathbf{L}_gf=\langle g,f\rangle =0$ for all $f\in C_c^{\infty}(\Omega)$.
	Therefore $g$ is the zero distribution.
	
	- Surjectivity : 
	For $\Lambda\in \big(X_{p,\theta}^{\gamma}(\Omega)\big)^\ast$, $\Lambda R_{\zeta_1}$ is in $\big(l^{\theta/p}_p(X_p^{\gamma})\big)^{\ast}\simeq l^{-\theta/p}_{p'}(X_{p'}^{-\gamma})$ (see, \textit{e.g.}, \cite[Theorem 2.11.2]{triebel2}).
	Therefore there exists $\{\widetilde{g}_n\}_{n\in\bZ}\in l^{-\theta/p}_{p'}(X_{p'}^{-\gamma})$ such that
	\begin{equation}\label{22090500341}
		\begin{aligned}
			\begin{gathered}
				\Lambda R_{\zeta_1}\{f_n\}=\sum_n\langle \widetilde{g}_n,f_n\rangle\quad\text{for all $\{f_n\}\in l_p^{\theta/p}(X_p^{\gamma})$}\,,\,\,\text{and} \\
				\big\|\{\widetilde{g}_n\}\big\|_{l^{-\theta/p}_{p'}(X_{p'}^{-\gamma})}\simeq_{\cI}\big\|\Lambda R_{\zeta_1}\big\|_{\left(l^{\theta/p}_p(X_p^{\gamma})\right)^{\ast}}\,.
			\end{gathered}
		\end{aligned}
	\end{equation}
	For any $f\in C_c^{\infty}(\Omega)$,
	\begin{align}\label{230330113}
		\begin{split}
			\Lambda f\,&=\Lambda \big(R_{\zeta_1}S_{\zeta_0}f\big)= \big(\Lambda R_{\zeta_1}\big)\big(S_{\zeta_0}f\big)\\
			&=\sum_{n\in\bN}\big\langle \widetilde{g}_n,(S_{\zeta_0}f)_n\big\rangle
			=\sum_{n\in\bN}e^{-nd}\big\langle  \widetilde{g}_n(e^{-n}\,\cdot\,)\zeta_{0,(n)},f\big\rangle \\
			&=\big\langle R_{\zeta_{0}}\{e^{-nd}\widetilde{g}_n\},f\big\rangle \,.
		\end{split}
	\end{align}
	Since
	$$
	\|\{e^{-nd}\widetilde{g}_n\}\|_{l_{p'}^{\theta'/p'}(X_{p'}^{-\gamma})}=\|\{\widetilde{g}_n\}\|_{l_p^{-\theta/p}(X_{p'}^{-\gamma})}<\infty\,,
	$$
	we have
	\begin{align}\label{230330114}
		\widetilde{g}:=R_{\zeta_0}\{e^{-nd}\widetilde{g}_n\}\in X_{p',\theta'}^{-\gamma}(\Omega)\,,
	\end{align}
	Consequently, \eqref{230330113} and \eqref{230330114} yield $\Lambda=\mathbf{L}_{\widetilde{g}}$, and teh surjectivity is proved.
	
	- \eqref{220904854} :
	Let $g\in X_{p',\theta'}^{-\gamma}(\Omega)$.
	For $\Lambda:=\mathbf{L}_g$, we recall $\{\widetilde{g}_n\}$ and  $\widetilde{g}:=R_{\zeta_0}\{e^{-nd}\widetilde{g}_n\}$ in \eqref{22090500341} -  \eqref{230330114}.
	Since $\mathbf{L}$ is bijective, $\widetilde{g}=g$.
	It is implied by \eqref{220526250}, \eqref{22090500341} -  \eqref{230330114} that
	\begin{align*}
		\|g\|_{X_{p',\theta'}^{-\gamma}(\Omega)}\lesssim\|\{\widetilde{g}_n\}\|_{l_{p'}^{-\theta/p}(X_{p'}^{-\gamma})}\simeq_{\cI}\big\|\Lambda R_{\zeta_1}\big\|_{\left(l^{\theta/p}_p(X_p^{\gamma})\right)^{\ast}}\lesssim \|\Lambda\|_{\left(X_{p,\theta}^{\gamma}(\Omega)\right)^{\ast}}\,.
	\end{align*}

	Although we have only proved
\begin{align}\label{230421100}
	\sup_{f\in C_c^{\infty}(\Omega),\,f\neq 0}\frac{|\langle g,f\rangle |}{\| f\|_{X_{p,\theta}^{\gamma}(\Omega)}}\lesssim_{\cI} \|g\|_{X_{p',\theta'}^{-\gamma}(\Omega)}
\end{align}
for $g\in X_{p',\theta'}^{-\gamma}(\Omega)$, the proofs of \eqref{220904441} and \eqref{220904854} imply that \eqref{230421100} holds for all $g\in\cD'(\Omega)$.

	The reflexivity of $X_{p,\theta}^{\gamma}(\Omega)$ follows from that $\big(X_{p,\theta}^{\gamma}(\Omega)\big)^{\ast\ast}\simeq \big(X_{p',\theta'}^{-\gamma}(\Omega)\big)^{\ast}\simeq X_{p,\theta}^{\gamma}(\Omega)$.

	(4) Although the formula for $\theta$ in \eqref{220526223} is different within \cite[Proposition 2.4]{Lo1}, the formula in \eqref{220526223} is sufficient for our purpose.
	Indeed, this proposition is implied by Proposition~\ref{230129407}, \cite[Theorem 1.2.4]{triebel}, and that 
				\begin{align*}
					\big[ l^{\theta_0/p_0}_{p_0}(X^{\gamma_0}_{p_0}), l^{\theta_1/p_1}_{p_1}(X^{\gamma_1}_{p_1})\big]_t\simeq_{N}  l^{\theta_t/p_t}_{p_t}(X^{\gamma_t}_{p_t})\,,
				\end{align*}
				where $N=N(d,p_i,\gamma_i,t;i=1,2)$ (see, \textit{e.g.}, \cite[Theorem 1.18.1, Theorem 2.4.2/1]{triebel}).
				
				(5) This proposition is implied by Proposition~\ref{230129407}, \cite[Theorem 1.2.4]{triebel}, and that 
				\begin{align*}
					\big( l^{\theta_0/p_0}_{p_0}(H^{\gamma_0}_{p_0}), l^{\theta_1/p_1}_{p_1}(H^{\gamma_1}_{p_1})\big)_{t,p_t}\simeq_N  l^{\theta_t/p_t}_{p_t}(B^{\gamma}_{p_t})
				\end{align*}
				(see, \textit{e.g.}, \cite[Theorem 1.18.1, Theorem 2.4.2/1.(a)]{triebel}).
			\end{proof}

			\begin{prop}[Properties of weighted Sobolev/Besov spaces - II]\label{220527502}
				\,\,
				
				\begin{enumerate}
					\item If $p\geq 2$, then
					$$
					\|f\|_{B_{p,\theta}^{\gamma}(\Omega)}\lesssim_{\cI}\|f\|_{H_{p,\theta}^{\gamma}(\Omega)}\,,
					$$ 
					and if $1<p\leq 2$, then 
					$$
					\|f\|_{H_{p,\theta}^{\gamma}(\Omega)}\lesssim_{\cI}\|f\|_{B_{p,\theta}^{\gamma}(\Omega)}\,.
					$$

					\item For any $s<\gamma$,
					$$
					\|f\|_{H_{p,\theta}^{s}(\Omega)}+\|f\|_{B_{p,\theta}^{s}(\Omega)}\lesssim_{\cI,s}\|f\|_{X_{p,\theta}^{\gamma}(\Omega)}\,.
					$$
					
					\item(Sobolev embedding) Let $p_i\in(1,\infty)$ and $\gamma_i$,  $\theta_i\in\bR$ for $i=0,\,1$, with that
					$$
					\gamma_0>\gamma_1\,\,,\quad \gamma_0-\frac{d}{p_0}=\gamma_1-\frac{d}{p_1}\,\,,\quad \frac{\theta_0}{p_0}=\frac{\theta_1}{p_1}\,.
					$$
					Then
					$$
					\|f\|_{X_{p_1,\theta_1}^{\gamma_1}(\Omega)}+\|f\|_{B_{p_1,\theta_1}^{\gamma_1}(\Omega)}\leq N\|f\|_{X_{p_0,\theta_0}^{\gamma_0}(\Omega)}
					$$
					where $N=N(d,p_i,\gamma_i,\theta_i;i=1,2)$.
					
					\item (Pointwise multiplier) For $k\in\bN_0$, let $a\in C^k_{loc}(\Omega)$ satisfy
					\begin{align*}
						|a|_{k}^{(0)}:=\sup_{\Omega}\sum_{|\alpha|\leq k}\rho^{|\alpha|}|D^{\alpha}a|<\infty\,\,.
					\end{align*}
					If $|\gamma|\leq k$ then
					\begin{align}\label{220608607}
						\|af\|_{H^{\gamma}_{p,\theta}(\Omega)}\lesssim_{\cI}|a|_{k}^{(0)}\|f\|_{H^{\gamma}_{p,\theta}(\Omega)}\,,
					\end{align}
					for all $f\in H^{\gamma}_{p,\theta}(\Omega)$, and if $|\gamma|<k$ then
					$$
					\|af\|_{B^{\gamma}_{p,\theta}(\Omega)}\lesssim_{\cI}|a|_{k}^{(0)}\|f\|_{B^{\gamma}_{p,\theta}(\Omega)}\,,
					$$
					for all $f\in B^{\gamma}_{p,\theta}(\Omega)$.
					
					\item For any $\eta\in C_c^{\infty}(\bR_+)$,
					\begin{align}\label{220526527}
						\sum_{n\in\bZ}e^{n\theta}\big\|\big(\eta_{(n)}f\big)(e^n\cdot)\big\|^p_{X^{\gamma}_{p}}\lesssim_{N}\|f\|_{X^{\gamma}_{p,\theta}(\Omega)}^p\,,
					\end{align}
					where $N=N(\cI,\eta)>0$.
					If $\eta$ additionally satisfies
					\begin{align}\label{220526529}
						\inf_{t\in\bR_+}\Big[\sum_{n\in\bZ}\eta(e^nt)\Big]>0\,,
					\end{align}
					then
					\begin{align}\label{220526530}
						\|f\|_{X^{\gamma}_{p,\theta}(\Omega)}^p\lesssim_{N} \sum_{n\in\bZ}e^{n\theta}\big\|\big(\eta_{(n)}f\big)(e^n\cdot)\big\|^p_{X^{\gamma}_{p}}\,,
					\end{align}
					where $N=N(\cI,\eta)>0$.
					
					\item For any $s\in\bR$,
					\begin{align}\label{220526558}
						\|\trho^{\,s} f\|_{X^{\gamma}_{p,\theta}(\Omega)}\simeq_{\cI,s}\|f\|_{X^{\gamma}_{p,\theta+sp}(\Omega)}\,.
					\end{align}
					
					\item For any $k\in\bN$,
					$$
					\|f\|_{X^{\gamma}_{p,\theta}(\Omega)}\simeq_{\cI,k}\sum_{i=0}^k\|D^if\|_{X^{\gamma-k}_{p,\theta+ip}(\Omega)}.
					$$

					\item For a fixed constant $A>1$, if $f$ is distribution on $\Omega$ and $f$ is supported on $\{x\in\Omega\,:\,A^{-1}\leq \rho(x)\leq A\}$, then $f\in X_{p,\theta}^{\gamma}(\Omega)$ if and only if $f\in X_{p}^{\gamma}$.
					Moreover, we have 
					\begin{align*}
						\|f\|_{X_p^{\gamma}(\bR^d)}\simeq_{\cI,A}\|f\|_{X_{p,\theta}^{\gamma}(\Omega)}\,.
					\end{align*}
					
					\item Let $t\in(0,1)$, and let $p_i\in(1,\infty),\,\theta_i,\,\gamma_i\in\bR$ ($i=1,\,2,\,t$) are constants satisfying \eqref{220526223}.
					Then
					$$
					\|f\|_{X_{p_t,\theta_t}^{\gamma_t}(\Omega)}\lesssim_N \|f\|_{X_{p_0,\theta_0}^{\gamma_0}(\Omega)}^{1-t}\|f\|_{X_{p_1,\theta_1}^{\gamma_1}(\Omega)}^{t}\,.
					$$
				\end{enumerate}
			\end{prop}
			
			\begin{proof}[Proof of Proposition~\ref{220527502}]\,\,
				
				(1) This follows from that $H_p^{\gamma}\subset B_p^{\gamma}$ if $p\geq 2$, and $B_p^{\gamma}\subset H_p^{\gamma}$ if $1<p\leq 2$ (see,\textit{e.g.}, \cite[Proposition 2.3.2/2.(iii)]{triebel2}).
				
				(2) This follows from that $X_p^{\gamma}\subset H_p^{s}\cap B_p^s$ (see,\textit{e.g.}, \cite[Proposition 2.3.2/2.(ii)]{triebel2}).
				
				(3) Note that $p_0<p_1$.
				Since $X_{p_0}^{\gamma_0}\subset X_{p_1}^{\gamma_1}\cap B_{p_1}^{\gamma_1}$ (see, \textit{e.g.}, \cite[Theorem 2.7.1]{triebel2}), we have
				\begin{align*}
					&\Big(\sum_{n\in\bZ}e
					^{n\theta_1}\|\big(f\zeta_{0,(n)}\big)(e^n\cdot)\|_{X_{p_1}^{\gamma_1}}^{p_1}\Big)^{1/p_1}+\Big(\sum_{n\in\bZ}e^{n\theta_1}\|\big(f\zeta_{0,(n)}\big)(e^n\cdot)\|_{B_{p_1}^{\gamma_1}}^{p_1}\Big)^{1/p_1}\\
					\leq \,&N\Big(\sum_{n\in\bZ}e^{n\theta_0}\|\big(f\zeta_{0,(n)}\big)(e^n\cdot)\|_{X_{p_0}^{\gamma_0}}^{p_1}\Big)^{1/p_1}\\\leq\,&N\Big(\sum_{n\in\bZ}e^{n\theta_0}\|\big(f\zeta_{0,(n)}\big)(e^n\cdot)\|_{X_{p_0}^{\gamma_0}}^{p_0}\Big)^{1/p_0}\,,
				\end{align*}
				where $N=N(d,p_i,\gamma_i;i=0,1)$.
				
				(4) 
				If either $k\geq |\gamma|$ and $X=H$ or $k>|\gamma|$ and $X=B$, then for any $f\in\cD'(\omega)$ and $a\in C^k(\bR^d)$,
				\begin{align}\label{2209051106}
					\|af\|_{X_p^{\gamma}}\lesssim_{d,p,\gamma}\|a\|_{C^k(\bR^d)}\|f\|_{X_p^{\gamma}}\,.
				\end{align}
				From direct calculation, one can observe that for any $k\in\bN_0$ and $a\in C_{\mathrm{loc}}^k(\Omega)$, 
				$$
				\|a(e^n\cdot\,)\zeta_{1,(n)}\|_{C^k}\leq N(d,k)|a|_k^{(0)}\,.
				$$
				By \eqref{2205070028} and \eqref{2209051106}, we have
				\begin{alignat*}{2}
					\|af\|_{X_{p,\theta}^{\gamma}(\Omega)}^p\,&=&&\sum_{n\in\bZ}e^{n\theta}\|a(e^n\cdot\,)\zeta_{1,(n)}\cdot f(e^n\cdot\,)\zeta_{0,(n)}\|_{X_p^{\gamma}}^p\\
					&\lesssim_{\cI,k}&&\sum_{n\in\bZ}e^{n\theta}|a|_k^{(0)}\|f(e^n\cdot\,)\zeta_{0,(n)}\|_{X_p^{\gamma}}^p=|a|_k^{(0)}\|f\|_{X_{p,\theta}^{\gamma}(\Omega)}^p\,.
				\end{alignat*}
				
				(5) \eqref{220526527} is directly implied by \eqref{220526250}.
				To prove the second assertion, we assume \eqref{220526529}.
				Put $\eta_0(t):=\sum_{n\in\bZ}\eta(e^nt)$, so that
				\begin{align}\label{220527242}
					\eta_0\in C^{\infty}(\bR_+)\,\,,\,\,\,\,\eta_0(e\,\cdot\,)=\eta_0(\,\cdot\,)\,\,,\,\, \text{and}\quad \sum_{n\in\bZ}\big(\eta/\eta_0\big)(e^n\,\cdot\,)= 1\quad\text{on}\quad \bR_+\,.
				\end{align}
				\eqref{220527242} implies that there exists $K\in\bN$ such that
				$$
				\sum_{|k|\leq K}\big(\eta/\eta_0\big)(e^k\cdot)=1\quad\text{on}\,\,\,[e^{-1},e]\,.
				$$
				Therefore we obtain that
				\begin{align}\label{230120246}
					\zeta_{0,(n)}=\zeta_{0,(n)}\sum_{|k|\leq K}(\eta/\eta_0)_{n-k}=\sum_{|k|\leq K}\eta_{(n-k)}\frac{\zeta_{0,(n)}}{\eta_{0,(n-k)}}=\sum_{|k|\leq K}\eta_{(n-k)}\big(\zeta_0/\eta_0\big)_{(n)}\,,
				\end{align}
				where the last inequality follows from the definition of $\eta_0$.
				By \eqref{2205070028}, \eqref{220526148}, and \eqref{230120246}, we have
				\begin{alignat*}{2}
					&&&\sum_{n\in\bZ}e^{n\theta}\|\big(\zeta_{0,(n)}f\big)(e^n\,\cdot\,)\|_{X_{p}^{\gamma}}^p\nonumber\\
					&\lesssim_N&&\sum_{|k|\leq K}\sum_{n\in\bZ}e^{n\theta}\|\big(\zeta_0/\eta_0\big)_{(n)}(e^n\,\cdot\,) \big(\eta_{(n-k)}f\big)(e^n\,\cdot\,)\|_{X_{p}^{\gamma}}^p\\
					&\lesssim_N&&\sum_{|k|\leq K}\sum_{n\in\bZ}e^{n\theta}\| \big(\eta_{(n-k)}f\big)(e^n\,\cdot\,)\|_{X_{p}^{\gamma}}^p\\
					&\lesssim_N&&\sum_{|k|\leq K}\sum_{n\in\bZ}e^{n\theta}\| \big(\eta_{(n-k)}f\big)(e^{n-k}\,\cdot\,)\|_{X_{p}^{\gamma}}^p\\
					&\lesssim_N&&\sum_{n\in\bZ}e^{n\theta}\| 	\big(\eta_{(n)}f\big)(e^{n}\,\cdot\,)\|_{X_{p}^{\gamma}}^p\,,
				\end{alignat*}
				where $N=N(d,p,\gamma,\theta,K)$.
				By \eqref{220608607} and \eqref{220527242}, the proof is completed.
				
				(6) Put $\eta(t)=t^s\zeta_0(t)$. Due to \eqref{220526211}, we have
				$$
				\inf_{t\in\bR_+}\sum_{n\in\bZ}\eta(e^nt)>0\,.
				$$
				Since
				$$
				\trho(x)^s\zeta_{0,(n)}(x)=e^{ns}\big(e^{-n}\trho(x)\big)^s\zeta_{0}(e^{-n}\trho(x))= e^{ns}\eta_{(n)}(x)\,,
				$$
				\eqref{220526558} is implied by (5) of this proposition.
				
				(7) We only need to prove for $k=1$.
				Note that
				\begin{align}\label{2209051118}
					\|\big(\zeta_{0,(n)}f\big)(e^n\cdot)\|_{X_p^{\gamma}}\simeq_{d,p,\gamma}\|\big(\zeta_{0,(n)}f\big)(e^n\cdot)\|_{X_p^{\gamma-1}}+e^n\|\big(D(\zeta_{0,(n)}f)\big)(e^n\cdot)\|_{X_p^{\gamma-1}}\,.
				\end{align}
				By direct calculation, we have
				\begin{align}\label{2207261257}
					\begin{split}
						D\big(\zeta_{0,(n)}f\big)=\zeta_{0,(n)}\big(Df\big)+e^{-n}(\zeta_0')_{(n)}(D\trho)f\,.
					\end{split}
				\end{align}
				By \eqref{220526527} and (4) of this proposition with \eqref{220719515}, we have
				\begin{align}\label{2209051119}
					\sum_{n\in\bZ}e^{n\theta}\|\big((\zeta_0')_{(n)}(D\trho)f\big)(e^n\cdot))\|_{X_p^{\gamma-1}}^p\lesssim_N\|(D\trho) f\|_{X_{p,\theta}^{\gamma-1}(\Omega)} \lesssim_N\|f\|_{X_{p,\theta}^{\gamma-1}(\Omega)}\,,
				\end{align}
				where $N=N(d,p,\theta,\gamma)$.
				By combining \eqref{2209051118} - \eqref{2209051119}, we obtain
				\begin{align*}
					\|f\|_{X_{p,\theta}^{\gamma}(\Omega)}^p&\simeq_\cI\sum_{n\in\bZ}e^{n\theta}\Big(\|\big(\zeta_{0,(n)}f\big)(e^n\cdot)\|_{X_p^{\gamma-1}}^p+e^n\|\big(D(\zeta_{0,(n)}f)\big)(e^n\cdot)\|_{X_p^{\gamma-1}}^p\Big)\\
					& \simeq_\cI \|f\|_{X_{p,\theta}^{\gamma-1}(\Omega)}^p+\|Df\|_{X_{p,\theta+p}^{\gamma-1}(\Omega)}\,.
				\end{align*}
				
				(8) Let $N_0$ be the constant in \eqref{220719515}, and take $B\in\bN$ such that
				$$
				\sum_{|n|\leq B}\zeta_0(e^n\cdot)\equiv 1\quad\text{on}\quad [(2N_0A)^{-1},2N_0A]\,,
				$$
				so that
				$$
				\sum_{|n|\leq B}\zeta_{0,(n)}\equiv 1\quad\text{on}\,\,E:=\{x\in\Omega\,:\,(2A)^{-1}\leq \rho(x)\leq 2A\}\,.
				$$
				Let $f$ be a distribution on $\Omega$ and supported on $\{x\in\Omega\,:\,A^{-1}\leq \rho(x)\leq A\}$.
				Then $f$ is also a distribution on $\bR^d$.
				Since $f\zeta_{0,(n)}\equiv 0$ for all $|n|>B$, it follows from \eqref{2205070028} that
				\begin{align*}
					&\|f\|_{X_{p,\theta}^{\gamma}(\Omega)}^p=\sum_{|n|\leq B}e^{n\theta}\|(\zeta_{0,(n)}f)(e^n\cdot)\|_{X_p^{\gamma}}^p\leq N\,\|f\|_{X_p^{\gamma}}^p\quad\text{and}\\
					&\|f\|_{X_p^{\gamma}}^p=\|\sum_{|n|\leq B}\big(\zeta_{0,(n)}f\big)\|_{X_p^{\gamma}}^p\leq N\sum_{|n|\leq B}e^{n\theta}\|(\zeta_{0,(n)}f)(e^n\cdot)\|_{X_p^{\gamma}}^p\leq N\|f\|_{X_{p,\theta}^{\gamma}}^p\,,
				\end{align*}
				where $N=N(d,p,\theta,\gamma,B)$.
				
				(9) Due to Proposition~\ref{220527502111}.(4) and the interpolation theory (see, \textit{e.g.}, \cite[Theorem 1.9.3/(f)]{triebel} and its proof), we obtain that for any $f\in X_{p_0,\theta_0}^{\gamma_0}(\Omega)\cap X_{p_1,\theta_1}^{\gamma_1}(\Omega)$, 
				\begin{align*}
					\|f\|_{X_{p_t}^{\gamma_t}(\Omega)}\leq N\|f\|_{\left[X_{p_0,\theta_0}^{\gamma_0}(\Omega),X_{p_1,\theta_1}^{\gamma_1}(\Omega)\right]_{t}}\leq \|f\|_{X_{p_0}^{\gamma_0}}^{1-t}\|f\|_{X_{p_1}^{\gamma_1}}^t\,,
				\end{align*}
				where $N=N(d,p_i,\theta_i,\gamma_i, t;i=1,\,2)$ and $\big[X_{p_0,\theta_0}^{\gamma_0}(\Omega),X_{p_1,\theta_1}^{\gamma_1}(\Omega)\big]_{t}$ is the complex interpolation space of  $X_{p_0,\theta_0}^{\gamma_0}(\Omega)$ and $X_{p_1,\theta_1}^{\gamma_1}(\Omega)$.
				Therefore the proof is completed.
			\end{proof}
			
			\begin{remark}
				As stated in \cite[Proposition 2.2.4]{Lo1}, Proposition~\ref{220527502}.(5) can be generalized as the following:
				\begin{itemize}
					\item 
					Let $\{\eta_n\}_{n\in\bZ}\subset C^{\infty}(\Omega)$ satisfies that
					\begin{enumerate}
						\item There exists a constant $\alpha>1$, $k_0\in\bN$ such that
						$$
						\mathrm{supp}(\eta_n)\subset \{x\in\Omega\,:\,\alpha^{n-k_0}< \rho(x)<\alpha^{n+k_0}\}\qquad \forall\quad n\in\bZ\,;
						$$
						\item There exist $\{N_m\}_{m\in\bN_0}\subset \bR_+$ such that for any $m\in\bN_0$, $\sup\limits_{\Omega}|D^m\eta_n|\leq N_m \alpha^{nm}$.
					\end{enumerate}
					Then for any $u\in X_{p,\theta}^{\gamma}(\Omega)$, \eqref{220526527} holds for $\{\eta_n\}$ instead of $\{\eta_{(n)}\}$ (where $N$ in \eqref{220526527} depends only on $d$, $p$, $\theta$, $\gamma$, $\alpha$, $k_0$, $\{N_m\}$). 
					Moreover, if there exists $\epsilon_0>0$ such that $\sum_{n\in\bZ}\eta_n\geq \epsilon_0$ on $\Omega$, then 
					\eqref{220526530} holds for $\{\eta_n\}$ (resp. $\alpha^n, \alpha^{n\theta}$) instead of $\{\eta_{(n)}\}$ (resp. $e^n$, $e^{n\theta}$), (where $N$ in \eqref{220526530} depends only on  $d$, $p$, $\theta$, $\gamma$, $\alpha$, $k_0$, $\{N_m\}$, $\epsilon_0$).
				\end{itemize}	
				\vspace{1mm}
				
				The proof of this statement is almost same with the proof of Proposition~\ref{220527502}.(5); note that there exists $K\in\bN$ depending only on $\alpha$ and $k_0$ such that for any $n_0\in\bZ$,
				\begin{alignat*}{2}
					&\#\{n\in\bZ\,:\,\left[e^{n-1},e^{n+1}\right]\cap \left[\alpha^{n_0-k_0},\alpha^{n_0+k_0}\right]\neq \emptyset\}&&\leq K\,;\\
					&\#\{n\in\bZ\,:\,\left[\alpha^{n-k_0},\alpha^{n+k_0}\right]\cap \left[e^{n_0-1},e^{n_0+1}\right]\neq \emptyset\}&&\leq K\,,
				\end{alignat*}
				where $\# A$ is the number of elements in a set $A$.
				The above statement implies that if $\eta\in C_c^{\infty}(\bR_+)$ satisfies \eqref{220526529}, then
				\begin{align}\label{230127236}
					\|f\|_{X_{p,\theta}^{\gamma}(\bR_+^d)}^p\simeq \sum\limits_{n\in\bZ}e^{n\theta}\|\eta(x_1)f(e^nx)\|_{X_{p}^{\gamma}}^p\qquad \forall \quad f\in\cD'(\bR^d_+)\,,
				\end{align}
				and if $\eta\in C_c^{\infty}(\bR_+)$ satisfies \eqref{220526529} for $2^n$ instead of $e^n$, then 
				\begin{align}\label{230127237}
					\|g\|_{X_{p,\theta}^{\gamma}(\bR^d\setminus\{0\})}^p\simeq \sum\limits_{n\in\bZ}2^{n\theta}\|\eta(|x|)g(2^nx)\|_{X_{p}^{\gamma}}^p\qquad \forall \quad g\in\cD'(\bR^d\setminus\{0\})\,.
				\end{align}
				In \cite{Krylov1999-3, Krylov1999-1, Krylov2001}, the space $H_{p,\theta}^{\gamma}(\bR^d_+)$  is defined by \eqref{230127236}.
				In addition, in \cite{LSJFA} the space $H_{p,\theta}^{\gamma}\big(\bR^d\setminus\{0\}\big)$ is defined by \eqref{230127237}.
			\end{remark}

			\vspace{2mm}

			\subsection{Auxiliary results}\label{2302131019}

			\begin{lemma}\label{220527700}
				Let $p\in(1,\infty)$, $\gamma,\,\theta\in\bR$.
				There exist linear maps
				$$
				\Lambda_i\,:\,X_{p,\theta}^{\gamma}(\Omega)\rightarrow \cD'(\Omega)\quad,\quad i=0,\,1,\,\ldots,\,d\,,
				$$
				such that for any $f\in X_{p,\theta}^{\gamma}(\Omega)$,
				$$
				f=\Lambda_0f+\sum_{i=1}^dD_i(\Lambda_if)
				$$
				and
				\begin{align}\label{220527656}
					\|\Lambda_0 f\|_{X_{p,\theta}^{\gamma+1}(\Omega)}+\sum_{i=1}^d\|\Lambda_i f\|_{X_{p,\theta-p}^{\gamma+1}(\Omega)}\leq N\|f\|_{X_{p,\theta}^{\gamma}(\Omega)}
				\end{align}
				where $N$ depends only on $d$, $p$, $\gamma$, $\theta$.
			\end{lemma}
			\begin{proof}
				Recall \eqref{220526211} and \eqref{230120312}.
				Put
				$$
				L_0=(1-\Delta)^{-1}\quad\text{and}\quad L_i=-D_i(1-\Delta)^{-1}\qquad\text{for}\,\,i=1,\,\ldots,\,d\,,
				$$
				which are linear operators on $X_{p}^{\gamma}$.
				It is implied by element properties of $X_p^{\gamma}$ that for any $g\in X_p^{\gamma}$,
				\begin{align}\label{2205101114}
					L_0g+\sum_{i=1}^dD_iL_ig=g\quad\text{and}\quad 
					\sum_{i=0}^d\|L_ig\|_{X_p^{\gamma+1}}\lesssim_{d,p,\gamma}\|g\|_{X_p^{\gamma}}\,.
				\end{align}
				Put
				\begin{align*}
					\Lambda_0f(x)\,		&=\sum_{n\in\bZ}\zeta_{1,(n)}(x)L_0\Big[\big(\zeta_{0,(n)}f\big)(e^n\cdot)\Big](e^{-n}x)\\
					&\quad -\sum_{i=1}^d\sum_{n\in\bZ}e^n\big(D_i\zeta_{1,(n)}\big)(x)L_i\Big[\big(\zeta_{0,(n)}f\big)(e^n\cdot)\Big](e^{-n}x)\\
					&=R_{\zeta_1}\big(L_0S_{\zeta_0}f\big)(x)-\sum_{i=1}^d\big(D_i\trho\big)(x)\cdot R_{\eta_0}\big(L_iS_{\zeta_0}f\big)(x)\,,\\
					\Lambda_if(x)\,&=\sum_{n\in\bZ}e^n\zeta_{1,(n)}(x)L_i\Big[\big(\zeta_{0,(n)}f\big)(e^n\cdot)\Big](e^{-n}x)\\
					&=\trho(x)\cdot R_{\eta_1}\big(L_iS_{\zeta_0}f\big)(x)\qquad\qquad \text{for}\quad i=1,\,\ldots,\,d\,,\\
				\end{align*}
				where $\eta_0(t):=\big(\zeta_1'\big)(t)$, $\eta_1(t):=t^{-1}\zeta_1(t)$, and 
				\begin{align*}
					\begin{gathered}
L_i\{f_n\}:=\{L_if_n\}\quad\text{for}\,\,\, \{f_n\}_{n\in\bZ}\in l_p^{\theta/p}(H_p^{\gamma})\,.
					\end{gathered}	
				\end{align*}
				Due to \eqref{2205101114}, we have 
				\begin{align*}
					\Lambda_0f+\sum_{i=1}^dD_i\Lambda_if=\,&\sum_{n\in\bZ}\left(\zeta_{1,(n)}(\,\cdot\,)\times \Big[\big(L_0+\sum_{i=1}^dD_iL_i\big)\big[(\zeta_{0,(n)}f)(e^n\cdot)\big]\Big](e^{-n}\,\cdot\,)\right)\\
					=\,&\sum_{n\in\bZ}\Big[\zeta_{1,(n)}\zeta_{0,(n)}f\Big]=f\,.
				\end{align*}
				Therefore we only need to prove \eqref{220527656}.
				Due to \eqref{220526250}, \eqref{2205101114}, and Proposition~\ref{220527502}.(5), we have
				\begin{align*}
					&\|\Lambda_0f\|_{H_{p,\theta}^{\gamma+1}(\Omega)}+\sum_{i=1}^d\|\Lambda_if\|_{H_{p,\theta-p}^{\gamma+1}}\\
					\lesssim_N&\|R_{\zeta_1}\big(L_0S_{\zeta_0}f\big)\|_{H_{p,\theta}^{\gamma+1}(\Omega)}+\sum_{i=1}^d\Big(\|R_{\eta_0}\big(L_iS_{\zeta_0}f\big)\|_{H_{p,\theta}^{\gamma+1}(\Omega)}+\|R_{\eta_1}\big(L_iS_{\zeta_0}f\big)\|_{H_{p,\theta}^{\gamma+1}(\Omega)}\Big)\\
					\lesssim_N&\sum_{i=0}^d\|L_iS_{\zeta_0}f\|_{l_p^{\theta/p}(H_p^{\gamma+1})}\,\,\lesssim_N\,\,\|S_{\zeta_0}f\|_{l_p^{\theta/p}(H_p^{\gamma})}\,\,\lesssim_N\,\,\|f\|_{H_{p,\theta}^{\gamma}(\Omega)}\,.
				\end{align*}
				Therefore the proof is completed.
			\end{proof}
			
			Recall that for a regular Harnack function $\Psi$ on $\Omega$,
			$$
			\Psi X_{p,\theta}^{\gamma}(\Omega):=\big\{f\,:\,\Psi^{-1}f\in X_{p,\theta}^{\gamma}(\Omega)\big\}\quad\text{and}\quad \|f\|_{\Psi X_{p,\theta}^{\gamma}(\Omega)}:=\|\Psi^{-1}f\|_{X_{p,\theta}^{\gamma}(\Omega)}\,.
			$$
			
			\begin{lemma}\label{22.04.11.3}
				Let $\eta\in C_c^{\infty}(\bR^d)$ satisfy
				$$
				\eta=1\,\,\,\,\text{on}\,\,B(0,1/2)\,\,,\quad \text{supp}(\eta)\subset B(0,1)\,\,,\quad \int_{\bR^d}\eta \dd x=1\,.
				$$
				For $i\in\bN$, let $N(i)\in\bN$ satisfy
				$$
				\text{supp}\Big(\sum_{|n|\leq i}\zeta_{0,(n)}\Big)\subset \big\{x\in\Omega\,:\,\big(N(i)/2\big)^{-1}\leq \rho(x)\leq N(i)/2\big\}\,.
				$$
				Let $\Lambda_i$, $\Lambda_{i,j}$, $\Lambda_{i,j,k}$ are linear functionals on $\cD'(\Omega)$ defined as
				\begin{align*}
					\Lambda_{i}f:=\Big(\sum_{|n|\leq i}\zeta_{0,(n)}\Big)f\,\,\,,\,\,\,\,\Lambda_{i,j}f(x)=\eta(j^{-1}x) \Lambda_if(x)\,\,\,,\,\,\,\,\Lambda_{i,j,k}f=\big(\Lambda_{i,j}f\big)^{(N(i)^{-1}k^{-1})}\,,
				\end{align*}
				where
				\begin{align*}
					h^{(\epsilon)}(x):=\int_{\bR^d}h(x-\epsilon y)\eta(y)dy:=\Big\langle h,\epsilon^{-d}\eta\big((x-\cdot)/\epsilon\big)\Big\rangle\,.
				\end{align*}
				Then for any regular Harnack function $\Psi$, the following hold:
				\begin{enumerate}
					\item For any $f\in \cD'(\Omega)$, $\Lambda_{i,j,k}f\in C_c^{\infty}(\Omega)$\,.
					\vspace{1mm}
					
					\item  For any $f\in \Psi X_{p,\theta}^{\gamma}(\Omega)$,
					\begin{align*}
						&\sup_i\|\Lambda_if\|_{\Psi X_{p,\theta}^{\gamma}(\Omega)}\leq N_1\|f\|_{\Psi	X_{p,\theta}^{\gamma}(\Omega)}\\
						&\,\sup_j\|\Lambda_{i,j}f\|_{\Psi X_{p,\theta}^{\gamma}(\Omega)}\leq N_2\|f\|_{\Psi X_{p,\theta}^{\gamma}(\Omega)}\\
						&\,\sup_k\|\Lambda_{i,j,k}f\|_{\Psi X_{p,\theta}^{\gamma}(\Omega)}\leq N_3\|f\|_{\Psi X_{p,\theta}^{\gamma}(\Omega)}
					\end{align*}
					where $N_1$, $N_2$, $N_3$ are constants independent of $f$.
					\vspace{1mm}
					
					\item For any $f\in \Psi X_{p,\theta}^{\gamma}(\Omega)$,
					\begin{align*}
						\lim_{k\rightarrow \infty}\Lambda_{i,j,k}f=\Lambda_{i,j}f\,\,\,,\,\,\,\lim_{j\rightarrow \infty}\Lambda_{i,j}f=\Lambda_{i}f\,\,\,,\,\,\,\lim_{i\rightarrow \infty}\Lambda_i f=f\quad\text{in}\quad \Psi H_{p,\theta}^{\gamma}(\Omega)\,.
					\end{align*}
				\end{enumerate}
				%
				%
				%
				%
				%
			\end{lemma}
			\begin{proof}
				(1) It follows directly from properties of distributions.
				
				(2), (3)
				Note the following elementary properties of $X_p^{\gamma}$: for any $F\in X_p^{\gamma}$,
				\begin{equation}\label{2206091022}
					\begin{aligned}
						\begin{gathered}
							\sup_{\epsilon>0}\|F^{(\epsilon)}\|_{X_p^{\gamma}}+\sup_{j\in\bN}\|\eta(j^{-1}\cdot)F\|_{X_p^{\gamma}}\leq N(d,p,\gamma,\eta)\|h\|_{X_p^{\gamma}}\,,\\
							\lim_{\epsilon\rightarrow 0}F^{(\epsilon)}=\lim_{j\rightarrow 	\infty}\eta(j^{-1}\cdot)F=F\qquad \text{in}\quad X_p^{\gamma}\,.
						\end{gathered}
					\end{aligned}
				\end{equation}
				
				\textbf{Step 1 : $\Lambda_i$}
				
				Let $f\in X_{p,\theta}^{\gamma}(\Omega)$.
				It is implied by \eqref{2205070028} and \eqref{220526148} that
				\begin{align}\label{2209121038}
					\|f-\Lambda_i f\|_{\Psi X_{p,\theta}^{\gamma}(\Omega)}^p\lesssim \sum_{|n|\geq i-1}\big\|\big(\Psi^{-1}f\zeta_{0,(n)}\big)(e^n\cdot)\big\|_{X_p^{\gamma}}^p\leq \|f\|_{\Psi X_{p,\theta}^{\gamma}(\Omega)}^p\,.
				\end{align}
				Therefore we have
				$$
				\sup_i\|\Lambda_if\|_{\Psi X_{p,\theta}^{\gamma}(\Omega)}\leq N\|f\|_{\Psi X_{p,\theta}^{\gamma}(\Omega)}\quad\text{and}\quad \lim_{i\rightarrow\infty}\|f-\Lambda_if\|_{\Psi X_{p,\theta}^{\gamma}(\Omega)}=0\,.
				$$
				where $N=N(d,p,\theta,\gamma)$.
				\vspace{1mm}
				
				\textbf{Step 2 : $\Lambda_{i,j}$}
				
				Note that $\Psi^{-1}\Lambda_if$ and $\Psi^{-1}\Lambda_{i,j}f$ are supported on
				$$
				\big\{x\in\Omega\,:\,N(i)^{-1}\leq \rho(x)\leq N(i)\big\}\,,
				$$
				It is implied by Proposition~\ref{220527502}.(8) and \eqref{2206091022} that
				\begin{align*}
					\begin{gathered}
						\big\|\Lambda_i f-\Lambda_{i,j}f\big\|_{\Psi X_{p,\theta}^{\gamma}(\Omega)}\simeq_{N_2}\big\|\big(1-\eta(j^{-1}\cdot)\big)\Psi^{-1}\Lambda_i f\big\|_{X_{p}^{\gamma}}\rightarrow 0\quad\text{as}\quad j\rightarrow\infty\,,\\
						\big\|\Lambda_{i,j}f\big\|_{\Psi X_{p,\theta}^{\gamma}(\Omega)}\simeq_{N_2}\big\|\Psi^{-1}\Lambda_{i,j}f\big\|_{X_{p}^{\gamma}}\lesssim_{N_2}\big\|\Psi^{-1}\Lambda_{i}f\big\|_{X_{p}^{\gamma}}\simeq_{N_2}\big\|f\big\|_{\Psi X_{p,\theta}^{\gamma}(\Omega)}\,,
					\end{gathered}
				\end{align*}
				where $N_2=N(d,p,\gamma,\theta,i,\eta)$.
				%
				\vspace{1mm}
				
				\textbf{Step 3 : $\Lambda_{i,j,k}$}
				
				Put
				$$
				K_{i,j}=\{x\in\Omega\,:\,N(i)^{-1}\leq \rho(x)\leq N(i)\,,\,\,|x|\leq 2j\}\,,
				$$
				and note that $K_{i,j}$ is compact subset of $\Omega$.
				Since $\Psi,\,\Psi^{-1}\in C^{\infty}(\Omega)$, Proposition~\ref{220527502}.(8) and \eqref{2205070028} yield that if $g\in\cD'(\Omega)$ is supported on $K_{i,j}$, then
				\begin{align}\label{230120545}
					\|g\|_{\Psi X_{p,\theta}^{\gamma}(\Omega)}:=\|\Psi^{-1} g\|_{X_{p,\theta}^{\gamma}(\Omega)}\simeq_N \|\Psi^{-1} g\|_{X_{p}^{\gamma}}\simeq_N\|g\|_{X_{p}^{\gamma}},
				\end{align}
				where $N=N(d,p,\gamma,\theta,i,j,\Psi)$.
				For any $k\in\bN$, $\Lambda_{i,j,k}f$ and $\Lambda_{i,j}$ are supported on $K_{i,j}$.
				Therefore it follows from \eqref{2206091022} and \eqref{230120545} that
				\begin{align*}
					\|\Lambda_{i,j,k}f\|_{\Psi X_{p,\theta}^{\gamma}(\Omega)}\,&\simeq_{N_3} \|\big(\Lambda_{i,j}f\big)^{(N(i)^{-1}k^{-1})}\|_{X_p^{\gamma}}\\
					&\lesssim_{N_3} \|\Lambda_{i,j}f\|_{X_p^{\gamma}}\simeq_{N_3}\|\Lambda_{i,j}f\|_{\Psi X_{p,\theta}^{\gamma}(\Omega)}\lesssim_{N_3}\|f\|_{\Psi X_{p,\theta}^{\gamma}(\Omega)}
				\end{align*}
				and
				$$
				\big\|\Lambda_{i,j}f-\Lambda_{i,j,k}f\big\|_{\Psi X_{p,\theta}}
				\simeq_{N_3}\big\|\Lambda_{i,j}f-\big(\Lambda_{i,j}f\big)^{(N(i)^{-1}k^{-1})}\big\|_{X_{p}^{\gamma}}\rightarrow 0\quad\text{as}\,\,k\rightarrow \infty\,,
				$$
				where $N_3=N(d,p,\theta,\gamma,\Psi,i,j,\eta)$.
			\end{proof}

			\vspace{2mm}

			\subsection{Equivalent norms}\label{0083}
			
			\begin{prop}\label{220528651}
				Let $\Phi$ be a regular Harnack function, $p\in (1,\infty)$, $k\in \bN_0$, and $\theta\in\bR$.
				There exists a constant $N=N(d,p,k,\theta,C_2(\Phi))$ such that
				\begin{align*}
					\|\Phi f\|_{H_{p,\theta}^{k}(\Omega)}^p\simeq_N \sum_{m=0}^k\int_{\Omega}|\rho^{m}D^mf|^p\Phi^p\rho^{\theta-d}\dd x\,.
				\end{align*}
			\end{prop}
			\begin{proof}
				Make use of Proposition~\ref{220527502}.(7) and Lemma~\ref{21.09.29.4}.(3) to obtain
				\begin{align}\label{2206091139}
					\|\Phi f\|_{\Psi H_{p,\theta}^{k}(\Omega)}\simeq \sum_{i=0}^k\|D^i (\Phi f)\|_{L_{p,\theta+ip}}\simeq \sum_{i=0}^k\|\Phi D^if\|_{L_{p,\theta+ip}}\,.
				\end{align}
				By \eqref{2206091139}, we only need to prove for $k=0$.
				Since $\|f\|_{H_p^0}^p=\|f\|_{L_p(\bR^d)}^p$, we obtain
				\begin{alignat*}{3}
					\|\Phi f\|_{L_{p,\theta}(\Omega)}^p&=&&\sum_{n\in\bZ}e^{n\theta}\int_{\Omega}|(\zeta_{0,(n)}\Phi f)(e^nx)|^p \dd x\\
					&=&&\int_{\Omega}|f|^p\Psi^p\Big(\sum_{n\in\bZ}e^{n(\theta-d)}|\zeta_{0,(n)}|^p\Big)\dd x\\
					&\simeq_{d,\theta}&&\int_{\Omega}|f|^p\Phi^p\rho^{\theta-d}\dd x\nonumber
				\end{alignat*}
				where the last inequality follows from \eqref{220526211}.
			\end{proof}
			
			\begin{prop}\label{220418435}
				Let $\Phi$ be a regular Harnack function, $p\in (1,\infty)$, $k\in\bN_0$, $\alpha\in (0,1)$, and $\theta\in\bR$.
				There exists a constant $N=N(d,p,k,\alpha,C_2(\Phi))$ such that 
				\begin{align}\label{220609106}
					\|\Phi f\|_{B_{p,\theta}^{k+\alpha}}^p \simeq_N\,&\sum_{i=0}^k\int_{\Omega}|\rho^k D^kf|^p\Phi^p\rho^{\theta-d}\dd x\\
					&+\int_{\Omega}\Big(\int_{|y-x|<\frac{\rho(x)}{2}}\frac{|D^{k}f(x)-D^{k}f(y)|^p}{|x-y|^{d+\alpha p}}dy\Big)\Phi(x)^p\rho(x)^{(k+\alpha)p +\theta-d}\dd x\nonumber
				\end{align}
			\end{prop}
			\begin{proof}
				\textbf{Step 1.} Our first claim is that
				\begin{align}\label{220609102}
					\|f\|_{B_{p,\theta}^{\alpha}(\Omega)}^p\simeq\,&\|f\|_{L_{p,\theta}(\Omega)}\\
					&+\int_{\Omega}\int_{|x-y|\leq \frac{\rho(x)}{2}}\frac{|(\trho^{(\theta-d)/p+\alpha}f)(x)-(\trho^{(\theta-d)/p+\alpha}f)(y)|^p}{|x-y|^{d+\alpha p}}\dd x\dd y\,.\nonumber
				\end{align}
				We note the following equivalent norm of Besov spaces:
				\begin{align}\label{2205181108}
					\|f\|_{B_{p}^{\alpha}}^p\simeq_{d,p,\alpha} \|f\|_{L_p}^p+\iint_{\bR^d\times \bR^d}\frac{|f(x)-f(y)|^p}{|x-y|^{d+\alpha p}}\dd x\dd y
				\end{align}
				(see,\textit{e.g.}, \cite[Theorem 2.5.7/(i)]{triebel2}).
				Recall that for $\xi\in C_c^{\infty}(\bR_+)$, we denote $\xi_{(n)}(x)=\xi(e^{-n}\trho(x))$.
				From \eqref{2205181108} we have
				\begin{alignat*}{2}
					\|f\|_{B_{p,\theta}^{\alpha}(\Omega)}^p&\simeq_N&& \sum_{k\in\bZ}e^{n\theta}\big\|\big(\zeta_{0,(n)}f\big)(e^n\,\cdot\,)\big\|_p^p\\
					&\quad&&+\sum_{k\in\bZ}e^{n\theta}\iint_{\bR^d\times \bR^d}\frac{\big|\big(\zeta_{0,(n)}f\big)(e^nx)-\big(\zeta_{0,(n)}f\big)(e^ny)\big|^p}{|x-y|^{d+\alpha p}}\dd x\dd y\\
					&=:&& I_1+I_2\,.
				\end{alignat*}
				Proposition~\ref{220528651} implies 
				$$
				I_1\simeq_{d,p,\theta} \|f\|_{L_{p,\theta}}^p\,.
				$$
				Change of variables implies
				\begin{alignat*}{3}
					I_2\,&=&&\sum_{k\in\bZ}\iint_{\bR^d\times \bR^d}\frac{|\zeta_{0,(n)}(x)f(x)-\zeta_{0,(n)}(y)f(y)|^p}{|x-y|^{d+\alpha p}}e^{n(\theta-d+\alpha p)}\dd x\dd y\\
					&=&&\sum_{k\in\bZ}\iint_{\bR^d\times \bR^d}\frac{|\eta_{(n)}(x)F(x)-\eta_{(n)}(y)F(y)|^p}{|x-y|^{d+\alpha p}}\dd x\dd y\\
					&\lesssim_p&& \sum_{k\in\bZ}\iint_{\left\{|x-y|\geq \rho(x)/2\right\}}\frac{|\eta_{(n)}(x)F(x)|^p+|\eta_{(n)}(y)F(y)|^p}{|x-y|^{d+\alpha p}}\dd x\dd y\\
					&&& +\sum_{k\in\bZ}\iint_{\left\{|x-y|\leq \rho(x)/2\right\}}\frac{|\eta_{(n)}(x)-\eta_{(n)}(y)|^p}{|x-y|^{d+\alpha p}}|F(x)|^p\dd x\dd y\\
					&&&+\sum_{k\in\bZ}\iint_{\left\{|x-y|\leq \rho(x)/2\right\}}|\eta_{(n)}(y)|^p\frac{|F(x)-F(y)|^p}{|x-y|^{d+\alpha p}}\dd x\dd y\\
					&=:&&I_{2,1}+I_{2,2}+I_{2,3}\,,
				\end{alignat*}
				where 
				$$
				F=\trho^{\,(\theta-d)/p+\alpha}f\quad\text{and}\quad \eta(t)=t^{-(\theta-d)/p-\alpha}\zeta_0(t)\,.
				$$
				Observe that for any $t>0$,
				\begin{align}\label{2206091204}
					\sum_{n\in\bZ}|\eta(e^{-n} t)|^p\simeq_N 1 \quad\text{and}\quad\sum_{n\in\bZ}e^{-np}|\eta'(e^{-n} t)|^p\lesssim_N t^{-p}\,,
				\end{align}
				where $N=N(d,p,\theta,\alpha)$.
				It follows from \eqref{2206091204} that
				\begin{align}\label{2206091221}
					I_{2,1}\simeq_N \int_{\Omega}\int_{y:|x-y|\geq \frac{\rho(x)}{2}}\frac{|F(x)|^p+|F(y)|^p}{|x-y|^{d+\alpha p}}\dd y\dd x\simeq_N \int_{\Omega}|f(x)|^p \rho(x)^{\theta-d}\dd x\,,
				\end{align}
				where $N=N(d,p,\theta,\alpha)$, and the last inequality is implied by that
				$$
				|x-y|\geq \rho(x)/2 \quad \Longrightarrow\quad  |x-y|\geq \rho(y)/3\,.
				$$
				To estimate $I_{2,2}$, observe that for $x,\,y\in\Omega$ with $|x-y|<\rho(x)/2$, 
				\begin{align}\label{220609621}
					\begin{split}
						\sum_{n\in\bZ}|\eta_{(n)}(x)-\eta_{(n)}(y)|^p\lesssim_N& \sum_n|x-y|^pe^{-np}\Big(\int_0^1|\eta'(e^{-n}\trho(x_r))|\dd r\Big)^p\\
						\leq\,\,\,\,& |x-y|^p\int_0^1\sum_{n}e^{-np}|\eta'(e^{-n}\trho(x_r))|^p \dd r\\
						\lesssim_N& |x-y|^p\int_0^1\trho(x_r)^{-p}\dd r\,,
					\end{split}
				\end{align}
				where $x_r=(1-r)x+ry$ and $N=N(d,p,\theta,\alpha)$.
				Here, the first inequality follows from that $|\nabla \trho|$ is bounded on $\Omega$, and the last inequality follows from \eqref{2206091204}.
				Since $\rho(x_r)\geq \rho(x)/2$, we have
				$$
				\sum_n|\eta_{(n)}(x)-\eta_{(n)}(y)|^p\lesssim_ N\,|x-y|^p\rho(x)^{-p}\,,
				$$
				where $N=N(d,p,\theta,\alpha)$.
				Consequently, we obtain
				\begin{align}\label{2206091223}
					I_{2,2}\,&\lesssim \int_{\Omega}\int_{y:|x-y|\leq \frac{\rho(x)}{2}}\frac{|F(x)|^p\rho(x)^{-p}}{|x-y|^{d-(1-\alpha)p}}\dd y\dd x\lesssim_{d,\alpha,p} \int|f(x)|^p\rho(x)^{\theta-d}\dd x.
				\end{align}
				Due to \eqref{2206091221} - \eqref{2206091223} and that
				$$
				I_{2,3}\lesssim I_2+I_{2,2}\lesssim \|f\|_{B_{p,\theta}^{\alpha}}^p\,,
				$$
				we have
				$$
				\|f\|_{B_{p,\theta}^{\alpha}(\Omega)}^p\simeq \|f\|_{L_{p,\theta}(\Omega)}+I_{2,3}\,.
				$$
				By applying \eqref{2206091204} to $I_{2,3}$, \eqref{220609102} is proved.
				
				\textbf{Step 2.}
				Now, we prove \eqref{220609106} for $k=0$.
				Denote $F:=\trho^{(\theta-d)/p+\alpha}f$.
				Since $\Phi\cdot \trho^{(\theta-d)/p+\alpha}$ is a regular Harnack function, if $|x-y|<\rho(x)/2$, then
				\begin{align}\label{2206091243}
					\begin{split}
						&\Big|\big|\Phi(x)F(x)-\Phi(y)F(y)\big|-\Phi(x)\trho(x)^{(\theta-d)/p+\alpha}\big|f(x)-f(y)\big|\Big|\\
						\leq \,&\big|\Phi(x)\trho(x)^{(\theta-d)/p+\alpha}-\Phi(y)\trho(y)^{(\theta-d)/p+\alpha}\big||f(y)|\\
						\leq \,&N |x-y|\cdot\Phi(y)\rho^{-1}(y)|F(y)|
					\end{split}
				\end{align}
				where $N=N(d,C_2(\Phi))$.
				By combining \eqref{220609102} (for $\Psi F$ instead of $f$), \eqref{2206091243}, and that
				\begin{align*}
					\int_{\Omega}\int_{y:|x-y|<\rho(y)}\frac{\big(|x-y|\cdot\Phi(y)\rho^{-1}(y)|F(y)|\big)^p}{|x-y|^{d+\alpha p}}\dd y\dd x\lesssim\int_{\Omega}|f(y)|^p\Phi(y)^p\rho(y)^{\theta-d}\dd y\,,
				\end{align*}
				we obtain \eqref{220609106} for $k=0$.
				
				\textbf{Step 3}. Let $k\geq 1$.
				The argument for \eqref{2206091139} (see with Proposition~\ref{220527502}.(2)) also implies that
				$$
				\|\Phi f\|_{B_{p,\theta}^{k+\alpha}(\Omega)}\simeq \sum_{i=0}^{k-1}\|\Phi D^if\|_{B_{p,\theta+ip}^{\alpha}(\Omega)}+\|\Phi D^kf\|_{B_{p,\theta+kp}^{\alpha}(\Omega)}\,.
				$$
				By Propositions~\ref{220527502}.(2) and (7), we have
				\begin{align*}
					\sum_{i=0}^{k-1}\|\Phi D^if\|_{L_{p,\theta+ip}(\Omega)}\lesssim &\sum_{i=0}^{k-1}\|\Phi D^if\|_{B_{p,\theta+ip}^{\alpha}(\Omega)}\\
					\lesssim &\sum_{i=0}^{k-1}\|\Phi D^if\|_{H_{p,\theta+ip}^{1}(\Omega)}\simeq \sum_{i=0}^{k}\|\Phi D^if\|_{L_{p,\theta+ip}(\Omega)}\\
					\lesssim &\sum_{i=0}^{k-1}\|\Phi D^if\|_{L_{p,\theta+ip}(\Omega)}+\|\Psi D^k f\|_{B_{p,\theta+kp}^{\alpha}(\Omega)}\,.
				\end{align*}
				Therefore, we have
				\begin{align*}
					\|\Phi f\|_{B_{p,\theta}^{k+\alpha}(\Omega)}\simeq \sum_{i=0}^k\|\Phi D^if\|_{B_{p,\theta+ip}^{\alpha}(\Omega)} \simeq \sum_{i=0}^{k-1}\|\Phi D^if\|_{L_{p,\theta+ip}(\Omega)}+\|\Psi D^k f\|_{B_{p,\theta+kp}^{\alpha}(\Omega)}\,.
				\end{align*}
				By Proposition~\ref{220528651} and the result of Step 2 (\eqref{220609106} for $k=0$), the proof is completed.
			\end{proof}

\vspace{2mm}

\section*{Acknowledgment}
The author wishes to express deep appreciation to Prof. Kyeong-Hun Kim for valuable comments.
The author is also sincerely grateful to Dr. Jin Bong Lee for careful assistance in proofreading and editing this paper.

\providecommand{\bysame}{\leavevmode\hbox to3em{\hrulefill}\thinspace}
\providecommand{\MR}{\relax\ifhmode\unskip\space\fi MR }
\providecommand{\MRhref}[2]{%
  \href{http://www.ams.org/mathscinet-getitem?mr=#1}{#2}
}
\providecommand{\href}[2]{#2}


%
%
%
%
%
%


\end{document}